\pdfoutput=1
\documentclass[oneside]{amsart}
\usepackage{geometry}
\PassOptionsToPackage{final}{graphicx}
\PassOptionsToPackage{final,hidelinks,bookmarks=false}{hyperref}
\usepackage{hyperref}
\usepackage{nima}
\usepackage{bbm}
\usepackage{tikz-cd}
\usepackage{amsaddr}
\usepackage{enumitem}
\usepackage[normalem]{ulem}

\DeclareMathOperator{\Lk}{Lk}
\DeclareMathOperator{\interior}{int}
\DeclareMathOperator{\boundary}{bd}

\theoremstyle{plain}
\newtheorem*{fact*}{Fact}
\newtheorem{fact}[thmctr]{Fact}
\newcommand{\factlabel}[1]{\label{fact:#1}}
\newcommand{\factref}[1]{\ref{fact:#1}}
\newcommand{\Factref}[1]{Fact~\factref{#1}}

\theoremstyle{definition}
\newtheorem*{rmkdefn*}{Remark/Definition}
\newtheorem{rmkdefn}[thmctr]{Remark/Definition}
\newcommand{\rmkdefnlabel}[1]{\label{rmkdefn:#1}}
\newcommand{\rmkdefnref}[1]{\ref{rmkdefn:#1}}
\newcommand{\Rmkdefnref}[1]{Remark/Definition~\rmkdefnref{#1}}

\newtheorem*{obs*}{Observation}
\newtheorem{obs}[thmctr]{Observation}
\newcommand{\obslabel}[1]{\label{obs:#1}}
\newcommand{\obsref}[1]{\ref{obs:#1}}
\newcommand{\Obsref}[1]{Observation~\obsref{#1}}

\title{Trees of Graphs as Boundaries of Hyperbolic Groups}

\author{Nima Hoda}
\address{Instytut Matematyczny,
  Uniwersytet Wroc\l awski\\
  pl.\ Grun\-wal\-dzki 2/4,
  50--384 Wroc{\l}aw, Poland
  \\ \ \\ DMA, École normale supérieure \\ Université
  PSL, CNRS \\ 75005 Paris, France
  \\ \ \\ Deptartment of Mathematics, Cornell University\\
  Ithaca, NY 14853, USA}
\email{nima@nimahoda.net}
\thanks{The first named author was partially supported by the ERC
  grant GroIsRan and by an NSERC Postdoctoral Fellowship}

\author{Jacek {\'S}wi\c atkowski}
\address{Instytut Matematyczny, Uniwersytet Wroc\l awski\\
  pl.\ Grun\-wal\-dzki 2/4, 50--384 Wroc{\l}aw, Poland}
\email{swiatkow@math.uni.wroc.pl}
\thanks{The second author was partially supported by (Polish) Narodowe Centrum
Nauki, grant no UMO-2017/25/B/ST1/01335.}

\date{\today}



\begin{document}

\begin{abstract}
  We characterize those 1-ended word hyperbolic groups whose Gromov
  boundaries are homeomorphic to \emph{trees of graphs} (i.e.\ to
  inverse limits of graphs that have particularly simple finitary
  descriptions). These are groups with the simplest connected Gromov
  boundaries of topological dimension 1. The characterization is
  expressed in terms of algebraic properties of the Bowditch JSJ
  splitting of the corresponding groups (i.e.\ the canonical JSJ
  splitting over 2-ended subgroups).
\end{abstract}

\maketitle

\tableofcontents

\section{Introduction}
In this paper we address the following classical and widely open problems:
\begin{enumerate}
\item
to describe explicitly the class of topological spaces which appear as Gromov boundaries
of word hyperbolic groups, and

\item
to characterize those word hyperbolic groups whose Gromov boundary is homeomorphic
to a given topological space.
\end{enumerate}

 Concerning the above problem (1), the paper
\cite{Swiatkowski:dense_amalgam:2016} provides its reduction to the case when the boundary is
connected, i.e., the corresponding group is 1-ended.
Some further partial insight into problem (1) is provided in  \cite{Pawlik:boundaries_markov_compacta:2015},
by showing that Gromov boundaries of
hyperbolic groups
belong to the class of
topological spaces called \emph{Markov compacta}. These are,
roughly speaking, the spaces which can be realized as inverse limits
of sequences of polyhedra built recursively out of an initial finite
polyhedron according to some finite set of modification (or substitution) rules.  In dimension 1, the
easiest to describe reasonably rich class of Markov compacta is that
of so called \emph{regular trees of graphs}, as introduced in
\Secref{reg_tree} of this paper. The main result of this paper is the complete 
characterization of
those 1-ended hyperbolic groups $G$ whose Gromov boundaries 
$\partial G$ are trees of
graphs (see \Mainthmref{1} below). 

In the statement of our main result, we refer to the concept of a
\emph{regular tree of graphs} (as mentioned above), as well as to the
more general concept of a \emph{tree of graphs}, as introduced in
\Secref{tr_gr}. The statement refers also to some features of the
\emph{Bowditch JSJ splitting} of a group $G$, as introduced in
\cite{Bowditch:cut_points:1998}, and recalled in detail in
\Secref{reduced_JSJ}.  This splitting is a canonical splitting of a
word hyperbolic group $G$ over 2-ended subgroups.  It is defined when $G$
is not cocompact Fuchsian.  ($G$ is cocompact Fuchsian if it acts
properly and cocompactly on the hyperbolic plane, a condition
equivalent to $\partial G \cong S^1$.)  Here we mention just a few
details concerning the Bowditch JSJ splitting, which are necessary for
the statement of \Mainthmref{1}.

Technically, the Bowditch JSJ splitting of $G$ is the action of $G$  
on the associated Bass-Serre tree ${T}_G$, called 
\emph{the Bowditch JSJ tree for $G$}.
This tree is naturally bipartitioned into \emph{black} and \emph{white} vertices.
The black vertices are either \emph{rigid} or \emph{flexible}, and their
stabilizers are the essential factors of the splitting (which are also called
rigid and flexible, respectively). The white vertex stabilizers
are maximal 2-ended subgroups of $G$.
The $G$-action on ${T}_G$ preserves all the above mentioned types
of vertices.
We introduce the notion of a \emph{rigid cluster} in $T_G$, as follows.
We say that two rigid vertices in $T_G$ are \emph{adjacent}
if their combinatorial distance in $T_G$ is equal to 2.
A \emph{rigid cluster} is a subtree of $T_G$ spanned on 
any equivalence class of the equivalence relation
on the set of rigid vertices of $T_G$ induced by adjacency.
A \emph{rigid cluster factor} of $G$ is the stabilizing subgroup of any rigid cluster,
i.e., the subgroup consisting of all elements which map this cluster to itself.
In particular, any rigid factor of $G$ is contained in some rigid cluster factor,
but in general rigid cluster factors are larger than rigid factors, and they are in some sense
``combined'' from collections of rigid factors. See 
\Ssecref{reduced_JSJ},
and in particular \Rmkref{3.5}, for more detailed explanations,
and for the description of a related splitting that we call 
the \emph{reduced Bowditch JSJ splitting} of $G$.

Recall finally that a graph is \emph{2-connected} if it is finite,
connected, nonempty, not equal to a single vertex, and cutpoint-free.


\begin{mainthm}
  \mainthmlabel{1} Let $G$ be a 1-ended hyperbolic group that is not
  cocompact Fuchsian.  Then the following conditions are equivalent:
\begin{enumerate}
\item \itmlabel{bd_g_tog} 
  $\partial G$ is homeomorphic to a tree of graphs;
  
\item \itmlabel{bd_g_reg_twocon_tog} 
  $\partial G$ is homeomorphic to a regular tree of 2-connected graphs;
  
\item \itmlabel{g_rigid_cluster_factors_vf} 
each rigid cluster factor of $G$ is virtually free.
\end{enumerate}
\end{mainthm}

As a complement to \Mainthmref{1}, we present also in the paper the
description of the structure of a regular tree of graphs appearing in
condition \pitmref{bd_g_reg_twocon_tog} of this result, given
explicitly in terms of the features of the corresponding Bowditch JSJ
splitting of $G$ (see \Thmref{S.1}, and the description of the
corresponding graphical connecting system $\mathcal{R}_G$ given in
\Secref{S}).  In fact, as we show in
\Rmkref{splitting_generalization}, an explicit description of
$\partial G$ as a regular tree of graphs can be similarly obtained
from a splitting of $G$ coming from a much more general class of
splittings over $2$-ended subgroups with maximal $2$-ended and
virtually free vertex groups.

One can view \Mainthmref{1} as confirming that regular trees of graphs are indeed
the simplest topological spaces in the complexity hierarchy among connected 1-dimensional
Gromov boundaries of hyperbolic groups. \Mainthmref{1} shows also that the zoo of 
connected 1-dimensional Gromov boundaries is quite complicated, 
and there is a big gap in it
between trees of graphs and the boundaries of those groups which contain as
their rigid factors groups whose boundary is either the Sierpinski curve or the Menger
curve (see \Rmkref{15.6} for more detailed discussion
of this issue). The closer understanding of the species inside this gap seems to be
an interesting challenge for future investigations.

As a byproduct of the methods developed for the proof of
\Mainthmref{1}, we present also another result, \Mainthmref{2}
below. In the light of the comments in the previous paragraph, it can
be viewed as characterizing those 1-ended hyperbolic groups $G$ whose
boundaries $\partial G$ belong to some subclass in the class of trees
of graphs, which consists of the simplest spaces among them, namely
the so called \emph{trees of $\theta$-graphs}.  A
\emph{$\theta$-graph} is any connected graph having precisely 2
vertices, whose any edge has both these vertices as its endpoints
(obviously, $\theta$-graphs typically contain multiple edges).  A
$\theta$-graph is \emph{thick} if it has at least three edges.  A
\emph{tree of thick $\theta$-graphs} is a tree of graphs whose all
consituent graphs are thick $\theta$-graphs (compare \Secref{tr_gr}).

\begin{mainthm}
  \mainthmlabel{2} Let $G$ be a 1-ended hyperbolic group that is not
  cocompact Fuchsian. Then the following conditions are equivelent:
\begin{enumerate}
\item \itmlabel{bdg_is_tree_of_theta_graphs}
$\partial G$ is homeomorphic to a tree of thick $\theta$-graphs;
\item 
$\partial G$ is homeomorphic to a regular tree of thick $\theta$-graphs;
\item \itmlabel{g_no_rigid_factors}
$G$ has no rigid factor in its Bowditch JSJ splitting.
\end{enumerate}
\end{mainthm}

\subsection{Examples.}
We now illustrate our results with few examples (see Figures~1--4). The details concerning these examples,
and other examples of a different nature, will be discussed and presented in \Secref{examples}. 
In each of the examples below, the 1-ended hyperbolic group $G_i$ is the fundamental
group of the corresponding ``graph of surfaces'' $M_i$.
Recall also that the Bowditch JSJ splitting of a group $G$ (as well as its reduced
Bowditch JSJ splitting) can be expressed
by means of the related graph of groups decomposition
$\mathcal{G}$ of $G$, and this is 
the approach that we follow
below. In all four discussed examples $G_i$,
$i\in\{ 1, 2, 3, 4 \}$, the edge groups of
the corresponding graph of groups decompositions $\mathcal{G}_i$
coincide with the corresponding incident white vertex groups,
so we skip the indication of the form of these edge groups
in the figures.

\begin{figure}[ht]
    \centering
   
        \includegraphics[width=\textwidth]{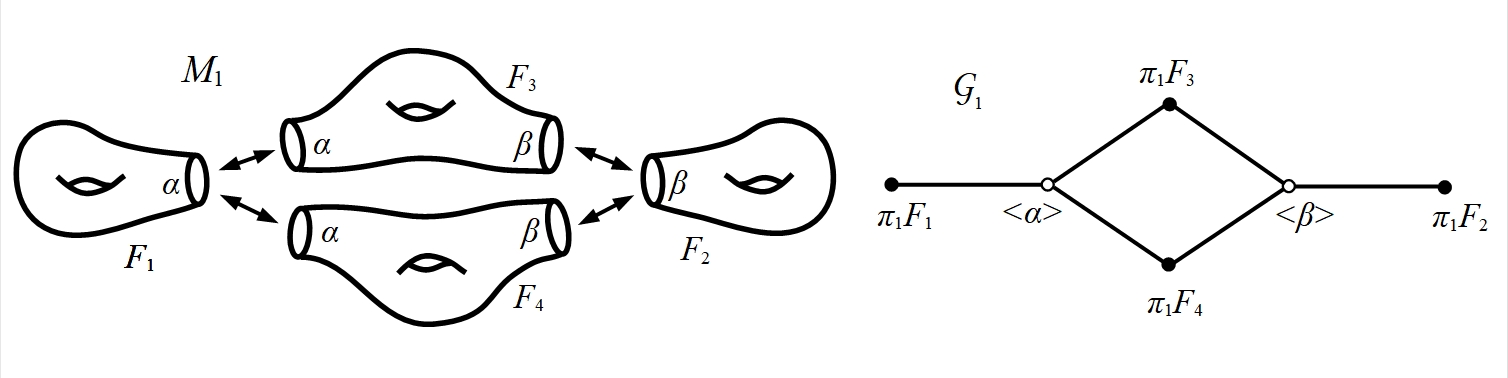}
        \caption{``Graph of surfaces'' $M_1$ and JSJ splitting
        $\mathcal{G}_1$ of its fundamental group $G_1=\pi_1M_1$.}
        \figlabel{fig1}
\end{figure}

Consider the group $G_1=\pi_1M_1$, presented in \Figref{fig1}. 
The Bowditch JSJ 
graph of groups decomposition of $G_1$, denoted $\mathcal{G}_1$, is presented in the right part
of this figure. Since all factors of this decomposition are easily seen to be flexible
(because they are the conjugates of the surface subgroups $\pi_1F_i$ equipped with
peripheral subgroups corresponding to the boundary curves), condition (3) of \Mainthmref{2}
holds for $G_1$, and hence the boundary $\partial G_1$ is homeomorphic
to some regular tree of thick $\theta$-graphs.

\Figref{fig2} shows an example in which the corresponding group $G_2=\pi_1M_2$
contains, up to conjugacy, precisely one rigid factor in its Bowditch JSJ splitting.
This factor corresponds to the subgroup $\pi_1H_0<\pi_1 M_2$, and it is obviously
isomorphic to the free group of rank 2. Moreover, it is not hard to observe that 
the rigid clusters in the Bowditch JSJ tree of $G_2$ coincide with the rigid vertices,
and hence the rigid cluster factors coincide with the rigid factors
(actually, in this case the reduced Bowditch JSJ splitting coincides with the ordinary
Bowditch JSJ splitting). It follows that the group $G_2$ satisfies condition \pitmref{g_rigid_cluster_factors_vf}
of \Mainthmref{1}, and hence its Gromov boundary $\partial G_2$
is homeomorphic to some regular tree of 2-connected graphs.
In \Secref{examples} we will give an explicit description of this regular tree of graphs,
in terms of the so called graphical connecting system $\mathcal{R}_{G_2}$
associated to $G_2$ in the way presented in \Secref{S}.

\begin{figure}[ht]
    \centering
   
        \includegraphics[width=\textwidth]{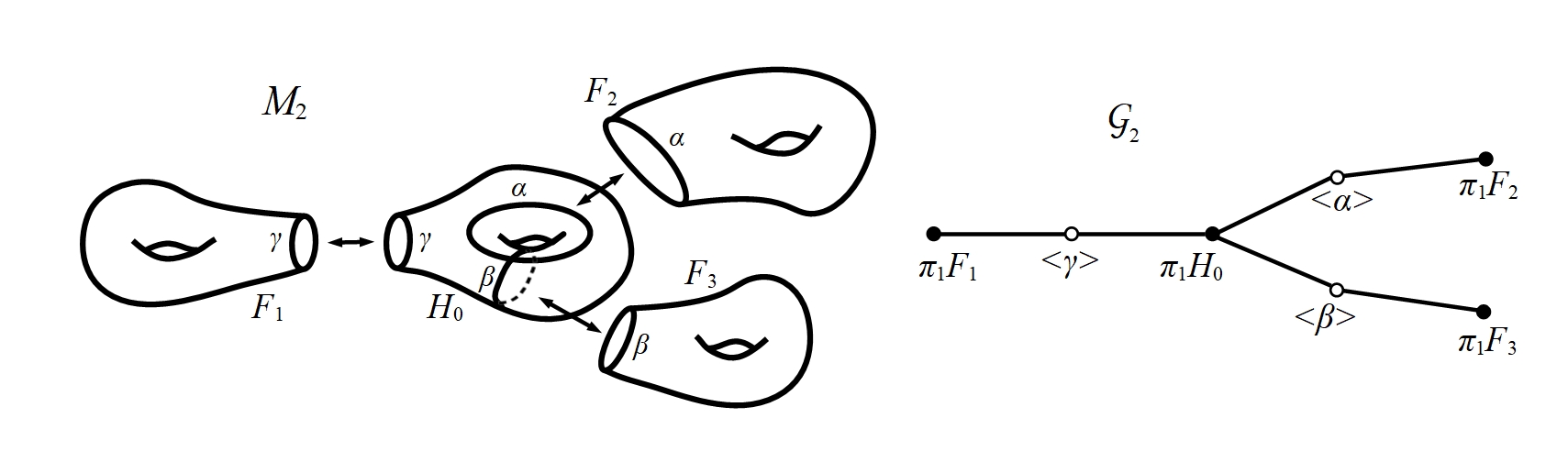}
        \caption{``Graph of surfaces'' $M_2$ and JSJ splitting
        $\mathcal{G}_2$ of its fundamental group $G_2=\pi_1M_2$.}
        \figlabel{fig2}
\end{figure}

The group $G_3=\pi_1M_3$ presented in \Figref{fig3} contains, up to conjugacy,
precisely two rigid factors in its Bowditch JSJ splitting, and they correspond to 
the subgroups $\pi_1H_i<\pi_1M_3$, $i=1,2$. It can be deduced from the
corresponding graph of groups decomposition $\mathcal{G}_3$ of $G_3$ that
the rigid clusters are in this case nontrivial, and that the rigid cluster factors of $G_3$
are the conjugates of the subgroup $\pi_1(H_1\cup H_2)$. (The corresponding
reduced Bowditch JSJ graph of groups decomposition of $G_3$ is presented
as $\mathcal{G}^r_3$ in the same figure.) Since obviously these rigid cluster factors
are free, condition \pitmref{g_rigid_cluster_factors_vf} of \Mainthmref{1} is again satisfied,
and thus the boundary $\partial G_3$ is also homeomorphic to some
regular tree of 2-connected graphs.

\begin{figure}[ht]
    \centering
   
        \includegraphics[width=\textwidth]{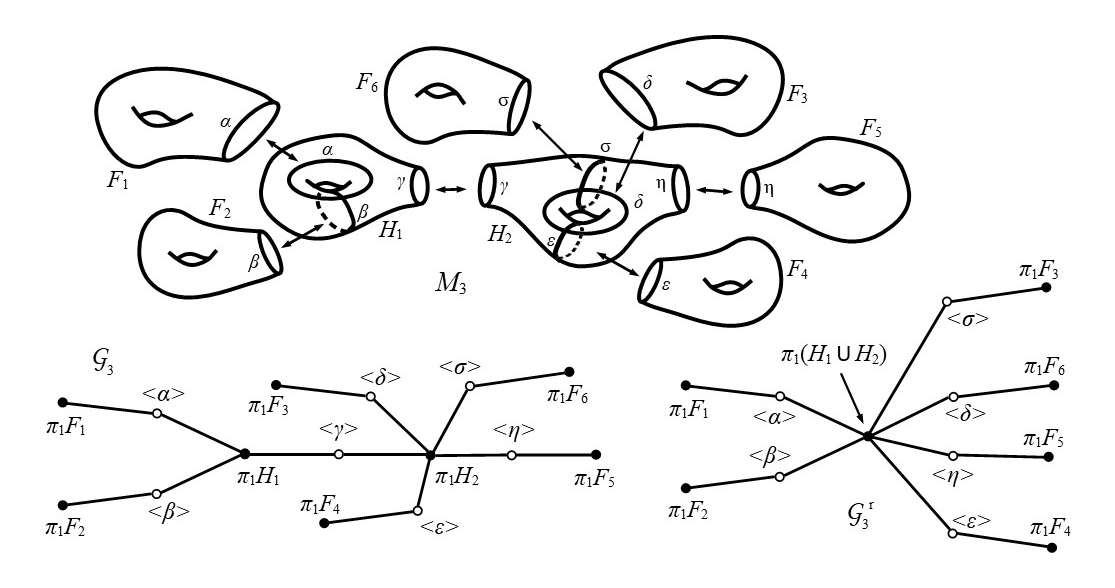}
        \caption{``Graph of surfaces'' $M_3$ and splittings
        $\mathcal{G}_3$ and $\mathcal{G}_3^r$ 
        of its fundamental group $G_3=\pi_1M_3$.}
        \figlabel{fig3}
\end{figure}

Our final example, presented in \Figref{fig4}, shows some 1-ended 
hyperbolic group
$G_4=\pi_1M_4$ whose boundary $\partial G_4$ is not a tree of graphs,
even though its all rigid factors are free. The rigid factors of $G_4$ correspond to the
subgroups $\pi_1H_i<\pi_1M_4$, for $i=1,2$, and the rigid cluster factors correspond
to the subgroup $\pi_1(H_1\cup H_2)$. Since that latter group is the fundamental group
of a closed surface, it is not virtually free, and hence condition \pitmref{g_rigid_cluster_factors_vf} of \Mainthmref{1}
is not fulfilled. As a consequence, the boundary $\partial G_4$ is not homeomorphic
to any tree of graphs.

\begin{figure}[ht]
    \centering
   
        \includegraphics[width=\textwidth]{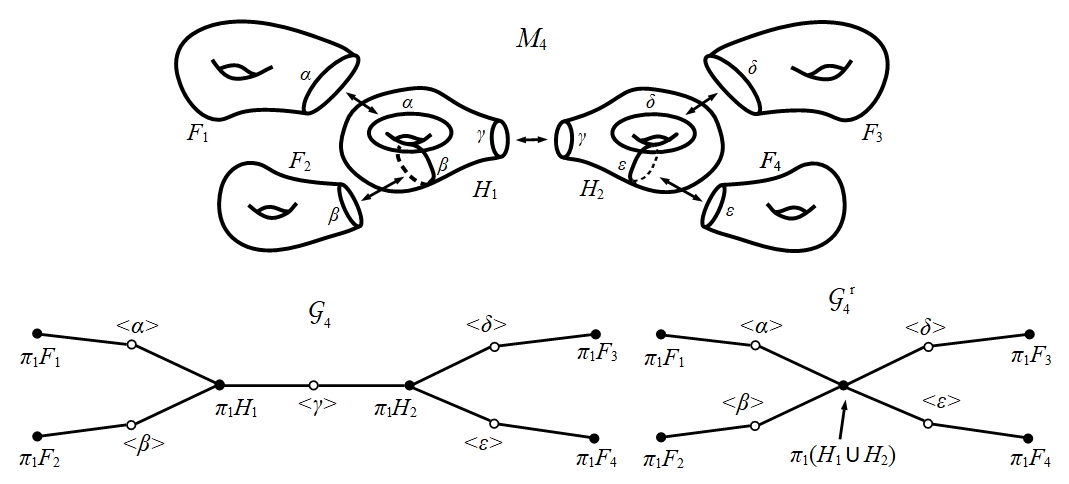}
        \caption{``Graph of surfaces'' $M_4$ and splittings
        $\mathcal{G}_4$ and $\mathcal{G}_4^r$ 
        of its fundamental group $G_4=\pi_1M_4$.}
        \figlabel{fig4}
\end{figure}

\subsection{Organization of the paper.}
The paper is organized as follows. The first part of the paper
(Sections~2--5) focuses on the proofs of the implications \pitmref{bd_g_tog}$\Rightarrow$\pitmref{g_rigid_cluster_factors_vf}
in \Mainthmref{1}, and (1)$\Rightarrow$(3) in \Mainthmref{2}. 
The second part of the paper (Sections~6--13) deals with the proofs of 
the implications \pitmref{g_rigid_cluster_factors_vf}$\Rightarrow$\pitmref{bd_g_reg_twocon_tog}
in \Mainthmref{1}, and (3)$\Rightarrow$(2) in \Mainthmref{2}. The remaining implications
(\pitmref{bd_g_reg_twocon_tog}$\Rightarrow$\pitmref{bd_g_tog}
in \Mainthmref{1}, and (2)$\Rightarrow$(1) in \Mainthmref{2}) necessary to get the full proofs
of Theorems \mainthmref{1} and \mainthmref{2} are obvious.

Our exposition in both parts of the paper consists of many preparations for the proofs,
including the complete descriptions and explanations of the terms appearing in the statements of our two main results.
More precisely, in the first part, the organization is as follows.
In \Secref{tree_sys} we recall the construction called
\emph{the limit of a tree system of metric compacta}.
It is used later, in \Secref{tr_gr}, to describe \emph{trees of graphs}
(and \emph{trees of $\theta$-graphs}), and it is also used throughout the
paper, as the main technical ingredient in the arguments.
In \Ssecref{bowditch_jsj_splitting} we recall the concept of the canonical (Bowditch) JSJ splitting
of a 1-ended hyperbolic group, and we describe related concepts of
the \emph{reduced} Bowditch JSJ splitting, and of \emph{rigid cluster
factors} of this reduced splitting. In \Secref{nec_conditions_tog} we study some configurations
of local cutpoints in trees of graphs (and trees of $\theta$-graphs),
and we deduce the proofs of the implications \pitmref{bd_g_tog}$\Rightarrow$\pitmref{g_rigid_cluster_factors_vf}
in \Mainthmref{1}, and (1)$\Rightarrow$(3) in \Mainthmref{2}. 

In the second part of the paper, the organization is as follows. 
In \Secref{reg_tree} we introduce \emph{regular trees of graphs}.
Sections~7--9 provide various technicalities and terminology which
allow to state in \Secref{S} (as \Thmref{S.1}) the result which describes
explicitly the form of the regular trees of graphs which appear
in our main results. In particular, we discuss in \Secref{E} some graphs
(the so called Whitehead graphs, and their extended variants) which are
naturally associated to the virtually free rigid factors (or rather rigid cluster factors)
of the JSJ splittings. 
Sections~10--12 contain some further technical preparations,
and in \Secref{P} we present the proofs of those parts of Theorems \mainthmref{1} and \mainthmref{2},
which identify the Gromov boundaries of the corresponding groups
as explicitly described regular trees of graphs.

In \Secref{examples}, we discuss the examples presented above in this
introduction, as well as some further examples, of a different sort,
namely some right-angled Coxeter groups.

In \Secref{final}, we make several final comments, including
explaining how the description of $\partial G$ as a regular tree of
graphs in terms of the Bowditch JSJ splitting can be generalized to a
much more general class of splittings.

\subsection{Acknowledgments.}
The second author expresses his thanks to Jason Manning for directing
his attention to Whitehead graphs as objects which could potentially
play a role in solving the problem studied in this paper (which indeed
turned out to be the case).

\section{Tree systems and their limits}
\seclabel{tree_sys}

In this preparatory section we present 
(or rather recall)
the construction called \emph{limit of a tree system of spaces}
(or, less formally, a \emph{tree of spaces}).  This construction is
a slight extension of that presented in Section~1 of \cite{Swiatkowski:trees_metric_compacta:2020}
(see also \cite{Pawlik_Zablocki:multidisks:2016}).
We will use this construction in the next sections to describe and
study \emph{trees of graphs}, the class of spaces on which we focus
in this paper.

\subsection{Some terminology and notation concerning trees}

We will mainly deal with trees $T$ where the vertex set $V_T$ is
endowed with a bipartition $V_T=V_T^b\sqcup V_T^w$ into
\emph{black} and \emph{white} vertices such that every edge of
$T$ has endpoints of different colours. We denote the set of all
(unoriented) edges of $T$ by $E_T$, and for each $e\in E_T$ we denote
by $b(e)$ and $w(e)$ the black and the white endpoint of $e$,
respectively. We will usually denote black vertices with letters $t,s$
and white ones with $u,w$.
We use also other letters, such as $v, z, x, y$, to denote vertices whose colour is
not essential.

We denote by $[v,z]$ the combinatorial (embedded) path in $T$
connecting a vertex $v$ to a vertex $z$.  (If $v,z$ are adjacent,
$[v,z]$ denotes also the edge with the endpoints $v,z$.)  A
\emph{ray} in $T$ is an infinite (embedded) combinatorial path
having an initial vertex. We usually denote rays by $\rho$, and we
denote by $\rho(0)$ the initial vertex and by $e_1(\rho)$ the initial
(unoriented) edge of a ray $\rho$.

For any $v\in V_T$, we denote by $N^T_v$, or simply by $N_v$, the set
of all vertices of $T$ adjacent to $v$. We denote by $d(v)=d_T(v)$ the
\emph{degree} of a vertex $v$ in $T$, i.e., the cardinality of the
set $N^T_v$.

We denote by $\partial T$ the set of \emph{ends} of $T$, i.e., the set
of equivalence classes of rays in $T$ with respect to the relation of
coincidence except possibly at some finite initial parts.  We denote
the end determined by a ray $\rho$ as $[\rho]$.

We consider subtrees of $T$ of some special form, that we call
b-subtrees.  A subtree $S$ of $T$ is a \emph{b-subtree} if for each
white vertex $u$ of $S$ the set $N^T_u$ is contained in the vertex set
of $S$. We apply to b-subtrees $S$ the earlier introduced notation
$V_S$, $V_S^b$, $V_S^w$, $E_S$, $N^S_v$, $d_S(v)$, etc.  For a
b-subtree $S$ of $T$ we define
\[ N_S=\{ u\in V_T: \text{$u\notin V_S$ and $u$ is adjacent to some
    $t\in V_S$} \}. \] Note that $N_S$ consists solely of white
vertices of $T$.  Note also that, in case when $S$ is equal to a
single black vertex $t$, this notation agrees with the earlier
introduced notation for the set $N_t$.

\subsection{Spaces with peripherals and tree systems}

Recall that a family of subsets of a compact metric space is
\emph{null} if for each $\epsilon>0$ all but finitely many of these
subsets have diameters less than $\epsilon$. 
Note 
that nullness of a family of subspaces does not depend on the choice
of a metric compatible with the topology of the corresponding space.

\begin{defn}
  \defnlabel{T.1} A \emph{space with peripherals} is a pair
  $(K,\Omega)$ such that $K$ is a compact metric space and $\Omega$ is
  a countable null family of pairwise disjoint closed subsets of $K$.  We call
  the sets from $\Omega$ \emph{the peripheral subspaces} of $K$, or
  simply \emph{the peripherals} of $K$.
\end{defn}

\begin{ex}
  \exlabel{T.1.5} Say that a graph $\Gamma$ is \emph{essential} if
  it is nonempty and contains no isolated vertex. (We allow that
  $\Gamma$ contains loop edges and multiple edges.)  Given a finite
  essential graph $\Gamma$, we describe the associated space with
  peripherals called \emph{the punctured graph $\Gamma$}, denoted
  $\Gamma^\circ$, as follows. Denote by $|\Gamma|$ the
  \emph{realization} of $\Gamma$, i.e., the underlying topological
  space of $\Gamma$, which is a compact metric space.  For each point
  $x\in|\Gamma|$ consider a \emph{normal neighbourhood} of $x$ in
  $\Gamma$, i.e., a connected closed neighbourhood $B(x)$ of $x$ whose
  intersection with the vertex set of $\Gamma$ either coincides with
  $\{x\}$ (when $x$ is a vertex of $\Gamma$), or is empty. Denote by
  $\interior B(x)$ and $\boundary B(x)$ the interior and the boundary
(in $|\Gamma|$)
  of a normal neighbourhood $B(x)$, respectively.  A \emph{standard
    dense family of normal neighbourhoods in} $\Gamma$ is any family
  of pairwise disjoint normal neighbourhoods in $\Gamma$ which
  contains normal neighbourhoods of all vertices of $\Gamma$, and
  whose union is dense in $|\Gamma|$. Note that each such family is
  automatically countable and null. Moreover, such family is unique up
  to a homeomorphism of the ambient space $|\Gamma|$ which is identity on the
  vertex set of $\Gamma$ and which maps all edges of $\Gamma$ to
  themselves.

  Let ${\mathcal N}_\Gamma$ be a standard dense family of normal
  neighbourhoods in $\Gamma$. Put
  \[ \Gamma^\circ:=|\Gamma|\setminus\bigcup\{
    \interior B:B\in{\mathcal N}_\Gamma \} \] and
  \[ \Omega=\{ \boundary B:B\in{\mathcal N}_\Gamma \}. \] Note that,
  by what was said above, $\Gamma^\circ$ is a compact metric space,
  and $\Omega$ is a countable and null family of closed subsets of $\Gamma^\circ$.
  \emph{The punctured graph $\Gamma$} is the space with peripherals
  $(\Gamma^\circ, \Omega)$.
\end{ex}

Note that the peripheral sets $\boundary B(x)$ of $\Gamma^\circ$
naturally correspond to vertices and edges of $\Gamma$, according to
whether $x$ is a vertex or an interior point of an edge. This
correspondence describes what we call a \emph{type} of a peripheral
set. The space with peripherals $(\Gamma^\circ, \Omega)$ is unique up
to a homeomorphism which respects types of peripheral sets.

If the underlying space $|\Gamma|$ of a graph $\Gamma$ is homeomorphic
to the circle $S^1$, it is not hard to realize that the space with
peripherals $(\Gamma^\circ,\Omega)$ as above does not depend, up to
homeomorphism, on the number of edges in $\Gamma$. This justifies the
notation $(S^1)^\circ$ for this space with peripherals, and the term
\emph{punctured circle}, with which we will be referring to it
later.

Punctured graphs will be used in the next section, to describe trees of graphs.

\begin{defn}
  \defnlabel{T.2}
  A \emph{tree system of spaces} is a tuple 
  $\Theta=(T,\{ K_t \},\{\Sigma_u\},\{\phi_e\})$ such that:
  \begin{enumerate}[label={(TS\arabic*)}]
  \item \itmlabel{(TS1)} $T$ is a countable tree (called the
    \emph{underlying tree} of $\Theta$), equipped with a proper
    colouring of its vertices in black and white (i.e. each edge has
    one of its vertices colored black and the other colored white),
    such that the degree $d(u)$ of each white vertex $u$ is finite;
  \item \itmlabel{(TS2)} to each black vertex $t\in V_T^b$ there is
    associated a nonempty compact metric space $K_t$, called a
    \emph{constituent space} of $\Theta$;
  \item \itmlabel{(TS3)} to each white vertex $u\in V_T^w$ there is
    associated a nonempty compact metric space $\Sigma_u$, called a
    \emph{peripheral space} of $\Theta$;
  \item \itmlabel{(TS4)} to each edge $e\in E_T$ there is associated a
    topological embedding
    \[\phi_e:\Sigma_{w(e)}\to K_{b(e)},\] called a \emph{connecting map} of
    $\Theta$;
  \item \itmlabel{(TS5)} for each black vertex $t\in V_T^b$ the pair
    \[(K_t,\{ \phi_{[t,u]}(\Sigma_u):u\in N_t \})\] is a space with
    peripherals, i.e., the family
    $\Omega_t:=\{ \phi_{[t,u]}(\Sigma_u):u\in N_t \}$ is null and
    consists of pairwise disjoint sets; pairs $(K_t,\Omega_t)$ are
    called the \emph{constituent spaces with peripherals} of
    $\Theta$.
  \end{enumerate}
\end{defn}

Note that, by condition \itmref{(TS5)}, a black vertex of the
underlying tree $T$ of a system $\Theta$ typically has infinite
(countable) degree.

\subsection{Limit of a tree system}\sseclabel{limit}

We now describe,
following the exposition in Section~1 of \cite{Swiatkowski:trees_metric_compacta:2020},   
an object called the \emph{limit} of a tree system
$\Theta=(T,\{ K_t \},\{\Sigma_u\},\{\phi_e\})$, denoted $\lim\Theta$,
starting with its description as a set. Denote by $\#\Theta$ the
quotient
\[(\bigsqcup_{t\in V_T^b}K_t)/\sim\]
of the disjoint union of the sets $K_t$ by the equivalence relation $\sim$
induced by the equivalences 
$\phi_{[u,t]}(x)\sim\phi_{[u,s]}(x)$ for any $u\in V_T^w$, any
$x\in\Sigma_u$, and any vertices $t,s\in N_u$. 
This set may be viewed as obtained from the family
$K_t:t\in V$ as a result of gluings provided by the maps 
\[
\phi_{[u,s]}\phi_{[u,t]}^{-1}:\phi_{[u,t]}(\Sigma_u)\to\phi_{[u,s]}(\Sigma_u).
\]
Observe that, by the fact that the peripheral subspaces in any $K_t$
are pairwise disjoint, the equivalence classes of the relation $\sim$
are easy to describe, and they are all finite (some of them have
cardinality 1, and others have cardinalities equal to the degrees of
the appropriate white vertices of $T$).  Observe also that any set
$K_t$ canonically injects in $\#\Theta$.  Define $\lim\Theta$ to be
the disjoint union $\#\Theta\sqcup \partial T$, where $\partial T$ is
the set of ends of $T$.

To put appropriate topology on the set $\lim\Theta$ we need more
terminology. Given a family $\mathcal A$ of subsets in a set $X$, we
say that $U\subset X$ is \emph{saturated} with respect to
$\mathcal A$ (shortly, \emph{$\mathcal A$-saturated}) if for any
$A\in{\mathcal A}$ we have either $A\subset U$ or $A\cap U=\emptyset$.

For any finite b-subtree $F$ of $T$, denote by 
\[
\Theta_F=(F,\{ K_t:t\in V_F^b \},
\{ \Sigma_u:u\in V_F^w \},
\{ \phi_e:e\in E_F \})
\]
the tree system $\Theta$ \emph{restricted} to $F$. The
set $K_F:=\#\Theta_F$, equipped with the natural quotient topology
(with which it is clearly a metrizable compact space), will be called
the \emph{partial union} of the system $\Theta$ related to $F$.
Observe that any partial union $K_F$ is canonically a subset in
$\#\Theta$, and thus also in $\lim\Theta$.

For any finite b-subtree $F$ and for any $u\in N_F$ denote by $t_u$
the (unique) vertex of $F$ adjacent to $u$ (which is necessarily a
black vertex). Put
${\Omega}_F=\{ \phi_{[u,t_u]}(\Sigma_u):u\in N_F \}$, and view the
elements of this family as subsets in the partial union $K_F$. Observe
that the family ${\Omega}_F$ consists of pairwise disjoint closed
subsets, and that this family is countable and null (so that $(K_F,\Omega_F)$ is a
space with peripherals).  Let $U\subset K_F$ be a subset which is
saturated with respect to $\Omega_F$.  Put
$N_U:=\{ u\in N_F:\phi_{[u,t_u]}(\Sigma_u)\subset U \}$ and
\[
  D_U:=\{ t\in V_T^b: \text{$[u,t]\cap F=\emptyset$ for some
    $u\in N_U$} \}
\]
(i.e. $D_U$ is the set of these vertices $t\in V_T^b\setminus V_F$
for which the shortest path in $T$ connecting $t$ to $F$
passes through a vertex $u\in N_U$).

Again, for any finite b-subtree $F$ of $T$ let $R_F$ be the set of all
rays $\rho$ in $T$ with $\rho(0)\in N_F$, and with $\rho$ disjoint
from $F$.  Note that the map $\rho\to[\rho]$ is then a bijection from
$R_F$ to the set $\partial T$ of all ends of $T$.  Given a subset
$U\subset K_F$ as above (i.e. saturated with respect to $\Omega_F$),
put
\[
  R_U:=\{ \rho\in R_F: \rho(0)\in N_U
  \},
\]
and put also $\partial_U:=\{ [\rho]\in \partial T:\rho\in R_U \}$.

Finally, for any subset $U\subset K_F$ as above, put
\[
{\mathcal G}(U):=U\cup(\bigcup_{t\in D_U}K_t)\cup \partial_U,
\]
where all summands in the above union are viewed as subsets of
$\lim\Theta$; clearly, ${\mathcal G}(U)$ is then also a subset in
$\lim\Theta$.  We consider the topology in the set $\lim\Theta$ given
by the basis $\mathcal B$ consisting of all sets ${\mathcal G}(U)$,
for all finite b-subtrees $F$ in $T$, and all open subsets
$U\subset K_F$ saturated with respect to $\Omega_F$.  We leave it as
an easy and instructive exercise to check that the family $\mathcal B$
is closed under finite intersections, and hence it satisfies the
axioms for a basis of a topology.  Moreover, we have the following
result, which is a slight extension of Proposition~1.C.1 of \cite{Swiatkowski:trees_metric_compacta:2020},
and which can be justified by the same proof as the latter
proposition.

\begin{prop}
  \proplabel{T.3} For any tree system $\Theta$ of metric compacta the
  limit $\lim\Theta$, with topology given by the above described basis
  $\mathcal B$, is a compact metrizable space.
\end{prop}

\subsection{Functoriality with respect to subtrees}

A tree system can be naturally restricted to any b-subtree of the
underlying tree.
Namely, let
$S\subset T$ be a b-subtree of
$T$. The \emph{restriction} of $\Theta$ to $S$ is the tree system
\[
\Theta_S=(S,\{ K_t:t\in V_S^b \},
\{ \Sigma_u:u\in V_S^w \},
\{ \phi_e:e\in E_S \}).
\]
Next observation follows by a direct argument (compare the proof
of Claim~1 inside the proof of Theorem~3.A.1 in \cite{Swiatkowski:trees_metric_compacta:2020}).

\begin{fact}
  \factlabel{T.4} The natural map
  $\#\Theta_S\cup \partial_S\to\#\Theta\cup \partial_T$ induced by the
  inclusion $S\to T$, and by the identities of the spaces
  $K_t:t\in V_S$, is injective. Moreover, this map viewed as a map
  from $\lim\Theta_S$ to $\lim\Theta$ is a topological embedding.
\end{fact}

In the case when $S=F$ is a finite b-subtree of $T$, it is not hard to
see that the limit $\lim\Theta_F$ coincides with the partial union
$K_F$.  It follows then from \Factref{T.4} that each of the partial
unions $K_F$ naturally embeds (as a topological space) in
$\lim\Theta$. In particular, each of the constituent spaces $K_t$
naturally embeds in $\lim\Theta$. A similar observation shows that
each of the peripheral spaces $\Sigma_u$ naturally embeds in
$\lim\Theta$. The following result is proved as Proposition~1.C.1 in
\cite{Swiatkowski:trees_metric_compacta:2020}.

\begin{fact}
  \factlabel{T.5} The families $K_t:t\in V^b_T$ and
  $\Sigma_u:u\in V^w_T$, viewed as families of subsets of the limit
  space $\lim\Theta$, are null. Moreover, the family
  $\Sigma_u:u\in V^w_T$ consists of pairwise disjoint subsets.
\end{fact}

\subsection{An isomorphism of tree systems}
\quad

Let 
$
\Theta=(T,\{ K_t \},\{\Sigma_u\},\{\phi_e\})$ 
and
$\Theta'=(T',\{ K'_t \},\{\Sigma'_u\},\{\phi'_e\})$ be two tree
systems of metric compacta. An \emph{isomorphism}
$F:\Theta\to\Theta'$ is a tuple 
$$
F=(\lambda, \{ f_t:t\in V^b_T \}, \{ g_u:u\in V^w_T \})
$$
such that:
\begin{enumerate}[label=(I\arabic*)]
\item \itmlabel{(I1)} $\lambda:T\to T'$ is an isomorphism of trees
  which respects colours of vertices;
\item \itmlabel{(I2)} for each $t\in V^b_T$ and each $u\in V^w_T$ the maps 
  $f_t:K_t\to K'_{\lambda(t)}$ and $g_u:\Sigma_u\to\Sigma'_{\lambda(u)}$ 
  are homeomorphisms;
\item \itmlabel{(I3)} for each $e\in E_T$ we have
  $\phi'_{\lambda(e)}\circ g_{w(e)}=f_{b(e)}\circ \phi_e$.
\end{enumerate}

An easy consequence of the definition of the limit (of a tree system
of metric compacta) is the following.

\begin{lem}
  \lemlabel{T.5.5} If $\Theta,\Theta'$ are isomorphic tree systems of
  metric compacta then their limits $\lim\Theta$ and $\lim\Theta'$ are
  homeomorphic.
  More precisely, an isomorphism $F:\Theta\to\Theta'$ induces
  (in the obvious way) a homeomorphism $\lim\Theta\to\lim\Theta'$.
\end{lem}

\subsection{An inverse system associated to a tree system}
\sseclabel{inverse_associated}

Let 
\[\Theta=(T,\{ K_t \},\{ \Sigma_u \},\{ \phi_e \})\]
be a tree
system.  For any finite b-subtree $F\subset T$, consider the family
$\Omega_F$ of peripheral subspaces of the partial union $K_F$, as
described in \Ssecref{limit}, and denote by $K_F^*=K_F/\Omega_F$ the quotient of
$K_F$
in which all the sets from $\Omega_F$ are shrunk to points.  More
precisely, let ${\mathcal D}_F$ be the decomposition of $K_F$
consisting of the sets of $\Omega_F$ and the singletons from the
complement $K_F\setminus\bigcup\Omega_F$, and let
$K_F^*=K_F/{\mathcal D}_F$. Since the family $\Omega_F$ is null, the
decomposition ${\mathcal D}_F$ is upper semicontinuous, and hence the
quotient $K_F^*$ is metrizable (see
\cite{Daverman:decompositions:1986}, Proposition~3 on page 14 and
Proposition~2 on page 13). Consequently, $K_F^*$ is a metric
compactum, and we call it the \emph{reduced partial union of $\Theta$}
related to $F$. We denote by $q_F:K_F\to K_F^*$ the quotient map
resulting from the above description of $K_F^*$.

For any pair $F_1\subset F_2$ of finite b-subtrees of $T$, define a
map $f_{F_1F_2}:K_{F_2}^*\to K_{F_1}^*$ as follows.  For each white
vertex $u\in N_{F_1}\cap V_{F_2}$ denote by ${\mathcal V}_u$ the set
of all black vertices $s\in V_{F_2}^b\setminus V_{F_1}$ such that the
shortest path in $T$ connecting $s$ with $F_1$ passes through $u$.
Then the family ${\mathcal V}_u:u\in N_{F_1}\cap V_{F_2}$ is a
partition of $V_{F_2}^b\setminus V_{F_1}$.  For each
$u\in N_{F_1}\cap V_{F_2}$, denote by $S_u$ the subtree of $T$ spanned
on the set ${\mathcal V}_u$, and note that it is actually a subtree of
$F_2$.  Viewing each $K_{S_u}=\lim\Theta_{S_u}$ canonically as a
subset in $K_{F_2}$, consider the corresponding subset
$q_{F_2}(K_{S_u})$ in $K_{F_2}^*$.  Observe that, by shrinking each of
the subsets $q_{F_2}(K_{S_u}):u\in N_{F_1}\cap V_{F_2}$ to a point we
get a quotient of $K_{F_2}^*$ which is canonically homeomorphic to
(and which we identify with) $K_{F_1}^*$. Take the corresponding
quotient map as $f_{F_1F_2}$, and observe that this map is continuous
and surjective.

Given any finite b-subtrees $F_1\subset F_2\subset F_3$ of $T$, it is
not hard to see that $f_{F_1F_2}\circ f_{F_2F_3}=f_{F_1F_3}$.
Consequently, the system
\[ {\mathcal S}_\Theta=(\{ K_F^*:F\in{\mathcal F}_T^b \},\{
  f_{FF'}:F\subset F' \})
\]
is an inverse system of metric compacta over the poset
${\mathcal F}_T^b$ of all finite b-subtrees of $T$. We call it the
\emph{standard inverse system associated to} $\Theta$.

Next result is a slight extension of Proposition~1.D.1 in
\cite{Swiatkowski:trees_metric_compacta:2020}, and its proof is the same as the proof of that
proposition (hence we omit it).

\begin{prop}
  \proplabel{T.6} Let $\Theta$ be a tree system, and let
  ${\mathcal S}_\Theta$ be the standard inverse system associated to
  $\Theta$. Then the limit $\lim\Theta$ is canonically homeomorphic to
  the inverse limit $\lim_{\leftarrow}{\mathcal S}_\Theta$.
\end{prop}

\subsection{Consolidation of a tree system}
\sseclabel{consolidation_of_tree_sys}

We describe an operation which turns one tree system of spaces into
another by merging the constituent spaces of the initial system, and
forming a new system out of bigger pieces (corresponding to a family
of pairwise disjoint b-subtrees in the underlying tree of the initial
system).  As we indicate below (see \Thmref{T.7}), this operation does
not affect the limit of a system.

Let $\Theta=(T,\{ K_t \},\{ \Sigma_u \},\{ \phi_e \})$ be a tree
system of metric compacta.  Let $\Pi$ be a \emph{partition of a
  tree $T$ into b-subtrees}, i.e., a family of b-subtrees $S\subset T$
such that the black vertex sets $V_S^b:S\in{\Pi}$ are pairwise
disjoint and cover all of $V_T^b$.  We allow that some of the subtrees
$S\in{\Pi}$ are just single black vertices of $T$.

We define a \emph{consolidation of $\Theta$ with respect to $\Pi$}
to be a tree system 
\[
\Theta^\Pi=
(T^{\Pi}, \{ K_S:S\in\Pi \}, \{ \Sigma^\Pi_u:u\in V^b_{T^\Pi} \}, 
\{ \phi^\Pi_e:e\in E_{T^\Pi} \})
\]
described as follows.  As a set $V^b_{T^{\Pi}}$ of black vertices of
$T^{{\Pi}}$ we take the family $\Pi$, and as a set $V^w_{T^{\Pi}}$ of
white vertices we take the set of those white vertices $u\in V^w_T$
which remain outside all of the subtrees $S\in\Pi$.  As the edge set
$E_{T^{\Pi}}$ we take the set
$\{ e\in E_T:e\notin\cup_{S\in{\Pi}}\,E_S \}$, so that each edge $e$
from this set is meant to connect the white vertex
$w(e)\in V^w_{T^\Pi}$ with the subtree $S\in\Pi$ (viewed as a black
vertex of $T^\Pi$) which contains the vertex $b(e)$ of
$T$. Intuitively speaking, the above described tree $T^\Pi$ is a
``dual tree'' of the partition $\Pi$.

To describe the constituent spaces $K_S$ of the tree system
$\Theta^\Pi$, for any subtree $S\in{\Pi}$ denote by $\Theta_S$ the
restriction of $\Theta$ to $S$, and put $K_S:=\lim\Theta_S$.  To
describe the peripheral spaces $\Sigma_u^\Pi$ and the connecting maps
$\phi_e^\Pi$, note that for any vertex $u\in V^w_{T^\Pi}$ (viewed as a
white vertex of $T$) and any edge $e$ with $w(e)=u$, the map
$\phi_e:\Sigma_u\to K_{b(e)}$ may be viewed as a map to the space
$K_S$ for this $S\in\Pi$ which contains the vertex $b(e)$.  Thus, for
$u$ and $e$ as above, we put $\Sigma_u^\Pi:=\Sigma_u$ and
$\phi_e^\Pi:=\phi_e$.  Now, viewing any subtree $S\in\Pi$ as a vertex
of $T^\Pi$, for any vertex $u\in N^{T^\Pi}_S$ denote by $t^S_u$ this
vertex of the subtree $S$ which is adjacent to $u$ in $T$, and note
that, as a consequence of \Factref{T.5}, the family
\[
\Omega^\Pi_S:=\{ \phi^\Pi_{[u,t^S_u]}(\Sigma^\Pi_u):u\in N^{T^\Pi}_S \}
\]
of subsets of the space $K_S$ is null.
It follows that the pair $(K_S,\Omega^\Pi_S)$ is a space with peripherals.
This justifies that the just described tuple 
\[
\Theta^{\Pi}=(T^{\Pi}, \{ K_S \}, \{ \Sigma^\Pi_u \},
\{ \phi^\Pi_e \})
\]
is a tree system of metric compacta.

The following result has been proved as Theorem~3.A.1 in 
\cite{Swiatkowski:trees_metric_compacta:2020}.

\begin{thm}
  \thmlabel{T.7} For any tree system $\Theta$ of metric compacta, and
  any consolidation $\Theta^{\Pi}$ of $\Theta$, the limits $\lim\Theta$ and
  $\lim\Theta^{\Pi}$ are canonically homeomorphic.
\end{thm}

\section{The reduced Bowditch JSJ splitting and rigid cluster factors}
\seclabel{reduced_JSJ}

In this section we introduce a new concept of the \emph{reduced Bowditch JSJ
splitting} of a 1-ended hyperbolic group, 
which is a variation of the classical Bowditch JSJ splitting.
In particular, we describe the factors of this splitting, especially the so called
\emph{rigid cluster factors} which appear in the statement of our main result, \Mainthmref{1}
of the introduction. We start with recalling basic features of the 
classical
Bowditch JSJ splitting.

\subsection{The Bowditch JSJ splitting of a 1-ended hyperbolic group}
\sseclabel{bowditch_jsj_splitting}


We refer the reader to \cite{Bowditch:cut_points:1998} or to Section~5
in \cite{Kim_Walsh:planar_boundaries:2022} for a detailed exposition
of the Bowditch JSJ splitting. Here we only mention briefly that this
splitting describes in some clever way the configuration of all local
cutpoints in the Gromov boundary $\partial G$ of a 1-ended hyperbolic
group $G$, and relates it to a certain canonical JSJ splitting of $G$
(so that the Bowditch JSJ tree appearing as part of the data of this
construction is canonically identified with the Bass-Serre tree of the
corresponding JSJ splitting of $G$).

More precisely, let $G$ be a 1-ended hyperbolic group which is not
cocompact Fuchsian (i.e. its Gromov boundary $\partial G$ 
is connected and not
homeomorphic to the circle $S^1$).  Recall that, under these
assumptions, $\partial G$ is a connected and locally connected compact
metrizable space without global cutpoints
\cite{Bowditch:cut_points:1998}.  Moreover, each local cutpoint $p$ of
$\partial G$ has finite degree (where degree of $p$ is, by definition,
the number of ends of a locally compact space
$\partial G\setminus\{ p \}$).  The set of all local cutpoints of
$\partial G$ is split into \emph{classes}, so that any two points from
the same class form a cut pair, any two points in the same class have
the same degree and that this degree is equal to the number of
connected components in the complement of the corresponding cut pair.
This splitting into classes is natural, in the sense that classes are
mapped to classes by homeomorphisms of $\partial G$.  Each class is
either infinite (in which case it consists of points of degree 2) or
it is \emph{of size} 2 (i.e. it consists of precisely 2 points).  Two
classes $\zeta_1$ and $\zeta_2$ are \emph{separated} if they lie in
distinct components of the complement $\bd G \setminus \{x,y\}$ of a
cut pair $x,y \in \eta$ from some other class $\eta$.  An object
called \emph{the Bowditch JSJ tree} for $G$, which we denote by
${T}_G$, has vertices $v$ (accompanied with some quasi-convex
subgroups $G(v)<G$) of the following four types:
  \begin{enumerate}[label=(v\arabic*)]
  \item \itmlabel{(v1)} vertices $v$ related to infinite classes of
    local cutpoints of degree 2, the so called \emph{necklaces}, for
    which the associated groups $G(v)$ are virtually free and coincide
    with the stabilizers of the corresponding necklaces under the
    action of $G$ on $\partial G$; the boundary
    $\partial G(v)\subset\partial G$ coincides then with the closure
    of the corresponding necklace, and it is homeomorphic to the
    Cantor set; such vertices $v$ have countable infinite degree in
    ${T}_G$;
  \item \itmlabel{(v2)} vertices $v$ related to classes of size 2, for
    which the associated groups $G(v)$ are equal to the
    $G$-stabilizers of the corresponding classes; these groups $G(v)$
    are maximal 2-ended subgroups of $G$ and their boundaries
    $\partial G(v)$ coincide with the corresponding classes; the degree
    of such a vertex $v$ in the tree ${T}_G$ is equal to the
    degree of any local cutpoint in the corresponding class, which is
    an integer $\ge2$;
  \item \itmlabel{(v3)} vertices $v$ related to infinite maximal
    families of pairwise not separated classes called \emph{stars},
    for which the associated groups $G(v)$ are equal to the
    $G$-stabilizers of unions of the corresponding stars; the degree of
    such a vertex $v$ in the tree ${T}_G$ is equal to the cardinality
    of the family of classes in the corresponding star, which is always
    infinite countable;
  \item \itmlabel{(v4)} vertices $v$ corresponding to pairs consisting
    of a necklace $C$ and a star $S$ containing this necklace; if we
    denote by $v_1$ and $v_3$ the vertices corresponding to the
    necklace $C$ and to the star $S$, respectively, then
    $G(v)=G({v_1})\cap G({v_3})$, which is a maximal 2-ended subgroup
    of $G$; the degree of such a vertex $v$ in the tree
    ${T}_G$ is equal to 2, and the boundary $\partial G(v)$
    is contained in $C$.
  \end{enumerate}
  The edges $e$ of ${T}_G$, and the associated groups $G(e)$, 
  have the following forms:

  \begin{enumerate}[label=(e\arabic*)]
  \item \itmlabel{(e1)} an edge $e$ connecting a pair of vertices
    $v_1,v_2$ of types (v1) and (v2), respectively; such an edge
    appears if and only if the class of size 2 corresponding to $v_2$
    is contained in the closure of the infinite class corresponding to
    $v_1$; moreover, in this case the class corresponding to $v_2$
    consists of local cutpoints of degrees $\ge3$;
  \item \itmlabel{(e2)} an edge $e$ connecting a pair of vertices
    $v_2,v_3$ of types (v2) and (v3), respectively; such an edge
    appears if and only if the class corresponding to $v_2$ belongs to
    the star corresponding to $v_3$;
  \item \itmlabel{(e3)} given any vertex $v$ of type (v4), let
    $v_1,v_3$ be the vertices related to the corresponding necklace
    and star; for each such $v$ we have in ${T}_G$ the edges
    $e=[v,v_1]$ and $e'=[v,v_3]$, and these are the only edges of
    ${T}_G$ adjacent to $v$.
  \end{enumerate}

In all cases \itmref{(e1)}--\itmref{(e3)} above the associated group
$G(e)$ is the intersection of the two vertex groups $G(v)$
corresponding to the endpoints of $e$. For each edge $e$ the group
$G(e)$ is 2-ended, and it is a finite index subgroup of the
corresponding group $G(v)$, where $v$ is the (unique) vertex of type
\itmref{(v2)} or \itmref{(v4)} in $e$.  We have a natural partition of
the vertices of ${T}_G$ into \emph{black} ones (vertices
of types \itmref{(v1)} and \itmref{(v3)}) and \emph{white} ones
(types \itmref{(v2)} and \itmref{(v4)}).  No two vertices of the same
colour are adjacent in ${T}_G$.  The degrees in
${T}_G$ of all white vertices are finite, and those of the
black vertices are countable infinite.

The tree ${T}_G$ is naturally acted upon by $G$. This action
preserves colours of the vertices, and the groups $G(v)$ and $G(e)$
coincide with the stabilizers of vertices $v$ and of edges $e$, with
respect to this action.

Groups $G(v)$ for vertices $v$ of type \itmref{(v1)} are called
\emph{flexible} or \emph{hanging} factors of the Bowditch JSJ
splitting of $G$, while those for vertices of type \itmref{(v3)} are
called \emph{rigid} factors. The corresponding vertices $v$ are
also called \emph{flexible} and \emph{rigid}, accordingly.  We
will adapt this terminology in the next sections of this paper.

The following observation is well known, and will be explained in
more details later in this paper
(see \Ssecref{F}, \Defnref{F.1.2} and the comments in the paragraph right before this definition; see also \Exref{X.1.1} together with \Defnref{X.1}).

\begin{fact}
  Let $v$ be a black vertex of ${T}_G$, i.e., either a rigid
  or a flexible vertex.  Put $K_v:=\partial G(v)$, and for any vertex
  $u$ of ${T}_G$ adjacent to $v$, $u\in N_v$, put
  $\Sigma_u:=\partial G(u)=\partial G([u,v])$.  Then the pair
  $(K_v,\{ \Sigma_u:u\in N_v \})$ is a space with peripherals.
  Moreover, if $v$ is flexible then this space with peripherals is
  (homeomorphic to) the punctured circle $(S^1)^\circ$.
\end{fact}

We now turn to the description of a certain tree system canonically associated
to the Bowditch JSJ splitting of $G$, and to the expression of the Gromov
boundary of $G$ as the limit of this tree system.
Denote by $V_{{T}_G}^b$ and $V_{{T}_G}^w$ the sets
of all black and white vertices of the tree ${T}_G$,
respectively.  Consider the following tree system $\Theta=\Theta(G)$:
\[
  \Theta(G)=({T}_G, \{K_t=\partial G(t):t\in V_{{\mathcal
      T}_G}^b\}, \{ \Sigma_u=\partial G(u):u\in V_{{T}_G}^w
  \}, \{ \phi_e:e\in E_{{T}_G} \}),
\]
where each $\phi_e:\partial G({w(e)})\to \partial G({b(e)})$ is the
natural inclusion.  We will call $\Theta(G)$ \emph{the tree system
  associated to the Bowditch JSJ splitting of $G$}, or concisely
\emph{the tree system associated to $G$}.  
The next result follows fairly directly by the arguments provided
by Alexandre Martin (see Corollary~9.19 in \cite{Martin:nonpos_cplx_of_groups:2014} or Theorem~2.25 in \cite{Cashen_Martin:qi_two_ended:2017}).
In the proof provided below we explain briefly how to apply these results
in our situation, by referring to the exposition in \cite{Cashen_Martin:qi_two_ended:2017}.

\begin{thm}
  \thmlabel{T.9} Let $G$ be a 1-ended hyperbolic group which is not
  cocompact Fuchsian, and let $\Theta(G)$ be the associated tree
  system.  Then the Gromov boundary $\partial G$ is
  homeomorphic to the limit of the tree system $\Theta(G)$, i.e.,
  \[\partial G\cong\lim\Theta(G).\]
\end{thm}

\begin{rmk}
  The homeomorphism $\partial G\cong\lim\Theta(G)$ in the above
  statement is natural in the sense that it is induced by the
  inclusions $\partial G(t)\subset\partial G$ and
  $\partial G(u)\subset\partial G$, where $t\in V_{{\mathcal T}_G}^b$,
  $u\in V_{{T}_G}^w$.
\end{rmk}

\begin{proof}
We sketch here an explanation of how \Thmref{T.9} follows from 
the result of Alexandre Martin mentioned right before the statement
of this theorem.

In Section~2.6.2 of \cite{Cashen_Martin:qi_two_ended:2017} 
Alexandre Martin and Christopher Cashen 
consider an object $X$ called 
\emph{a proper geodesic hyperbolic tree of spaces}, and the associated
object $\partial X$ (which may be viewed as the associated tree
of boundaries). 
It is not hard to see that any limit $\lim\Theta(G)$, viewed as a set,
can be identied with appropriately chosen Martin's set $\partial X$. 
(Namely, take as the underlying tree $T$ of $X$ the Bass-Serre
tree $T_G$, take as the vertex spaces $X_v$ of $X$ the corresponding
vertex groups $G(v)$, and as the edge spaces $X_e$ the
corresponding edge stabilizers $G(e)$.)
A less clear fact is that in case of $\lim\Theta(G)$ (viewed as special
case of Martin's $\partial X$) the topology on $\partial X$ described in Definitions 2.19
and 2.21 in \cite{Cashen_Martin:qi_two_ended:2017} coincides with the topology on $\lim\Theta(G)$ 
provided in \Ssecref{limit}
of the present paper.   We devote the remaining part of this proof to the 
justification of this coincidence of topologies, noting that then
\Thmref{T.9} follows fairly directly from Theorem~2.22 in \cite{Cashen_Martin:qi_two_ended:2017}.

Martin introduces his topology on $\partial X$ by means of describing
families of neighbourhoods of points. (These neighbourhoods are not necessarily open in the resulting topology, but they contain
open subneighbourhoods of the corresponding points.) 
In the case where $\partial X=\lim\Theta(G)$, we need to show that
the identity map is a homeomorphism from $\partial X$ equipped
with Martin's topology to $\partial X$ equipped with our topology
(as described in \Ssecref{limit}). However, since both these topological
spaces are compact metrizable, it is sufficient to show that
the identity map is continuous.

Suppose we are given a topological space $Z$ whose topology is given
by a family $N(z):z\in Z$ of neighbourhood systems for points $z$, 
and a topological space $Z'$ whose topology is induced 
by a basis $\mathcal{B}$
of open sets.
Recall that then a map $f:Z\to Z'$ is continuous iff for any 
open set $E\in{\mathcal{B}}$
and any $z\in f^{-1}(E)$ there is a neighbourhood $N\in N(z)$ 
of point $z$ such that 
$N\subset f^{-1}(E)$. In particular, if $Z=Z'$ (as sets) then
the identity map $\hbox{id}_Z:Z\to Z'$ is continuous (with respect to the 
topologies mentioned above) iff for any $E\in{\mathcal{B}}$ and any $z\in E$
there is $N\in N(z)$ such that $N\subset E$.

Suppose now that $\partial X=\lim\Theta(G)
$. 
Recall that $\lim\Theta(G)$,
as a set, is the disjoint union of the following subsets: $\partial T_G$,
$\Sigma_u:u\in V^w_{T_G}$ and $\dot K_t:t\in V^b_{T_G}$,
where $\dot K_t=K_t\setminus\bigcup\{ \Sigma_u:u\in N_t \}$.
By a direct application of Definitions 2.19 and 2.21 of \cite{Cashen_Martin:qi_two_ended:2017} in our context
we get that Martin's topology on  $\partial X=\lim\Theta(G)$ is 
described by the following families of neighbourhoods of points:
\begin{enumerate}
\item
for a point $\eta\in\partial T_G\subset\lim\Theta(G)$,
and for any neighbourhood $U$ of $\eta$ in $V_{T_G}\cup\partial T_G$,
we have a neighbourhood $V_U(\eta)$ of $\eta$ given by
$$
V_U(\eta)=(U\cap\partial T_G)\cup \bigcup\{ \dot K_t:t\in U \} \cup
\bigcup\{ \Sigma_u: \{ u \}\cup N_u\subset U \};
$$

\item
for a point $\xi\in\dot K_t\subset\lim\Theta(G)$, 
and for any neighbourhood $U_t$
of $\xi$ in $K_t$, define $\hbox{Cone}_{\{ U_t \}}(\xi)$
as the set consisting of all those $x\in V_{T_G}\cup\partial T_G
\setminus\{ t \}$ for which we have $\varphi_e(\Sigma_{w(e)})\subset U_t$,
where $e$ is the last edge on the path from $x$ to $t$ in $T_G$;
we then have a neighbourhood $V_{\{ U_t \}}(\xi)$ of $\xi$ given by
$$
V_{\{ U_t \}}(\xi)=(\partial T_G\cap \hbox{Cone}_{\{ U_t \}}(\xi))\cup
\bigcup\{ K_s:s\in\hbox{Cone}_{\{ U_t \}}(\xi) \}
\cup U_t;
$$

\item
for a point $\xi\in\Sigma_{u_0}\subset\lim\Theta(G)$, 
and for any family
${\mathcal{U}}=\{ U_{u_0} \}\cup\{ U_t:t\in N_{u_0} \}$ of neighbourhoods
of $\xi$ in the spaces $\Sigma_{u_0}$ and $\{ K_t:t\in N_{u_0} \}$,
respectively, define $\hbox{Cone}_{\mathcal{U}}(\xi)$ as the set 
consisting of all those $x\in V_{T_G}\cup\partial T_G
\setminus(\{ u_0 \}\cup N_{u_0})$ for which we have 
$\varphi_e(\Sigma_{w(e)})\subset U_{b(e)}$,
where $e$ is the last edge on the shortest path in $T_G$ from $x$ to 
$N_{u_0}$;
we then have a neighbourhood $V_{\mathcal{U}}(\xi)$ of $\xi$ given by
$$
V_{\mathcal{U}}(\xi)=(\partial T_G\cap \hbox{Cone}_{\mathcal{U}}(\xi))\cup
\bigcup\{ K_s:s\in\hbox{Cone}_{\mathcal{U}}(\xi) \}
\cup\big[\bigcup\{ U_t:t\in N_{u_0} \}\setminus U_{u_0}\big]\cup
$$
$$
\cup\big[U_{u_0}\cap\bigcap\{U_t:t\in N_{u_0}\}
\big].
$$
\end{enumerate}

We are now ready for showing that the identity map of the set 
$\partial X=\lim\Theta(G)$, for the Martin's topology in the source, 
and for the topology as in \Ssecref{limit} 
in the target, is continuous.
To do this, consider any open set $E\in {\mathcal{B}}$
(where $\mathcal{B}$ is as in \Ssecref{limit}), i.e. a set of form
$E={\mathcal{G}}(W)$ for an appropriate $\Omega_F$-saturated open subset
$W$ in a partial union $K_F$. We will proceed with the argument in the following three cases. First, let $\eta\in E$ be a point which belongs to
$\partial T_G$. We need to find a neighbourhood of the form 
$N=V_U(\eta)$ as in (1) above, such that $N\subset E$.
Let $t_0$ be this vertex of the b-subtree $F$ which is ``closest'' to $\eta$
in $T_G$, and let $e=[u_0,t_0]$ be the last edge on the path in $T_G$
from $\eta$ to $t_0$. Put $U$ to be the set of all those elements $z$
in $(V_{T_G}\cup\partial T_G)\setminus V_F$ for which the path
in $T_G$ from $z$ to $t_0$ passes through $u_0$ (including $z=u_0$).
$U$ is obviously a neighbourhood of $\eta$ in $V_{T_G}\cup\partial T_G$,
and it is not hard to observe that the corresponding neighbourhood $V_U(\eta)$ is a subset of $E$, as required.

In the second case, let $\xi\in E$ be a point which belongs to 
$\dot K_t$, for some $t\in V^b_{T_G}$. 
We need to find a neighbourhood of the form 
$N=V_{\{ U_t \}}(\xi)$ as in (2) above, such that $N\subset E$.
Consider first the subcase when $t\notin V_F$, and let $e$ be the first
edge on the path in $T_G$ connecting $t$ with the nearest vertex of $F$.
Put $U_t=K_t\setminus\varphi_e(\Sigma_{w(e)})$. Obviously,
$U_t$ is then a neighbourhood of $\xi$ in $K_t$, and the corresponding
neighbourhood $V_{\{ U_t \}}(\xi)$ is a subset of $E$, as required.
In the second subcase, let $t\in V_F$. 
Observe that then $\xi\in W$.
Let $e_1,\dots,e_k$ be the edges
of $F$ adjacent to $t$, and put 
$U_t=(K_t\cap W)\setminus\cup_{i=1}^k\varphi_{e_i}(\Sigma_{w(e_i)})$.
Again, it is easily seen that $U_t$ is a neighbourhood of $\xi$ in $K_t$,
and that the corresponding neighbourhood $V_{\{ U_t \}}(\xi)$ 
is a subset of $E$, as required.

In the final third case, let $\xi\in E$ be a point which belongs to
$\Sigma_{u_0}$, for some $u_0\in V^w_{T_G}$.
We need to find a neighbourhood of the form 
$N=V_{\mathcal{U}}(\xi)$ as in (3) above, such that $N\subset E$.
We will do this by considering three subcases related to the 
possible position of $u_0$ with respect to $F$. First,
consider the subcase where $u_0\in V_F$. 
Observe that then $\xi\in W$.
Note also that, since $F$
is a b-subtree of $T_G$, the set $N_{u_0}$ of all neighbours of $u_0$
in $T_G$ is actually contained in $V_F$. 
For each $t\in N_{u_0}$, let ${\mathcal{E}}_t$ be the set of edges of $F$
adjacent to $t$ and distinct from $[u_0,t]$. For each $t\in N_{u_0}$ 
put
$U_t=(W\cap K_t)
\setminus\bigcup\{ \varphi_e(\Sigma_{w(e)}):e\in{\mathcal{E}}_t \}$;
put also $U_{u_0}=W\cap\Sigma_{u_0}$ and 
${\mathcal{U}}=\{ U_{u_0} \}\cup \{ U_t:t\in N_{u_0} \}$.
Note that for each $t\in N_{u_0}$, $U_t$ is a neighbourhood of $\xi$
in $K_t$, and similarly $U_{u_0}$ is a neighbourhood 
of $\xi$ in $\Sigma_{u_0}$. Moreover, it is not hard to see that the
corresponding neigbourhood $V_{\mathcal{U}}(\xi)$ is then contained in $E$,
as required. In the second subcase, we suppose that $u_0\notin V_F$,
but $N_{u_0}\cap F\ne\emptyset$.
For each $t\in N_{u_0}\setminus V_F$ put $U_t=K_t$.
For $t\in N_{u_0}\cap V_F$, let ${\mathcal{E}}_t^F$
be the set of all edges in $F$ adjacent to $t$, and put 
$U_t=
(W\cap K_t)\setminus\bigcup\{ \varphi_e(\Sigma_{w(e)}):e\in{\mathcal{E}}_t \}$.
Put also $U_{u_0}=\Sigma_{u_0}$ and 
${\mathcal{U}}=\{ U_{u_0} \}\cup \{ U_t:t\in N_{u_0} \}$.
As before, 
note that for each $t\in N_{u_0}$, $U_t$ is a neighbourhood of $\xi$
in $K_t$, and similarly $U_{u_0}$ is a neighbourhood 
of $\xi$ in $\Sigma_{u_0}$. Likewise, it is not hard to see that the
corresponding neigbourhood $V_{\mathcal{U}}(\xi)$ is then contained in $E$,
as required.
Finally, in the third subcase, we suppose that 
$N_{u_0}\cap F=\emptyset$. Denote by $t_1\in N_{u_0}$
this vertex which lies on the path from $u_0$ to $F$.
Denote also by $u_1$ the white vertex adjacent to $t_1$ and
distinct from $u_0$ which also lies on the path from $u_0$ to $F$.
For any $t\in N_{u_0}\setminus\{ t_1 \}$, put $U_t=K_t$,
and put 
$U_{t_1}=K_{t_1}\setminus\varphi_{[u_1,t_1]}(\Sigma_{u_1})$.
Put $U_{u_0}=\Sigma_{u_0}$ and 
${\mathcal{U}}=\{ U_{u_0} \}\cup \{ U_t:t\in N_{u_0} \}$,
and note again that $V_{\mathcal{U}}(\xi)$ is contained in $E$,
as required.

This finishes the proof.

\end{proof}

\subsection{The reduced Bowditch JSJ splitting and its rigid cluster factors}
\sseclabel{reduced_JSJ}

Let $G$ be a 1-ended hyperbolic group which is not cocompact Fuchsian
and let ${T}_G$ be the Bowditch JSJ tree of $G$.  Let $\sim$ be the
smallest equivalence relation on the vertex set of ${T}_G$ for which
$u \sim v$ whenever $u$ and $v$ are rigid vertices at distance $2$.
Equivalence classes of $\sim$ are either white vertex singletons,
flexible vertex singletons or rigid classes (i.e. sets of rigid
vertices).  The convex hulls of the rigid classes are the \emph{rigid
  clusters} of ${T}_G$.  Each rigid cluster of ${T}_G$ is a subtree of
$T_G$ (possibly a single vertex).

It is not hard to observe that embedded cycles in the quotient graph 
${T}_G / {\sim}$ have
length two, and there are no loop edges in this quotient.  
So, by replacing any multi-edge of
${T}_G / {\sim}$ with a single edge, we obtain a tree. 
The latter tree contains some vertices of degree 1, namely those which
correspond to the white vertices in ${T}_G$ whose all neighbours are rigid.
We delete all those vertices (and all edges adjacent to them), thus
obtaining a tree denoted as ${T}^r_G$, which
we call the \emph{reduced Bowditch JSJ tree} for $G$.  The action
of $G$ descends to ${T}^r_G$ and there is a natural map
${T}_G \to {T}^r_G$ which is $G$-equivariant.
(This $G$-equivariant map can be described as composition of the
quotient map ${T}_G\to{T}_G/{\sim}$ with the map which collapses
to points all those edges of the tree obtained from ${T}_G/{\sim}$
which are adjacent to the degree 1 vertices.)
Technically speaking, the \emph{reduced Bowditch JSJ splitting} for $G$
is the above induced action of $G$ on the tree ${T}^r_G$.

We now describe closer the nature of the just defined reduced Bowditch
JSJ splitting, and in particular of its factors.  Observe that the
bipartite structure of the tree ${T}_G$ descends to ${T}^r_G$, so that
we can speak of black and white vertices of the latter. The white
vertices of ${T}^r_G$ are naturally in bijective correspondence with
those white vertices of ${T}_G$ which are adjacent to at least one
flexible vertex. Their stabilizers (for the induced action of $G$ on
${T}^r_G$) coincide with the corresponding stabilizers for the
$G$-action on ${T}_G$, and so they are maximal 2-ended subgroups of
$G$.  There are two kinds of black vertices in ${T}^r_G$.  The first
kind of black vertices of ${T}^r_G$ are those in bijective
correspondence with the flexible vertices of ${T}_G$, and so we still
call them \emph{flexible}.  Their stabilizers under the action of $G$
on ${T}^r_G$ do not change, and we call them the \emph{flexible
  factors} of the reduced Bowditch JSJ splitting of $G$.  The second
kind of black vertices of ${T}^r_G$ are those in bijective
correspondence with the rigid clusters of ${T}_G$, and we call them
the \emph{rigid cluster vertices} of ${T}^r_G$. Their stabilizers,
which coincide with the stabilizers of the corersponding rigid
clusters in the $G$-action on ${T}_G$, are called the \emph{rigid
  cluster factors} of the reduced Bowditch JSJ splitting of $G$.

The next observations follow directly from the above description.

\begin{fact}\factlabel{3.3} 
Let $v'$ be a flexible vertex of the reduced Bowditch tree ${T}^r_G$,
and let $v$ be the vertex of ${T}_G$ corresponding to $v'$
(so that the vertex stabilizers $G(v')$ and $G(v)$ coincide).
Then the family of edge stabilizers of the edges adjacent to $v'$ in ${T}^r_G$
coincides with the family of edge stabilizers of the edges adjacent to $v$ in
${T}_G$.

\end{fact}

\begin{fact}\factlabel{3.4}
Each rigid cluster vertex $v$ of the reduced Bowditch tree ${T}^r_G$
is \emph{isolated}, i.e., there is no rigid cluster vertex in ${T}^r_G$
at distance 2 from $v$. Equivalently, all vertices of ${T}^r_G$
lying at distance 2 from a rigid cluster vertex $v$ are flexible.
\end{fact}

\begin{rmk}\rmklabel{3.5}
  Given a rigid cluster $S\subset{T}_G$, and the corresponding rigid
  cluster factor $G(v_S)$, consider the action of $G(v_S)$ on $S$.
  The vertex stabilizers of this action are easily seen to coincide
  with the corresponding vertex stabilizers of the $G$-action on
  ${T}_G$.  Thus, this action describes some splitting of the rigid
  cluster factor $G(v_S)$, along maximal 2-ended subgroups, whose
  factors are some appropriate rigid factors of the Bowditch JSJ
  splitting of $G$. This observation is helpful for understanding the
  nature of rigid cluster factors, even though we will make no
  essential use of it.
\end{rmk}

Since the factors of any splitting of a hyperbolic group along
quasi-convex subgroups are also quasi-convex (see e.g. Proposition~1.2
in \cite{Bowditch:cut_points:1998}), and since 2-ended subgroups are
quasi-convex, we get the following.

\begin{cor}\corlabel{cluster_quasiconvex}
Each rigid cluster factor of a 1-ended hyperbolic group $G$ is quasi-convex in $G$.
In particular, each rigid cluster factor is a non-elementary hyperbolic group.
\end{cor}

The reduced Bowditch JSJ splitting of a group $G$ induces in a natural way
a  tree system described as follows:
\[
\Theta^r(G):=({T}^r_G, \{ K_t=\partial G(t):t\in V^b_{{T}^r_G} \},
\{ \Sigma_u=\partial G(u):u\in V^w_{{T}^r_G} \}, 
\{ \varphi_e:e\in E_{{T}^r_G} \}),
\]
where each $\phi_e:\partial G({w(e)})\to \partial G({b(e)})$ is the
natural inclusion.
The 
same argument as in the proof
of \Thmref{T.9} yields also the following result.

\begin{thm}
  \thmlabel{T.9+} Let $G$ be a 1-ended hyperbolic group which is not
  cocompact Fuchsian, and let $\Theta^r(G)$ be the associated tree
  system (related to the reduced Bowditch JSJ splitting of $G$).  
Then the Gromov boundary $\partial G$ is naturally
  homeomorphic to the limit of the tree system $\Theta^r(G)$, i.e.,
  \[\partial G\cong\lim\Theta^r(G).\]
\end{thm}

\section{General trees of graphs}
\seclabel{tr_gr}

In this section we describe trees of graphs---a class of topological spaces which is
the main object of study in this paper. As we will see, these spaces are compact metrizable,
and the most interesting subclass (that of trees of 2-connected graphs) 
consists of spaces that are
connected, locally connected, cutpoint-free, and of topological dimension 1. 
Throughout all of this section we make use of the terminology and notation introduced in \Secref{tree_sys}.
We start with basic definitions.

\begin{defn}
  \defnlabel{G.1} A \emph{tree system of punctured graphs} is a tree
  system of spaces in which:
  \begin{enumerate}[label=(\arabic*)]
  \item \itmlabel{G.1(1)} each constituent space with peripherals is a
    punctured graph (as described in \Exref{T.1.5});
  \item \itmlabel{G.1(2)} each white vertex of the underlying tree has
    degree 2.
  \end{enumerate}
  We will use the notation
  \[ \Theta=(T,\{ \Gamma_t^\circ \}, \{ \Sigma_u \}, \{ \phi_e \})
  \] for a tree system of punctured graphs (where for each
  $t\in V^b_T$, $\Gamma_t$ is some finite essential graph, i.e.,
  nonempty and without isolated vertices).
\end{defn}

\begin{defn}
  \defnlabel{G.2} A \emph{tree of graphs} is any topological space
  which is homeomorphic to the limit of a tree system of punctured
  graphs.
\end{defn}

\subsection{Trees of graphs as inverse limits of graphs}
\sseclabel{tog_as_inv_lims_of_graphs}

We will now present an alternative description of trees of graphs, as
inverse limits of some special inverse systems of graphs.  (As we will
see below, these inverse systems in fact coincide, up to isomorphism,
with the systems ${\mathcal S}_\Theta$ as described in \Ssecref{inverse_associated}, associated to the
corresponding tree systems $\Theta$ of punctured graphs.)  
This alternative approach to trees of graphs is
related to (and generalizes) that given in Section~2 of \cite{Swiatkowski:refl_trees_of_graphs:2021}, where
only a special case of the so called \emph{reflection trees of
  graphs} has been treated, using a similar description.

Let $\Theta$ be a tree system of punctured graphs.  For any graph
$\Gamma_t$ whose punctured version $\Gamma_t^\circ$ appears as a
constituent space in $\Theta$, denote by $D_t\subset|\Gamma_t|$ a
countable dense subset of the realization $|\Gamma_t|$ containing all
vertices of $\Gamma_t$.  Note that such $D_t$ is unique up to a
homeomorphism of $|\Gamma_t|$ that is the identity on the vertices and
that preserves all the edges. Consider also the inverse system
${\mathcal S}_\Theta=[\{ K_F^{\ast} \}, \{ f_{F_1F_2} \}]$ associated
to $\Theta$.  Observe that, viewing any black vertex $t\in V^b_T$ as a
b-subtree of $T$, and denoting by
$\{ [\phi_{[u,t]}(\Sigma_u)]:u\in N_t \}$ the set of points of the
quotient space $K_t^{\ast}$ obtained from $\Gamma_t^\circ$ by
shrinking the peripheral sets $\phi_{[u,t]}(\Sigma_u):u\in N_t $,
there is a homeomorphism $h_t:K_t^{\ast}\to|\Gamma_t|$ such that:
\begin{enumerate}[label=(\arabic*)]
\item \itmlabel{G.X(1)} $h_t$ maps the subset
  $\{ [\phi_{[u,t]}(\Sigma_u)]:u\in N_t \}$ of $K_t^*$ onto the subset
  $D_t$ of $|\Gamma_t|$;
\item \itmlabel{G.X(2)} $h_t$ respects types of the shrunk peripheral
  sets $[\phi_{[u,t]}(\Sigma_u)]$, i.e., if $\phi_{[u,t]}(\Sigma_u)$ is
  the boundary of a normal neighbourhood of some vertex of $\Gamma$,
  then $h_t([\phi_{[u,t]}(\Sigma_u)])$ is equal to this vertex, and if
  $\phi_{[u,t]}(\Sigma_u)$ is the boundary of a normal neighbourhood
  of some interior point of some edge $a$ of $\Gamma$, then
  $h_t([\phi_{[u,t]}(\Sigma_u)])$ is equal to some interior point of
  the same edge $a$.
\end{enumerate}
To better understand $h_t$, consider an edge $e \subset |\Gamma_t|$
and its punctured version
$e^{\circ} \subset \Gamma_t^{\circ} \subset |\Gamma_t|$.  Up to
homeomorphism, we can view $e^{\circ}$ as the standard Cantor set on a
closed interval $I \subset e$.  Let $h \colon I \to I$ be the Cantor
function with respect to $e^{\circ}$ and note that $h$ descends to a
continuous map $h^e \colon (e^\circ)^* \to I$ from the image 
$(e^\circ)^*$ of
$e^\circ$ in $K_t^{\ast}$.  Then $h_t|_{(e^\circ)^*}$ can be taken to be the
composition of $h^e$ with an orientation preserving homeomorphism
$I \to e$ which maps dyadic rationals in $I$ onto the subset $D_t\cap e$.

The above identifications of the spaces $K_t^*$ of the inverse system
${\mathcal S}_\Theta$ with realizations $|\Gamma_t|$ of the graphs
$\Gamma_t$ induce further identifications of all other spaces $K_F^*$
of this inverse system with certain graphs $X_F$ which are obtained
out of the graphs $\Gamma_t:t\in V^b_F$ by means of operations of
``connected sum'' for graphs. We now describe these graphs $X_F$,
together with the description of maps between them, which correspond
to the bonding maps of the inverse system ${\mathcal S}_\Theta$.

Let $X$ be the realization of a finite graph, and let $x\in X$. The
\emph{blow-up} of $X$ at $x$ is the space obtained from $X$ by
deleting $x$, and by taking the completion of the resulting space
which consists of adding as many points as the number of connected
components into which $x$ splits its normal neighbourhood $B(x)$ in
$X$. We denote this blow-up by $X^\#(x)$, and the set of points added
at the completion by $P_x$ (we use the term \emph{blow-up divisor at
  $x$} for this set $P_x$).  Similarly, if $J$ is a finite subset of
$X$, the \emph{blow-up} of $X$ at $J$, denoted $X^\#(J)$, is the space
obtained from $X$ by performing blow-ups at all points of $J$ (in
arbitrary order).  Moreover, for any two finite subsets
$J_1\subset J_2$ of $X$, denote by
$\rho^X_{J_2J_1}: X^\#(J_2)\to X^\#(J_1)$ the \emph{blow-down} map
which maps the sets $P_x:x\in J_2\setminus J_1$ to the corresponding
points $x$, and which is the identity on the remaining part of
$X^\#(J_2)$.

Now, let $F$ be a finite b-subtree of the underlying tree $T$ of the tree
system $\Theta$. 
Recalling that we have identified graphs $|\Gamma_t|$ with the corresponding
spaces $K_t^*$,
for each $t\in V^b_F$ consider the subset 
$J_t^F\subset|\Gamma_t|$ given by
$J_t^F:=\{ h_t([\phi_{[t,u]}(\Sigma_u)]):u\in N^F_t \}$.
Define $X_F$ as quotient of the disjoint union  
\[
  \bigsqcup_{t\in V^b_F} |\Gamma_t|^\#(J^F_t)
\]
by the equivalence relation $\sim$ determined by the following
gluing maps.  Recall that, by condition \itmref{G.1(2)} of
\Defnref{G.1}, for any white vertex $u$ of $F$ there are exactly two
black vertices in $F$ adjacent to $u$. Denoting these adjacent black
vertices by $t$ and $s$, consider the blow-up divisors
$P_{h_t([\phi_{[t,u]}(\Sigma_u)])} \subset |\Gamma|^\#(J_t^F)$ and
$P_{h_s([\phi_{[s,u]}(\Sigma_u)])} \subset |\Gamma|^\#(J_s^F)$.  Note
that there is a natural bijection
\[
  b_{s,t}: P_{h_s([\phi_{[s,u]}(\Sigma_u)])} \to P_{h_t([\phi_{[t,u]}(\Sigma_u)])}
\]
induced by the map
$\phi_{[t,u]}\circ\phi^{-1}_{[s,u]}:\phi_{[s,u]}(\Sigma_u) \to
\phi_{[t,u]}(\Sigma_u)$ (under the canonical identification of each of
the two divisors with the corresponding peripheral subspace), and that
$b_{t,s}=b_{s,t}^{-1}$. The equivalence relation $\sim$ is then
determined by all of the maps $b_{s,t}$ as above related to all white
vertices $u$ of $F$.  Note that the described quotient space $X_F$
may be viewed as obtained by appropriate iterated connected sum of the
family $\Gamma_t:t\in V^b_F$ of graphs
(see the paragraph right before \Lemref{puncture_component_closure} for a precise description of the concept of connected sum for graphs).  Note also that $X_F$ is in a
natural way homeomomorphic to the space $K^*_F$, via homeomorphism
induced by the homeomorphisms $h_t^{-1}:|\Gamma_t|\to K^*_t$, for
$t\in V^b_F$.

In order to make our alternative description of the inverse system
${\mathcal S}_\Theta$ complete, we need to describe the bonding maps
between the spaces $X_F$ which, under the above identifications with
the spaces $K_F^*$, correspond to the bonding maps $f_{FF'}$.  To do
this, consider any two finite b-subtrees $F\subset F'$ of $T$, and
define the map $\pi_{FF'}:X_{F'}\to X_F$ to be a natural map from the
larger connected sum to the smaller one, given as follows.  For
$t\in V^b_F$, the restriction of $\pi_{FF'}$ to
$|\Gamma_t|^\#(J_t^{F'})$ (viewed as a subset in $X_{F'}$) coincides
with the blow-down map $\rho^{|\Gamma_t|}_{J_t^{F'}J_t^F}$ (whose
image $|\Gamma_t|^\#(J_t^F)$ is viewed as a subset in $X_F$).  For
$t\in V^b_{F'}\setminus V^b_F$, the subset
$|\Gamma_t|^\#(J_t^{F'})\subset X_{F'}$ is mapped by $\pi_{FF'}$ to
the point
$h_s([\phi_{[u,s]}(\Sigma_u)])\in |\Gamma_s|^\#(J_s^F)\subset X_F$,
where $s\in V^b_F$ is the vertex of the subtree $F$ that is closest to
$t$, and where $u\in V^w_{F'}$ is the vertex adjacent to $s$ on the
shortest path in $T$ from $t$ to $s$.

We skip a straightforward verification of the following fact, the second part of which
follows from the first part in view of \Propref{T.6}.

\begin{fact}
  \factlabel{G.3} The inverse systems
  ${\mathcal S}_\Theta=[\{ K_F^* \}, \{ f_{FF'} \}]$ and
  $[ \{ X_F \}, \{ \pi_{FF'} \} ]$ are isomorphic.  As a consequence,
  we have
  \[
    \lim\Theta\cong \lim_{\longleftarrow}[ \{ X_F \}, \{ \pi_{FF'} \} ].
  \]
\end{fact}

  Expression of a tree of graphs $\lim\Theta$ as the inverse limit of
  graphs $X_F$, as above, turns out to be very useful for establishing
  various topological properties of trees of graphs.
The next corollary records the most obvious such properties
(which follow directly from the corresponding well known properties of
inverse limits, e.g. the estimate of the topological dimension of the limit by supremum of the topological dimensions of the terms in the inverse system). 
Further properties, for a subclass of trees of 2-connected graphs, 
are discussed in the next subsection. 

\begin{cor}\corlabel{top_dim_estimate}
Any tree of graphs is a compact metrizable topological space of topological
dimension $\le1$.
\end{cor}

\subsection{Trees of 2-connected graphs}

We now discuss a subclass in the class of trees of graphs, so called
\emph{trees of 2-connected graphs}, which is most interesting from
the perspective of studying the topology of Gromov boundaries
of hyperbolic groups. As we will see, the spaces from this subclass are
connected, locally connected, cutpoint-free, 
and they have topological dimension equal to 1.  

A graph is \emph{2-connected} if it is finite, connected, nonempty,
not equal to a single vertex, and contains no cutpoint.  We say
that a tree system of punctured graphs
$\Theta=(T,\{ \Gamma_t^\circ \}, \{ \Sigma_u \}, \{ \phi_e \})$ in
which all the graphs $\Gamma_t$ are 2-connected is a \emph{tree system
  of punctured 2-connected graphs}; likewise, a \emph{tree of
  2-connected graphs} is the limit of a tree system of punctured
2-connected graphs. 

In our analysis of trees of 2-connected graphs we will need the following
rather well known observation.

\begin{fact}
  \factlabel{compact_connected} Let
  ${\mathcal S} = (\{ K_{\alpha}\}_{\alpha},\{ f_{\alpha
    \alpha'}\}_{\alpha \le \alpha'})$ be an inverse system of compact metric
  spaces
  with surjective bonding maps.
  Let $C$ be
  a closed subspace of some $K_{\alpha}$.  Then the preimage
  $f_{\alpha}^{-1}(C)$ of $C$ in the inverse limit of $\mathcal S$ 
  is disconnected if and only if
  $f_{\alpha \hat \alpha}^{-1}(C)$ is disconnected for some
  $\hat \alpha \ge \alpha$.  
\end{fact}



For what follows, we use the notation $\pi_0(X)$ to denote the set of
connected components of a space $X$.


\begin{lem}\lemlabel{pair_link_map}
  Let $X$ be the realization of a 2-connected graph and let $x$ and
  $y$ be two points of $X$.  The inclusion
  $\iota_{x,y} \colon P_x \hookrightarrow X^{\#}(\{x,y\})$ is
  $\pi_0$-surjective.  Consequently, the map $\iota_{x,y}$ is
  $\pi_0$-bijective if and only if $\deg(x)$ is equal to the number of
  components of $X^{\#}(\{x,y\})$.
\end{lem}
\begin{proof}
  The blow-down map $X^{\#}(\{x,y\}) \to X^{\#}(y)$ is the quotient map
  identifying all points of $P_x$.  Since $X$ has no cutpoints, the
  blow-up $X^{\#}(y)$ is connected.  It follows that $P_x$ intersects
  every component of $X^{\#}(\{x,y\})$, i.e., the inclusion
  $\iota_{x,y}$ is $\pi_0$-surjective.
  
  The last sentence of the statement follows from the fact that
  $|P_x| = \deg(x)$ and $\deg(x)$ is finite.
\end{proof}

\begin{rmk}\rmklabel{cut_pair_map}
  Let $X$ be the realization of a finite graph and let $J$ be a finite
  subset of $X$.  The inclusion
  $X \setminus J \hookrightarrow X^{\#}(J)$ is $\pi_0$-bijective.
  Thus there exists a map
  $\bigcup_{x\in J} P_x \to \pi_0(X \setminus J)$ given by composing
  $\bigcup_{x\in J} P_x \to \pi_0\bigl(X^{\#}(J)\bigr)$ with the
  $\pi_0$-inverse
  $\pi_0\bigl(X^{\#}(J)\bigr) \to \pi_0(X \setminus J)$ of the
  inclusion.  By \Lemref{pair_link_map}, if $X$ is 2-connected and
  $J = \{x,y\}$ then the restriction
  $P_x \to \pi_0(X \setminus \{x,y\})$ is surjective always and is
  bijective if $\deg(x) = \bigl|\pi_0(X \setminus \{x,y\})\bigr|$.
\end{rmk}

The description in \Ssecref{tog_as_inv_lims_of_graphs} of trees of
punctured graphs as limits of inverse systems of graphs 
alludes to the notion of a
connected sum of graphs.  Here we give an explicit definition.  If
$X_1$ and $X_2$ are realizations of finite graphs with $x_1 \in X_1$
and $x_2 \in X_2$ and $\ell \colon P_{x_1} \to P_{x_2}$ is a bijection
between the blow-up divisor $P_{x_1}$ of $x_1$ in $X_1$ and the
blow-up divisor $P_{x_2}$ of $x_2$ in $X_2$ then the \emph{connected
  sum} $X_1 \#_{\ell} X_2$ is the quotient of
$X_1^{\#}(x_1) \sqcup X_2^{\#}(x_2)$ by the equivalence relation
identifying each $p \in P_{x_1}$ with $\ell(p) \in P_{x_2}$.  For
$i \in \{1,2\}$, the \emph{projection} $X_1 \#_{\ell} X_2 \to X_i$ is
the map that is identity on
$X_i \setminus \{x_i\} = X_i^{\#}(x_i) \setminus P_{x_i} \subset X_1
\#_{\ell} X_2$ and that sends the remaining points of
$X_1 \#_{\ell} X_2$ to $x_i$.

\begin{lem}\lemlabel{puncture_component_closure}
  Let $X$ be a connected and locally connected $T_1$ space, let
  $x \in X$ and let $C$ be a component of $X \setminus \{x\}$.  Then
  the closure $\bar C$ in $X$ of $C$ is $C \cup \{x\}$.
\end{lem}
\begin{proof}
  Since $X$ is $T_1$, the set $X \setminus \{x\}$ is open in $X$ and
  so, by local connectedness, the components of $X \setminus \{x\}$
  are open in $X$.  In particular, the component $C$ is open and so, since
  $X$ is connected and $C$ is a proper subset of $X$, it cannot be
  that $C$ is also closed in $X$.  It suffices then to show that
  $\bar C \subset C \cup \{x\}$.  But the complement of $C \cup \{x\}$
  is the union of the non-$C$ components of $X \setminus \{x\}$, which
  is open.  Thus $C \cup \{x\}$ is closed and so $\bar C$ is indeed
  contained in $C \cup \{x\}$.
\end{proof}

\begin{cor}\corlabel{twin_component_closure}
  Let $X$ be a connected, locally connected, cutpoint-free $T_1$
  space.  Let $x,y \in X$ and let $C$ be a component of
  $X \setminus \{x,y\}$.  The closure $\bar C$ of $C$ in $X$ is
  $C \cup \{x,y\}$.
\end{cor}
\begin{proof}
  If $x = y$ then we can directly apply \Lemref{puncture_component_closure}
  so we assume that $x \neq y$.  Put $X_y = X \setminus \{y\}$
  and note that $X_y$ is connected (because $X$ is cutpoint-free).  Then
  $C$ is a component of $X_y \setminus \{x\}$ so, by
  \Lemref{puncture_component_closure}, the closure of $C$ in $X_y$ is
  $C \cup \{x\}$.  Then $\bar C$ is either $C \cup \{x\}$ or
  $C \cup \{x,y\}$.  Swapping the roles of $x$ and $y$ we see that
  $\bar C = C \cup \{x,y\}$.
\end{proof}

\begin{lem}\lemlabel{conn_sum_map_connected}
  Let $X_1$ and $X_2$ be realizations  of
  graphs, suppose that $X_2$ is 2-connected, and let
  $x_1 \in X_1$ and $x_2 \in X_2$.  Let
  $\ell \colon P_{x_1} \to P_{x_2}$ be a bijection between the blow-up
  divisor $P_{x_1}$ of $x_1$ in $X_1$ and the blow-up divisor
  $P_{x_2}$ of $x_2$ in $X_2$.  Let $A$ be a connected subspace of
  $X_1$.  Then the preimage of $A$ under the
  projection $X_1 \#_{\ell} X_2 \to X_1$ is connected.
\end{lem}
\begin{proof}
  If $A$ is disjoint from $x_1$
  then the preimage is homeomorphic to $A$ so we may assume, without
  loss of generality, that $x_1 \in A$.  Any connected subspace of a
  graph is locally connected so, by
  \Lemref{puncture_component_closure}, the closure in $X_1$ of every
  component of $A^{-} = A \setminus \{x_1\}$ contains $x_1$.  Thus we
  see that the closure in $X_1^{\#}(x_1)$ of every component of
  $A^{-}$ intersects the blow-up divisor $P_{x_1}$.  The preimage of
  $A$ in $X_1 \#_{\ell} X_2$ is the union of $A^{-}$ and
  $X_2^{\#}(x_2)$, where the latter is connected 
  since $X_2$ is 2-connected.  But
  $A^{-} \cup X_2^{\#}(x_2)$ is equal to the union of $X_2^{\#}(x_2)$
  and the closures of the components of $A^{-}$ in
  $A^{-} \cup P_{x_1}$, each of which intersects
  $P_{x_1} = P_{x_2} \subset X_2^{\#}(x_2)$.
\end{proof}

\begin{lem}
  \lemlabel{conn_sum_nice} Let $X_1$ and $X_2$ be realizations of
  2-connected graphs and let $x_1 \in X_1$ and $x_2 \in X_2$.  Let
  $\ell \colon P_{x_1} \to P_{x_2}$ be a bijection between the blow-up
  divisor $P_{x_1}$ of $x_1$ in $X_1$ and the blow-up divisor
  $P_{x_2}$ of $x_2$ in $X_2$.  Then the connected sum
  $X = X_1 \#_{\ell} X_2$ is 2-connected.
\end{lem}
\begin{proof}
  It is clear that $X$ is not empty and not equal to a single point.
  So we need only to show that $X$ is connected and has no cutpoint.
  That $X$ is connected follows from \Lemref{conn_sum_map_connected}
  by setting $A = X_1$.

  To see that $X$ has no cutpoints, let $p \in X$.  Up to
  homeomorphism, either $p \in X_1^{\#}(x_1) \setminus P_{x_1}$ or
  $p \in X_2^{\#}(x_2) \setminus P_{x_2}$ so we may assume, without
  loss of generality, that $p \in X_1^{\#}(x_1) \setminus P_{x_1}$.
  Then $X \setminus \{p\}$ is the preimage of
  $A = X_1 \setminus \{p\}$ under the projection map
  $X_1 \#_{\ell} X_2 \to X_1$.  Since $X_1$ has no cutpoints, the
  subspace $A$ is connected so, by \Lemref{conn_sum_map_connected},
  its preimage under the projection map is connected.
\end{proof}

\begin{lem}
  \lemlabel{asc_union_compact_conn} Let $X$ be the realization of a
  finite graph and let $A \subset X$ be a connected subspace.  Then
  $A$ is an ascending union of compact connected subspaces of $X$.
\end{lem}
\begin{proof}
  Let $\bar A$ be the closure of $A$.  Since $X$ is the realization of
  a finite graph, the set $\bar A \setminus A$ must be finite since
  otherwise there would exist a common accumulation point of
  $\bar A \setminus A$ and $A$ in a closed edge, which would
  contradict connectedness of $A$.
  
  We metrize $X$ as a geodesic metric space with each edge isometric
  to the unit interval $[0,1]$.  Let
  $A_n = A \setminus N_{\frac{1}{n}}(\bar A \setminus A)$ where
  $N_{\frac{1}{n}}(\bar A \setminus A)$ is the set of all points at
  distance less than $\frac{1}{n}$ to $\bar A \setminus A$.  Then
  $A_n = \bar A \setminus N_{\frac{1}{n}}(\bar A \setminus A)$, and so it
  is closed.  Thus the $A_n$ are compact.  Moreover, since
  $\bar A \setminus A$ is finite, the ascending union $\bigcup_n A_n$
  is equal to $A$.  It remains to show that the $A_n$ are connected
  for $n$ large enough.

  Let $L$ be large enough that for any $x \in \bar A \setminus A$, the
  metric ball $B_{\frac{1}{L}}(x)$ intersects no vertex of $X$ aside
  from $x$, in the case where $x$ is itself a vertex.  Take
  $p,q \in A_n$ with $n \ge L$.  Since $A$ is a connected subspace of
  a graph, there exists an embedded path $P \subset A$ with endpoints
  $p$ and $q$.  We claim that $P$ is disjoint from
  $N_{\frac{1}{n}}(\bar A \setminus A)$.  For the sake of finding a
  contradiction, assume that $P$ intersects
  $N_{\frac{1}{n}}(\bar A \setminus A)$, then $P$ intersects
  $B_{\frac{1}{n}}(x)$ for some $x \in \bar A \setminus A$.  But
  $x \notin A$ so the intersection $B_{\frac{1}{n}}(x) \cap A$ is
  contained in $B_{\frac{1}{n}}(x) \setminus \{x\}$, which is a
  disjoint union of open segments since $n \ge L$.  Let $U \cap P$ be
  the intersection of one of these open segments $U$ with $P$.  Since
  $\bar U$ contains $x$ and $P$ is closed and disjoint from
  $x \notin A$, we have $U \cap P \subsetneq U$.  This implies that an
  endpoint of $P$ is contained in
  $U \subset B_{\frac{1}{n}}(x) \subset N_{\frac{1}{n}}(\bar A
  \setminus A)$, which contradicts $p,q \in A_n$.  Thus $P$ is
  disjoint from $N_{\frac{1}{n}}(\bar A \setminus A)$ and so is
  contained in $A_n$.
\end{proof}

\begin{lem}\lemlabel{preimage_connected}
  Let $\Theta$ be a tree system of punctured 2-connected graphs with
  limit $Z$.  Let
  ${\mathcal S}_\Theta=(\{ K_F^{\ast}:F\in{\mathcal F}_T^b \},\{
  f_{FF'}:F\subset F' \})$ be the standard inverse system associated
  to $\Theta$.  Let $F$ be a finite b-subtree of $T$ and let $A$ be a
  connected subspace of the reduced partial union graph $K_F^{\ast}$.
  Then the preimage of $A$ in $Z$ is connected.
\end{lem}
\begin{proof}
  By \Lemref{asc_union_compact_conn}, the connected subspace $A$ is a
  union $\bigcup_n C_n$ of compact connected subspaces
  $C_1 \subset C_2 \subset \cdots$.  Thus the preimage $f_F^{-1}(A)$
  of $A$ is an ascending union of the preimages $f_F^{-1}(C_n)$ of the
  $C_n$, where $f_F \colon Z \to K_F^{\ast}$ is the projection.  It
  suffices then to prove that the $f_F^{-1}(C_n)$ are connected or,
  more generally, that the preimage in $Z$ of any \emph{compact}
  connected subspace $C$ of $K_F^{\ast}$ is connected.

  By \Factref{compact_connected}, it suffices to prove that
  $f_{FF'}^{-1}(C)$ is connected for any b-subtree $F' \supset F$.  We
  prove this by induction on the size of $F' \setminus F$.  The base
  case $F' = F$ is immediate.  Otherwise, the reduced partial union
  graph $K_{F'}^{\ast}$ is the connected sum of the reduced partial
  union graph $K_{F''}^{\ast}$ and the realization of a 2-connected
  graph, for some b-subtree $F''$ with $F' \supsetneq F'' \supset F$.
  Then $f_{FF'}^{-1}(C) = f_{F''F'}^{-1}(\hat C)$ where
  $\hat C = f_{FF''}^{-1}(C)$, which is connected by the inductive
  hypothesis.  
  By \Lemref{conn_sum_map_connected}, the
  preimage $f_{F''F'}^{-1}(\hat C) = f_{FF'}^{-1}(C)$ is connected.
\end{proof}

\begin{lem}
  \lemlabel{limit_very_connected} Let $\Theta$ be a tree system of
  punctured 2-connected graphs.  Then the limit $Z$ of $\Theta$ is
  connected, locally connected and cutpoint-free.
\end{lem}
\begin{proof}
  Let
  ${\mathcal S}_\Theta=(\{ K_F^{\ast}:F\in{\mathcal F}_T^b \},\{
  f_{FF'}:F\subset F' \})$ be the standard inverse system associated
  to $\Theta$.  By \Lemref{conn_sum_nice}, any reduced partial union
  $K_F^{\ast}$ is connected.  Then, since an inverse limit of
  connected metric compacta is connected (or by
  \Factref{compact_connected}),
  $Z$ is connected as well.

  Let $p \in Z$.  Let $U$ be an open neighborhood of $p$.  Then there
  is an open subneighborhood of $p$ of the form $f_F^{-1}(V)$ where
  $f_F \colon Z \to K_F^{\ast}$ is the projection to some reduced
  partial union graph $K_F^{\ast}$ and $V$ is an open neighborhood of
  $f_F(p)$ in $K_F^{\ast}$.  Since graphs are locally connected, there
  is a connected open subneighborhood $W\subset V$ of $f_F(p)$.  By
  \Lemref{preimage_connected}, the preimage $f_F^{-1}(W)$ is
  connected.  Thus $f_F^{-1}(W)$ is a connected open neighborhood of
  $p$ contained in $U$ and this shows that $Z$ is locally connected.
  
  Let $p \in Z$.  To establish that $Z$ is cutpoint-free, we will
  show that $Z \setminus \{p\}$ is connected.  
%
  Let
  $F_1 \subset F_2 \subset \cdots$ be an ascending sequence of finite
  b-subtrees with $T = \bigcup_n F_n$.  
  Since this sequence is cofinal in the poset ${\mathcal F}_T^b$,
  the set $Z \setminus \{p\}$ is
  the ascending union of the preimages of the  sets
  $K_{F_n}^{\ast} \setminus \{f_{F_n}(p)\}$.  
  It is thus sufficient to show that all these preimages are connected.
  However,
  by \Lemref{conn_sum_nice} (appropriately iterated), the subspaces
  $K_{F_n}^{\ast} \setminus \{f_{F_n}(p)\}$ are all connected
  and so, by applying \Lemref{preimage_connected}, their preimages
  in $Z$ are connected as well.
\end{proof}

In the next lemma we show that, for trees of 2-connected
graphs, the estimate for topological dimension
given in \Corref{top_dim_estimate} actually yields equality.

\begin{lem}
  \lemlabel{G.5} Let $Z$ be a tree of 2-connected graphs. Then $Z$ is
  of topological dimension 1.
\end{lem}

\begin{proof}
  Note that by the general estimate for the dimension of inverse
  limits, and by the fact that $\dim X_F=1$ for all $F$, we get that
  $\dim Z\le1$.  To prove the converse estimate, we will show that $Z$
  contains an embedded copy of the circle $S^1$.  Consider any
  sequence $(F_n)$ of finite b-subtrees of $T$ with the following
  properties:
  \begin{enumerate}[label=(\arabic*),start=0]
  \item \itmlabel{G.Y(0)} $F_1$ is a subtree consisting of a single
    black vertex, say $t_1$;
  \item \itmlabel{G.Y(1)} for each $n\ge1$ we have
    $F_n\subset F_{n+1}$ and $F_{n+1}$ has exactly one black vertex,
    say $t_{n+1}$, not contained in $V^b_{F_n}$;
  \item \itmlabel{G.Y(2)} $\bigcup_{n\ge1}F_n=T$.
  \end{enumerate}

  \noindent
  Obviously, such a sequence $(F_n)$ always exists, and it is cofinal
  in the poset of all finite b-subtrees of $T$ (ordered by inclusion).
  Consequently, if
  \[
    {\mathcal I}=[\{ K^*_{F_n}:n\ge1 \}, \{ f_{F_nF_{n+1}}:n\ge1 \}]
    =[\{ X_{F_n}:n\ge1 \}, \{ \pi_{F_nF_{n+1}}:n\ge1 \}]
  \]
  is the inverse sequence obtained by restriction of
  ${\mathcal S}_\Theta$ to the subposet $(F_n)$, then the
  corresponding tree of graphs $\lim\Theta$ is homeomorphic to the
  inverse limit $\lim_{\leftarrow}{\mathcal I}$.  Note that, since
  $X_{F_1}$ has no cut vertex, it contains a cycle, and we pick one
  such cycle $C_1\subset X_{F_1}$.  We next describe inductively a
  sequence of cycles $C_n\subset X_{F_n}$ such that for each $n\ge1$
  we have $C_{n+1}\subset \pi_{F_nF_{n+1}}^{-1}(C_n)$. Having defined
  the cycle $C_n$, let $x\in X_{F_n}$ be the point such that
  $X_{F_{n+1}}$ is obtained by connected sum of $X_{F_n}$ with
  $\Gamma_{t_{n+1}}$ at $x$.  More precisely, view $X_{F_{n+1}}$ as
  obtained from the blow-ups $X_{F_n}^\#(x)$ and
  $\Gamma_{t_{n+1}}^{\#}(y)$ by appropriate gluing of the blow-up
  divisors. If $x\notin C_n$, we may view $C_n$ as a subset of
  $X_{F_{n+1}}$, and we put $C_{n+1}:=C_n$ in this case.  If
  $x\in C_n$, let $p,q$ be the points of the blow-up divisor
  $P_x\subset X_{F_n}^\#(x)$ corresponding to these two components
  into which $x$ locally splits $X_{F_n}$ which intersect $C_n$.  Put
  also
  $C_n^\#=(C_n\setminus\{ x \})\cup\{ p,q \}\subset X_{F_n}^\#(x)$.
  Denote by $p'$ and $q'$, respectively, the images of $p$ and $q$ in
  $P_y\subset \Gamma_{t_{n+1}}^{\#}(y)$ through the gluing map of the
  corresponding connected sum. Consider also an arc
  $A\subset \Gamma_{t_{n+1}}^{\#}(y)$ connecting $p'$ with $q'$ (which
  exists by the assumption that $\Gamma_t$ has no cut vertex).  Put
  $C_{n+1}:=C_n^\#\cup A$.

  Now, we get an inverse sequence
  \[
    {\mathcal C}=[(C_n)_{n\ge1}, (\pi_{F_nF_{n+1}}|_{C_{n+1}})_{n\ge1}]
  \]
  whose limit $\lim_\leftarrow{\mathcal C}$ is obviously a subspace of
  $Z$.  Observe that, since $\mathcal C$ is a sequence of circles,
  with bonding maps which are near-homeomorphisms (i.e. can be
  approximated by homeomorphisms), by the result of M. Brown
  \cite{MBrown:inverse_limits:1960} the limit
  $\lim_\leftarrow{\mathcal C}$ is also a circle, which completes the
  proof.
\end{proof}

\section{Necessary conditions for tree of
  graphs boundary}
\seclabel{nec_conditions_tog}

This section is devoted to the proofs of the implications
\pitmref{bd_g_tog}$\Rightarrow$\pitmref{g_rigid_cluster_factors_vf} of \Mainthmref{1} and
\pitmref{bd_g_tog}$\Rightarrow$\pitmref{g_no_rigid_factors} of
\Mainthmref{2}.

\subsection{Proof of the impliciation
  \pitmref{bd_g_tog}$\Rightarrow$\pitmref{g_rigid_cluster_factors_vf} of \Mainthmref{1}}
In this subsection we will introduce subnecklaces, prove some results
about sufficiently connected trees of graphs and use these results to
prove the impliciation
\pitmref{bd_g_tog}$\Rightarrow$\pitmref{g_rigid_cluster_factors_vf} of \Mainthmref{1}.

The following definition will allow us to relate the topology of trees
of graphs with the topology of boundaries of $1$-ended hyperbolic
groups.

\begin{defn}
  \defnlabel{subnecklace} Let $X$ be a connected, locally connected,
  cutpoint-free, compact Hausdorff space.  An infinite subspace
  $N \subset X$ is a \emph{subnecklace} if it satisfies the following
  conditions.
  \begin{enumerate}
  \item \itmlabel{subnecklace_cutpair} For any $p,q \in N$, the space
    $X \setminus \{p,q\}$ has two components.
  \item \itmlabel{subnecklace_cutpoint} For any $p \in N$, the space
    $X \setminus \{p\}$ has two ends (i.e. $p$ is a local cutpoint of degree 2).
  \end{enumerate}
\end{defn}


\begin{rmk}
  \rmklabel{subnecklaces_subsets_of_necklaces} Let $G$ be a 1-ended
  hyperbolic group which is not cocompact Fuchsian.  Recall the
  description in \Ssecref{bowditch_jsj_splitting} of the classes of
  local cutpoints of $\bd G$.  The subnecklaces of $\bd G$ are
  precisely the infinite subsets of the necklace classes of $\bd G$.
\end{rmk}

We now turn to the analysis of subnecklaces in connected, locally connected
and cutpoint-free trees of graphs. We start with terminological preparations,
and with some preparatory technical observations.
Let $\Theta=(T,\{ \Gamma_t^\circ \}, \{ \Sigma_u \}, \{ \phi_e \})$ be
a tree system of punctured graphs, let
${\mathcal S}_\Theta=(\{ K_F^{\ast}:F\in{\mathcal F}_T^b \},\{
f_{FF'}:F\subset F' \})$ be the standard inverse system associated to
$\Theta$ and assume that the limit 
$Z=\lim\Theta=\lim_{\leftarrow}\mathcal{S}_\Theta$ 
is connected and locally
connected.  As described in \Ssecref{tog_as_inv_lims_of_graphs}, for a
finite b-subtree $F$ of $T$, the reduced partial union
$K_F^{\ast} = K_F/\Omega_F$ of $\Theta$ 
is a graph.
Let $D_F$ be the image of the union of
the peripheral subspaces $\Omega_F$ in
$K_F^{\ast}$ and note that $D_F$ is dense, countable and contains the
vertex set of $K_F^{\ast}$.  Note that any
$p \in K_F^{\ast} \setminus D_F$ has singleton preimage in the limit
$Z$.  We thus call any such $p$ the \emph{stable} point of
$K_F^{\ast}$.  We will casually abuse notation and refer to the preimage
in $Z$ or $K_{F'}^{\ast}$ of a stable point $p \in K_F^{\ast}$ by the
same name $p$.

Let $K$ be a space that is homeomorphic to the realization of a graph.
A \emph{topological edge} $e$ of $K$ is a subspace $e \subset K$
satisfying the following conditions:
\begin{enumerate}
\item $e$ is homeomorphic to $[0,1] \subset \R$;
\item aside from its endpoints, every point of $e$ has degree two in
  $K$.
\end{enumerate}
If $K$ is the realization of a graph $\Gamma$ then every edge of
$\Gamma$ is a topological edge of $K$.

\begin{lem}
  \lemlabel{compact_edges_connected} Let
  ${\mathcal S}_\Theta=(\{ K_F^{\ast}:F\in{\mathcal F}_T^b \},\{
  f_{FF'}:F\subset F' \})$ be the standard inverse system associated
  to a tree system of punctured graphs $\Theta$ and assume that the
  limit $Z$ is connected, locally connected and has no cutpoints.  Let
  $F$ be a finite b-subtree of $T$ and let $e$ be a topological edge of
  $K_F^{\ast}$ with stable endpoints.  Then $f_F^{-1}(e)$ is
  connected, where $f_F\colon Z \to K_F^{\ast}$ is the projection.
\end{lem}

\begin{rmk*}
Note that in the case when $Z$ is a tree of 2-connected graphs, 
the assertion  above follows by \Lemref{preimage_connected}. 
Here we will get this assertion in a more general setting.
\end{rmk*}

\begin{proof}
  Let $p$ and $q$ be the endpoints of $e$.  By
  \Factref{compact_connected}, it suffices to prove that
  $f_{FF'}^{-1}(e)$ is connected for every $F' \supset F$.  For the
  sake of reaching a contradiction, suppose that $e' = f_{FF'}^{-1}(e)$
  is not connected for some $F' \supset F$.  Let $C$ be the component
  of $e'$ that contains the endpoint $p$.  Since $e'$ is closed in
  $K_{F'}^{\ast}$, the set $V = K_{F'}^{\ast} \setminus C$ is open and
  nonempty in $K_{F'}^{\ast}$.  Since $e'$ is obtained from $e$ by
  performing finitely many connected sum operations, the set
  $U = C \setminus \{p\}$ must also be nonempty.

  We will argue that $q \notin C$ so that
  $U = C \cap \bigl(e' \setminus \{p,q\}\bigr) = C \cap
  f_{FF'}^{-1}(e \setminus \{p,q\})$.  For the sake of reaching a contradiction,
  suppose $C$ contains $q$ in addition to $p$.  Let $C'$ be the union
  of the remaining components of $e'$.  Then $C'$ is open in
  $f_{FF'}^{-1}(e \setminus \{p,q\})$, which is open in $K_{F'}^{\ast}$, so $C'$
  is open in $K_{F'}^{\ast}$.  On the other hand, since $e'$ has
  finitely many components (being homeomorphic to a finite graph), we
  see that $C'$ is closed in $e'$, which is closed in $K_{F'}^{\ast}$,
  so $C'$ is closed in $K_{F'}^{\ast}$.  Then $K_{F'}^{\ast}$ is
  disconnected. 
  Since the bonding maps in ${\mathcal S}_\Theta$ are obviously surjective,
  this contradicts connectedness of $Z$ by
  \Factref{compact_connected}.  So we have
  $U = C \cap f_{FF'}^{-1}(e \setminus \{p,q\})$.

  Since $e'$ has finitely many components, the component $C$ must be
  open in $e'$.  Then $U = C \cap f_{FF'}^{-1}(e \setminus \{p,q\})$ is open
  in $f_{FF'}^{-1}(e \setminus \{p,q\})$, which is open in $K_{F'}^{\ast}$.
  We conclude that $U$ and $V$ are nonempty disjoint open sets of
  $K_{F'}^{\ast}$ with $U \sqcup V = K_{F'}^{\ast} \setminus \{p\}$.
  By taking preimages in $Z$ and recalling that $p$ is a stable point
  we see that the preimage of $p$ in $Z$ is a cutpoint of $Z$, a
  contradiction.
\end{proof}

\begin{cor}
  \corlabel{open_edges_conn} Let
  ${\mathcal S}_\Theta=(\{ K_F^{\ast}:F\in{\mathcal F}_T^b \},\{
  f_{FF'}:F\subset F' \})$ be the standard inverse system associated
  to a tree system of punctured graphs $\Theta$ and assume that the
  limit $Z$ is connected, locally connected and has no cutpoints.  Let
  $F$ be a finite b-subtree of $T$ and let $e$ be an open interval of
  a topological edge of $K_F^{\ast}$.  Then $f_F^{-1}(e)$ is connected,
  where $f_F\colon Z \to K_F^{\ast}$ is the projection.
\end{cor}
\begin{proof}
  Every such $e$ is an ascending union
  $e_1 \subset e_2 \subset \cdots$ of topological edges with stable
  endpoints.  Then $f_F^{-1}(e)$ is an ascending union of the spaces
  $f_F^{-1}(e_n)$, which are all connected by
  \Lemref{compact_edges_connected}.
\end{proof}

\begin{lem}\lemlabel{singleton_local_basis}
  Let $f \colon X \to Y$ be a continuous surjective map of compact
  metric spaces and let $p \in Y$ have singleton preimage $\{\hat p\}$
  under $f$.  Let $U$ be an open neighbourhood of $\hat p$ in $Z$.
  Then there exists an open neighborhood $V$ of $p$ such that
  $f^{-1}(V) \subset U$.
\end{lem}
\begin{proof}
  Suppose no such $V$ exists.  Then, for each $n \in \N$, there exists
  a point $p_n \in X \setminus U$ such that
  $f(p_n) \in B_{\frac{1}{n}}(p)$.  Since $X$ is compact, there is a
  convergent subsequence $(p_{n_k})_k$.  The limit $p'$ of this
  subsequence is also contained in the closed set $X \setminus U$.  In
  particular $p' \neq \hat p$.  But, by continuity of $f$, we have
  $f(p') = \lim_{k \to \infty}f(p_{n_k}) = p$, contradicting
  $f^{-1}(p) = \{\hat p\}$.
\end{proof}

\begin{prop}
  \proplabel{edge_necklace} Let
  ${\mathcal S}_\Theta=(\{ K_F^{\ast}:F\in{\mathcal F}_T^b \},\{
  f_{FF'}:F\subset F' \})$ be the standard inverse system associated
  to a tree system of punctured graphs $\Theta$ and assume that the
  limit $Z$ is connected, locally connected and has no cutpoints.  Let
  $F$ be a finite b-subtree of $T$ and let $e$ be a topological edge of
  the graph $K_F^{\ast}$.  Then the preimage $N$ in $Z$ of the set of
  stable points of $e$ is a subnecklace of $Z$.
\end{prop}
\begin{proof}
  The stable points of $K_F^{\ast}$ are contained in the interiors of
  edges so, if an endpoint of $e$ is 
  stable we can slightly enlarge
  $e$ to avoid this.  Thus, without loss of generality, the endpoints
  of $e$ are unstable.  We will start by proving that $N$ satisfies
  condition \pitmref{subnecklace_cutpair} of \Defnref{subnecklace}.
  Let $a$ and $b$ be the endpoints of $e$.  Let $p$ and $q$ be
  distinct stable points of $e$.  Then $p, q \in e \setminus \{a,b\}$.
  For points $x, y \in e$ let $(x,y)$ denote the open interval of $e$
  bounded by $x$ and $y$.  Without loss of generality, the components
  of $e \setminus \{a,b,p,q\}$ are $(a,p)$, $(p,q)$ and $(q,b)$.
  By \Corref{open_edges_conn}, the preimage
  $f_F^{-1}\bigl((p,q)\bigr)$ is connected.  So, to ensure that
  $Z \setminus \{p,q\}$ has two components it suffices to prove that
  $Y = Z \setminus \Bigl(f_F^{-1}\bigl((p,q)\bigr)\cup \{p,q\}\Bigr)$
  is connected.

  For the sake of reaching a contradiction assume that $Y$ is
  disconnected.  Then there exist open subsets $U$ and $V$ of $Z$ such
  that $Y \subset U \cup V$, $Y \cap U \neq \emptyset$ and
  $Y \cap V \neq \emptyset$.  Since $Y$ is open in $Z$ we can assume
  that $U$ and $V$ are subsets of $Y$ so that $U \sqcup V = Y$.  By
  \Corref{open_edges_conn}, the preimages $f_F^{-1}\bigl((p,b)\bigr)$
  and $f_F^{-1}\bigl((q,b)\bigr)$ are connected.  The latter preimage
  $f_F^{-1}\bigl((q,b)\bigr)$ is contained in $Y$ so, without loss of
  generality, we have $f_F^{-1}\bigl((q,b)\bigr) \subset V$.  The
  intersection $Y \cap f_F^{-1}\bigl((p,b)\bigr)$ is equal to
  $f_F^{-1}\bigl((q,b)\bigr) \subset V$ so
  $V' = V \cup f_F^{-1}\bigl((p,b)\bigr)$ is open and disjoint from
  $U$.  But then
  $U \cup V' = Y \cup f_F^{-1}\bigl((p,b)\bigr) = Z \setminus \{p\}$
  which is a contradiction since $Z$ has no cutpoints.  Thus $Y$ is
  connected.  This completes the proof of condition
  \pitmref{subnecklace_cutpair}.

  It remains to prove that $N$ satisfies condition
  \pitmref{subnecklace_cutpoint} of \Defnref{subnecklace}.  Let $p$
  be a stable point of $e$.  Then $p \in e \setminus \{a,b\}$.  Let
  $(U_n)_{n \in \N}$ be a descending sequence of open intervals of $e$
  that contain $p$ such that $\cap_n U_n = \{p\}$.  By
  \Lemref{singleton_local_basis} and \Corref{open_edges_conn}, the
  sets $\bigl(f_F^{-1}(U_n)\bigr)_n$ form a basis of connected open
  neighbourhoods of $p \in Z$.  Since $U_n \setminus \{p\}$ is
  disconnected and $f_F$ is surjective, the preimage
  $f_F^{-1}(U_n\setminus\{p\})=f_F^{-1}(U_n) \setminus \{p\}$ 
  in $Z$ must also be disconnected.
  Since $U_n \setminus \{p\}$ has two components, each of which is an
  open interval of $e$, the preimage $f_F^{-1}(U_n) \setminus \{p\}$
  must have exactly two components.  It follows that
  $Z \setminus \{p\}$ has exactly two ends, which completes the proof
  of condition \pitmref{subnecklace_cutpoint}.
\end{proof}

The next result presents some property crucial
from the perspective of our application of subnecklaces in the 
main argument of this subsection (provided at its very end). 

\begin{prop}
  \proplabel{necklace_intersection} Let $Z$ be a tree of
  graphs. Assume that $Z$ is connected, locally connected and has no
  cutpoints.  Let $A \subset Z$ be closed, connected and of
  cardinality at least $2$.  Then, for some subnecklace $N$ of $Z$,
  the intersection $A \cap N$ is uncountable.
\end{prop}
\begin{proof}
  Let
  ${\mathcal S}_\Theta=(\{ K_F^{\ast}:F\in{\mathcal F}_T^b \},\{
  f_{FF'}:F\subset F' \})$ be the standard inverse system associated
  to a tree system of punctured graphs $\Theta$ whose limit is $Z$.
  Since $A$ has cardinality at least $2$, for some finite b-subtree
  $F$, the image $f_F(A)$ in $K_F^{\ast}$ has cardinality at least
  $2$.  Since $A$ is connected, so is $f_F(A)$.  But $K_F^{\ast}$ is a
  graph so then $f_F(A)$ must contain a topological edge $e$ of
  $K_F^{\ast}$.  But $e$ contains uncountably many stable points and
  by \Propref{edge_necklace}, the preimages of these stable points in
  $Z$ form a subnecklace. Since these preimages are also easily seen
  to be contained in $A$, the assertion follows.
\end{proof}

Later in this subsection, in the main argument, we will confront the 
property of subnecklaces established above, in \Propref{necklace_intersection},
with the following property of necklaces in Gromov boundaries
of hyperbolic groups.

\begin{lem}
  \lemlabel{rigid_necklace_intersection} Let $G$ be a $1$-ended
  hyperbolic group that is not cocompact Fuchsian.  Let $H$ be a rigid
  cluster factor of $G$ and let $N$ be a necklace class of $\bd G$.
  Then $\bd H \cap N$ has cardinality at most $2$.
\end{lem}
\begin{proof}
  The closure of $N$ is the limit set $\bd K$ of a flexible factor $K$
  of $G$.  Since $H$ and $K$ are stabilizers of vertices of the
  reduced Bowditch JSJ tree $\mathcal{T}^r_G$ the intersection
  $H \cap K$ is contained in an edge stabilizer of $\mathcal{T}^r_G$.
  Thus $H \cap K$ is either finite or $2$-ended.  Since $H$ and $K$
  are both quasiconvex
  (e.g. by an observation in the paragraph right before \Corref{cluster_quasiconvex}), 
  we have $\bd H \cap \bd K = \bd (H \cap K)$ as
  subspaces of $\bd G$ \cite[Corollary to
  Theorem~13]{swenson:quasiconvex_groups:1997}.  Thus
  $\bd H \cap N \subset \bd H \cap \bd K = \bd (H \cap K)$ and so
  $\bd H \cap N$ has cardinality at most $2$.
\end{proof}

Before proving the main result, we recall two more properties of
hyperbolic groups which we need for the argument.

\begin{lem}\lemlabel{one_ended_subgroup}
  Let $G$ be a hyperbolic group and assume that $G$ is not virtually
  free (in particular, it is not finite and not virtually cyclic).  
  Then $G$ has a $1$-ended quasi-convex subgroup.
\end{lem}
\begin{proof}
  By the Dunwoody's Accessibility Theorem
  \cite{dunwoody:accessibility:1985}, there is a splitting of $G$ as a
  finite graph of groups whose edge groups are finite and whose vertex
  groups have at most $1$ end.   
  Since graph of finite groups always yields a virtually free group
  (see Proposition~11 in \cite{Serre:trees:2003}), at least one of the 
  vertex groups
  above must be 1-ended. Finally, this vertex group is quasi-convex
  e.g. by the discussion in the paragraph right before \Corref{cluster_quasiconvex}.
\end{proof}

Next rather well known result follows e.g. by 
 \cite[
  Theorem~2.28]{kapovich_benakli:boundaries:2002}.

\begin{lem}\lemlabel{bdry_not_singleton}
  Let $G$ be a hyperbolic group.  Then the boundary $\bd G$ cannot be
  a singleton.
\end{lem}

We are now ready to prove the main result of this section: the
impliciation
\pitmref{bd_g_tog}$\Rightarrow$\pitmref{g_rigid_cluster_factors_vf} of \Mainthmref{1}.  We restate
and prove the implication below.

\begin{thm}
  Let $G$ be a 1-ended hyperbolic group that is not cocompact
  Fuchsian.  If $\partial G$ is homeomorphic to a tree of graphs then
  each rigid cluster factor of $G$ is virtually free
\end{thm}

\begin{proof}
  For the sake of reaching a contradiction, assume that $\bd G$ is a
  tree of graphs and $G$ has a rigid cluster factor $H$ that is
  not virtually free.  Then, by \Lemref{one_ended_subgroup}, the
  rigid cluster factor $H$ has a $1$-ended quasi-convex subgroup $K$.
  Then, since $H$ is quasi-convex in $G$, the boundary $\bd K$ is a
  closed subspace of $\bd G$.  Since $K$ is $1$-ended, the boundary
  $\bd K$ is connected.  By \Lemref{bdry_not_singleton}, the boundary
  $\bd K$ has at least two points.  Moreover, the boundary $\bd G$ is
  connected, locally connected and has no cutpoints
  \cite[Proposition~7.1,
  Theorem~7.2]{kapovich_benakli:boundaries:2002}.  Thus we can apply
  \Propref{necklace_intersection} to the subspace
  $\bd K \subset \bd G$ to conclude that $\bd K$ has uncountable
  intersection with some subnecklace $N' \subset \bd G$.  Recalling
  \Rmkref{subnecklaces_subsets_of_necklaces}, we see that $\bd K$ has
  uncountable intersection with a necklace class $N$ of $\bd G$.
  Then, since $\bd K \subset \bd H$, we see that $\bd H$ has
  uncountable intersection with $N$.  This contradicts
  \Lemref{rigid_necklace_intersection}.
\end{proof}

\subsection{Twin graphs and trees of thick $\theta$-graphs}
\sseclabel{twin_graphs}

Recall that a \emph{$\theta$-graph} is a finite graph with two
vertices, at least one edge and no loop edges.  Note that a
$\theta$-graph is 2-connected if it has at least two edges.  We say
that a $\theta$-graph is \emph{thick} if it has at least three edges.
In this subsection we will study connectedness properties of finite
connected sums of $\theta$-graphs which we call \emph{twin graphs} and
characerize in terms of their Bowditch cut pairs.  We will use these
results in the following subsection to study trees of thick
$\theta$-graphs and to prove the implication
\pitmref{bdg_is_tree_of_theta_graphs}$\Rightarrow$\pitmref{g_no_rigid_factors}
of \Mainthmref{2}.

In the remainder of this section we treat graphs as $1$-dimensional
cellular complexes and thus do not differentiate between a graph and
its geometric realization.

\begin{defn}
  Let $X$ be a 2-connected graph.  Two points $x$ and $y$ of $X$ are
  \emph{twins} if the degree of $x$, the degree of $y$ and the number
  of components of $X \setminus \{x,y\}$ all coincide.  That is, the
  points $x$ and $y$ are distinct and they are in the same Bowditch
  class as described in \Ssecref{bowditch_jsj_splitting}.
\end{defn}

\begin{rmk}\rmklabel{twin_blowup}
  The subspace $X \setminus \{x,y\}$ has the same number of components
  as the blow-up $X^{\#}(\{x, y\})$.  Thus if
  $\iota_{x,y} \colon P_{x} \to \pi_0\bigl(X^{\#}(\{x, y\})\bigr)$ is
  the map sending each point of the blow-up divisor to its component
  in the blow-up and $\iota_{y,x}$ is defined similarly then, by
  \Lemref{pair_link_map}, the points $x$ and $y$ are twins if and only
  if $\iota_{x,y}$ and $\iota_{y,x}$ are injective.
\end{rmk}

\begin{rmk}\rmklabel{deg2_inf_twins}
  Any degree $2$ point of any 2-connected graph has infinitely many
  twins.  Indeed, any two points in the interior of the same
  topological edge are twins.
\end{rmk}

The following proposition, which justifies the terminology
\emph{twin}, is proved by Bowditch in the more general setting of
Peano continua without global cutpoints.

\begin{lem}[Bowditch {\cite[Lemma~3.8]{Bowditch:cut_points:1998}}]\lemlabel{twins_unique}
  Let $X$ be a 2-connected graph and let $x \in X$.  If
  $\deg(x) \ge 3$ then $x$ has at most one twin.
\end{lem}

\begin{defn}
  A point $x$ of a graph $X$ is an \emph{essential vertex} if it has
  degree other thant $2$.  Note that in a $2$-connected graph, any
  essential vertex has degree at least $3$.
\end{defn}

\begin{defn}
  A 2-connected graph $X$ is a \emph{twin graph} if every point in $X$
  has a twin.  If $x \in X$ is an essential vertex then we denote the
  unique twin of $x$ by $\bar x$ and we call $x$ and $\bar x$
  \emph{essential twins}.
\end{defn}

\begin{ex}
  \exlabel{theta_twin} Any 2-connected $\theta$-graph is a twin
  graph.
\end{ex}

\begin{prop}\proplabel{connected_sum_twins}
  Let $X_1$ and $X_2$ be 2-connected graphs with $x_1 \in X_1$ and
  $x_2 \in X_2$.  Let $\ell \colon P_{x_1} \to P_{x_2}$ be a bijection
  between the blow-up divisor $P_{x_1}$ of $x_1$ in $X_1$ and the
  blow-up divisor $P_{x_2}$ of $x_2$ in $X_2$.  Let $X$ be the
  connected sum $X_1 \#_{\ell} X_2$ and let
  $f_i \colon X_1 \#_{\ell} X_2 \to X_i$ be the projection, for each
  $i \in \{1,2\}$.
  \begin{enumerate}
  \item \itmlabel{twins_in_part} Let $i \in \{1,2\}$ and let
    $x, y \in X_i$ with $x \neq x_i$ and $y \neq x_i$.  Then $x$ and
    $y$ are twins in $X$ if and only if $x$ and $y$ are twins in
    $X_i$.
  \item \itmlabel{twins_bw_parts} For each $i \in \{1,2\}$, let
    $y_i \in X_i$ with $y_i \neq x_i$, let
    $\iota_{x_i,y_i} \colon P_{x_i} \to \pi_0\bigl(X_i^{\#}(\{x_i,
    y_i\})\bigr)$ be the map sending each point of the blow-up divisor
    to its component in the blow-up and let $\sigma_i$ be the
    partition of $P_{x_i}$ into fibers of $\iota_{x_i,y_i}$.  Then
    $y_1$ and $y_2$ are twins in $X$ if and only if
    $\deg(y_i) = \bigl|\pi_0\bigl(X_i \setminus \{ x_i, y_i
    \}\bigr)\bigr|$, for each $i$, and $\ell(\sigma_1) = \sigma_2$.
    In this case, for each $i \in \{1,2\}$, the inclusion
    $X_i^{\#}(\{x_i,y_i\}) \hookrightarrow X^{\#}(\{y_1,y_2\})$
    induces a bijection on connected components.
  \end{enumerate}
\end{prop}
\begin{proof}
  We begin by proving \pitmref{twins_in_part}.  By symmetry, we may
  assume that $i = 1$.  Since $X_2$ is connected and cutpoint-free,
  the blow-up $X_2^{\#}(x_2)$ is connected.  Then the restriction
  $X_1 \#_{\ell} X_2 \setminus \{x,y\} \to X_1 \setminus \{x,y\}$ of
  the projection $X_1 \#_{\ell} X_2 \to X_1$ (which is the quotient
  identifying $X_2^{\#}(x_2)$ to a point) preserves the number of
  connected components.  Then, since $x$ and $y$ have the same degrees
  in $X_1 \#_{\ell} X_2$ as they do in $X_1$, they are twins in
  $X_1 \#_{\ell} X_2$ if and only if they are twins in $X_1$.
  
  We now prove \pitmref{twins_bw_parts}.  Suppose first that
  $\deg(y_i) \neq \bigl|\pi_0\bigl(X_i \setminus \{ x_i, y_i
  \}\bigr)\bigr|$.  Without loss of generality, we can assume $i = 1$.
  Notice that $\deg(y_1)$ is equal to the cardinality of the blow-up
  divisor $P_{y_1}$ and that
  $\bigl|\pi_0\bigl(X_1 \setminus \{ x_1, y_1 \}\bigr)\bigr|$ is equal
  to the number of components of the blow-up $X_1^{\#}(\{x_1,y_1\})$.
  Then, by \Lemref{pair_link_map}, we have distinct
  $p, p' \in P_{y_1}$ that are contained in the same component $C$ of
  $X_1^{\#}(\{x_1,y_1\})$.  But $X_1^{\#}(x_1)$ is a subspace of $X$
  with $y_1 \in X_1^{\#}(x_1)$ and $y_2 \notin X_1^{\#}(x_1)$ so
  $X_1^{\#}(\{x_1,y_1\})$ is a subspace of $X^{\#}(\{y_1,y_2\})$.
  Then $C$ is a connected subspace of $X^{\#}(\{y_1,y_2\})$ containing
  $p,p' \in P_{y_1}$, which by \Rmkref{twin_blowup}, implies that
  $y_1$ and $y_2$ are not twins.

  Suppose now that $\ell(\sigma_1) \neq \sigma_2$.  Without loss of
  generality (by swapping, if necessary, the indices $i \in \{1,2\}$)
  there are $p, p' \in P_{x_1}$ contained in the same component $C_1$
  of $X_1^{\#}(\{x_1,y_1\})$ such that $\ell(p)$ and $\ell(p')$ are
  contained in distinct components $C_2$ and $C_2'$ of
  $X_2^{\#}(\{x_2,y_2\})$.  By \Lemref{pair_link_map}, there are
  $p_2, p_2' \in P_{y_2}$ with $p_2 \in C_2$ and $p_2' \in C_2'$.  We
  can view $X^{\#}(\{y_1,y_2\})$ as the connected sum of
  $X_1^{\#}(y_1)$ and $X_2^{\#}(y_2)$ at $x_1$ and $x_2$, along
  $\ell \colon P_{x_1} \to P_{x_2}$.  Then $C_1 \cup C_2 \cup C_2'$ is
  a connected subspace of $X^{\#}(\{y_1,y_2\})$ containing both $p_2$
  and $p_2'$ which, by \Rmkref{twin_blowup}, implies that $y_1$ and
  $y_2$ are not twins.  Thus we have proved the forward implication of
  \pitmref{twins_bw_parts}.

  To prove the reverse implication of \pitmref{twins_bw_parts} assume,
  for each $i$, that
  $\deg(y_i) = \bigl|\pi_0\bigl(X_i \setminus \{ x_i, y_i
  \}\bigr)\bigr|$ and that $\ell(\sigma_1) = \sigma_2$.  By
  \Rmkref{twin_blowup}, it suffices to prove that
  $X^{\#}(\{y_1,y_2\})$, $X_1^{\#}(\{ x_1, y_1 \})$ and
  $X_2^{\#}(\{ x_2, y_2 \})$ all have the same number of components.
  For components $C_1$ of $X_1^{\#}(\{ x_1, y_1 \})$ and $C_2$ of
  $X_2^{\#}(\{ x_2, y_2 \})$, let $C_1 \sim C_2$ be the intersection
  relation: $C_1 \cap C_2 \neq \emptyset$ in $X^{\#}(\{y_1,y_2\})$.
  The equality $\ell(\sigma_1) = \sigma_2$ implies that the
  intersection relation induces a bijection between
  $\pi_0\bigl(X_1^{\#}(\{ x_1, y_1 \}))$ and
  $\pi_0\bigl(X_2^{\#}(\{ x_2, y_2 \}))$.  Moreover, the $\sim$-pairs
  are in bijection with the components of $X^{\#}(\{y_1,y_2\})$ so we
  are done.  This also makes clear that
  $X_i^{\#}(\{x_i,y_i\}) \hookrightarrow X^{\#}(\{y_1,y_2\})$ is a
  bijection on components for each $i$.
\end{proof}

\begin{cor}\corlabel{connected_sum_deg2_twins}
  Let $X_1$ and $X_2$ be 2-connected graphs with $x_1,y_1 \in X_1$ and
  $x_2,y_2 \in X_2$.  Assume that the $x_i$ and $y_i$ are all distinct
  and all have degree $2$.  Let $X = X_1 \# X_2$ be a connected sum of
  $X_1$ and $X_2$ at $x_1$ and $x_2$.  Then $y_1$ and $y_2$ are twins
  in $X$ if and only if $x_i$ and $y_i$ are twins in $X_i$, for each
  $i$.  In this case, for each $i \in \{1,2\}$, the inclusion
  $X_i^{\#}(\{x_i,y_i\}) \hookrightarrow X^{\#}(\{y_1,y_2\})$ induces
  a bijection on connected components.
\end{cor}

\begin{prop}\proplabel{twin_sum}
  Any connected sum of twin graphs is a twin graph.
\end{prop}
\begin{proof}
  Let $X_1$ and $X_2$ be twin graphs with points $x_1 \in X_1$ and
  $x_2 \in X_2$.  Let $\ell \colon P_{x_1} \to P_{x_2}$ be a bijection
  of blow-up divisors.  By \Lemref{conn_sum_nice}, the connected sum
  $X = X_1 \#_{\ell} X_2$ is 2-connected so we need only show that
  every point has a twin.  By \Rmkref{deg2_inf_twins}, it suffices to
  consider an essential vertex, i.e., a point $y_1 \in X$ of degree at
  least $3$.  Then, without loss of generality, we can assume
  $y_1 \in X_1 \setminus \{x_1\} \subset X$.  Let $\bar y_1$ be the
  twin of $y_1$ in $X_1$, which is unique by \Lemref{twins_unique}.
  By \Propref{connected_sum_twins}\pitmref{twins_in_part}, if
  $\bar y_1 \neq x_1$ then $\bar y_1$ remains the twin of $y_1$ in
  $X$.  So we can assume that $\bar y_1 = x_1$ so that
  $\deg(x_2) = \deg(x_1) = \deg(y_1) \ge 3$.  Let $y_2$ be the twin of
  $x_2$ in $X_2$.  We will prove that $y_2$ is the twin of $y_1$ in
  $X$.  Since $x_1$ and $y_1$ are twins in $X_1$ and $x_2$ and $y_2$
  are twins in $X_2$, the degree condition of
  \Propref{connected_sum_twins}\pitmref{twins_bw_parts} is satisfied.
  Moreover, by \Rmkref{twin_blowup}, the fibers of $\iota_{x_i,y_i}$
  are singletons so $\ell$ preserves the partitions of the $P_{x_i}$
  into fibers.  Thus the remaining condition of
  \Propref{connected_sum_twins}\pitmref{twins_bw_parts} is satisfied
  and $y_1$ and $y_2$ are twins in $X$.
\end{proof}


  

We will need the following modified version of a result of Bowditch.

\begin{lem}[Bowditch {\cite[Lemma~3.18]{Bowditch:cut_points:1998}}]\lemlabel{bowditch_twins_not_separated}
  Let $X$ be a connected, locally connected, cutpoint-free $T_1$ space
  and let $x, x' \in X$.  Assume that $X \setminus \{x,x'\}$ has at
  least three components.  If $y, z \in X$ are contained in distinct
  components $C_y$ and $C_z$ of $X \setminus \{x, x'\}$ then
  $X \setminus \{y,z\}$ is connected.
\end{lem}
\begin{proof}
  No component $C'$ of $X \setminus \{y,z\}$ can be disjoint from
  $\{x,x'\}$.  Otherwise, by \Corref{twin_component_closure}, we would
  have $\bar C' = C' \cup \{y,z\}$ so that $C_y \cup \bar C' \cup C_z$
  would be a connected subspace of $X \setminus \{x, x'\}$
  contradicting the fact that $C_y$ and $C_z$ are distinct components
  of $X \setminus \{x, x'\}$.  Thus $X \setminus \{y,z\}$ has at most
  two components, each of which contains either $x$ or $x'$.

  Assume for the sake of finding a contradiction that
  $X \setminus \{y,z\}$ is disconnected.  Then, by the previous
  paragraph, it has exactly two components $C_x$ and $C_{x'}$ which
  contain $x$ and $x'$, respectively.  Since $X \setminus \{x, x'\}$
  has at least three components, there is a component $C$ of
  $X \setminus \{x, x'\}$ that is distinct from $C_y$ and $C_z$.  By
  \Corref{twin_component_closure}, we have $\bar C = C \cup \{x,x'\}$
  so that $C_x \cup \bar C \cup C_{x'}$ is a connected subspace of
  $X \setminus \{y, z\}$ contradicting the fact that $C_x$ and
  $C_{x'}$ are distinct components of $X \setminus \{y, z\}$.
\end{proof}

\begin{lem}
  \lemlabel{inner_twin} Let $X$ be a twin graph not homeomorphic to
  the circle.  There exists an essential twin pair $\{x, \bar x\}$ of
  $X$ such that at most a single component of
  $X \setminus \{x, \bar x\}$ contains any essential vertices of $X$.
\end{lem}
\begin{proof}
  For an essential twin pair $\{x, \bar x\}$, let $n_{\{x, \bar x\}}$
  be the maximal number of essential vertices contained in a component of
  $X \setminus \{x, \bar x\}$.  The \emph{complexity} of
  $\{x, \bar x\}$ is the nonnegative integer
  $c_{\{x, \bar x\}} = n - n_{\{x, \bar x\}} - 2$, where $n$ is the
  number of essential vertices of $X$.  Note that $c_{\{x, \bar x\}}$
  is the number of essential vertices in all but a ``largest'' component of
  $X \setminus \{x, \bar x\}$.  So it suffices to find an essential
  twin pair $\{x, \bar x\}$ for which $c_{\{x, \bar x\}} = 0$.  This
  is the base case of an induction starting with an arbitrary
  essential twin pair $\{x, \bar x\}$.  It suffices to show that if
  $c_{\{x, \bar x\}} > 0$ then we can find an essential twin pair
  $\{y, \bar y\}$ for which $c_{\{y, \bar y\}} < c_{\{x, \bar x\}}$.
  Let $C'$ be a component of $X \setminus \{x, \bar x\}$ containing
  $n_{\{x, \bar x\}}$ essential vertices.  Since
  $n_{\{x, \bar x\}} = n - c_{\{x, \bar x\}} - 2 < n - 2$, some other
  component $C$ of $X \setminus \{x, \bar x\}$ contains an essential
  vertex $y$.  We will show that
  $c_{\{y, \bar y\}} < c_{\{x, \bar x\}}$.

  By \Lemref{bowditch_twins_not_separated}, the twin $\bar y$ of $y$ is also
  contained in $C$.  By \Corref{twin_component_closure}, the closure
  $\bar C'$ of $C'$ is $C' \cup \{x, \bar x\}$.  Then
  $C' \cup \{x, \bar x\}$ is connected and so must be contained in
  some component of $X \setminus \{y, \bar y\}$.  But
  $C' \cup \{x, \bar x\}$ contains $n_{\{x, \bar x\}} + 2$ essential
  vertices so $n_{\{y, \bar y\}} \ge n_{\{x, \bar x\}} + 2$ and so
  $c_{\{y, \bar y\}} = n - n_{\{y, \bar y\}} - 2 \le n - n_{\{x, \bar
    x\}} - 4 < c_{\{x, \bar x\}}$, as required.
\end{proof}

\begin{thm}\thmlabel{twin_iff_thetasum}
  Let $X$ be a graph that is not homeomorphic to the circle.  Then $X$
  is a twin graph if and only if it can be obtained from thick
  $\theta$-graphs by finitely many connected sum operations.
\end{thm}
\begin{proof}
  The ``if'' part follows from \Propref{twin_sum} and the fact that
  thick $\theta$-graphs are twin graphs.

  To prove the ``only if'' part, suppose $X$ is a twin graph.  By
  \Lemref{inner_twin}, there exists a pair of essential twins $x$ and
  $\bar x$ of $X$ such that at most a single component $C$ of
  $X \setminus \{x, \bar x\}$ contains any essential vertices of $X$.
  It follows from \Corref{twin_component_closure} that any component
  of $X \setminus \{x, \bar x\}$ without essential vertices is a
  single open topological edge with endpoints $x$ and $\bar x$.  Thus
  if $C$ also contains no essential vertices then $X$ is a thick
  $\theta$-graph.  This is the base case of an induction on the number
  of essential vertices of $X$.

  Suppose $C$ contains at least one essential vertex.  By
  \Lemref{pair_link_map} there is a topological edge $e_x$ with
  endpoint $x$ and whose interior is contained in $C$.  Similarly,
  there is a topological edge $e_{\bar x}$ with endpoint $\bar x$ and
  whose interior is contained in $C$.  Let $p_x$ be a point in the
  interior of $e_x$ and let $p_{\bar x}$ be a point in the interior of
  $e_{\bar x}$.  Then $X^{\#}(\{p_{x}, p_{\bar x}\})$ has two components,
  one of which contains $x$ and $\bar x$ and the other of which is
  naturally a subspace of $C$.  Gluing together the pairs of points of
  the blow-up divisors
  $P_{p_x} \cup P_{p_{\bar x}} \subset X^{\#}(\{p_{x}, p_{\bar x}\})$ that are
  contained in the same component results in a disjoint union of two
  graphs: a thick $\theta$-graph $X_1$ with essential vertices $x$ and
  $\bar x$ and a graph $X_2$ with two less essential vertices than
  $X$.  Moreover, the original graph $X$ is a connected sum of $X_1$
  and $X_2$.  Thus, by induction, it suffices to show that $X_2$ is a
  twin graph.  This follows immediately from
  \Propref{connected_sum_twins}(1).
\end{proof}

\begin{lem}\lemlabel{deg2_nontwins_separated}
  Let $X$ be a twin graph and let $x,y \in X$ be distinct and of
  degree $2$.  If $X \setminus \{x,y\}$ is connected then $x$ and $y$
  are contained in distinct components of $X \setminus \{v, \bar v\}$
  for some essential twin pair $v,\bar v$.
\end{lem}
\begin{proof}
  Since $X \setminus \{x,y\}$ is connected, the twin graph $X$ cannot
  be homeomorphic to the circle.  By \Thmref{twin_iff_thetasum}, it is
  a connected sum of thick $\theta$-graphs.  We proceed by induction
  on the number of $\theta$-graph factors.  In the base case $X$ is a
  $\theta$-graph and so $x$ and $y$ must be contained in distinct
  components of $X \setminus \{v, \bar v\}$ where $\{v, \bar v\}$ is
  the unique essential twin pair of $X$.

  In the inductive step, we can write $X$ as a connected sum
  $X_1 \# X_2$ where $X_1$ is a twin graph and $X_2$ is a thick
  $\theta$-graph and the connected sum is performed at points
  $x_1 \in X_1$ and $x_2 \in X_2$.  Since performing a connected sum
  with a $\theta$-graph at an essential vertex does not change the
  topology, we can assume that $\deg(x_1)=\deg(x_2)=2$.  Up to
  homeomorphism, we can assume that $x$ and $y$ are not contained in
  the blow-up divisor $P_{x_1} = P_{x_2} \subset X_1 \# X_2$.  In fact, we can
  assume that $x$ and $y$ are not contained in
  $e \setminus \{v, \bar v, x_2\} \subset X_1 \# X_2$ where $e$ is the
  topological edge of $X_2$ containing $x_2$ and $\{v, \bar v\}$ is
  the unique essential twin pair of $X_2$.  We consider first the case
  where one of the points, say $x$, is contained in
  $X_2 \setminus e \subset X_1 \# X_2$.  Then $y$ cannot be contained
  in the topological edge of $X_2$ that contains $x$ since otherwise
  $X \setminus \{x,y\}$ would be disconnected.  Hence $x$ and $y$ are
  contained in distinct components of the complement of
  $\{v, \bar v\}$, which is a twin pair of $X$ by
  \Propref{connected_sum_twins}\pitmref{twins_in_part}.  It remains
  then to consider the case where
  $x, y \in X_1 \setminus \{x_1\} \subset X_1 \# X_2$.  By
  \Propref{connected_sum_twins}\pitmref{twins_in_part}, the degree $2$
  points $x$ and $y$ are not twins in $X_1$ and so
  $X_1 \setminus \{x,y\}$ is connected.  Then, by induction, there
  exists an essential twin pair $\{w, \bar w\}$ of $X_1$ separating
  $x$ from $y$.  By
  \Propref{connected_sum_twins}\pitmref{twins_in_part}, the pair
  $\{w, \bar w\}$ continues to be a twin pair in $X_1 \# X_2$.  Since the connected
  sum operation with the $\theta$-graph $X_2$ preserves components of
  the complement of $\{w, \bar w\}$ we see that $\{w, \bar w\}$
  continues to separate $x$ and $y$ in $X = X_1 \# X_2$.
\end{proof}

\subsection{Proof of the impliciation \pitmref{bdg_is_tree_of_theta_graphs}$\Rightarrow$\pitmref{g_no_rigid_factors}
  of \Mainthmref{2}}

We now turn our attention from twin graphs towards trees of thick
$\theta$-graphs in order to prove the implication
\pitmref{bdg_is_tree_of_theta_graphs}$\Rightarrow$\pitmref{g_no_rigid_factors}
of \Mainthmref{2}, i.e., we will show that if $G$ is a $1$-ended
hyperbolic group that is not cocompact Fuchsian and $\bd G$ is a tree
of thick $\theta$-graphs then $G$ has no rigid JSJ factors.

Recall from \Ssecref{bowditch_jsj_splitting} that the \emph{degree} of
a point $p$ in a compact space $X$ is the number of ends of the
locally compact space $X \setminus\{ p \}$.  Since $X$ is compact, any
compact subset of $X\setminus\{ p \}$ is the complement of some
open neighborhood $U$ of $p$.  If $U$ is also connected then the
degree of $p$ is at least the number of components of
$U \setminus \{p\}$ since, in this case, each component of
$U \setminus \{p\}$ contains at least one end of $X \setminus \{ p\}$.
It follows that, if $p$ has a basis $\{U_{\alpha}\}_{\alpha}$ of
connected open neighborhoods then the degree of $p$ is the supremum of
the number of components of the $U_{\alpha} \setminus \{p\}$.

\begin{lem}\lemlabel{degree_characterization}
  Let $\Theta$ be a tree system of punctured thick $\theta$-graphs,
  with underlying tree $T$.  Let $Z$ be the limit of $\Theta$ and let
  $x \in Z$.  Then $x$ has degree at least $3$ if and only if there
  exists a finite b-subtree $F_0 \subset T$ such that, for any finite
  b-subtree $F \supset F_0$, the projection of $x$ to the reduced
  partial union $K_{F}^{\ast}$ has degree at least $3$.  In this case,
  for any $F \supset F_0$, the degree of the projection of $x$ to
  $K_{F}^{\ast}$ is equal to the degree of $x$.
\end{lem}
\begin{proof}
  We first prove the contrapositive of the forward implication.
  Assume that for any finite b-subtree $F_0 \subset T$ there exists a
  finite b-subtree $F \supset F_0$ such that the projection $x_F$ of
  $x$ to the reduced partial union $K_F^{\ast}$ has degree less than
  $3$.  Let $U$ be an open neighborhood of $x$ in $Z$.  Then there
    exists a reduced partial union $K_F^{\ast}$ and an open
    neighborhood $V_F$ of the projection $x_F$ of $x$ in $K_F^{\ast}$
    such that the preimage of $V_F$ in $Z$ is contained in $U$.  By
  our assumption, we can choose $K_F^{\ast}$ so that $x_F$ has degree
  $2$ and so we can choose $V_F$ so that it is homeomorphic to an open
  interval.  By \Lemref{preimage_connected}, the preimage $V$ of $V_F$
  in $Z$ is connected.  By the paragraph preceding the statement of the
    lemma, it suffices to show that $V \setminus \{x\}$ has at most
    two components in order to establish that $x$ has degree at most
    $2$.  Now, let $K_{\hat F}^{\ast}$ be any reduced partial union
  with $\hat F \supset F$, let $V_{\hat F}$ be the preimage of $V_F$
  in $K_{\hat F}^{\ast}$ and let $x_{\hat F}$ be the projection of $x$
  to $K_{\hat F}^{\ast}$.  Then $V_{\hat F}$ is homeomorphic to the
  complement of a point of degree $2$ in a finite connected sum of
  $2$-connected graphs.  To see this, view $V_F$ as the complement of
  a point in the circle $S^1$ and then perform the connected sum
  operations required to obtain $K_{\hat F}^{\ast}$ from $K_F^{\ast}$.
  It follows then from \Lemref{pair_link_map} and
  \Rmkref{cut_pair_map} that $V_{\hat F} \setminus \{x_{\hat F}\}$ has
  at most two components.  But $V \setminus \{x\}$ is the ascending
  union of the preimages of the
  $V_{\hat F} \setminus \{x_{\hat F}\} \subset K_{\hat F}^{\ast}$ as
  $\hat F$ ranges over all finite b-subtrees containing $F$.  Hence
  $V$ itself has at most two components.  This establishes the forward
  implication.

  To prove the reverse implication, assume that $F_0$ is a finite
  b-subtree and that in any finite b-subtree $F \supset F_0$, the
  projection $x_F$ of $x$ to the reduced partial union $K_F^{\ast}$
  has degree at least $3$.  Since $K_F^{\ast}$ is obtained from
  $K_{F_0}^{\ast}$ by a sequence of connected sum operations with
  thick $\theta$-graphs and the projection of $x$ to each of these has
  degree at least $3$, we see that the degree of $x_F$ is equal to the
  degree of the projection $x_{F_0}$ of $x$ in $K_{F_0}^{\ast}$.
  Indeed, the degree is unchanged when the connected sum is performed
  away from the projection of $x$ and, if the connected sum is
  performed at the projection of $x$, then the projection of $x$ in
  the new reduced partial union can only be at the remaining essential
  vertex of the $\theta$-graph summand, which has the same degree.

  It remains to show that $\deg(x) = \deg(x_{F_0})$.  Let $U$ be an
  open neighborhood of $x$ and let $F \supset F_0$ be large enough
  that a normal neighborhood $V_F$ of the projection $x_F$ of $x$ in
  $K_F^{\ast}$ has the property that the preimage $V$ of $V_F$ in $Z$
  is contained in $U$.  By \Lemref{preimage_connected}, the preimage
  $V$ is connected.  It suffices to prove that $V \setminus \{x\}$ has
  the same number of components as $V_F \setminus \{x_F\}$, which has
  $\deg(x_{F_0})$-many components.  As above, we will describe
  $V \setminus \{x\}$ as an ascending union of preimages of spaces
  $V_{\hat F} \setminus \{x_{\hat F}\}$ where $V_{\hat F}$ is the
  preimage of $V_F$ in a reduced partial union $K_{\hat F}^{\ast}$
  with $\hat F \supset F$ and $x_{\hat F}$ is the projection of $x$ to
  $K_{\hat F}^{\ast}$.  But $V_{\hat F}$ is obtained from $V_F$ by
  performing a series of connected sum operations with thick
  $\theta$-graphs.  These connected sum operations do not change the
  topology when performed at the projection of $x$, since we are
  summing with a thick $\theta$-graph and the new projection of $x$ is
  at the remaining essential vertex of the thick $\theta$-graph.  When
  the connected sum operation is performed away from the projection of
  $x$ it does not disconnect (or join) the components of the
  complement of the projection of $x$, since we are summing with a
  $\theta$-graph and these are $2$-connected.  Thus we see that
  $V_{\hat F} \setminus \{x_{\hat F}\}$ has the same number of
  components as $V_F \setminus \{x_F\}$ and the inclusion of the
  preimage of $V_F \setminus \{x_F\}$ in
  $V_{\hat F} \setminus \{x_{\hat F}\}$ is bijective on components.
  Thus, by \Lemref{preimage_connected}, the preimages of the
  $V_{\hat F} \setminus \{x_{\hat F}\}$ in $Z$ with $\hat F$ ranging
  over a cofinal sequence of finite b-subtrees containing $F$
  expresses $V \setminus \{x\}$ as an ascending union of spaces, each
  having $\deg(x_{F_0})$-many components and where the inclusion of
  one in the next induces a bijection on components.  It follows that
  $V \setminus \{x\}$ has $\deg(x_{F_0})$-many components.
\end{proof}

\begin{lem}\lemlabel{finset_deg2}
  Let $\Theta$ be a tree system of punctured thick $\theta$-graphs,
  with underlying tree $T$.  Let $Z$ be the limit of $\Theta$, let
  $S \subset Z$ be a finite set of degree $2$ points and let
  $F_0 \subset T$ be a finite b-subtree.  Then there exists a finite
  b-subtree $F \supset F_0$ such that the projection of each $x \in S$
  to the reduced partial union $K_F^{\ast}$ has degree $2$.
\end{lem}
\begin{proof}
  We proceed by induction on $|S|$.  For the base case $|S|=0$ we can
  take $F = F_0$.  Assume now that $|S| > 0$ and $x_0 \in S$.  By
  induction, we have a finite b-subtree $F_{-} \supset F_0$ such that
  the projection of each $x \in S \setminus \{x_0\}$ to the reduced
  partial union $K_{F_{-}}^{\ast}$ has degree $2$.  By
  \Lemref{degree_characterization} there exists a finite b-subtree
  $F_{+} \supset F_{-}$ such that the projection of $x_0$ to the
  reduced partial union $K_{F_{+}}^{\ast}$ has degree $2$.  Recall
  that $K_{F_{+}}^{\ast}$ is obtained from $K_{F_{-}}^{\ast}$ by a
  series of connected sum operations, each of which corresponds to
  growing the b-subtree inside of $F_{+}$.  For a point $x \in S$,
  these operations do not change the degree of the projection of $x$
  unless the connected sum is performed at the projection.  It follows
  that if we perform only those operations that occur at the
  projection of $x_0$, we will obtain a reduced partial union
  $K_F^{\ast}$, with $F_{-} \subset F \subset F_{+}$, such that the
  degrees of projections of the points of $S \setminus \{x_0\}$ are as
  they are in $K_{F_{-}}^{\ast}$ and the degree of the projection of
  $x_0$ is as it is in $K_{F_{+}}^{\ast}$.
\end{proof}



\begin{lem}\lemlabel{deg2_pair_2_comps}
  Let $\Theta$ be a tree system of punctured thick $\theta$-graphs,
  with underlying tree $T$.  Let $Z$ be the limit of $\Theta$, let
  $x, y \in Z$ be degree $2$ points.  Then $Z \setminus \{x,y\}$ has
  at most $2$ components.
\end{lem}
\begin{proof}
  By \Lemref{finset_deg2}, there exists a cofinal sequence of
  b-subtrees $F_0 \subset F_1 \subset F_2 \subset \cdots$ such that
  $x$ and $y$ project to points $x_{F_n}$ and $y_{F_n}$ of degree $2$
  in the reduced partial unions $K_{F_n}^{\ast}$.  By
  \Rmkref{cut_pair_map}, each
  $K_{F_n}^{\ast} \setminus \{x_{F_n}, y_{F_n}\}$ has at most two
  components, each of which has connected preimage in $Z$ by
  \Lemref{preimage_connected}.  But $Z \setminus \{x,y\}$ is the
  ascending union of the preimages of the
  $K_{F_n}^{\ast} \setminus \{x_{F_n}, y_{F_n}\}$ and so has at most
  two components.
\end{proof}

\begin{lem}\lemlabel{deg2_limit_twins}
  Let $\Theta$ be a tree system of punctured thick $\theta$-graphs,
  with underlying tree $T$.  Let $Z$ be the limit of $\Theta$ and let
  $x,y \in Z$ be degree $2$ points. Then $Z \setminus \{x,y\}$ is
  connected if and only if there exists a finite b-subtree
  $F \subset T$ such that the projections $x_F$ and $y_F$ of $x$ and
  $y$ to the reduced partial union $K_F^{\ast}$ have degree $2$ and
  $K_F^{\ast} \setminus \{x_F,y_F\}$ is connected.  In this case, for
  any finite b-subtree $\hat F \supset F$ for which $x,y$ have degree
  $2$ projections $x_{\hat F},y_{\hat F}$ in the reduced partial union
  $K_{\hat F}^{\ast}$, the complement
  $K_{\hat F}^{\ast} \setminus \{x_{\hat F},y_{\hat F}\}$ is
  connected.
\end{lem}
\begin{rmk*}
  In the last sentence of \Lemref{deg2_limit_twins}, the condition
  that $x_{\hat F}$ and $y_{\hat F}$ have degree $2$ can be removed.
  We avoid proving the more general statement since we do not need it
  here.
\end{rmk*}
\begin{proof}[Proof of \Lemref{deg2_limit_twins}]
  By \Lemref{finset_deg2} there is a cofinal chain
  $F_0 \subset F_1 \subset F_2 \subset \cdots$ of finite b-subtrees of
  $T$ such that the projections $x_{F_i}$ and $y_{F_i}$ of $x$ and $y$
  have degree $2$ in the reduced partial unions $K_{F_{i}}^{\ast}$.
  Then $Z \setminus \{x,y\}$ is the ascending union of the preimages
  of the $K_{F_{i}}^{\ast} \setminus \{x_{F_i}, y_{F_i}\}$.
  
  By \Rmkref{cut_pair_map}, each
  $K_{F_i}^{\ast} \setminus \{x_{F_i}, y_{F_i}\}$ has either $1$ or
  $2$ components, corresponding to whether or not $x_{F_i}$ and
  $y_{F_i}$ are twins in $K_{F_i}^{\ast}$.  For some $i$, let $V$ be
  the set of vertices of $F_{i+1} \setminus F_i$ that are adjacent to
  $F_i$ and, for each $v \in V$, let $F_v$ be the component of
  $F_{i+1} \setminus \{v\}$ that does not contain $F_i$.  The reduced
  partial union $K_{F_{i+1}}^{\ast}$ is obtained from $K_{F_i}^{\ast}$
  by performing connected sum operations (in arbitrary order) with the
  reduced partial unions $K_{F_v}^{\ast}$ at various points of
  $K_{F_i}^{\ast}$.  By \Lemref{conn_sum_nice}, the $K_{F_v}^{\ast}$
  are all $2$-connected so the connected sum operations performed at
  points of $K_{F_i}^{\ast} \setminus \{x_{F_i},y_{F_i}\}$ preserves
  the components of the complement of the projection of $\{x,y\}$.  By
  \Corref{connected_sum_deg2_twins}, a connected sum operation
  performed at $x_{F_i}$ or $y_{F_i}$ can only reduce the number of
  components of the complement of the projection of $\{x,y\}$ and will
  otherwise preserve the components.  It follows that if
  $K_{F_{i}}^{\ast} \setminus \{x_{F_i}, y_{F_i}\}$ is connected then
  so is $K_{F_{i+1}}^{\ast} \setminus \{x_{F_{i+1}}, y_{F_{i+1}}\}$
  and that if
  $K_{F_{i+1}}^{\ast} \setminus \{x_{F_{i+1}}, y_{F_{i+1}}\}$ has $2$
  components then the inclusion of the preimage of
  $K_{F_{i}}^{\ast} \setminus \{x_{F_i}, y_{F_i}\}$ in
  $K_{F_{i+1}}^{\ast} \setminus \{x_{F_{i+1}}, y_{F_{i+1}}\}$ induces
  a bijection on components.

  We are now ready to prove the statement.  To prove the reverse
  implication, assume that there exists a finite b-subtree $F$ as in
  the statement, i.e., a finite b-subtree $F \subset T$ such that the
  projections $x_F$ and $y_F$ of $x$ and $y$ to the reduced partial
  union $K_F^{\ast}$ have degree $2$ and
  $K_F^{\ast} \setminus \{x_F,y_F\}$ is connected.  Then we can obtain
  a cofinal chain of finite b-subtrees
  $F_0 \subset F_1 \subset F_2 \subset \cdots$ as above with
  $F = F_0$.  By the above argument, the
  $K_{F_{i}}^{\ast} \setminus \{x_{F_i}, y_{F_i}\}$ are all connected.
  By \Lemref{preimage_connected}, the preimages of the
  $K_{F_{i}}^{\ast} \setminus \{x_{F_i}, y_{F_i}\}$ in $Z$ are then
  also connected so that their ascending union, which is equal to
  $Z \setminus \{x,y\}$, must also be connected.

  To prove the reverse implication, assume that no such $F$ exists as
  in the statement.  Then the
  $K_{F_{i}}^{\ast} \setminus \{x_{F_i}, y_{F_i}\}$ all have $2$
  components for a cofinal chain of finite b-subtrees
  $F_0 \subset F_1 \subset F_2 \subset \cdots$ as above.  As argued
  above, the inclusion of the preimage of
  $K_{F_{i}}^{\ast} \setminus \{x_{F_i}, y_{F_i}\}$ in
  $K_{F_{i+1}}^{\ast} \setminus \{x_{F_{i+1}}, y_{F_{i+1}}\}$ induces
  a bijection on components.  Thus the preimages of the
  $K_{F_{i}}^{\ast} \setminus \{x_{F_i}, y_{F_i}\}$ in $Z$ express
  $Z \setminus \{x,y\}$ as a disjoint union of two ascending unions of
  connected open subspaces.  It follows that $Z \setminus \{x,y\}$ has
  two components.

  The last sentence of the statement follows by the same argument
  given above showing that
  $K_{F_{i+1}}^{\ast} \setminus \{x_{F_{i+1}}, y_{F_{i+1}}\}$ has no
  more components than
  $K_{F_i}^{\ast} \setminus \{x_{F_i}, y_{F_i}\}$.
\end{proof}


\begin{lem}\lemlabel{essential_lifts}
  Let $\Theta$ be a tree system of punctured thick $\theta$-graphs,
  with underlying tree $T$.  Let $Z$ be the limit of $\Theta$.  Let
  $F_0 \subset T$ be a finite b-subtree and let $v$ be an essential
  vertex of the reduced partial union $K_{F_0}^{\ast}$.  There exists
  a unique $\tilde v \in Z$ such that
  \begin{enumerate}
  \item $\tilde v$ projects to $v$ in $K_{F_0}^{\ast}$ and
  \item for any finite b-subtree $F \supset F_0$, the projection of
    $\tilde v$ to $K_F^{\ast}$ has the same degree as $v$.
  \end{enumerate}
  If $u$ is another essential vertex of $K_{F_0}^{\ast}$ then the
  preimage in $Z$ of each component of
  $K_{F_0}^{\ast} \setminus \{u,v\}$ is contained in a component of
  $Z \setminus \{\tilde u, \tilde v\}$ and the map
  $\pi_0(K_{F_0}^{\ast} \setminus \{u,v\}) \to \pi_0(Z \setminus
  \{\tilde u, \tilde v\})$ induced by these inclusions is a bijection.
\end{lem}
\begin{proof}
  We will define $\tilde v$ as the limit of a system of lifts
  $v_F \in K_F^{\ast}$ of $v$ where $F \supset F_0$.  These lifts need
  to be consistent with respect to the bonding maps so it suffices to
  show how to lift $v_F$ to $v_{\hat F}$ for finite b-subtrees
  $\hat F \supset F \supset F_0$ for which no b-subtree lies properly
  between $\hat F$ and $F$.  Then the reduced partial union
  $K_{\hat F}^{\ast}$ is obtained from $K_F^{\ast}$ by performing a
  connected sum operation with a $\theta$-graph $K_t^{\ast}$.  In the
  case that the connected sum operation is performed away from $v_F$
  then we let $v_{\hat F} = v_F$ in $K_{\hat F}^{\ast}$.  In the case
  that the connected sum operation is performed at $v_F$, we let
  $v_{\hat F}$ be the essential vertex of $K_t^{\ast}$ that is not
  involved in the connected sum operation.  Notice that, in either
  case, $v_{\hat F}$ is the unique vertex of $K_{\hat F}^{\ast}$ of
  degree $\deg(v_F)$ that projects onto $v_F$ in $K_F^{\ast}$.  Thus
  the limit $\tilde v \in Z$ of the $v_F$ is the unique point in $Z$
  projecting to $v$ in $K_{F_0}^{\ast}$ and projecting to an essential
  vertex of degree $\deg(v_{F_0})$ in every $K_F^{\ast}$ with
  $F \supset F_0$.

  We now prove the final part of the lemma.  Let $u$ be another
  essential vertex of $K_{F_0}^{\ast}$.  As above, consider finite
  b-subtrees $\hat F \supset F \supset F_0$ for which no b-subtree
  lies properly between $\hat F$ and $F$ and decompose
  $K_{\hat F}^{\ast}$ as $K_F^{\ast} \# K_t^{\ast}$.  Since
  $\theta$-graphs are $2$-connected, if the connected sum operation is
  performed away from $u_F$ and $v_F$ then each component of
  $K_F^{\ast} \setminus \{u_F, v_F\}$ has connected preimage in
  $K_{\hat F}^{\ast} \setminus \{u_{\hat F},v_{\hat F}\} = K_{\hat
    F}^{\ast} \setminus \{u_F,v_F\}$.  Moreover, the inclusion in
  $K_{\hat F}^{\ast} \setminus \{u_F,v_F\}$ of the preimage of
  $K_F^{\ast} \setminus \{u_F, v_F\}$ is bijective on components.  If
  the connected sum operation is performed at $u_F$ or $v_F$, say at
  $v_F$, then $K_F^{\ast} \setminus \{u_F, v_F\}$ is a subspace of
  $K_{\hat F}^{\ast} \setminus \{u_{\hat F},v_{\hat F}\} = K_{\hat
    F}^{\ast} \setminus \{u_F,v_{\hat F}\}$.  The closure of
  $K_F^{\ast} \setminus \{u_F, v_F\}$ in
  $K_{\hat F}^{\ast} \setminus \{u_F,v_{\hat F}\}$ is
  $K_F^{\ast} \setminus \{u_F, v_F\} \cup P_{v_F}$ where $P_{v_F}$ is
  the blow-up divisor of $v_F$.  We see that
  $K_{\hat F}^{\ast} \setminus \{u_F,v_{\hat F}\}$ is obtained from
  $K_F^{\ast} \setminus \{u_F, v_F\} \cup P_{v_F}$ by gluing a copy of
  the half-open interval $[0,1)$ along $0$ to each point in $P_{v_F}$.
  Thus $K_{\hat F}^{\ast} \setminus \{u_F,v_{\hat F}\}$ has a
  deformation retraction onto
  $K_F^{\ast} \setminus \{u_F, v_F\} \cup P_{v_F}$ and it follows that
  the inclusion of $K_F^{\ast} \setminus \{u_F, v_F\}$ in
  $K_{\hat F}^{\ast} \setminus \{u_{\hat F},v_{\hat F}\}$ is bijective
  on components.  Let $F_0 \subset F_1 \subset F_2 \subset \cdots$ be
  a cofinal sequence of finite b-subtrees 
  such that for each $n\ge0$ no b-subtree lies properly between $F_n$
  and $F_{n+1}$.
  For each $n$, let
  $X_n = K_{F_n}^{\ast} \setminus \{u_{F_n}, v_{F_n}\}$ and let
  $\tilde X_n$ be the preimage in $Z$ of $X_n$.  Then
  $Z \setminus \{\tilde u, \tilde v\}$ is the ascending union of the
  $\tilde X_n$.  By \Lemref{preimage_connected}, the preimage in $Z$
  of any component of any $X_n$ is connected.  Consequently, the
  preimage in $Z$ of any component of
  $X_0 = K_{F_0}^{\ast} \setminus \{u, v\}$ is contained in a
  component of $Z \setminus \{\tilde u, \tilde v\}$.  By the above
  argument, the inclusion of $\tilde X_n$ in $\tilde X_{n+1}$ induces
  a bijection on components.  It follows that the map
  $\pi_0(K_{F_0}^{\ast} \setminus \{u,v\}) \to \pi_0(Z \setminus
  \{\tilde u, \tilde v\})$ sending a component $C$ of
  $K_{F_0}^{\ast} \setminus \{u, v\}$ to the component of
  $Z \setminus \{\tilde u, \tilde v\}$ that contains the preimage of
  $C$ is bijective.
\end{proof}

Recall the description of Bowditch classes in
\Ssecref{bowditch_jsj_splitting}: a pair of points $\{x, y\}$ is a
\emph{Bowditch cut pair} if $\deg(x) = \deg(y) \ge 2$ and the common
degree of $x$ and $y$ is equal to the number of connected components
in the complement of $\{x,y\}$.  The \emph{degree} of $\{x,y\}$ is
defined as the common degree of $x$ and $y$.  In the case of graphs we
have used the terminology \emph{twin pairs} for Bowditch cut pairs.
We will also consider Bowditch cut pairs in the case of trees of
punctured thick $\theta$-graphs.

\begin{cor}\corlabel{separation_by_bowditch_cutpair}
  Let $\Theta$ be a tree system of punctured thick $\theta$-graphs,
  with underlying tree $T$.  Let $Z$ be the limit of $\Theta$ and let
  $x,y \in Z$.  Assume that there exists a finite b-subtree
  $F \subset T$ such that the projections $x_F$ and $y_F$ of $x$ and
  $y$ in the reduced partial union $K_F^{\ast}$ are contained in
  distinct components of the complement
  $K_F^{\ast} \setminus \{v, \bar v\}$ of an essential twin pair
  $\{v, \bar v\}$.  Then there exists a Bowditch cut pair $\{z, z'\}$
  in $Z$ of degree $\deg(v)$ such that
  \begin{enumerate}
  \item $z$ and $z'$ project to $v$ and $\bar v$, and
  \item $x$ and $y$ are contained in distinct components of
    $Z \setminus \{z, z'\}$.
  \end{enumerate}
\end{cor}
\begin{proof}
  Let $z, z' \in Z$ be the lifts of $v$ and $\bar v$ given by
  \Lemref{essential_lifts}.  \Lemref{degree_characterization} ensures
  that $z$ and $z'$ have the same degree as $v$.  The final
  sentence of \Lemref{essential_lifts} ensures that
  $Z \setminus \{z,z'\}$ has $\deg(z)$-many components and
  separates $x$ from $y$.
\end{proof}

\begin{cor}\corlabel{deg3_pairs_project_to_twins}
  Let $\Theta$ be a tree system of punctured thick $\theta$-graphs,
  with underlying tree $T$.  Let $Z$ be the limit of $\Theta$.  Let
  $\{x,x'\}$ be a Bowditch cut pair of $Z$ of degree $d \ge 3$.
  Then there exists a finite b-subtree $F_0 \subset T$ such that, for
  any finite b-subtree $F \supset F_0$, the projections of $x$ and
  $x'$ to the reduced partial union $K_{F}^{\ast}$ are twins of
  degree $d$.
\end{cor}
\begin{proof}
  By \Lemref{degree_characterization}, there exists a finite b-subtree
  $F_0 \subset T$ such that, for any finite b-subtree $F \supset F_0$,
  the projections $x_F$ and $x'_F$ of $x$ and $x'$ in the
  reduced partial union $K_F^{\ast}$ have degree $d$.  Then $x$ and
  $x'$ are the canonical lifts of $x_F$ and $x'_F$ given by
  \Lemref{essential_lifts}.  It follows from the last sentence of
  \Lemref{essential_lifts} that
  $K_F^{\ast} \setminus \{x_F, x'_F\}$ has the same number of
  components, $d$, as does $Z \setminus \{x, x'\}$ and so $x_F$
  and $x'_F$ are twins in $K_F^{\ast}$.
\end{proof}

\begin{lem}\lemlabel{separating_bowditch_pairs}
  Let $\Theta$ be a tree system of punctured thick $\theta$-graphs,
  with underlying tree $T$.  Let $Z$ be the limit of $\Theta$, and let
  $\{x,x'\}$ and $\{y, y'\}$ be disjoint Bowditch cut pairs of
  $Z$. 
  If
  $\{x, x'\}$ has degree at least $3$ and $\{y, y'\}$ has
  degree at least $3$ then there exists a Bowditch cut pair
  $\{z, z'\}$ of degree $2$ such that $\{x, x'\}$ and
  $\{y, y'\}$ are contained in distinct components of
  $Z \setminus \{z, z'\}$.
\end{lem}
\begin{proof}
  It follows from \Corref{deg3_pairs_project_to_twins} that there exists a finite
  b-subtree $F_0 \subset T$ such that, for any finite b-subtree
  $F \supset F_0$, the projection $\{x_F, x'_F\}$ of
  $\{x, x'\}$ in the reduced partial union $K_F^{\ast}$ is a twin
  pair of the same degree as $\{x, x'\}$ and the projection
  $\{y_F, y'_F\}$ of $\{y, y'\}$ in $K_F^{\ast}$ is a twin
  pair of the same degree as $\{y, y'\}$.  Then
  $x, x', y, y'$ are the canonical lifts of
  $x_F, x'_F, y_F, y'_F$ given by \Lemref{essential_lifts} and
  so $\{x_F, x'_F\}$ and $\{y_F, y'_F\}$ must be distinct essential twin
  pairs of $K_F^{\ast}$.  Recall that $K_F^{\ast}$ is obtained from
  thick $\theta$-graphs by a series of connected sum operations.  By
  first performing the connected sum operations at essential vertices,
  we can assume that the connected sum operations take place only at
  degree $2$ points so that, by
  \Propref{connected_sum_twins}\pitmref{twins_in_part}, the connected
  sum operations preserve essential twins.  Then $\{x_F, x'_F\}$
  and $\{y_F, y'_F\}$ correspond to essential twin pairs in distinct thick
  $\theta$-graph summands and, consequently, some blow-up divisor
  $\{q_1, q_2\} = P_q \subset K_F^{\ast}$ of a degree $2$ point $q$ at
  which a connected sum operation was performed separates
  $\{x_F, x'_F\}$ and $\{y_F, y'_F\}$ into distinct
  components.  Since connected sums of thick $\theta$-graphs are
  2-connected, the complement $K_F^{\ast} \setminus \{q_1, q_2\}$ has
  two components.  Since $q_1$ and $q_2$ are stable points (i.e. no
  connected sum operation is performed at these points to obtain a
  larger reduced partial union) they have singleton preimages
  $\{\tilde q_1\}$ and $\{\tilde q_2\}$ in $Z$ and
  $Z \setminus \{\tilde q_1, \tilde q_2\}$ is the preimage in $Z$ of
  $K_F^{\ast} \setminus \{q_1, q_2\}$.  Then
  \Lemref{preimage_connected} implies that
  $Z \setminus \{\tilde q_1, \tilde q_2\}$ has $2$ components and
  \Lemref{degree_characterization} implies that $\tilde q_1$ and
  $\tilde q_2$ have degree $2$ so that we can take
  $\{z, z'\} = \{\tilde q_1, \tilde q_2\}$.
\end{proof}

Recall from \Ssecref{bowditch_jsj_splitting} that the reflexive
closure of the Bowditch cut pair relation is an equivalence relation.
We call the equivalence classes \emph{Bowditch classes}.  Recall that
a Bowditch class $\zeta$ of degree at least $3$ has cardinality
$|\zeta| = 2$, i.e., consists of a single Bowditch cut pair.

We are now ready to prove the main result of this section: the
implication
\pitmref{bdg_is_tree_of_theta_graphs}$\Rightarrow$\pitmref{g_no_rigid_factors}
of \Mainthmref{2}.  We restate and prove the implication below.

\begin{thm}
  Let $G$ be a $1$-ended hyperbolic group which is not cocompact
  Fuchsian.  If $\bd G$ is homeomorphic to a tree of thick
  $\theta$-graphs then $G$ has no rigid factor in its (Bowditch) JSJ
  splitting.
\end{thm}
\begin{proof}
  Let $T$ be the underlying tree of the tree system of 
  punctured graphs whose
  limit is homeomorphic to $\bd G$.  We identify $\bd G$ with this
  limit throughout the proof.
  
  Recall from \Ssecref{bowditch_jsj_splitting} that a rigid factor of
  $G$ corresponds to a vertex of type (v3) in the Bowditch JSJ
  tree $T_G$.  Such a vertex corresponds to an infinite maximal family
  $\mathcal{F}$ of pairwise not separated Bowditch classes of $\bd G$.
  So for any two classes $\zeta, \eta \in \mathcal{F}$ there does not
  exist a Bowditch class $\xi$ and $x,y \in \xi$ such that $\zeta$ and
  $\eta$ are contained in distinct components of
  $\bd G \setminus \{x,y\}$.

  We split the rest of the proof into the following two cases:
  I. $\mathcal{F}$ has two distinct classes $\zeta$ and $\eta$ of
  degree at least $3$; II. $\mathcal{F}$ has two distinct classes
  $\zeta$ and $\eta$ of degree $2$.

  \paragraph*{\textbf{Case I. $\mathcal{F}$ has two distinct classes $\zeta$
    and $\eta$ of degree at least $3$}}
  Then $\zeta$ and $\eta$ each consists of a single Bowditch cut pair
  of degree at least $3$.  Then, by
  \Lemref{separating_bowditch_pairs}, there is a Bowditch cut pair
  $\{z, z'\}$ of degree $2$ such that $\zeta$ and $\eta$ are
  contained in distinct components of $\bd G \setminus \{z, z'\}$,
  contradicting the fact that $\zeta$ and $\eta$ are not separated.

  \paragraph*{\textbf{Case II. $\mathcal{F}$ has two distinct classes $\zeta$ and
    $\eta$ of degree $2$}}
  Take $x \in \zeta$ an $y \in \eta$.  By \Lemref{deg2_pair_2_comps},
  there are at most two components of $\bd G \setminus \{x,y\}$.  But
  $\zeta$ and $\eta$ are distinct Bowditch classes so the points $x$
  and $y$ do not form a Bowditch cut pair.  Thus
  $\bd G \setminus \{x,y\}$ does not have two components and so it
  must be connected.  By \Lemref{deg2_limit_twins}, there is a finite
  b-subtree $F \subset T$ for which the projections $x_F$ and $y_F$ of
  $x$ and $y$ in the reduced partial union $K_F^{\ast}$ have degree
  $2$ and the complement $K_F^{\ast} \setminus \{x_F, y_F\}$ is
  connected.  By \Thmref{twin_iff_thetasum}, the reduced partial sum
  $K_F^{\ast}$ is a twin graph.  By \Lemref{deg2_nontwins_separated},
  there is an essential twin pair $v, \bar v$ such that $x_F$ and
  $y_F$ are contained in distinct components of
  $K_F^{\ast} \setminus \{v, \bar v\}$.  Then, by
  \Corref{separation_by_bowditch_cutpair}, the points $x$ and $y$ are
  contained in distinct components $C_x$ and $C_y$ of the complement
  $\bd G \setminus \{z, z'\}$ of a Bowditch cut pair
  $\{z, z'\}$ in $\bd G$ of degree $\deg(v) \ge 3$.  Since moreover, by
  \Lemref{bowditch_twins_not_separated}, we have $\zeta \subset C_x$
  and $\eta \subset C_y$, this yields a contradiction with the fact that $\zeta$ and
  $\eta$ are not separated.
\end{proof}

\section{Regular trees of graphs}
\seclabel{reg_tree}

In this section we describe a class of tree systems of punctured
graphs which depend uniquely on certain finite collections of defining
data (which we call \emph{graphical connecting systems}, see
\Defnref{R.1} below).  We call the limits of such tree systems the
\emph{regular trees of graphs}.  Although we do not justify this 
in the present paper (since this is not essential for our results), 
the regular trees of graphs belong to a larger
class of spaces called Markov compacta, which are 
described by certain recursive
procedures based on some finite collections of data (see
\cite{Dranishnikov:cohomol_markov:2007,Pawlik:boundaries_markov_compacta:2015,Czapp:constructive:2022}). Since,
due to the main result from
\cite{Pawlik:boundaries_markov_compacta:2015},
Gromov boundaries of
hyperbolic groups are all homeomorphic to Markov compacta, the regular
trees of graphs are good candidates for spaces that can be potentially
realized as Gromov boundaries of some hyperbolic groups.  Actually,
the main goal of the present paper is to characterize those 1-ended
hyperbolic groups $G$ whose Gromov boundaries $\bd G$ are
(homeomorphic to some) trees of graphs, and to show that these trees
of graphs are then regular.

Recall that in this paper we allow graphs to have multiple edges and
loop edges.  Given a graph $\Gamma$ and its vertex $v$, \emph{the
  link of $\Gamma$ at $v$}, denoted $\Lk(v)$, is the set of oriented
edges of $\Gamma$ issuing from $v$. Equivalently, $\Lk(v)$ can be
described as the blow-up divisor $P_v$ of
  the blow-up $|\Gamma|^\#(v)$ at $v$
  (as described in \Ssecref{tog_as_inv_lims_of_graphs}).

Given an oriented edge $\epsilon$ of a graph $\Gamma$, we denote by
$\bar\epsilon$ the same edge taken with the opposite orientation;
we denote by $|\varepsilon|$ the unoriented edge which underlies
the oriented edge $\varepsilon$.

\begin{defn} \defnlabel{R.1} A \emph{graphical connecting system}
is a triple $\mathcal{R}=(\Gamma, \mathbf{a}, \mathcal{A})$ such that:
  \begin{enumerate}[label=(\arabic*)]
  \item \itmlabel{R.1(1)} $\Gamma=\sqcup_{j=1}^k\Gamma_j$ is the
    disjoint union of a finite nonempty collection of connected finite
    graphs $\Gamma_j$, each of which is nonempty and not equal to a
    single vertex (in particular, $k\ge1$);
  \item \itmlabel{R.1(2)}
    $\mathbf{a}=(a,\{ \alpha_v \}_{v\in V_\Gamma})$ is a pair
    consisting of an involution $a:V_\Gamma\to V_\Gamma$ of the vertex
    set of $\Gamma$ which preserves degrees of vertices, and of
    bijections $\alpha_v:\Lk(v)\to\Lk(a(v))$ such that for any
    $v\in V_\Gamma$ we have $\alpha_{a(v)}=\alpha_v^{-1}$; we will call $\mathbf{a}$ the $V$-\emph{involution} 
      of $\mathcal{R}$;
  \item \itmlabel{R.1(3)} $\mathcal{A}$ is a set of pairs
    $(\epsilon_1,\epsilon_2)$ of (not necessarily distinct) oriented
    edges of $\Gamma$ closed under the involutions
    $(\epsilon_1,\epsilon_2)\to(\epsilon_2,\epsilon_1)$ and
    $(\epsilon_1,\epsilon_2)\to(\bar\epsilon_1,\bar\epsilon_2)$, and
    such that for any oriented edge $\epsilon$ of $\Gamma$ there is at
    least one pair in $\mathcal{A}$ containing $\epsilon$; we will call
    $\mathcal{A}$ the set of \emph{$E$-connections} of
    $\mathcal{R}$;
  \item \itmlabel{R.1(4)} we require also the following
    \emph{transitivity} property: for any two distinct components
    $\Gamma_i,\Gamma_j$ of $\Gamma$ there is a sequence
    $j(0),j(1),\dots,j(m)$ of indices from $\{ 1,\dots,k\}$ such that
    $j(0)=i$, $j(m)=j$, and for each $1\le n\le m$ at least one of the
    following two conditions holds:
    \begin{enumerate}[label=(\roman*)]
    \item \itmlabel{R.1(4)(i)} for some $v\in V_{\Gamma_{j(n-1)}}$
      we have $a(v)\in V_{\Gamma_{j(n)}}$;
    \item \itmlabel{R.1(4)(ii)} for some oriented edge
      $\epsilon$ of $\Gamma_{j(n-1)}$ there is an oriented edge
      $\epsilon'$ of $\Gamma_{j(n)}$ such that
      $(\epsilon,\epsilon')\in{\mathcal A}$.
    \end{enumerate}
  \end{enumerate}
\end{defn}

Before presenting the next crucial definition (\Defnref{6.2})
we need some further terminological and notational preparations.
Let ${\mathcal N}_\Gamma$ be a standard family of normal
neighbourhoods of $\Gamma$, as described in \Exref{T.1.5}, and let
$\Gamma^\circ:=|\Gamma|\setminus\bigcup\{ \interior B:B\in{\mathcal
  N}_\Gamma \}$ be the space with peripherals, called the punctured
graph $\Gamma$, as described in the same example.  Denote by
$\Omega^E_\Gamma= \{ {\rm bd}(B):B\in{\mathcal N}_\Gamma, B\cap
V_\Gamma=\emptyset \}$ the set of those peripherals in $\Gamma^\circ$
which correspond to those $B\in{\mathcal N}_\Gamma$  
which are normal neighbourhoods
of interior points of the edges of
$\Gamma$. Similarly, put
$\Omega_\Gamma^V=\{ {\rm bd}(B):B\in{\mathcal N}_\Gamma, B\cap
V_\Gamma\ne\emptyset \}$.  We will call the elements of the above sets
\emph{E-peripherals} and \emph{V-peripherals},
respectively. Observe that each E-peripheral is a set of cardinality
2, i.e., a doubleton. Each such set can be equipped with one of the two
possible orders, so that it becomes an ordered pair of points, and we
will call any such order an \emph{orientation} associated to this
E-peripheral.  We denote by $\widetilde\Omega^E_\Gamma$ the set of all
oriented E-peripherals of $\Gamma^\circ$, i.e., the set
$\{ (p,q):\{ p,q \}\in\Omega^E_\Gamma \}$.  Observe that each of the
oriented E-peripherals may be either compatible or incompatible with
the fixed orientation of the edge $e$ of $\Gamma$ in the interior of
which it is contained. 
Actually, in the definition below we use the convention 
of identifying orientations of E-peripherals of $\Gamma^\circ$
with the corresponding compatible orientations of the edges of $\Gamma$
containing those peripherals.
For any oriented peripheral
$\delta\in\widetilde\Omega^E_\Gamma$, we denote by $|\delta|$ the
corresponding underlying (unoriented) peripheral, and by $\bar\delta$
the same peripheral $|\delta|$ oriented oppositely to $\delta$.

\begin{defn}\defnlabel{6.2}
Given a graphical connecting system $\mathcal{R}=(\Gamma, \mathbf{a}, \mathcal{A})$,
with $\Gamma$ split into connected components $\Gamma_j$,
we say that a tree system of punctured graphs
$\Theta=(T, \{ K_t \}, \{ \Sigma_u \}, \{ \varphi_e \})$
is \emph{compatible} with $\mathcal{R}$,
if the following conditions are satisfied: 

\begin{enumerate}

\item
$T$ is the unique tree bipartitioned into black and white vertices such that
the degree of any white vertex is 2, and the degree of any black vertex is countable infinite;

\item
for any black vertex $t$ of $T$ there is an index $j(t)$ such that 
$K_t\cong\Gamma^\circ_{j(t)}$
(i.e. $K_t$ is identified with $\Gamma^\circ_{j(t)}$, as a space with peripherals);

\item
for each white vertex $u$ of $T$, if $t_1,t_2$ are its (black) neighbours, then
\begin{enumerate}
\item
if the image $\varphi_{[u,t_1]}(\Sigma_u)$ is a V-peripheral of 
$K_{t_1}\cong\Gamma^\circ_{j(t_1)}$, and if it corresponds to a vertex 
$v\in V_{\Gamma_{j(t_1)}}$, then $\Gamma_{j(t_2)}$ is the component of
$\Gamma$ which contains the vertex $a(v)$, and the image
$\varphi_{[u,t_2]}(\Sigma_u)$
is the V-peripheral of $K_{t_2}\cong\Gamma^\circ_{j(t_2)}$
corresponding to $a(v)$; moreover, the composition map 
$\varphi_{[u,t_2]}\circ\varphi_{[u,t_1]}^{-1}$ coincides with
the bijection $\alpha_v$ (after natural identifications of the vertex links
with the sets of elements of the V-peripherals, at the vertices $v$ and $a(v)$);
\item
if $\varphi_{[u,t_1]}(\Sigma_u)$ is an E-peripheral of 
$K_{t_1}\cong\Gamma^\circ_{j(t_1)}$, then $\varphi_{[u,t_2]}(\Sigma_u)$ 
is an E-peripheral of 
$K_{t_2}\cong\Gamma^\circ_{j(t_2)}$;
moreover, if $e_i$ is the edge of $\Gamma$ containing 
$\varphi_{[u,t_i]}(\Sigma_u)$, for $i=1,2$, and if $\varepsilon_1$ is any oriented edge
with $|\varepsilon_1|=e_1$, then denoting by $\varepsilon_2$ the oriented edge
related to $e_2$ whose orientation is induced from that of $\varepsilon_1$
via $\varphi_{[u,t_2]}\circ\varphi_{[u,t_1]}^{-1}$, we have 
$(\varepsilon_1,\varepsilon_2)\in\mathcal{A}$;
 \end{enumerate}

\item
for any black vertex $t$ of $T$ 
and 
for any oriented edge $\varepsilon$ of $\Gamma_{j(t)}$, denote by 
$\widetilde\Omega(\varepsilon)$ the family of those oriented E-peripherals in
$K_t\cong\Gamma^\circ_{j(t)}$ which are contained in $|\varepsilon|$, and which
are oriented consistently with the orientation of $\varepsilon$;
put also $O_\varepsilon:=\{ \varepsilon':(\varepsilon,\varepsilon')\in\mathcal{A} \}$, and
consider the map 
$f_\varepsilon:\widetilde\Omega(\varepsilon)\to O_\varepsilon$
described as follows:
if
$\omega\in\widetilde\Omega(\varepsilon)$ and $u$ is the white vertex
of $T$ for which $|\omega|=\varphi_{[u,t]}(\Sigma_u)$ and $t'$ is the
black vertex of $T$ adjacent to $u$ and distinct from $t$ then
$f_\varepsilon(\omega)$ is the oriented edge $\varepsilon'$ of
$\Gamma_{j(t')}$ containing $\varphi_{[u,t']}(\Sigma_u)$ and oriented
in the way induced by the map
$\varphi_{[u,t']}\circ\varphi_{[u,t]}^{-1}$ from the orientation of
$\varepsilon$; it is required that for each
$\varepsilon'\in O_\varepsilon$ the union
$\bigcup\{ |\omega|: \omega\in\widetilde\Omega(\varepsilon) \hbox{ and
}f_\varepsilon(\omega)=\varepsilon'\}$ is dense (and in particular
nonempty) in $|\varepsilon|\cap\Gamma^\circ$.

\end{enumerate}

\end{defn}

Though technically involved, the above definitions 
(i.e. Definitions \defnref{R.1} and \defnref{6.2}) are in fact quite natural,
and their significance stems from the following lemma, in view of which
to any graphical connecting system $\mathcal{R}$ there is associated
some unique up to isomorphism tree system of punctured graphs
(which we denote $\Theta[{\mathcal{R}}]$),
and some unique up to
homeomorphism topological space having the form of a tree of graphs
(which we denote $\mathcal{X}(\mathcal{R})$, and which is given as 
$\mathcal{X}(\mathcal{R}):=\lim\Theta[{\mathcal R}]$).

\begin{lem}\lemlabel{6.3}
For any graphical connecting system $\mathcal{R}$ a tree system of punctured graphs
compatible with $\mathcal{R}$ exists. Moreover, any two tree systems of punctured
graphs compatible with $\mathcal{R}$ are isomorphic.
\end{lem}

\begin{proof}
Given any graphical connecting system 
$\mathcal{R}=(\Gamma , \mathbf{a}, \mathcal{A})$,
we describe a tree system $\Theta[{\mathcal R}]$ of
punctured graphs 
which is by construction compatible with $\mathcal{R}$. 
After this rather long construction we provide also some less detailed hints
for concluding the uniqueness up to isomorphism.

Recall that we denote by $\widetilde\Omega_\Gamma^E$ the set of
all oriented E-peripherals in the punctured graph $\Gamma^\circ$.
We will now describe some auxiliary labelling $\lambda$ of the family
$\widetilde\Omega_\Gamma^E$.  The labels of this labelling will be in
the set $O_\Gamma$ of all oriented edges of $\Gamma$.  
Observe that,
by condition \itmref{R.1(3)} of \Defnref{R.1}, 
for any oriented edge $\epsilon$ of
$\Gamma$
the corresponding set $O_\varepsilon$ (as defined in the statement of
condition (4) of \Defnref{6.2}) is
nonempty.  Take as $\lambda$ any labelling of
$\widetilde\Omega_\Gamma^E$ such that:
\begin{enumerate}[label=(L\arabic*)]
\item \itmlabel{(L1)} for any $\epsilon\in O_\Gamma$, and for
  any oriented peripheral $\delta$ contained in the interior of
  $|\epsilon|$ whose orientation is compatible with that of
  $\epsilon$, the label of $\delta$ is in the subset $O_\epsilon$,
  i.e., $\lambda(\delta)\in O_\epsilon$;
\item \itmlabel{(L2)} for any
  $\delta\in\widetilde\Omega_\Gamma^E$, we have
  $\lambda(\bar\delta)= \overline{\lambda(\delta)}$;
\item \itmlabel{(L3)} for each $\epsilon\in O_\Gamma$ and any
  $\epsilon'\in O_\epsilon$ the union of the family
  \[
    \{ |\delta|\in\Omega^E_\Gamma:
    \text{$|\delta|\subset\interior |\epsilon|, \delta$ is oriented compatibly
      with $\epsilon$, and $\lambda(\delta)=\epsilon'$}\}
  \]
  is dense in $|\epsilon|\cap\Gamma^\circ$.
\end{enumerate}

It is not hard to realize (and we omit the details) that a labelling
$\lambda$ as above exists, and is unique in the sense of the first
part of the following claim.  We omit the proof, which is a standard
back-and-forth argument.

\begin{claim} \claimlabel{R.2} For any two labellings
  $\lambda,\lambda'$ as above, satisfying conditions
  \itmref{(L1)}--\itmref{(L3)}, there is a homeomorphism
  $h:|\Gamma|\to |\Gamma|$ which fixes all vertices of $\Gamma$, which
  preserves all edges of $\Gamma$, which preserves the family
  ${\mathcal N}_\Gamma$, and whose restriction
  $h^\circ=h|_{\Gamma^\circ}$ is a homeomorphism of $\Gamma^\circ$
  such that 
  $\lambda'(h^\circ(\delta))=\lambda(\delta)$ for any
  $\delta\in\widetilde\Omega_\Gamma^E$ (where $h^\circ(\delta)$ is
  viewed as an oriented E-peripheral with orientation induced from
  $\delta$ via $h^\circ$).  Moreover, if
  $\delta,\delta'\in\widetilde\Omega^E_\Gamma$ are contained in the
  interior of the same edge of $\Gamma$, if they are consistently
  oriented, and if $\lambda(\delta)=\lambda'(\delta')$, then $h$ above
  can be chosen so that $h^\circ(\delta)=\delta'$.
\end{claim}

We are now ready to describe the tree system $\Theta[{\mathcal R}]=(T,
\{ \Gamma_t^\circ \}, \{ \Sigma_u \}, \{ \phi_e \} )$.  As the
underlying tree $T$ take the unique bipartite tree, with the vertex
set partitioned into black and white vertices, such that the degree of
any white vertex is equal to 2, and the degree of any black vertex is
countable infinite.  Now, we describe resursively the punctured graphs
$\Gamma_t^\circ$ associated to the black vertices of $T$. Let $t_0$ be
any black vertex of $T$ (which we will view as the initial
vertex). Put $\Gamma_{t_0}$ to be any connected component $\Gamma_{i}$
of $\Gamma$.  Fix also some auxiliary bijection $\beta_{t_0}$ between
the set $N_{t_0}$ of all (white) vertices of $T$ adjacent to $t_0$ and
the set of peripherals in the punctured graph $\Gamma_{t_0}^\circ$.
(We will be choosing appropriate bijections $\beta_t:t\in V^b_T$
simultaneously with the description of the constituent spaces
$\Gamma_t^\circ$; these bijections will play an essential role at
later stages of the construction of the tree system $\Theta[{\mathcal
R}]$.)  Now, for each $k\ge0$, denote by $T_k$ the b-subtree of $T$
spanned by all vertices of $T$ lying at distance $\le 2k$
from $t_0$.  Suppose that for all $t\in V^b_{T_k}$ we have already
chosen the punctured graphs $\Gamma_t^\circ$ and the corresponding
bijections $\beta_t$ from the sets $N_t$ (of all neighbours of $t$ in
the tree $T$) to the sets of peripherals in $\Gamma_t^\circ$, so that
the following conditions hold:
\begin{enumerate}[label=(R\arabic*)]
\item \itmlabel{(R1)} each $\Gamma_t$ is a copy of one of the
  components of $\Gamma$, and we denote this component by
  $\Gamma_{i(t)}\subset\Gamma$;
\item \itmlabel{(R2)} for any two black vertices $t,t'$ of $T_k$
  lying at combinatorial distance 2, denoting by $u$ the unique common
  neighbour in $T$ of those two vertices, identifying the constituent
  spaces $\Gamma_t^\circ$, $\Gamma_{t'}^\circ$ with the corresponding
  punctured components $\Gamma_{i(t)}^\circ$,
  $\Gamma_{i(t')}^\circ \subset\Gamma^\circ$, viewing them naturally
  as subspaces in $\Gamma^\circ$, and viewing $\lambda$ as the
  labelling of E-peripherals in both $\Gamma_t^\circ$ and
  $\Gamma_{t'}^\circ$ (via the above mentioned identifications), we
  have
  \begin{enumerate}
  \item \itmlabel{(R2)(a)} if $\beta_t(u)$ is an E-peripheral
    contained in the interior of an edge $|\epsilon|$ of $\Gamma_t$
    (where $\epsilon\in O_\Gamma$) then $\beta_{t'}(u)$ is also an
    E-peripheral; moreover, there is a choice of orientation on the
    edge containing the peripheral $\beta_{t'}(u)$, with which we
    denote this edge by $\epsilon'$, such that if we orient
    $\beta_t(u)$ and $\beta_{t'}(u)$ consistently with $\epsilon$ and
    $\epsilon'$, thus getting the oriented peripherals
    $\delta,\delta'$, respectively, we have
    $\lambda(\delta)=\epsilon'$ and $\lambda(\delta')=\epsilon$;
  \item \itmlabel{(R2)(b)} if $\beta_t(u)$ is a V-peripheral, i.e.,
    $\beta_t(u)={\rm bd}(B)$ for $B$ a normal neighborhood of a vertex
    $v$ of $\Gamma_{i(t)}$ then $\beta_{t'}(u)$ is also a
    V-peripheral; moreover, if $v'$ is the corresponding vertex of
    $\Gamma_{i(t')}$ for $\beta_{t'}(u)$, then, viewing $v,v'$ as
    vertices of $\Gamma$, we have $a(v)=v'$.
  \end{enumerate}
\end{enumerate}

\noindent As the main recursive step of the description of constituent
spaces, we describe now the spaces $\Gamma_t^\circ$, and the
associated bijections $\beta_t$, for all vertices
$t\in V^b_{T_{k+1}}\setminus V^b_{T_k}$, so that the above conditions
\itmref{(R1)} and \itmref{(R2)} will be satisfied with $k$ replaced by
$k+1$. Consider any vertex $t\in V^b_{T_{k+1}}\setminus V^b_{T_k}$,
and denote by $s$ the unique vertex in $V^b_{T_k}$ at combinatorial
distance 2 from $t$, and by $u$ the unique common neighbour of both
$t$ and $s$ (which is a white vertex). Suppose first that the subspace
$\beta_s(u)\subset\Gamma_s^\circ$ is an E-peripheral, i.e., it is
contained in the interior of a single edge, say $|\epsilon|$, of
$\Gamma_s$, and denote by $\delta$ this same peripheral oriented
consistently with $\epsilon$.  Recall that, under our identification
conventions, the oriented peripheral $\delta$ has a label, and denote
this label by $\epsilon'$. Put $\Gamma_t$ to be a copy of this component of
$\Gamma$ which contains $|\epsilon'|$, and denote this component by
$\Gamma_{i(t)}$. Take as $\beta_t$ any bijection from $N_t$ to the set
of peripherals of $\Gamma_t^\circ$ for which $\beta_t(u)$ is contained
in the interior of $|\epsilon'|$ and, after orienting it consistently
with $\epsilon'$, it is labeled with $\epsilon$.  Suppose now that
$\beta_s(u)$ is a V-peripheral, and denote by $v$ this vertex of
$\Gamma_s$ for which $\beta_s(u)=\partial B$, where $B$ is a normal
neighbourhood of $v$.  Identifying $\Gamma_s$ with the component
$\Gamma_{i(s)}$ of $\Gamma$, put $v':=a(v)$. Denote by $\Gamma_{i(t)}$
this component of $\Gamma$ which contains $v'$, and put
$\Gamma_t=\Gamma_{i(t)}$.  Finally, take as $\beta_t$ any bijection
from $N_t$ to the set of peripherals of $\Gamma_t^\circ$ for which
$\beta_t(u)=\partial B'$, where $B'$ is a normal neighbourhood of
$v'$.  We skip a straightforward verification that this provides data
as required.

We now turn to describing the sets $\Sigma_u$ and the maps $\phi_e$ of
$\Theta[{\mathcal R}]$. Let $u$ be any white vertex of $T$, and denote
by $t,s$ the two black vertices adjacent to $u$. Suppose first that
$\beta_t(u)$ is an E-peripheral of $\Gamma_t^\circ$, i.e., it is
contained in the interior of some edge, say $|\epsilon|$, of
$\Gamma_t$.  Observe that, by the construction, the peripheral
$\beta_s(u)$ is then also contained in the interior of an edge of the
corresponding graph $\Gamma_s$, and denote this edge by
$|\epsilon'|$. Then $\epsilon$ and $\epsilon'$ induce orientations on
$\beta_t(u)$ and $\beta_s(u)$ giving us oriented E-peripherals
$\delta$ and $\delta'$, respectively.  Without loss of generality
(changing $\epsilon'$ into $\overline{\epsilon'}$ and $\delta'$ into
$\overline{\delta'}$, if necessary) we have
$\lambda(\delta)=\epsilon'$ and $\lambda(\delta')=\epsilon$.  Put
$\Sigma_u$ to be a set consisting of two points, with discrete
topology, and consider an auxiliary order on this set.  Put
$\phi_{[t,u]}:\Sigma_u\to|\delta|$ and
$\phi_{[s,u]}:\Sigma_u\to|\delta'|$ to be the bijections which
preserve the distinguished orders.

Finally, suppose that $\beta_t(u)$ is a V-peripheral, and denote by
$v$ this vertex of $\Gamma_t$ for which $\beta_t(u)=\partial B$, where
$B$ is a normal neighbourhood of $v$. Similarly, denote by $v'$ this
vertex of $\Gamma_s$ for which $\beta_s(u)=\partial B'$, where $B'$ is
a normal neighbourhood of $v'$. Note that, by (R2)(b), we have
$a(v)=v'$, and that, identifying $\Gamma_t$ and $\Gamma_s$ with the
corresponding components of $\Gamma$, we have a bijection
$\hat\alpha_v:\beta_t(u)\to\beta_s(u)$ naturally induced by the
bijection $\alpha_v$. Take as $\Sigma_u$ any set of the same
cardinality as $\beta_t(u)$ (and as $\beta_s(u)$), with discrete
topology. Take as the maps $\phi_{[u,t]}$ and $\phi_{[u,s]}$ any
bijections for which the composition
$\phi_{[u,s]}\circ\phi_{[u,t]}^{-1}$ coincides with $\hat\alpha_v$.

Having completed the construction of $\Theta[\mathcal{R}]$,
we skip the details of a verification that this system is compatible with $\mathcal{R}$.
We also note that an isomorphism between any two tree systems compatible
with $\mathcal{R}$, as in the second assertion of the lemma,
can be constructed recursively, in a straightforward way,
using \Claimref{R.2} (and more precisely, its final assertion). To
start such recursive construction, one needs to refer to the
transitivity condition \itmref{R.1(4)} of \Defnref{R.1}. We omit
further details.
\end{proof}


\begin{ex} \exlabel{R.4} For any finite connected graph $\Gamma$, let
  ${\mathcal R}^r_\Gamma=(\Gamma, \mathbf{a}, \mathcal{A})$ be a
  graphical connecting system in which $a={\rm id}_{V_\Gamma}$,
  $\alpha_v={\rm id}_{\Lk(v)}$ for each $v\in V_\Gamma$, and
  ${\mathcal A}=\{ (\epsilon,\epsilon):\epsilon\in O_\Gamma \}$, where
  $O_\Gamma$ denotes the set of all oriented edges of $\Gamma$.  We
  call ${\mathcal R}^r_\Gamma$ \emph{the reflection connecting system
    for $\Gamma$}. The associated tree systems
  $\Theta[{\mathcal R}^r_\Gamma]$ lead to a class of regular trees of
  graphs (see \Defnref{R.5} below) which are called \emph{reflection
    trees of graphs} and which have been described and studied earlier
  in \cite{Swiatkowski:refl_trees_of_graphs:2021}.
\end{ex}

\begin{defn} \defnlabel{R.5} A \emph{regular tree of graphs} is any
space of the form ${\mathcal X}({\mathcal R})=\lim\Theta[{\mathcal R}]$, 
where $\mathcal R$ is a
graphical connecting system,
and where $\Theta[{\mathcal R}]$ is the (unique up to an isomorphism)
tree system of punctured graphs compatible with $\mathcal R$. 
Moreover, we call any space  ${\mathcal X}({\mathcal R})$
as above
\emph{the regular tree of graphs for} $\mathcal{R}$.
\end{defn}

\begin{defn} \defnlabel{R.6} A graphical connecting system ${\mathcal
R}=(\Gamma, \mathbf{a}, \mathcal{A})$ is \emph{2-connected} if all the
connected components $\Gamma_j$ of $\Gamma$ are 2-connected.
\end{defn}

Note that, if $\mathcal R$ is 2-connected, then obviously the
corresponding system $\Theta[{\mathcal R}]$ is a tree system of
punctured 2-connected graphs, and consequently, the regular tree of
graphs ${\mathcal X}(\mathcal R)$ is then connected, locally
connected, cutpoint-free and has topological dimension 1 (see
\Lemref{limit_very_connected} and \Lemref{G.5} above).  We call the
spaces ${\mathcal X}(\mathcal R)$ as above \emph{the regular trees of
  2-connected graphs}.

\section{$V$-trees of graphs and their decompositions into core and
  arms}
\seclabel{V}

In this section we begin preparations for the proof of the implication \pitmref{g_rigid_cluster_factors_vf}$\Rightarrow$\pitmref{bd_g_reg_twocon_tog}
in \Mainthmref{1} 
(and the implication (3)$\Rightarrow$(2) of \Mainthmref{2}). The V-tree systems as discussed in this section will appear later as
subsystems in some special regular tree systems of graphs
that we will be considering, and their closer understanding is thus
a first step in understanding the latter, and in relating the latter with Gromov
boundaries of hyperbolic groups that are under our interest.

\begin{defn}[connecting $V$-system]
  \defnlabel{V.1}
A \emph{connecting $V$-system} is a pair ${\mathcal V}=(\Gamma, {\mathbf a})$ such that:
\begin{enumerate}
\item[(1)] $\Gamma$ is a finite graph which is essential (i.e. is nonempty and has no isolated vertex), 
and we assume it has no loop edges; 
\item[(2)] ${\mathbf a}=(a,\{ \alpha_v \}_{v\in V_\Gamma})$
is a pair consisting
of an involution $a:V_\Gamma\to V_\Gamma$ of the vertex set of $\Gamma$
which preserves degrees of vertices, and of bijections
$\alpha_v:\hbox{Lk}(v)\to\hbox{Lk}(a(v))$ such that for any $v\in V_\Gamma$
we have 
$\alpha_{a(v)}=\alpha_v^{-1}$,
where links are taken in $\Gamma$.
\end{enumerate}
\end{defn}

A connecting $V$-system is thus a structure consisting of slightly less data
than a graphical connecting system, as described in \Defnref{R.1},
and we insist on graphs $\Gamma$ to have a more restrictive form (no loop edges).
Intuitively, a $V$-system structure contains only the information about connected sum operations
at vertices of the involved copies of graphs, as encoded in the $V$-involution 
$\mathbf a$. Although in our later applications graphs $\Gamma$
of $V$-systems will be typically connected, we do not need to assume this
here, in the description of the general concept.

We now pass to describing a certain tree system $\Theta[{\mathcal V}]=\Theta[\Gamma,{\mathbf a}]$ canonically
associated to a connecting $V$-system $\mathcal V$. Its limit $\lim\Theta[{\mathcal V}]$
will be called \emph{the $V$-tree of graphs for $\mathcal V$}, and it will be denoted
${\mathcal X}_V({\mathcal V})$, or ${\mathcal X}_V(\Gamma,{\mathbf a})$, 
where the role of the subscript $V$ is just to indicate
the $V$-tree nature of the resulting space. Actually, the $V$-tree 
${\mathcal X}_V({\mathcal V})$ will be natuarally a space with peripherals,
and the description of a family of peripherals will be a part of its definition.

Similarly as in \Secref{reg_tree}, we consider a standard family ${\mathcal N}_\Gamma$
of normal neighbourhoods in $\Gamma$, as described in 
Example \exref{T.1.5},
and the space with peripherals $\Gamma^\circ$, called the punctured graph
$\Gamma$, described in the same example. Consider also the set 
$\Omega^V_\Gamma$ of $V$-peripherals of $\Gamma^\circ$,
as described right after \Defnref{R.1}.

As the underlying tree $T$ of the $V$-tree system 
$\Theta[{\mathcal V}]=(T, \{ \Gamma_t^\circ \}, \{ \Sigma_u \}, \{ \phi_e \})$
we take the unique tree, with the vertex set bipartitioned into black
and white vertices, such that the degree of any white vertex is equal to 2,
and the degree of any black vertex coincides with the (finite) cardinality of the
vertex set $V_\Gamma$ of $\Gamma$.
The system $\Theta[{\mathcal V}]$ is then described uniquely up to an
isomorphism of tree systems by the following requirements:

\begin{enumerate}
\item[(v1)] for each black vertex $t$ of $T$ the constituent space
with peripherals
at $t$ is equal to the pair $(\Gamma^\circ,\Omega^V_\Gamma)$;
this means in particular that the vertices in $N_t$ (i.e. the vertices adjacent to $t$)
are in fixed bijective correspondence with the family $\Omega_\Gamma^V$;
\item[(v2)] for each white vertex $u$ of $T$, denote by $e,e'$
the two edges of $T$ adjacent to $u$, and by $P_e,P_{e'}$ the $V$-peripherals
of $\Gamma_{b(e)}^\circ$  and $\Gamma_{b(e')}^\circ$ associated to $u$
via the bijective correspondences mentioned above in condition (v1);
then,
denoting by $v,v'$ the vertices of $\Gamma$ corresponding to 
the $V$-peripherals $P_e,P_{e'}$, respectively,
we have $a(v)=v'$; moreover,
 the set $\Sigma_u$ and the bijections
$\phi_e:\Sigma_u\to P_e,\phi_{e'}:\Sigma_u\to P_{e'}$ 
are given so that the composition 
$\phi_{e'}\circ\phi_e^{-1}:P_e\to P_{e'}$
coincides with the map $\alpha_v$,
where we refer to the natural identifications $P_e={\rm Lk}(v,\Gamma)$
and $P_{e'}={\rm Lk}(v',\Gamma)$.
\end{enumerate}

\noindent
We skip a straightforward justification of the fact that a tree system
satisfying the above requirements exists and is unique up to
an isomorphism of tree systems (the argument is similar to that presented
in \Secref{reg_tree} for the tree system $\Theta[{\mathcal R}]$;
see the proof of \Lemref{6.3}).

Now, let ${\mathcal X}_V({\mathcal V}):=\lim\Theta[{\mathcal V}]$. 
We describe the
family $\Omega_V({\mathcal V})$ of natural peripheral subsets for ${\mathcal X}_V({\mathcal V})$.
Let $P$ be any $E$-peripheral in any constituent space $\Gamma^\circ_t$,
$t\in V^b_T$.
Clearly, $P$ is then naturally a subset of the limit $\lim\Theta[{\mathcal V}]$.
Moreover, for any two distinct $E$-peripherals $P,P'$ as above, 
the induced subsets of $\lim\Theta[{\mathcal V}]$ are disjoint. 
It is also not hard
to realize (by deducing it e.g. from Fact \factref{T.4} in view of 
Fact \factref{T.5}) that the family $\Omega_V({\mathcal V})$
of all such subsets $P\subset\lim\Theta[{\mathcal V}]$,
for all $t\in V^b_T$, is null, and thus we take it 
as the family of natural peripherals for ${\mathcal X}_V({\mathcal V})$.

\begin{ex}[$V$-tree of segments]
  \exlabel{V.2}
A graph which coincides with a single edge will be called \emph{the segment},
and we donote this graph by $I$. The \emph{punctured segment} $I^\circ$
is then a space with peripherals in which two of the peripheral sets are 
singletons, while the others are doubletons. If we delete from the set
of peripherals of $I^\circ$ the two singleton peripherals, we get a new
space with peripherals 
(with the same underlying space, but smaller family of peripherals)
which we call \emph{the internally punctured segment},
and which we denote $I^\circ_{\rm int}$. Alternatively, $I^\circ_{\rm int}$
can be described as obtained by deleting from $I$ the interiors of a dense and null
family of pairwise disjoint normal neighbourhoods of points lying in 
the interior of $I$, and by taking the boundaries of those neighbourhoods
as peripherals.

Observe that there are exactly two essentially distinct connecting $V$-systems
${\mathcal V}=(\Gamma,{\mathbf a})$ with $\Gamma=I$: the first one with the involution
$a:V_I\to V_I$ equal to the identity of the vertex set $V_I$, 
and the second one with $a$ equal
to the transposition of the doubleton $V_I$.
In both cases the resulting $V$-tree ${\mathcal X}_V({\mathcal V})$ will be called a
\emph{$V$-tree of segments}. Moreover, we obviously have the following.
\end{ex}

\begin{fact}
  \factlabel{V.3}
  Any $V$-tree of segments is, as a space with peripherals,
  homeomorphic to the internally punctured segment
  $I^\circ_{\rm int}$.
\end{fact}

We now pass to describing some natural decomposition of any $V$-tree of graphs
${\mathcal X}_V({\mathcal V})$ into convenient pieces. 
These pieces will be called the \emph{core} and
the \emph{arms} 
of ${\mathcal X}_V({\mathcal V})$,
and they will have a form of spaces with peripherals. We will show (in \Lemref{V.9})
that decomposition into these pieces has a tree-like nature,
i.e., that the space ${\mathcal X}_V({\mathcal V})$ is canonically homeomorphic to a
certain tree of these pieces. This new expression of ${\mathcal X}_V({\mathcal V})$ as a tree
of spaces does not coincide with the original expression of this space 
as a $V$-tree of graphs, and this reinterpretation of the spaces of the form 
${\mathcal X}_V({\mathcal V})$ will be one of our tools in the proof of
the implications \pitmref{g_rigid_cluster_factors_vf}$\Rightarrow$\pitmref{bd_g_reg_twocon_tog}
in \Mainthmref{1} and (3)$\Rightarrow$(2) in \Mainthmref{2}.

Before describing the pieces and the decomposition,
we need to introduce more terminology. 
A \emph{punctured edge} of a punctured graph $\Gamma^\circ$ is the intersection
of $\Gamma^\circ$ (viewed as subspace of $\Gamma$) with any edge $e$ of 
$\Gamma$, and we denote such a punctured edge by $e^\circ$. 
Obviously, since we assume that $\Gamma$ contains no loop edges, each
$e^\circ$ naturally has the structure of a space with peripherals
with which it is homeomorphic to the punctured segment $I^\circ$.
The \emph{endpoints} of a punctured edge $e^\circ$ are the two points of 
$e^\circ$ which correspond to the two singleton peripherals of this space.
The set of all endpoints in all punctured edges of $\Gamma^\circ$
is clearly in a natural bijective correspondence with the union of all vertex links
of the graph $\Gamma$, and we will identify each such endpoint with
the corresponding point in the appropriate vertex link.
Now, given a $V$-tree system 
$\Theta[{\mathcal V}]=(T, \{ \Gamma_t^\circ \}, \{ \Sigma_u \}, \{ \phi_e \})$,
we say that
punctured edges $e_1^\circ$ of $\Gamma_{t_1}^\circ$ and $e_2^\circ$ of 
$\Gamma_{t_2}^\circ$ are \emph{$\Theta$-adjacent}
if some endpoints $x_1,x_2$ in $e_1^\circ,e_2^\circ$ get identified
in the limit space ${\mathcal X}_V({\mathcal V})$, i.e.,
if $t_1,t_2$ are distinct black vartices of $T$ with common adjacent white vertex,
which we denote $u$,  and if we have
$x_i\in \phi_{[t_i,u]}(\Sigma_u)$, for $i=1,2$, and
$\phi_{[t_1,u]}^{-1}(x_1)=\phi_{[t_2,u]}^{-1}(x_2)$.
We then say that $e_1^\circ$ and $e_2^\circ$ are $\Theta$-adjacent
\emph{at their endpoints $x_1,x_2$}.
Note that if $t_1,t_2,u$ are as above, and if $e_1^\circ$ is a punctured
edge of $\Gamma_{t_1}^\circ$ having an endpoint 
$x_1\in\varphi_{[t_1,u]}(\Sigma_u)$, then there is precisely one
punctured edge $e_2^\circ$ in $\Gamma_{t_2}^\circ$ which is
$\Theta$-adjacent to $e_1^\circ$.

\begin{defn}[{line in $\Theta[{\mathcal V}]$}]
  \defnlabel{V.4}
  A \emph{line} in a $V$-tree system $\Theta[{\mathcal V}]$ is a bi-infinite sequence
$(e_n^\circ)_{n\in \Z}$ of punctured edges from the punctured graphs
$\Gamma_t^\circ:t\in V^b_T$ such that for each $n$ the edge $e_n^\circ$
is $\Theta$-adjacent to $e_{n-1}^\circ$ at one of its endpoints, 
and it is $\Theta$-adjacent 
to $e_{n+1}^\circ$
at its other endpoint.
We do not distinguish sequences as above which differ by a shift 
or by inversion of the indices (i.e. we view two such sequences as representing the same line). 
\end{defn}

Obviously, since $\Gamma$ has no loop edges,
any punctured edge in any constituent graph $\Gamma_t^\circ$
belongs to a unique line in $\Theta[{\mathcal V}]$.
It is also not hard to observe that in each $V$-tree system $\Theta[{\mathcal V}]$
for which $\Gamma$ has more than two vertices,
there are many (actually countably infinitely many) lines. 
We denote the set of all lines in $\Theta[{\mathcal V}]$ by 
${\mathcal L}({\mathcal V})$.
Since each term $e_n^\circ$ of any line $L$ is a punctured edge in some copy $\Gamma_t$
of the graph $\Gamma$, we will often refer to it as an \emph{appearance} in $L$
of the corresponding (unpunctured) edge of $\Gamma$.

Let $L=(e^\circ_n)_{n\in \Z}$ be a line in $\Theta[{\mathcal V}]$, and for each $n$
denote by $t_n$ this black vertex of $T$ for which $e_n^\circ$ is a punctured
edge of $\Gamma_{t_n}^\circ$, and by $u_n$ the unique white vertex of $T$
adjacent to both $t_{n-1}$ and $t_n$. The bi-infinite sequence
\[\dots,u_n,t_n,u_{n+1},t_{n+1},\dots\]
is then a (bi-infinite) combinatorial geodesic in $T$, which we denote by
$\gamma_L$ and call \emph{the geodesic induced by $L$.}
(We view such geodesics as subsets in the vertex set $V_T$ rather than as
sequences, so that a shift of the indices or reversal of the order of terms
yields the same geodesic.)
We make a record of the following properties of lines and their induced geodesics,
whose straightforward proofs we omit. (These proofs are based on the easily seen 
observation that lines,
viewed as sequences of appearances of edges from $\Gamma$, are periodic.)

\begin{lem}
  \lemlabel{V.4.1}
  \begin{enumerate}
\item[(1)]
Given two distinct lines $L_1\ne L_2$ in $\Theta[{\mathcal V}]$,
their corresponding geodesics $\gamma_{L_1},\gamma_{L_2}$ either
coincide, or have bounded intersection in $T$ (including the case
of empty intersection). In particular, the doubleton subsets 
$\partial\gamma_{L_1},\partial\gamma_{L_2}$ in $\partial T$ either coincide or are disjoint.
\item[(2)]
There is a universal finite bound on the cardinality of a set of lines $L$
in $\Theta[{\mathcal V}]$
which share the corresponding doubletons $\partial\gamma_L$.
\end{enumerate}
\end{lem}

Consider the equivalence relation $\sim_\Theta$ on the set of edges of the graph $\Gamma$
described as follows: we have $e\sim_\Theta e'$
iff $e$ and $e'$ have appearances in the same line $L\in{\mathcal L}({\mathcal V})$.
Note that this equivalence relation is well defined since if some two lines $L_1,L_2$ contain
appearances of the same edge of $\Gamma$, then any edge of $\Gamma$ appears in $L_1$
if and only if it appears in $L_2$. Thus, we get a partition of the edge set $E_\Gamma$
into classes of the relation $\sim_\Theta$.

\begin{defn}[arm of ${\mathcal X}_V({\mathcal V})$]
  \defnlabel{V.5}
An \emph{arm} in a $V$-tree of graphs ${\mathcal X}_V({\mathcal V})$ is its subspace
$A_L$ of the following form. Given any line $L\in{\mathcal L}({\mathcal V})$,
$L=(e_n^\circ)_{n\in Z}$, take as $A_L$ the union of the set
$\partial\gamma_L\subset\partial T$ (consisting of the two ends 
of the geodesic $\gamma_L$)
and the image in $\#\Theta[{\mathcal V}]$ of 
the union $\bigcup_{n\in Z}e_n^\circ$.
We view $A_L$ as a space with peripherals, where the family of peripherals
consists of only one peripheral set, namely $\partial\gamma_L$.
We distinguish also in $A_L$ some other subsets, namely the ones corresponding
to all doubleton peripherals in all punctured edges $e_n^\circ$ constituting $A_L$;
we call these subsets the \emph{supplementary peripherals} of $A_L$, in contrast with
the above mentioned subset $\partial\gamma_L$, which we call the \emph{basic
peripheral} of $A_L$.
\end{defn}

For each arm $A_L$ we consider the subspace topology induced from
${\mathcal X}_V({\mathcal V})$. It is not hard to observe 
by referring to Fact \factref{T.4} that, when equipped with this
topology, any arm $A_L$ naturally gets the form of a $V$-tree of segments.
In particular, by Fact \factref{V.3}, we obtain the following.

\begin{fact}
  \factlabel{V.6} Any arm $A_L$, as a space equipped with the family
  of all supplementary peripherals, is homeomorphic to the internally
  punctured segment $I^\circ_{\rm int}$. Moreover, its basic
  peripheral subset $\partial\gamma_L$ coincides then with the set of
  endpoints of $I^\circ_{\rm int}$.
\end{fact}

\begin{defn}[core of ${\mathcal X}_V({\mathcal V})$]
  \defnlabel{V.7}
  The \emph{core} of ${\mathcal X}_V({\mathcal V})=\lim\Theta[{\mathcal V}]$ is the subspace 
$C({\mathcal V})$
corresponding to the
subset $\partial T\subset\lim\Theta[{\mathcal V}]$. It is equipped (as part of its 
structure) with the family of doubleton subsets
$\Omega^c_{\mathcal V}:=\{ \partial\gamma_L:L\in{\mathcal L}
({\mathcal V}) \}$
(with which it is a space with peripherals, see Remark \rmkref{V.8}.2 below).
\end{defn}

\begin{rmk}
  \rmklabel{V.8}
\begin{enumerate}
\item[1.] Using the fact that the tree $T$ is locally finite, it is not hard to see
that the subspace topology in the subset $\partial T\subset\lim\Theta[{\mathcal V}]$ coincides with the standard topology in $\partial T$, and hence $\partial T$
is a Cantor set (except the elementary cases when $\Gamma$ has two vertices, in which
$\partial T$ is a doubleton).
\item[2.] The family $\Omega^c_{\mathcal V}$ of subsets of the core is null.
To see this, it is sufficient to
note that the distances in $T$ from a fixed
base point to the geodesics $\gamma_L$  diverge to $\infty$. This, in turn, easily follows from the fact
that only finitely many geodesics of the form $\gamma_L$ passes
through any vertex of $T$, and the fact that $T$ is locally finite.
The family $\Omega^c_{\mathcal V}$ consists also of sets that are pairwise disjoint,
see \Lemref{V.4.1}(1).
\item[3.] The sets in the family $\Omega^c_{\mathcal V}$, viewed as subsets
of the $V$-tree ${\mathcal X}_V({\mathcal V})$, clearly coincide with the basic peripherals in all arms of this $V$-tree.
\end{enumerate}
\end{rmk}

\begin{ex}[core and arm in a $V$-tree of segments]
  \exlabel{V.8.1}
Note that for any connecting $V$-system of segments ${\mathcal V}=(I,{\mathbf a})$
(as described in \Exref{V.2}), there is precisely one line $L$ in this system.
The $V$-tree of segments ${\mathcal X}_V({\mathcal V})$ has then precisely one arm, $A_L$,
and the core $C({\mathcal V})$ is a doubleton which coincides with the boundary
$\partial\gamma_L$ and with the set of endpoints of the arm $A_L$.
Actually, as a space with peripherals the $V$-tree ${\mathcal X}_V({\mathcal V})$
coincides with the arm $A_L$ equipped with its supplementary peripherals,
and the basic peripheral of $A_L$ coincides with the core $C({\mathcal V})$.
\end{ex}

We now generalize V-trees of segments to V-trees of $\theta$-graphs.
We start with describing the $\theta$-graphs---a family of graphs which
appear explicitly in the statement of \Mainthmref{2}, and which also appear implicitly
in \Mainthmref{1}, since the regular trees of graphs mentioned in conditions (2) and (3)
of this result necessarily involve certain $\theta$-graphs 
(compare \Thmref{S.1}
and the description of the graphical connecting system ${\mathcal R}_G$
in \Secref{S}).

\begin{ex}[V-tree of $\theta$-graphs]
  \exlabel{V.8.2}
A $\theta$-\emph{graph}  is a finite connected graph $\Xi$ which has precisely two vertices
and no loop edges. The number of edges is thus a natural number $k\ge1$, and each edge
has both vertices of $\Xi$ as its endpoints. We will denote the $\theta$-graph with $k$ edges
by $\Xi_k$. Obviously, $\Xi_1$ is simply the segment, as described in \Exref{V.2}.

Generalizing the exposition of \Exref{V.2}, we define \emph{the internally punctured
$\theta$-graph $(\Xi_k)^\circ_{{\rm int}}$} as the space with peripherals 
obtained from $\Xi_k$
by deleting the interiors of any dense family of pairwise disjoint 
normal neighbourhoods of non-vertex points
of $\Xi_k$, and by taking the boundaries of those neighbourhoods as peripherals.
Such a space is obviously unique up to a homeomorphism respecting peripherals,
its peripherals are all doubletons, and vertices of $\Xi_k$ belong to this space,
and are not contained in any peripheral. 

Observe that any connecting $V$-system ${\mathcal V}=(\Xi_k,{\mathbf a})$ of $\theta$-graphs
has precisely $k$ lines $L_1,\dots,L_k$. Moreover, the arms $A_1,\dots,A_k$ of the
corresponding $V$-tree of $\theta$-graphs ${\mathcal X}_V({\mathcal V})$ have a common pair
of endpoints, and this pair coincides with the core $C({\mathcal V})$
of this $V$-tree. The set of peripherals of ${\mathcal X}_V({\mathcal V})$ coincides with the union
of the sets of supplementary peripherals of the arms $A_i$. Thus, as a space with peripherals,
${\mathcal X}_V({\mathcal V})$ is homeomorphic to the internally punctured $\theta$-graph 
$(\Xi_k)^\circ_{\rm int}$.
The core $C({\mathcal V})$ of this $V$-tree coincides then with the vertex set of $\Xi_k$
(viewed as the subset of $(\Xi_k)^\circ_{\rm int}$). The only peripheral of this core
coincides with the whole of the core.

For our later purposes, we now describe some special connecting $V$-system 
${\mathcal V}_{\Xi_k}=(\Xi_k,{\mathbf a}_k)$, which we call \emph{the standard connecting $V$-system
for $\Xi_k$}. The $V$-involution ${\mathbf a}_k=(a_k,\{ \alpha_u, \alpha_v \})$
consists of the transposition $a_k$ of the two vertices $u,v$ of $\Xi_k$,
and both $V$-maps $\alpha_u:\hbox{Lk}(u)\to\hbox{Lk}(v)$, $\alpha_v:\hbox{Lk}(v)\to\hbox{Lk}(u)$
are ``tautological'' in the following sense: for any element $p$ in the link $\hbox{Lk}(u)$ or
$\hbox{Lk}(v)$ its image through the corresponding $V$-map 
coincides with this element in the other
vertex link which is induced by the same edge of $\Xi_k$ as $p$.
Observe that all $k$ lines in so described standard system ${\mathcal V}_{\Xi_k}$
consist of appearances of a single edge of the graph $\Xi_k$, and those edges are distinct
for distict lines.
\end{ex}

Coming back to general connecting $V$-systems $\mathcal V$,
we now describe the tree system $\Psi[{\mathcal V}]$ consisting of the core and the arms of the $V$-tree ${\mathcal X}_V({\mathcal V})$ as the constituent spaces of this system. The underlying tree $T_\Psi$
of $\Psi[{\mathcal V}]$
is described as follows. There are three kinds of vertices in $T_\Psi$:
\begin{enumerate}
\item[(1)] a single ``central'' black vertex $t_c$ corresponding to the core of ${\mathcal X}_V({\mathcal V})$;
\item[(2)] black vertices $t_A$ corresponding to arms $A$ of 
${\mathcal X}_V({\mathcal V})$;
\item[(3)] white vertices $u_P$ corresponding to peripherals 
$P\in\Omega^c_{\mathcal V}$.
\end{enumerate}

\noindent
The set of edges is described as follows.
For each of the white vertices $u_P$ there is an edge $[u_P,t_c]$
in $T_\Psi$. For each arm $A=A_L$, putting $P=\partial\gamma_L$,
there is an edge $[t_A,u_P]$ in $T_\Psi$. Note that, due to \Lemref{V.4.1}(2),
the tree $T_\Psi$ has finite degree at each white vertex $u_P$.

As the constituent spaces for $\Psi[{\mathcal V}]$ we take
$K_{t_c}=C({\mathcal V})$ and $K_{t_A}=A$. As the peripheral spaces
for $\Psi[{\mathcal V}]$ we take $\Sigma_{u_P}=P$, and as the maps 
$\phi_{[u,t]}:\Sigma_u\to K_t$ we take the natural inclusions
of the forms $P\subset C[{\mathcal V}]$ and $\partial\gamma_L\subset A_L$.

We have the following expression of the space ${\mathcal X}_V({\mathcal V})$ as a tree of spaces.

\begin{lem}
  \lemlabel{V.9} The inclusions of the form
  $A_L\subset {\mathcal X}_V({\mathcal V})$, for all
  $L\in{\mathcal L}({\mathcal V})$, and
  $C[{\mathcal V}]\subset {\mathcal X}_V({\mathcal V})$ induce a well
  defined map
  $i_{\mathcal V}:\lim\Psi[{\mathcal V}]\to {\mathcal X}_V({\mathcal
    V})$, and this map is a homeomorphism.
\end{lem}
\begin{proof}
The induced map $i_{\mathcal V}$ is obviously a bijection. 
Since both spaces $\lim\Psi[{\mathcal V}]$ and ${\mathcal X}_V({\mathcal V})$ are compact
and metrizable,
it is enough to show that the map $i_{\mathcal V}$ is continuous.
We leave it to the reader as a straightforward exercise concerning the topology
on limits of tree systems. 
\end{proof}

\begin{ex}[wedge of arms]
  \exlabel{V.9.1}
We present here some alternative expression of the $\mathcal V$-tree of $\theta$-graphs
${\mathcal X}_{V}[{\mathcal V}_{\Xi_k}]$ as tree of arms. This expression, a bit 
simpler than that provided by the tree system $\Psi[{\mathcal V}_{\Xi_k}]$, 
will be used in the proof of 
the implication \pitmref{g_rigid_cluster_factors_vf}$\Rightarrow$\pitmref{bd_g_reg_twocon_tog}
of \Mainthmref{1} given in \Secref{P}.

Consider the tree system $\Psi_w[{\mathcal V}_{\Xi_k}]$ described as follows.
The underlying tree $T_w$ of $\Psi_w[{\mathcal V}_{\Xi_k}]$ consists of a single
white vertex $u_c$ and precisely $k$ black vertices $t_A$ corresponding to the arms
$A$ of ${\mathcal X}_V({\mathcal V}_{\Xi_k})$, and of precisely $k$ edges of the form $[u_c,t_A]$.
As the constituent spaces for $\Psi_w[{\mathcal V}_{\Xi_k}]$ we take $K_{t_A}=A$,
and as the peripheral space we take $\Sigma_{u_c}=C[{\mathcal V}_{\Xi_k}]$.
As the maps $\phi_{[u_c,t_A]}:\Sigma_{u_c}\to K_{t_A}$
we take the natural inclusions of the form $C[{\mathcal V}_{\Xi_k}]\subset A$.

Since, by an easy observation, we have a natural identification 
$\lim\Psi_w[{\mathcal V}_{\Xi_k}]=\lim\Psi[{\mathcal V}_{\Xi_k}]$,
it follows then from \Lemref{V.9}
 that the inclusions 
$A\subset {\mathcal X}_{V}({\mathcal V}_{\Xi_k})$ induce a natural identification of the limit space
$\lim\Psi_w[{\mathcal V}_{\Xi_k}]$ with the $V$-tree ${\mathcal X}_{V}({\mathcal V}_{\Xi_k})$.
\end{ex}

As the last subject in this section, we discuss orientability of arms in V-trees of graphs.

\begin{defn}[orientable line and arm]
  \defnlabel{V.10}
A line $L$ in a V-tree system $\Theta[{\mathcal V}]$ is \emph{orientable} if any two appearances
$e_k^\circ$, $e_m^\circ$ in $L$ of any edge $e$ of $\Gamma$ are consistently oriented.
The latter means that in the corresponding arm $A_L$, viewed as an internally punctured segment,
the orientations of subsegments $e_k^\circ$ and $e_m^\circ$ induced from a fixed orientation
of $e$ are consistent.  A line $L$ is \emph{non-orientable}, if some edge $e$ of $\Gamma$ has some
two inconsistently oriented appearances in $L$.  Accordingly, we speak of
orientability and non-orientability of the corresponding arms $A_L$.
\end{defn}

We have the following useful characterization of orientability of lines (and arms).

\begin{lem}
  \lemlabel{V.10a} Let ${\mathcal V}=(\Gamma, {\mathbf a})$ be a
  connecting V-system.  Then the V-tree system $\Theta[{\mathcal V}]$
  contains a nonorientable line if and only if there is
  $v\in V_\Gamma$ such that $a(v)=v$ and there is $p\in\hbox{Lk}(v)$
  such that $\alpha_v(p)=p$. In particular, if the involution
  $a:V_\Gamma\to V_\Gamma$ is fixed point free then all lines in the
  V-tree system $\Theta[{\mathcal V}]$ are orientable.
\end{lem}
\begin{proof}
If there are $v$ and $p$ such that $a(v)=v$ and $\alpha_v(p)=p$,
denote by $e$ this edge of $\Gamma$ issuing from $v$ which corresponds to $p$.
The equalities above obviously imply that in any line $L$ containing an appearance of $e$
we have another appearance of $e$, adjacent to this first appearance, and oppositely oriented.
Thus any such line $L$ is nonorientable.

To prove the converse implication, observe that all appearances of any
oriented edge $\epsilon$ of $\Gamma$ in lines of $\Theta[{\mathcal V}]$ are followed by
adjacent appearances
of some other oriented edge uniquely determined by $\epsilon$
(which is not necessarily distinct from $\epsilon$). 
Thus, if some line contains oppositely oriented appearances
of some edge, it must also contain, somewhere in between, oppositely oriented and \emph{adjacent} appearances of some 
other edge, which clearly implies existence of $v$ and $p$ as in the statement.
This completes the proof.
\end{proof}

\begin{rmkdefn}
  \rmkdefnlabel{V.11} 
\begin{enumerate}
\item[(1)]
If a line $L$ is orientable, and if we fix an order of the doubleton $\partial \gamma_L$ 
(i.e. of the set of endpoints of the arm $A_L$), then this order induces orientations
at all those edges of $\Gamma$, which have appearances in $L$. We 
call such orientations of the above mentioned edges the 
\emph{order-induced} orientations (with respect to the above fixed order of 
$\partial \gamma_L$). 

\item[(2)] Let ${\mathcal V}=(\Gamma,{\mathbf a})$ be a connecting $V$-system,
and let $P\subset{\mathcal C}({\mathcal V})$ be a peripheral of the core ${\mathcal C}({\mathcal V})$
of the $V$-tree ${\mathcal X}_V({\mathcal V})$. Consider the family ${\mathcal L}_P$ of all lines
$L$ in $\Theta[{\mathcal V}]$ such that $\partial \gamma_L=P$. Note that ${\mathcal L}_P$
is nonempty and finite. Define the set of \emph{edges of $\Gamma$ associated to $P$},
denoted $E_\Gamma^P$,
as the set of all edges which have an appearance in any line $L\in{\mathcal L}_P$.
\item[(3)]
Suppose that the $V$-involution map $a:V_\Gamma\to V_\Gamma$
of $\mathcal{V}$ is fixed point free. Let $P$ be some doubleton
peripheral as in 
(2) above, and suppose that an edge $e$ of $\Gamma$ has appearances
in two distinct lines $L_1,L_2\in\mathcal{L}_P$.
Fix some order of $P$.
The argument similar as in the proof of \Lemref{V.10a}
shows that the order-induced (from the above fixed order of $P$)
orientations of $e$ determined by the two above mentioned appearances
of $e$ in $L_1$ and $L_2$ coincide.
It follows that any fixed order of the doubleton $P$
uniquely induces orientations on all edges from the set $E_\Gamma^P$. 
We again call the so described orientations the \emph{order-induced} orientations (with respect to the above fixed order of 
the peripheral $P$). 
\end{enumerate}
\end{rmkdefn}

\begin{obs}
  \obslabel{V.12} Note that in any standard connecting $V$-system 
  ${\mathcal V}_{\Xi_k}$ of $\theta$-graphs
(as described in \Exref{V.8.2}), 
each line (and consequently
each arm in the corresponding $V$-tree 
${\mathcal X}_V({\mathcal V}_{\Xi_k})$) is orientable. 
Note that, fixing an order for the common set of endpoints of the arms of 
${\mathcal X}_V({\mathcal V}_{\Xi_k})$, 
we get order-induced orientations of the edges
of the $\theta$-graph $\Xi_k$ which are consistent (in the sense that they induce
the same order on the set of two vertices of $\Xi_k$).
\end{obs}

More complicated examples of connecting $V$-systems and $V$-trees will appear in \Ssecref{Y},
in connection with abstract virtually free factors and their Whitehead graphs.

\section{Abstract factors}
\seclabel{abs_fac}

In this section we prepare some further tools and terminology
necessary for our proof of the implication \pitmref{g_rigid_cluster_factors_vf}$\Rightarrow$\pitmref{bd_g_reg_twocon_tog} of
\Mainthmref{1} (and the implication (3)$\Rightarrow$(2) of \Mainthmref{2}).  In
particular, we discuss the so called \emph{Whitehead graphs} associated
to virtually free factors of the considered groups $G$.  Such graphs
(or graphs closely related to them) will appear in our expressions of
the Gromov boundaries $\partial G$ as regular trees of graphs.

\subsection{Abstract factors and their decorations}
\sseclabel{F}

\begin{defn}[abstract factors]
  \defnlabel{F.1}
An \emph{abstract factor} is a pair $(K,\{ H_i \}_{i\in I})$, where $K$ is a
nonelementary hyperbolic group and $\{ H_i \}_{i\in I}$ is a finite collection
of maximal 2-ended subgroups of $K$ which are pairwise non-conjugate.
Actually, we will view $\{ H_i \}_{i\in I}$ rather as a family of conjugacy classes of
maximal 2-ended subgroups, and we will not distinguish abstract factors
which have the same $K$, and for which the conjugacy classes of the subgroups
coincide. 
\end{defn}

\begin{ex}
  \exlabel{F.1.1}
  Given a 1-ended hyperbolic group $G$, any black vertex group $G(v)$ of the 
Bowditch JSJ splitting of $G$
(as well as of the reduced Bowditch JSJ splitting of $G$), 
equipped with the family ${\mathcal H}_v$ of conjugacy classes
of its edge subgroups $G([v,u])=G(v)\cap G(u)$ with $u\in N_v$, is an abstract factor.
In particular, we may apply all of the discussion below to such factors
of the corresponding splittings.
\end{ex}

It is obviously true that the conjugate subgroup $gH_ig^{-1}<K$ depends
only on the coset in $K/H_i$ to which the conjugating element $g$
belongs.  Moreover, since $H_i$ is maximal 2-ended, it coincides with
the stabilizer of $\bd H_i$ in $\bd G$ (the stabilizer is a 2-ended
subgroup containing $H_i$) which implies that $H_i$ is
self-normalizing.  Thus the cosets and conjugates of $H_i$ are in
bijective correspondence and we will freely refer to a conjugate
$gH_ig^{-1}$ by way of its corresponding coset $gH_i\in K/H_i$.  Since
it is well known that for distinct maximal 2-ended subgroups $H,H'$ of
a hyperbolic group $K$ their boundaries $\partial H, \partial H'$
viewed as subsets in the Gromov boundary $\partial K$ are disjoint
(see e.g. the sentence right before Lemma~1.1 in
\cite{Bowditch:cut_points:1998}), it follows that the family of
boundaries of the conjugates
\[
\partial (gH_ig^{-1}):i\in I, g\in K/H_i \leqno{(1)}
\]
consists of pairwise disjoint subsets of $\partial K$. 
It is also not hard to observe that this family of subsets of $\partial K$ is null.
To see it, note that any conjugate $gH_ig^{-1}$ lies at finite Hausdorff distance
(for any word metric $d_K$ in $K$) from the coset $gH_i$, and thus we have
$\partial(gH_ig^{-1})=\partial (gH_i)$. Moreover, 
for any fixed $i\in I$ and for the identity element $e\in K$ the distances
\[
d_K(e,gH_i):gH_i\in K/H_i
\]
clearly diverge to infinity. Since the cosets $gH_i$ of $K/H_i$ are uniformly
quasiconvex in $K$ 
(in particular, they lie at uniformly finite distance from geodesics in $K$
connecting their boundary points), it follows that the diameters 
$\hbox{diam}(\partial(gH_i)):gH_i\in K/H_i$
converge to $0$, which obviously implies that the family (1) is null.
This allows us to introduce the following concept.

\begin{defn}[space with peripherals induced by an abstract factor]
  \defnlabel{F.1.2}
The \emph{space with peripherals
associated to an abstract factor $(K,\{ H_i \}_{i\in I})$} is the pair
\[
(\partial K, \{ \partial (gH_ig^{-1}):i\in I, gH_i\in K/H_i \}).
\]
Peripherals of the form $\partial(gH_ig^{-1})$ are called $H_i$-\emph{peripherals},
or peripherals \emph{of type} $H_i$.
\end{defn}

We now pass to discussing orientability phenomona for peripheral subgroups
of abstract factors.

\begin{defn}
  \defnlabel{F.2} A 2-ended group $H$ is \emph{orientable} if each
  element $h\in H$ fixes $\partial H$ pointwise. Otherwise $H$ is
  \emph{non-orientable}.  If $H$ is orientable, an \emph{orientation}
  for $H$ is a choice of any order on the set $\partial H$.
\end{defn}

Note that $H$ as above is non-orientable iff for some $h\in H$ the
induced homeomorphism of $\partial K$ transposes the points of the
doubleton $\partial H$. Assuming now that $H$ is a maximal 2-ended
subgroup of a hyperbolic group $K$, if $H$ is orientable then any
conjugate of $H$ in $K$ is also orientable, and a similar property
holds for non-orientability. Moreover, if $H$ is orientable, any fixed
orientation for $H$ induces orientations for all conjugates
$gHg^{-1}$.  In particular, an orientation for $H$ induces also orders
on the boundary sets $\partial(gHg^{-1})$ (which are doubletons).
(Here we use the fact that the stabilizer of the doubleton
$\partial H\subset\partial K$ coincides with $H$.)
As a matter of facts, $H$ is non-orientable iff for some $h\in H$ the left
multiplication by $h$ reverses the ends of $H$, so that non-orientability
(and orientability) is in fact an intrinsic property of $H$. However, for the purposes of this paper, the corresponding earlier metioned ``extrinsic''
manifestation of this property is more important.

Coming back to the space with peripherals associated to an abstract factor
$(K,\{ H_i \})$, we say that an $H_i$-peripheral is \emph{orientable} if the corresponding
subgroup $H_i$ is orientable; otherwise an $H_i$-peripheral is \emph{non-orientable}.
An \emph{orientation} for an orientable peripheral is a choice of an order,
and obviously it can be compatible or incompatible with a fixed orientation
of the corresponding 2-ended subgroup $H_i$.

\begin{defn}[decoration of an abstract factor]
  \defnlabel{F.3}
\emph{Orientation data} for an abstract factor  $(K,\{ H_i \}_{i\in I})$
is information about orientability and non-orientability of the subgroups $H_i$
(and thus of all subgroups from their conjugacy classes) and, for each orientable
subgroup $H_i$, a choice of an orientation.
Such orientation data induces a \emph{decoration of the associated space with peripherals},
which consists of information about the types ($H_i$-types) of the peripherals,
and about their orientability, and if
a peripheral is orientable, this decoration indicates also 
one of the orientations of
this peripheral, namely this one which is 
compatible with the fixed orientation of the corresponding
2-ended subgroup.
\end{defn}

We make a record of the following obvious observations,
which follow from the well known properties of the 
action
of a hyperbolic group on its boundary.

\begin{lem}
  \lemlabel{F.4}
Let  $(K,\{ H_i \}_{i\in I})$ be an abstract factor.
\begin{enumerate}
\item[(1)]
Each homeomorphism of the action of $K$ on $\partial K$ preserves any
decoration of the associated space with peripherals. 
More precisely, this means that any such homeomorphism preserves the types of the
peripherals, it maps non-orientable
peripherals to non-orientable ones, and it maps orientable peripherals to
orientable ones, and in the latter case, it preserves the orientations indicated by
the decoration.
\item[(2)] For any given type $H_i$, $K$ acts transitively on the peripherals of this type.
\item[(3)] 
If a peripheral $P$
in the associated space with peripherals is non-orientable, then there is a homeomorphism
of this space with peripherals which preserves types $H_i$ of all peripherals,
preserves any decoration, and maps $P$ to itself via the transposition.
\end{enumerate}
\end{lem}

\subsection{Flexible factors}
\sseclabel{X}

\begin{defn}[abstract flexible factor]
  \defnlabel{X.1}
An abstract factor $(K,\{ H_i \}_{i\in I})$ is called an
\emph{abstract flexible factor} if the
associated space with peripherals is homeomorphic (as a space with peripherals) to the
punctured circle $(S^1)^\circ$.
\end{defn}

\begin{ex}
  \exlabel{X.1.1}
Let $G(v)$ be a flexible factor of the Bowditch JSJ-splitting 
of a 1-ended 
hyperbolic group $G$.
Then $(G(v),{\mathcal H}_v)$, i.e., the abstract factor associated to $G(v)$
as in \Exref{F.1.1}, 
is an abstract flexible factor, in the sense of \Defnref{X.1}. For justification, see Section~5
in \cite{Bowditch:cut_points:1998}, where flexible factors are called \emph{maximal hanging Fuchsians (MHF)},
or the paragraph after Definition~5.9 in \cite{Kim_Walsh:planar_boundaries:2022}.

Similarly, let $G(v')$ be a flexible factor of the reduced Bowditch splitting of a 1-ended
hyperbolic group $G$. In view of \Factref{3.3}, the abstract factor $(G(v'),{\mathcal H}_{v'})$
associated to $G(v')$ is also a flexible abstract factor.
\end{ex}




\begin{rmk}
  \rmklabel{X.1.2}
\begin{enumerate}
\item[(1)]Since, as a topological space, $(S^1)^\circ$ is a Cantor set, it follows that the underlying
group $K$ of an abstract flexible factor is a virtually free group.
\item[(2)] Recall that the natural embedding of the punctured circle $(S^1)^\circ$ 
in the circle $S^1$ is uniquely described, up to a homeomorphism of the ambient $S^1$,
by the following condition: for any peripheral doubleton $P$ of $(S^1)^\circ$
there is a segment $I_P$ contained in $S^1$ whose set of endpoints coincides with $P$
and whose interior is disjoint with $(S^1)^\circ$. We will call this embedding
\emph{the natural extension of $(S^1)^\circ$ to a circle}. Each segment $I_P$
as above will be called \emph{the segment induced by $P$}. Note that any homeomorphism
$h$ of $(S^1)^\circ$ preserving the family of peripherals can be extended
in an essentially unique way to a homeomorphism $\bar h:S^1\to S^1$ of the
natural extension.
Thus,
it makes sense to speak of \emph{orientation-preserving} and \emph{orientation-reversing}
homeomorphisms $h$, by referring to the corresponding property of the induced $\bar h$.
\end{enumerate}
\end{rmk}

\Rmkref{X.1.2}(2) leads to the following concept.

\begin{defn}[orientability of an abstract flexible factor]
  \defnlabel{X.2}
An abstract flexible factor $(K,\{ H_i \}_{i\in I})$ is \emph{orientable}
if for each $g\in K$ the induced homeomorphism of $\partial K=(S^1)^\circ$
is orientation-preserving (in the sense explained in \Rmkref{X.1.2}(2)). An \emph{orientation}
of such a factor is a choice of orientation for the corresponding extended circle $S^1$.
\end{defn}

Note that an abstract flexible factor $(K,\{ H_i \}_{i\in I})$ is \emph{non-orientable}
if for some $g\in K$  the induced homeomorphism of $\partial K=(S^1)^\circ$
is orientation-reversing. The next lemma relates orientability properties of an abstract
flexible factor $(K,\{ H_i \}_{i\in I})$ with orientability properties of its peripheral 
subgroups $H_i$. We omit the straightforward proof.

\begin{lem}
  \lemlabel{X.3}
Let ${\mathcal F}=(K,\{ H_i \}_{i\in I})$ be an abstract flexible factor.
\begin{enumerate}
\item[(1)] If some peripheral subgroup $H_i$ is non-orientable (in the sense of
\Defnref{F.2}) then $\mathcal F$ is also non-orientable.
\item[(2)] If $\mathcal F$ is orientable then all of its peripheral subgroups are also orientable.
\end{enumerate}
\end{lem}

The reader should keep in mind that, given an abstract flexible factor $\mathcal F$,  
orientability of all its peripheral subgroups does not
necessarily lead to orientability of the factor. For example, if $K$ is the fundamental
group of a compact connected hyperbolic non-orientable surface with non-empty boundary, and if we take as its peripherals  the subgroups corresponding to boundary components,
then the obtained flexible factor is non-orientable, though all its peripheral subgroups
are cyclic, and hence orientable.

\begin{defn}[compatibility of orientations of a flexible factor and of
  its peripheral subgroup]
\defnlabel{X.4}
Given an orientable abstract flexible factor ${\mathcal F}=(K,\{ H_i \}_{i\in I})$, with a fixed orientation,
and an orientable peripheral subgroup $H\in \{ H_i \}_{i\in I}$, also with a fixed
orientation (i.e. an order of the doubleton $P=\partial H$), we say that these orientations are \emph{compatible} if the orientation induced from $P$ on the induced segment
$I_P$ in the extended circle $S^1$ coincides with the orientation
of this segment induced from
the orientation of the whole $S^1$ which represents the fixed orientation of $\mathcal F$.
\end{defn}

We make a record of the two observations below, which follow easily from
the fact that each orbit of a hyperbolic group in the action on its Gromov boundary
is dense.

\begin{fact}
  \factlabel{X.5}
Let ${\mathcal F}=(K,\{ H_i \}_{i\in I})$ be an abstract flexible factor.

\begin{enumerate}
\item[(1)] Suppose that $\mathcal F$ is orientable, with one of the orientations fixed.
If we also fix one of the orientations of a peripheral subgroup $H_i$,
then either the induced orientations of all $H_i$-peripherals are compatible
with the orientation of $\mathcal F$, or they are all incompatible.
In any case, the union of all $H_i$-peripherals is dense in $\partial K=(S^1)^\circ$.

\item[(2)] Suppose that $\mathcal F$ is non-orientable. Consider the extended circle $S^1$
for $\partial K=(S^1)^\circ$, and fix one of its orientations. 
If a subgroup $H_i$ is orientable, and if one of its orientations is fixed,
consider the induced orientations for all conjugates $gH_ig^{-1}$,
and the corresponding orientations of all $H_i$-peripherals.
Then the family of those $H_i$-peripherals whose just described orientations
 are compatible
(respectively, incompatible) with the above fixed orientation of $S^1$ is non-empty,
and the union of this family is dense in $\partial K=(S^1)^\circ$.
\end{enumerate}
\end{fact}

\subsection{Abstract virtually free factors and their Whitehead
  graphs}
\sseclabel{Y}

In this subsection we first recall the concept of a Whitehead graph,
as applied to abstract factors that are free.
We then extend all of our discussion 
to abstract virtually free factors. The main result in this subsection, \Corref{Y.6},
identifies the space with peripherals associated to an abstract virtually free factor
with the core of the associated V-tree of Whitehead graphs. This result will play
a crucial role in our proof of the implication \pitmref{g_rigid_cluster_factors_vf}$\Rightarrow$\pitmref{bd_g_reg_twocon_tog} of \Mainthmref{1}, 
given in \Secref{P}, and it will be applied
to virtually free rigid cluster factors of the reduced Bowditch JSJ splitting of a group $G$.

\begin{defn}
  \defnlabel{Y.1} An \emph{abstract free factor} is an abstract factor
  $(F,\{ H_i \})$ in which $F$ is a finitely generated non-abelian
  free group.
\end{defn}

\begin{rmk}
  Note that, since $F$ is torsion free, each of the subgroups $H_i$ is
  a maximal infinite cyclic subgroup of $F$.
\end{rmk}

To each abstract free factor $(F,{\mathcal H})$, equipped with a basis $B$ of free
generators for $F$, there is associated some finite graph, called 
the \emph{Whitehead graph} of the factor, see Section~2E in \cite{Cashen:line_patterns:2016}. 
We now recall the description of this graph.

\begin{defn}[Whitehead graph of an abstract free factor]
  \defnlabel{Y.2} Given an abstract free factor  ${\mathcal Q}=(F,\{ H_i \})_{1\le i\le k}$, 
let  $B=(x_1,\dots,x_m)$ be some 
base set of free generators
for $F$. For each $1\le i\le k$  fix a generator $h_i$ of the cyclic subgroup $H_i$, 
and denote by $w_i$ the unique cyclically
reduced word over $B\cup B^{-1}$ in the conjugacy class of $h_i$. 
(Here we view $w_i$ as a cyclic word, with letters occuring in cyclic order
rather than in linear one.)
The \emph{Whitehead graph}
for $\mathcal Q$ with respect to $B$ is the graph $W=W_{\mathcal Q}$
with the vertex set 
$V_{W_{\mathcal Q}}=\{ x_1,\dots,x_m,x_1^{-1},\dots,x_m^{-1} \}$
and with one edge connecting distinct vertices $x,y\in V_{W_{\mathcal Q}}$ for every occurence
of the subword $x^{-1}y$ in the cyclic words $w_1,\dots,w_k$.
(We use the convention that if $w_i=x$ has length 1 then there is precisely one occurence of $xx$ in $w_i$, and no occurence of any other word of length 2; 
moreover, if $w_i=uz$ has length 2, then $w_i$ has two subwords of length 2,
namely $uz$ and $zu$.)
\end{defn}

\begin{rmk}
  \rmklabel{Y.2.1}

  \begin{enumerate}
\item[(1)]
Observe that, since the subgroups $H_i<F$ are maximal cyclic, their generators are not
the powers in $F$, and hence the cyclic words $w_i$ are not periodic. For this reason,
any two occurences of the same subword in any fixed $w_i$ are essentially distinct.
This dispels some potential doubts concerning the above definition.

\item[(2)] Note that the Whitehead graph $W_{\mathcal Q}$ is obviously
  a graph with no loop edges, but it may contain multiple
  edges. $W_{\mathcal Q}$ depends not only on the factor $\mathcal Q$,
  but also on the choice of a basis $B$ in the free group $F$.  The
  choices for which the number of edges in corresponding Whitehead
  graphs is minimal lead to so called \emph{minimal Whitehead graphs}
  for the factor $\mathcal Q$.  As we will show later (see
  \Propref{M.2}), in the situations considered in this paper the
  minimal Whitehead graphs will always be 2-connected (i.e. finite,
  connected, nonempty, not equal to a single vertex and
  cutpoint-free), so in particular they will be essential
  (i.e. nonempty and with no isolated vertices).  Since all our
  further observations apply to any choice of a minimal Whitehead
  graph for a factor $\mathcal Q$, we will not bother about this
  choice, and we will speak of \emph{the} Whitehead graph associated
  to $\mathcal Q$, meaning any of the minimal Whitehead graphs as
  above.
\end{enumerate}

\end{rmk}

Any Whitehead graph $W_{\mathcal Q}$ which is essential
 is canonically equipped with
some $V$-involution ${\mathbf a}_{\mathcal Q}$, so that it yields an associated
connecting $V$-system ${\mathcal V}_{\mathcal Q}=(W_{\mathcal Q},{\mathbf a}_{\mathcal Q})$. 
This $V$-involution consists of
the involution $a_{\mathcal Q}:V_{W_{\mathcal Q}}\to V_{W_{\mathcal Q}}$ given by
$a_{\mathcal Q}(x)=x^{-1}$ for each $x\in V_{W_{\mathcal Q}}$
(where we use the convention that $(x_i^{-1})^{-1}=x_i$), and of the maps 
$\alpha_x:\hbox{Lk}(x)\to\hbox{Lk}(x^{-1})$ described as follows. 
For any $y\in V_{W_{\mathcal Q}}$ and any occurence of $y$ in words $w_1,\dots,w_k$
consider the subwords of the form $xy$ and $yz$, one for each of the two forms, 
which extend this occurence of $y$, and denote by $\sigma,\tau$ the corresponding occurences
of these subwords, respectively.
Denote by $e_{\sigma}$ and $e_{\tau}$ the edges of $W$ corresponding to the
above occurences of $xy$ and $yz$. Denote also by
$[e_{\sigma}]_y$ and $[e_{\tau}]_{y^{-1}}$ the points in the links 
$\hbox{Lk}(y)$ and $\hbox{Lk}(y^{-1})$ induced in those links by
the edges $e_{\sigma}$ and $e_{\tau}$, respectively.
Finally, put $\alpha_y([e_{\sigma}]_y)=[e_{\tau}]_{y^{-1}}$ 
and $\alpha_{y^{-1}}([e_{\tau}]_{y^{-1}})=[e_{\sigma}]_y$;
note that
the collection of all assignments as above fully describes all the maps
$\alpha_y:y\in V_{W_{\mathcal Q}}$, and that the conditions required in \Defnref{V.1}
are satisfied.
We skip straightforward details of the justification.

The next observation follows directly from \Lemref{V.10a}, in view of the fact
that the involution 
$a_{\mathcal Q}$, as given above, is fixed point free.

\begin{lem}
  \lemlabel{Y.2a} For any essential Whitehead graph $W_{\mathcal Q}$
  (associated to an abstract free factor $\mathcal Q$), each line in
  the corresponding V-tree system $\Theta[{\mathcal V}_{\mathcal Q}]$
  is orientable.
\end{lem}

We now show that, for any abstract free factor ${\mathcal Q}=(F, \{ H_i \}_{i\in I})$,
the core $C({\mathcal V}_{\mathcal Q})$ of the $V$-tree ${\mathcal X}_V({\mathcal V}_{\mathcal Q})$
naturally carries an action of the group $F$.
To do this, we will first identify naturally the underlying tree $T$ of the $V$-tree system
$\Theta[{\mathcal V}_{\mathcal Q}]$ with the first barycentric subdivision of the
Cayley graph ${\rm Cay}(F,B)$ of the group $F$ with respect to its
basis $B=\{ x_1,\dots,x_m \}$.

Recall that the tree $T$ is bipartite, with all white vertices of
degree 2.  Thus $T$ is the first barycentric subdivision of a certain
tree, which we denote $T_*$. The vertices of $T_*$ are naturally
identified with the black vertices of $T$. Our goal is to identify
$T_*$ with ${\rm Cay}(F,B)$.  To do this, we associate to each edge
$e$ of $T_*$ some orientation and some label from the generating set
$B$.  Let $t,s$ be the endpoints of $e$, and view them as black
vertices of $T$. Let $u$ be the barycenter of $e$, which we view as an
appropriate white vertex of $T$ (the unique white vertex of $T$
adjacent to both $t$ and $s$).  The constituent spaces at $t$ and $s$
of the tree system $\Theta[{\mathcal V}_{\mathcal Q}]$, denoted
$\Gamma_t^\circ$ and $\Gamma_s^\circ$, are the copies of the punctured
Whitehead graph $W_{\mathcal Q}$. The images in $\Gamma_t^\circ$ and
in $\Gamma_s^\circ$ of the peripheral space $\Sigma_u$ are some
$V$-peripherals related to vertices of $\Gamma_t$ and $\Gamma_s$, and
we denote these vertices by $z$ and $y$, respectively. Identifying
both $z$ and $y$ as vertices of the Whitehead graph $W_{\mathcal Q}$,
we have $z,y\in B\cup B^{-1}$, and by the form of the associated
$V$-involution $a_{\mathcal Q}$, we have $y=z^{-1}$. Without loss of
generality, assume that $z\in B$.  Orient then the edge $e$
consistently with the order $(z,y)$, and label it with $z$.  It is
then easy to observe that for any vertex $v$ of $T_*$ we have exactly
$m$ oriented edges issuing from $v$, labelled bijectively by the
elements from $B$, and exactly $m$ oriented edges incoming to $v$,
also labelled bijectively with the elements of $B$. Thus the tree
$T_*$ is locally isomorphic to the Cayley graph ${\rm Cay}(F,B)$,
which is also a tree.  To get an explicit identification of $T_*$ with
${\rm Cay}(F,B)$, fix a vertex of $T_*$ and identify it with the
vertex of ${\rm Cay}(F,B)$ corresponding to the unit of $F$. This
identification obviously extends to a unique label preserving
isomorphism of oriented graphs, and we take this isomorphism as the
desired identification.

Given the identification of $T_*$ with ${\rm Cay}(F,B)$ as above, the
natural action of $F$ on ${\rm Cay}(F,B)$ induces an action of $F$ on
$T_*$. This clearly yields also an action of $F$, by automorphisms, on
$T$.  Since the core $C({\mathcal V}_{\mathcal Q})$ coincides with the
boundary $\partial T$, we equip it with the induced action of $F$ on
$\partial T$.  Observe that, being orientation and label preserving,
the above described action of $F$ on $T$ preserves also the lines of
the $V$-tree system $\Theta[{\mathcal V}_{\mathcal Q}]$, as explained
in some detail in the proof of \Lemref{Y.3}.  As a consequence, the
induced action of $F$ on the core
${\mathcal C}({\mathcal V}_{\mathcal Q})=\partial T$ preserves the
peripherals in this core.

Note that, since $F$ is a free group, its Gromov boundary $\partial F$ is the Cantor set.
Thus, the space with peripherals associated to an abstract free factor is a Cantor set
equipped with some family of doubleton subsets.
The next lemma exhibits an important relationship between this space with peripherals
and the $V$-tree of Whitehead graphs induced by the corresponding factor.

\begin{lem}
  \lemlabel{Y.3} Let ${\mathcal Q}$ be an abstract free factor  equipped with an essential
Whitehead graph $W_{\mathcal Q}$, and let 
${\mathcal V}_{\mathcal Q}=(W_{\mathcal Q},{\mathbf a}_{\mathcal Q})$ be the connecting
$V$-system associated to it (as described above).
Then the space with peripherals associated
to $\mathcal Q$ is homeomorphic (as a space with peripherals) to the core
of the $V$-tree of graphs ${\mathcal X}_V({\mathcal V}_{\mathcal Q})$,
through an $F$-equivariant homeomorphism (i.e. a homeomorphism which
respects the $F$-actions on the two spaces).
\end{lem}
\begin{proof}
Recall that the space with peripherals associated to ${\mathcal Q}=(F,\{ H_i \}_{i\in I})$ is the pair
\[
(\partial F, \{ \partial(gH_ig^{-1}) : i\in I, g\in F/H_i \}),
\]
and that the peripherals of the core $C({\mathcal V}_{\mathcal Q})$ have the form $\partial\gamma_L$,
where $L$ runs through the set of all lines of the system $\Theta[{\mathcal V}_{\mathcal Q}]$. 
As a required homeomorphism between the core $C({\mathcal V}_{\mathcal Q})=\partial T$
and the associated space with peripherals $\partial F=\partial [{\rm Cay}(F,B)]$
take the homeomorphism induced on the boundaries by the indentification 
of $T$ with the first barycentric subdivision of ${\rm Cay}(F,B)$,
as described right above the statement of the lemma. 
Denote this homeomorphism by $h$.
It is obvious that $h$ is $F$-equivariant, by the definition of the action of $F$ on $\partial T$.
We need to show that $h$ preserves the families of the peripherals.
 
Let $L$ be a line of the $V$-tree system 
$\Theta[{\mathcal V}_{\mathcal Q}]$.
Observe that any two consecutive edges of $L$ correspond then to 
some consecutive occurences of subwords of form $xy$ and $yz$
in precisely one of the cyclically reduced words $w_i$. 
We will emphasize this fact by saying that the line $L$ \emph{corresponds} 
to this generator $w_i$  (of the corresponding peripheral subgroups $H_i$). 
Let $P=\partial\gamma_L$ 
be the corresponding peripheral
of the core $C({\mathcal V}_{\mathcal Q})$. Observe that the geodesic $\gamma_L$ coincides then,
via the identification of $T_*$ with ${\rm Cay}(F,B)$, with the axis of the automorphism of ${\rm Cay}(F,B)$
given by left multiplication through some conjugate $gw_ig^{-1}$. Since this axis stays at finite
Hausdorff distance in ${\rm Cay}(F,B)$ from the subgroup $gH_ig^{-1}<F$, it has the same
boundary at infinity as this subgroup, and thus we get 
\[h(\partial\gamma_L)=\partial(gH_ig^{-1}).
\]
So $h$ maps peripherals to peripherals.
A similar argument shows that $h^{-1}$ also maps peripherals to peripherals,
and thus $h$ is a homeomorphism of spaces with peripherals, which completes the proof.
\end{proof}

\begin{rmk}
  \rmklabel{Y.3.1}
  Observe that the choices of the generators $h_i$ in the subgroups $H_i<F$ determine (uniquely) the choices of the generators $gh_ig^{-1}$ in the conjugate
  subgroups $gH_ig^{-1}<F$, and hence also the
  orders
on all peripherals in the space of peripherals $\partial F$ associated to $\mathcal Q$.
On the other hand, each $w_i$ induces orientations on all edges of any line $L$
in $\Theta[{\mathcal V}_{\mathcal Q}]$  corresponding
to $w_i$. It follows that the choices of $w_i$ induce orientations 
on all geodesics $\gamma_L$,
and thus they induce orders on all peripherals in the core $C({\mathcal V}_{\mathcal Q})$.
It is not hard to see that the homeomorphism $h$ from the above proof of \Lemref{Y.3} preserves the above
mentioned orders of the peripherals.
\end{rmk}

We now pass to virtually free abstract factors.

\begin{defn}
  \defnlabel{Y.3b} An \emph{abstract virtually free factor} is an abstract factor 
$\hat{\mathcal Q}=(\hat F, \{ \hat H_i \})$
in which $\hat F$ is a finitely generated infinitely ended virtually free group.
\end{defn}

Consider the associated space with peripherals for $\hat{\mathcal Q}$
\[
(\partial\hat F, \{ \partial\hat H_i^g : 1\le i\le k, \,g\in\hat F/\hat H_i \})
\]
(here we again use the convention of conjugating through a coset $g\in\hat F/\hat H_i$, 
since the result of
the conjugation does not depend on a choice of a coset representative).
Our goal is to give an appropriate extension of the concept of Whitehead graph
associated to a factor, which will apply to virtually free factors as well.
To do this, we simply 
associate to $\hat{\mathcal Q}$ some appropriate free factor structure for a finite index free
subgroup of $\hat F$. The details are provided below, with culmination in \Corref{Y.6}.

Let $F$ be a free finite index subgroup of $\hat F$.
For each $i\in\{ 1,\dots,k \}$ and any $\hat g\in\hat F$ consider the subgroup
$\hat H_i^{\hat g}\cap F<  F$, which is torsion-free and of finite index in 
$\hat H_i^{\hat g}$, and which is hence infinite cyclic. Since we clearly have
$\hbox{Stab}_F\partial\hat H_i^{\hat g}=\hat H_i^{\hat g}\cap F$,
this subgroup is a maximal 2-ended subgroup of $F$. Note that $F$ acts
on the set $\{\hat H_i^{\hat g}\cap F:\hat g\in\hat F/\hat H_i\}$
of such subgroups by conjugations, as in
\[
f(\hat H_i^{\hat g}\cap F)f^{-1}=\hat H_i^{f\cdot\hat g}\cap F.
\]
Since $F$ has finite index in $\hat F$, the above action by conjugations 
has finitely many orbits. Denote by $g_{i,j}:1\le j\le k_i$ some elements
of $\hat F$ such that\
\[
\hat H_i^{g_{i,j}}\cap F:1\le j\le k_i
\]
is a family of representatives of all orbits of the above action of $F$ by conjugation.
Since the stabilizer of each such representative satisfies
\[
\hbox{Stab}_F(\hat H_i^{g_{i,j}}\cap F)=\hat H_i^{g_{i,j}}\cap F,
\]
its orbit can be described as the set
\[
\{\hat H_i^{f\cdot g_{i,j}}\cap F : f\in F/(\hat H_i^{g_{i,j}}\cap F)\}.
\]
Put
\[
{\mathcal P}:=\{ \hat H_i^{g_{i,j}}\cap F : 1\le i\le k, \,\, 1\le j\le k_i \}.
\]
Then $(F,{\mathcal P})$ is obviously an abstract free factor.
Since we clearly have $\partial F=\partial\hat F$ and 
$\partial(\hat H_i^{g_{i,j}}\cap F)=\partial \hat H_i^{g_{i,j}}$,
we get the following.

\begin{lem}
  \lemlabel{Y.4}
  The spaces with peripherals 
\[
(\partial\hat F, \{ \partial\hat H_i^g : 1\le i\le k, \,g\in\hat F/\hat H_i \})
\hbox{\quad  and  \quad} (\partial F, \{ \partial H^f:H\in{\mathcal P},\, f\in F/H \})
\]
are canonically homeomorphic.
\end{lem}

We now extend the notion of a Whitehead graph to virtually free abstract factors.

\begin{defn}
  \defnlabel{Y.5}
Let $\hat{\mathcal Q}=(\hat F, \{ \hat H_i \})$ be a virtually free abstract factor.
The \emph{Whitehead graph} $W=W_{\hat{\mathcal Q}}$ of this factor is any graph of the following form.
Choose any free subgroup $F<\hat F$ of finite index, and a basis $B$ for $F$, and consider
the induced peripheral system $\mathcal P$ of subgroups of $F$, as described above. Take
as $W_{\hat{\mathcal Q}}$ the Whitehead graph for the free factor 
${\mathcal Q}=(F,{\mathcal P})$ with respect to $B$.
Take also as the accompanying connecting $V$-system ${\mathcal V}_{\hat{\mathcal Q}}$
the $V$-system ${\mathcal V}_{\mathcal Q}$ for $W_{\mathcal Q}$.
\end{defn}

\begin{rmk}\rmklabel{8.24}
Similarly as in the case of a free abstract factor, the Whitehead graph
$W_{\hat{\mathcal Q}}$ for $\hat{\mathcal Q}$ is not uniquely determined, and depends
on the choices of $F$ and $B$. 
In our later applications of this concept, we will deal with any choice of a subgroup $F$,
and such a choice of a basis $B$ for $F$ that the corresponding Whitehead graph
for $(F,{\mathcal P})$ is minimal. As we will show
(see Propositions \propref{M.1} and \propref{M.2}), 
under such choices the Whitehead
graphs appearing in our context will always be 2-connected, and in particular essential.
Since all further statements concerning the graph $W_{\hat{\mathcal Q}}$
apply equally well to all choices as above, we will speak of \emph{the} Whitehead graph
associated to the factor $\hat{\mathcal Q}$, meaning any of the graphs as described above.
\end{rmk}

As a consequence of the above definition, and in view of \Lemref{Y.4}, we have 
the following corollary to \Lemref{Y.3}.

\begin{cor}
  \corlabel{Y.6}
Let $\hat{\mathcal Q}$ be a virtually free abstract factor, and 
suppose that the associated Whitehead graph $W_{\hat{\mathcal Q}}$ is essential. 
Let ${\mathcal V}_{\hat{\mathcal Q}}=(W_{\hat{\mathcal Q}},{\mathbf a}
)$ 
be the associated connecting $V$-system.
Then the associated space with
peripherals for $\hat{\mathcal Q}$ is homeomorphic (as a space with peripherals) to the core
${\mathcal C}({\mathcal V}_{\hat{\mathcal Q}})$
of the corresponding $V$-tree of graphs ${\mathcal X}_V({\mathcal V}_{\hat{\mathcal Q}})$.
\end{cor}

We now discuss some further features of Whitehead graphs associated to virtually free abstract
factors. We refer to the notation as above in this section, additionally denoting by $\mathcal H$
the set of conjugacy classes of the peripheral subgroups in an abstract virtually free
factor $\hat{\mathcal Q}=(\hat F, \{ \hat H_i \})$. We first describe some natural labelling 
$\lambda:E_{W_{\hat{\mathcal Q}}}\to{\mathcal H}$ of the edges of the Whitehead graph
$W_{\hat{\mathcal Q}}$, with values in the set $\mathcal H$. Any edge $e$ in $W_{\hat{\mathcal Q}}$
corresponds to an occurence of a subword of length 2 in a cyclically reduced
word $w$ from the conjugacy class of a generator of precisely one cyclic peripheral subgroup
$\hat H_i^{g_{i,j}}\cap F\in{\mathcal P}$. We then put $\lambda(e)=[\hat H_i]$,
where $[\hat H_i]\in{\mathcal H}$ denotes the conjugacy class of the subgroup 
$\hat H_i$ in $\hat F$. So described $\lambda$ is clearly surjective, and it splits the set
$E_{W_{\hat{\mathcal Q}}}$ into subsets corresponding to the conjugacy classes in $\mathcal H$.
Typically, the subset corresponding to a fixed label $[\hat H_i]$ consists of more than
one edge, and actually its cardinality is equal to the sum of the lengths $|w|$ of those
cyclically reduced words $w$ which correspond to the generators of the subgroups
$\hat H_i^{g_{i,j}}\cap F:1\le j\le k_i$. We make an obvious observation that
each of the maps $\alpha_x$ from the involution system 
${\mathbf a}=(a, \{ \alpha_x \})$ associated
to $W_{\hat{\mathcal Q}}$ preserves the labelling $\lambda$, in the following sense.
If $e$ is an edge of $W_{\hat{\mathcal Q}}$ adjacent to $x$, and $e'$ is an edge
adjacent to $a(x)$, then denoting by $e_x$ and $e'_{a(x)}$ the points in the corresponding
vertex links induced by the corresponding edges, 
the equality $\alpha_x(e_x)=e'_{a(x)}$ implies that $\lambda(e)=\lambda(e')$.


The following observations concern some further features of the $V$-tree system 
$\Theta[{\mathcal V}_{\hat{\mathcal Q}}]$
associated to a virtually free abstract factor $\hat{\mathcal Q}$.

\begin{obs}
  \obslabel{Y.7}
\begin{enumerate}
\item[(1)]
Since the subgroups $\hat H_i^{g_{i,j}}\cap F$ are maximal cyclic in $F$,
no two lines in $\Theta[{\mathcal V}_{\hat{\mathcal Q}}]$ have the same ends, i.e.,
for any peripheral $P$ of the core $C({\mathcal V}_{\hat{\mathcal Q}})$
the set ${\mathcal L}_P=\{ L:\partial\gamma_L=P \}$ consists of a single line.

\item[(2)] It follows from the above observation (1) that 
for any peripheral $P$ of the core $C({\mathcal V}_{\hat{\mathcal Q}})$
there is precisely one arm in the V-tree ${\mathcal X}_V({\mathcal V}_{\hat{\mathcal Q}})$
which is attached to $C({\mathcal V}_{\hat{\mathcal Q}})$ along $P$.


\item[(3)] Since, due to \Lemref{Y.2a}, the line $L\in\mathcal{L}_P$,
  as above in (1), is orientable in
  $\Theta[{\mathcal V}_{\hat{\mathcal Q}}]$, any fixed order of the
  peripheral $P$ induces orientations on all the edges of
  $W_{\hat{\mathcal Q}}$ associated to $P$ (i.e. all edges having
  appearances in $L$).  In accordance with \Rmkdefnref{V.11}(1), we
  call such orientations \emph{order-induced} (with respect to the
  fixed order of $P$).
\end{enumerate}
\end{obs}


\section{The (extended) Whitehead graph of a virtually free rigid cluster factor of $G$}
\seclabel{E}

Let $v$ be a rigid cluster vertex in the tree $T^r_G$ of the reduced Bowditch JSJ splitting of a hyperbolic group $G$ (see \Secref{reduced_JSJ}).
We then call the group $G(v)$ a \emph{rigid cluster factor} of $G$.
Denote by $N_v$ the set of vertices in $T^r_G$ adjacent to $v$, and denote by
${\mathcal H}_v$ the set of the $G(u)$-conjugacy classes of the subgroups $G(u)\cap G(v):u\in N_v$.
Since the family ${\mathcal H}_v$ is finite,
the pair $(G(v),{\mathcal H}_v)$ is an abstract factor, which we denote briefly by ${\mathcal Q}_v$.
Denote also by $(X_v,{\mathcal D}_v)$ the space with peripherals associated to ${\mathcal Q}_v$,
where ${\mathcal D}_v$ is the family of peripheral (doubleton) subsets of $X_v=\partial G(v)$.

The next lemma exhibits some consequence of \Factref{3.4}
(which says that the rigid cluster vertices of ${T}^r_G$ are isolated)
for the structure
of the reduced Bowditch JSJ splitting of $G$.

\begin{lem}
  \lemlabel{E.1}
Let $v$ be a rigid cluster vertex of the tree ${T}^r_G$.
Then for any
$u\in N_v$ we have $G(u)\cap G(v)=G(u)$. As a consequence, each $G(u)$ as above
is a subgroup in $G(v)$, and the family ${\mathcal H}_v$ actually consists of the conjugacy
classes of the subgroups $G(u):u\in N_v$.
\end{lem}
\begin{proof}
Suppose on the contrary that $G(u)\cap G(v)$ is a proper subgroup of $G(u)$. This implies that
the orbit of the vertex $v$ under the action of $G(u)$ on $T^r_G$ contains at least one
vertex $v'\ne v$. By $G$-invariance, this vertex $v'$ is also a rigid cluster vertex, 
and since it is
adjacent to $u$, we get a contradiction with the fact that $v$ is isolated
(i.e. with \Factref{3.4}). 
\end{proof}

\subsection{The Whitehead graph $W_v$ of a virtually free rigid cluster
factor $G(v)$}
\sseclabel{E.1}

Suppose now that the group $G(v)$ (at a rigid cluster vertex $v$ of $T^r_G$) is a virtually free group,
and recall that it is then necessarily non-elementary (see \Corref{cluster_quasiconvex}). 
Obviously, ${\mathcal Q}_v=(G(v),{\mathcal H}_v)$ 
is then an abstract
virtually free factor. Denote by $W_v$ any minimal Whitehead graph for ${\mathcal Q}_v$, as described in \Ssecref{Y}
(see \Defnref{Y.5} and \Rmkref{8.24}), 
for any choice of a finite index free subgroup of $G(v)$. 

In the next two propositions we derive basic properties of the graphs $W_v$ as above.

\begin{prop}
  \proplabel{M.1}
Let ${\mathcal Q}=(F,{\mathcal H})$ be a free abstract factor, and let $W_{\rm min}$
be a minimal Whitehead graph for $\mathcal Q$. Then the following conditions are
equivalent:
\begin{enumerate}
\item[(1)] $W_{\rm min}$ is a 2-connected graph (i.e.  finite,
  connected, nonempty, not equal to a single vertex, and
  cutpoint-free);
\item[(2)] $F$ has no splitting as a free product relative to
  $\mathcal H$ (i.e. a splitting as a free product with nontrivial
  factors where the classes in $\mathcal H$ are all elliptic).
\end{enumerate}
\end{prop}
\begin{proof}
We say that a vertex $x$ of a graph $X$ is a \emph{cut vertex} if the space
$|X|\setminus\{ x \}$ has more connected components than $|X|$. 
It is an easy observation that any minimal Whitehead graph of a free abstract factor
has no cut vertex (see e.g. Lemma~2.3 in \cite{Cashen:cyclic_splittings:2017}) so any component is either cutpoint-free or an isolated edge. 
On the other hand, by Proposition~2.6 from \cite{Cashen:line_patterns:2016},
the following three conditions are equivalent:
\begin{enumerate}
\item[(1)] $F$ splits as a free product relative to $\mathcal H$,  
\item[(2)] every minimal Whitehead graph for $\mathcal Q$
is not connected, 
\item[(3)] some minimal Whitehead graph for $\mathcal Q$ is not connected.
\end{enumerate}

The proposition follows directly from these two observations and the
fact that $F$ is non-elementary.
\end{proof}

\begin{prop}
  \proplabel{M.2}
Let $G$ be a 1-ended hyperbolic group which is not cocompact Fuchsian,
and let ${\mathcal Q}_v=(G(v), {\mathcal H}_v)$ be any rigid cluster factor of the Bowditch JSJ splitting for $G$. If the group $G(v)$ is virtually free, 
then any minimal Whitehead graph for ${\mathcal Q}_v$ (for any choice of a finite index 
free subgroup $F<G(v)$) is 2-connected.
\end{prop}
\begin{proof}
  Suppose on the contrary that some minimal Whitehead graph $W$ for
  ${\mathcal Q}_v$ (related to a finite index free subgroup $F$ of
  $G(v)$, and to the family $\mathcal P$ of peripheral subgroups of
  $F$ induced from ${\mathcal H}_v$, as described in \Ssecref{Y}, right before \Lemref{Y.4}) is not 2-connected.  It
  follows from \Propref{M.1} that $F$ splits as a free product relative to
  $\mathcal P$.  Our next goal is to deduce from this fact the
  following.

\begin{claim}
  \claimlabel{9.4}
  $G(v)$ splits relative to ${\mathcal H}_v$ over a finite subgroup
  (i.e. acts on a tree without global fixed points such that edge
  stabilizers are finite and subgroups from ${\mathcal H}_v$ are
  conjugate into vertex stabilizers).
\end{claim}

To prove the claim,
note that since $G(v)$ is hyperbolic and the subgroups from ${\mathcal H}_v$ 
are maximal 2-ended 
and pairwise nonconjugate, $G(v)$
is relatively hyperbolic with respect to ${\mathcal H}_v$, and the same is true
for the pair $(F,{\mathcal P})$. We will argue by referring to some results
from the theory of relatively hyperbolic groups. 
In particular, we will refer to the properties of the so called \emph{Bowditch boundary} 
$\partial_B(G^\star,{\mathcal Q}^\star)$ of a relatively hyperbolic
group $G^\star$ (with respect to a family ${\mathcal Q}^\star=\{ Q_i \}$ of subgroups). By a result of Tran \cite{Tran:various_boundaries_relhyp:2013}, if $G^\star$ is hyperbolic then  $\partial_B(G^\star,{\mathcal Q}^\star)$ is homeomorphic to the quotient of
the Gromov boundary $\partial G^\star$ through a decomposition
provided by the family of subspaces 
\[\partial (gQ_ig^{-1}): Q_i\in{\mathcal Q^\star},\,\,g\in G^\star.\]
Note that, in view of \Lemref{Y.4}, this means that
the Bowditch boundaries $\partial_B(G(v),{\mathcal H}_v)$
and $\partial_B(F,{\mathcal P})$ are canonically homeomorphic.

By Proposition~10.1 in \cite{Bowditch:relatively_hyperbolic:2012}, the
Bowditch boundary $\partial_B(G^\star,{\mathcal Q}^\star)$ is
connected iff $G^\star$ does not split over a finite subgroup relative
to ${\mathcal Q}^\star$. Since $F$ splits over the trivial group
relative to $\mathcal{P}$, it follows that the Bowditch boundary
$\partial_B(F,{\mathcal P})$ is not connected. As a consequence, the
Bowditch boundary $\partial_B(G(v),{\mathcal H}_v)$ (which is
homeomorphic to $\partial_B(F,{\mathcal P})$) is also not
connected. By the same Proposition~10.1 in
\cite{Bowditch:relatively_hyperbolic:2012}, $G(v)$ splits over a
finite subgroup relative to ${\mathcal H}_v$, which completes the
proof of the claim.

Coming back to the proof of \Propref{M.2},
it is not hard to observe that if $G$ contains a rigid cluster factor which splits over
a finite subgroup relative to its peripheral subgroups, the group
$G$ itself also splits nontrivially over  a finite subgroup
(as an amalgameted free product or an HNN-extension).
Due to a well known result of Stallings, the latter implies that $G$ is not 1-ended,
contradicting one of our assumptions, and thus the proposition follows. 
\end{proof}

As we have already mentioned in \Rmkref{8.24}, 
in our further considerations we will use the convention of choosing
as ``the Whitehead graph $W_v$'' any minimal Whitehead graph
for $\mathcal{Q}_v$. In particular, $W_v$ will be always 2-connected.

\subsection{The extended Whitehead graph $\widetilde W_v$ of a virtually free rigid cluster
factor $G(v)$}
\sseclabel{E.2}

Under the notation of \Ssecref{E.1},
each conjugacy class $\xi\in{\mathcal H}_v$ is represented by a group $G(u)$ for some
$u\in N_v$ (see \Lemref{E.1}). Denote by $n_\xi$ the degree of the vertex $u$ in the tree $T^r_G$.
This number $n_\xi$ is independent of the choice of $G(u)$ as above.
Recall also (e.g. from the description in Sections \ssecref{bowditch_jsj_splitting} and \ssecref{reduced_JSJ})
that for each $\xi\in{\mathcal H}_v$ we have $n_\xi\ge2$.

Let $W_v$ be the Whitehead graph for ${\mathcal Q}_v$, as described 
in \Ssecref{E.1}.
The \emph{extended Whitehead graph} $\widetilde W_v$ for ${\mathcal Q}_v$ is the graph obtained
from the $W_v$ as follows. 
As the vertex set $V_{\widetilde W_v}$ we take the set $V_{W_v}$ of vertices of $W_v$.
To describe the edge set,
consider the labelling $\lambda$ of edges of $W_v$, as described in \Ssecref{Y}, right
after \Corref{Y.6}, and denote it by $\lambda_v$
(in order to distinguish it from the analogous labellings for Whitehead graphs
$W_{v'}$ related to other rigid cluster vertices $v'$ of $T^r_G$).
For any $\xi\in{\mathcal H}_v$, 
replace each edge $e$ of $W_v$ labelled with $\xi$ (i.e. such that $\lambda_v(e)=\xi$)  by $n_\xi-1$ edges
$e_1,\dots,e_{n_\xi-1}$ having the same endpoints as $e$.
We will call $e_1,\dots,e_{n_\xi-1}$ as above the \emph{edges of $\widetilde W_v$ corresponding
to an edge $e$ of $W_v$}. 
(As we will see later, the reason for choosing the number $n_\xi-1$
comes from the fact that the vertex $u$ of $T^r_G$, as above,
has precisely $n_\xi-1$ flexible neighbours, in addition to the single rigid
neighbour $v$.)
As part of its structure,
equip the so described graph $\widetilde W_v$ with its own involution system 
$\tilde{\mathbf a}_v=(\tilde a,\{ \tilde\alpha_x \})$ obtained from the corresponding involution
system ${\mathbf a}_v=(a,\{ \alpha_x \})$ for $W_v$ as follows.
Identifying naturally the vertices of $\widetilde W_v$ with the vertices of $W_v$,
put $\tilde a:=a$. Moreover, 
referring to the notation concerning vertex links as introduced right before
\Obsref{Y.7},
for any pair of edges $e,e'$ of $W_v$ such that $\alpha_x(e_x)=e'_{a(x)}$,
put $\tilde\alpha_x((e_i)_x):=(e'_i)_{a(x)}$ for $i=1,\dots,n_\xi-1$ 
(where $\xi=\lambda_v(e)$, and at the same time $\xi=\lambda_v(e')$, 
due to invariance of the labelling under the maps
$\alpha_x$). Finally, as the last part of the structure of the extended Whitehead graph, 
we equip it with the  labelling $\tilde\lambda_v$ of the edge set $E_{\widetilde W_v}$,
induced from the labelling $\lambda_v:E_{W_v}\to{\mathcal H}_v$ 
as follows.
As the set of values for this new labelling we take the set 
$\widetilde{\mathcal H}_v:=\{ (\xi,j):\xi\in{\mathcal H}_v, 1\le j\le n_\xi-1 \}$, and we put
$\tilde\lambda_v(e_i):=(\lambda_v(e),i)$ for any edge $e$ of $W_v$
and any $i=1,\dots,n_\xi-1$. Obviously, this induced labelling
$\tilde\lambda_v$ is invariant under the maps $\tilde\alpha_x$ of the involution system
$\tilde{\mathbf a}_v$, in the same sense as $\lambda_v$ is invariant under the maps $\alpha_x$.
Moreover, again similarly as with $\lambda_v$, this induced labelling $\tilde\lambda_v$
associates the same label to some two edges of $\widetilde W_v$ if and only if 
copies of these two edges are involved in a common arm of the $V$-tree
${\mathcal X}_V((\widetilde W_v,\tilde{\mathbf a}_v))$.
In particular, we may say that $\tilde\lambda_v$ uniquely
associates the labels from $\widetilde{\mathcal H}_v$
to arms in ${\mathcal X}_V((\widetilde W_v,\tilde{\mathbf a}_v))$.

The next observation is a fairly direct consequence of \Propref{M.2}.

\begin{fact}
  \factlabel{E.2}
The extended Whitehead graph $\widetilde W_v$ of any virtually free 
rigid cluster factor of the group $G$ is 2-connected (in particular, it is essential).
\end{fact}

We now pass to basic observations concerning $V$-trees of extended Whitehead graphs.

\begin{lem}
  \lemlabel{E.4}
Let $v$ be a rigid cluster vertex in the tree $T^r_G$ 
of the reduced Bowditch JSJ splitting
for $G$, and let ${\mathcal V}_v=(W_v,{\mathbf a}_v)$ and
$\widetilde{\mathcal V}_v=(\widetilde W_v,\tilde{\mathbf a}_v)$ be the associated
connecting $V$-systems.
\begin{enumerate}
\item[(1)] The core of the $V$-tree of graphs ${\mathcal X}_V(\widetilde{\mathcal V}_v)$
has a canonical identification (as a space with peripherals) with the core of ${\mathcal X}_V({\mathcal V}_v)$.
As a consequence, the 
associated space with peripherals for the factor
${\mathcal Q}_v=(G(v),{\mathcal H}_v)$ is homeomorphic (as a space with peripherals) to the core
of the $V$-tree of extended Whitehead graphs ${\mathcal X}_V(\widetilde{\mathcal V}_v)$.
\item[(2)] The $V$-tree ${\mathcal X}_V(\widetilde{\mathcal V}_v)$ can be obtained from the
$V$-tree ${\mathcal X}_V({\mathcal V}_v)$ by leaving the core unchanged and by
replacing each arm $A$ in the latter with $n_\xi-1$ copies
of this arm
(sharing only their endpoints, which coincide with the endpoints of $A$), 
where $\xi=\lambda_v(e)$ for any edge $e$ of $W_v$ involved
in the arm $A$. In particular, the $V$-tree ${\mathcal X}_V(\widetilde{\mathcal V}_v)$ has a 
description 
as a tree of arms and the core, 
$\Psi[\widetilde{\mathcal V}_v]$,
as presented
in \Lemref{V.9} and in the paragraphs right before it.
\end{enumerate}
\end{lem}

\begin{proof}
Since the vertex sets of the graphs $W_v$ and $\widetilde W_v$ canonically coincide,
and since the involutions $a:V_{W_v}\to V_{W_v}$ and 
$\tilde a_v:V_{\widetilde W_v}\to V_{\widetilde W_v}$ also coincide,
there is a canonical identification of the underlying trees $T$ and $\widetilde T$ of the tree
systems $\Theta[{\mathcal V}_v]$ and $\Theta[\widetilde{\mathcal V}_v]$
associated to the connecting $V$-systems ${\mathcal V}_v$ and $\widetilde{\mathcal V}_v$.
This identification results from the fact that the vertices in both trees correspond to
(punctured) copies of the graphs, and edges correspond to gluings of the $V$-peripherals
associated to vertices, according to the schemes provided by the $V$-involutions.
Now, let $L=(e_n^\circ)_n$ be a line in 
$\Theta[\widetilde{\mathcal V}_v]$, and recall that $\xi:=\lambda_v(e_n)$
is then independent of $n$.
Under the above identification of $T$ with $\widetilde T$, it is not hard to observe that to each line
$(e_n^\circ)_n$ in $\Theta[{\mathcal V}_v]$ there correspond $n_\xi-1$ lines
$((e_n)_i^\circ)_n:i=1,\dots,n_\xi-1$ in the tree system $\Theta[\widetilde{\mathcal V}_v]$,
and that there are no other lines in the latter tree system. Moreover, the geodesics in the tree
$\widetilde T$ corresponding to the lines $((e_n)_i^\circ)_n:i=1,\dots,n_\xi-1$ 
all coincide, and they are equal to the geodesic corresponding to the line $(e_n^\circ)_n$
under our identification of $T$ with $\widetilde T$. Since the cores under our interest
coincide with the Cantor sets appearing as boundaries of the corresponding trees,
and since the peripherals in the cores correspond to doubletons given by ends of geodesics
induced by lines, part (1) of the lemma clearly follows. Part (2) follows similarly due to the fact
that arms in the $V$-trees are induced by lines.
\end{proof}

We now pass to discussing some properties of the decomposition of the V-tree
${\mathcal X}_V(\widetilde{\mathcal V}_v)$ into core and arms. We begin with a natural
concept of types of the peripherals in the core $C(\widetilde{\mathcal V}_v)$.

\begin{defn}
  \defnlabel{E.5}
Under notation as in \Defnref{V.7},
let $P=\partial\gamma_L$ 
be a peripheral in the core $C(\widetilde{\mathcal V}_v)$ (which is identified with the core $C(\mathcal V_v)$ as in \Lemref{E.4}(1)), and let ${\mathcal H}_v$
be the set of conjugacy classes of the peripheral subgroups of the abstract factor
$G(v)$. Recall that ${\mathcal L}_P$ is the set of lines of $\mathcal V_v$ whose corresponding arms contain $P$ and that there is necessarily only one such line so that ${\mathcal L}_P$ is a singleton (see \Obsref{Y.7}). The \emph{type} of $P$ is defined as the element
$\xi=\lambda_v(e)\in{\mathcal H}_v$, where $e$ is any edge of $W_v$ whose copy appears
in the unique line $L\in{\mathcal L}_P$ (all such edges have the same label,
and the line is unique by \Obsref{Y.7}(1)).
\end{defn}

This assignment of types to the peripherals of the core $C(\widetilde{\mathcal V}_v)$
is natural in the sense provided by the following observation,
which can be deduced directly from the definitions and descriptions of the involved objects.

\begin{fact}
  \factlabel{E.6}
Let $(X_v,{\mathcal D}_v)$ be the space with peripherals associated to 
a virtually free rigid cluster factor
$(G(v),{\mathcal H}_v)$.
Under the identification of $(X_v,{\mathcal D}_v)$ with the core $C(\widetilde{\mathcal V}_v)$,
the types (from ${\mathcal H}_v$) of the peripherals in ${\mathcal D}_v$, as described in 
\Defnref{F.1.2}, coincide with the types of the corresponding peripherals of
$C(\widetilde{\mathcal V}_v)$, as described in \Defnref{E.5} above.
\end{fact}

As a consequence of \Obsref{Y.7}(1)(2) and \Lemref{E.4}(2), we get the following.

\begin{cor}
  \corlabel{E.7}
Let $P$ be any peripheral in the core $C(\widetilde{\mathcal V}_v)$,
and let $\xi\in{\mathcal H}_v$ denote its type.
Then there is precisely
$n_\xi-1$ arms in ${\mathcal X}_V(\widetilde{\mathcal V}_v)$ which are attached to 
$C(\widetilde{\mathcal V}_v)$ along $P$. Moreover, among these arms we have precisely 
one arm for each label $(\xi,j):1\le j\le n_\xi-1$ (for the labelling $\tilde\lambda_v$).
\end{cor}

As a final element in this section,
we make a record of a few further rather obvious observations concerning 
orientation phenomena in the extended Whitehead
graphs and in their associated V-tree systems.

\begin{rmk}
  \rmklabel{E.5}
\begin{enumerate}
\item[(1)]
Note that the edges $e_1,\dots,e_{n_\xi-1}$ which correspond to an edge $e$ of $W_v$
in the above description of 
the extended Whitehead graph $\widetilde W_v$ have the same endpoints as
$e$  (under the canonical identification of the vertex sets of $W_v$ and $\widetilde W_v$).
It follows that any orientation for $e$ induces naturally orientations for all $e_1,\dots,e_{n_\xi-1}$.
\item[(2)] Let $P$ be a peripheral of the core $C(\widetilde{\mathcal V}_v)$. Then it is also a peripheral
of the core $C({\mathcal V}_v)$. Recall that $E_{W_v}^P$ and $E_{\widetilde W_v}^P$
denote then the sets of edges  associated to $P$ in the corresponding graphs $W_v$ and $\widetilde{W}_v$, 
as described in \Rmkdefnref{V.11}(2).
Using the same notation as in the previous remark, we have that $e\in E_{W_v}^P$ if and only if
all edges $e_1,\dots,e_{n_\xi-1}$ belong to $E_{\widetilde W_v}^P$.
\item[(3)] Under notation as in the above remark (2), if we fix an order of the peripheral $P$,
we get \emph{order-induced} orientations at all edges of $\widetilde W_v$ associated to $P$ (see \Obsref{Y.7}(3)).
Moreover, these orientations are compatible (in the sense of remark (1) above) with the
order-induced orientations of the edges of $W_v$ associated to $P$
(with respect to the same fixed order of $P$). 
\end{enumerate}
\end{rmk}

\section{Structure of a regular tree of graphs for $\partial G$ in terms of the reduced 
Bow\-ditch JSJ 
splitting of $G$}
\seclabel{S}

Let $G$ be a 1-ended hyperbolic group which is not cocompact Fuchsian,
and suppose it satisfies condition \pitmref{g_rigid_cluster_factors_vf} of \Mainthmref{1} (i.e. suppose that each rigid cluster factor of the reduced Bowditch JSJ splitting of $G$ is virtually free).
In this section we associate to $G$ 
some graphical connecting system ${\mathcal R}_G=(\Gamma,{\mathbf a}, {\mathcal A})$ (see \Defnref{R.1})
for which, as we will show in \Secref{P}, the regular tree of graphs ${\mathcal X}({\mathcal R}_G)$ is 
homeomorphic to the Gromov boundary $\partial G$. 
This connecting system ${\mathcal R}_G$ will be explicitly determined by various
features of the reduced Bowditch JSJ splitting of $G$, 
and its description constitutes a natural complement to condition \pitmref{bd_g_reg_twocon_tog} of \Mainthmref{1},
and to condition (2) of \Mainthmref{2}.
We thus rephrase the implication \pitmref{g_rigid_cluster_factors_vf}$\Rightarrow$\pitmref{bd_g_reg_twocon_tog} in \Mainthmref{1} as a more precise statement, 
presented below as \Thmref{S.1}.
(We postpone the proof of \Thmref{S.1}, and hence of the corresponding
implication \pitmref{g_rigid_cluster_factors_vf}$\Rightarrow$\pitmref{bd_g_reg_twocon_tog} in \Mainthmref{1}, until \Secref{P}.)
We start by
describing the graph $\Gamma$ appearing as part of the
data in ${\mathcal R}_G$.

Let $Z$ be a fixed set of representatives of orbits of the action of $G$ on the subset of
the vertex set 
of $T^r_G$ consisting of the vertices of the following two kinds:
\begin{enumerate}
\item[(1)] rigid cluster vertices;
\item[(2)] white vertices of type (v2) (as described in Sections \ssecref{bowditch_jsj_splitting} and \ssecref{reduced_JSJ})
which are not adjacent to any rigid cluster vertex.
\end{enumerate}

\noindent
If $z\in Z$ is of type (v2), put $\Gamma_z:=\Xi_{k(z)}$ (the $\theta$-graph), 
where $k(z)$ is the degree of $z$ in $T^r_G$
(which is easily seen to coincide with the degree of the corresponding vertex in $T_G$);
note that, by the description of $T_G$ and $T^r_G$ given in \Secref{reduced_JSJ}, 
for each such $z$ we necessarily have $k(z)\ge3$. 
If $z\in Z$ is a rigid cluster vertex, put $\Gamma_z:=\widetilde W_z$
(the extended Whitehead graph for $G(z)$, as described in \Secref{E}). Finally, put 
$\Gamma:=\bigsqcup_{z\in Z}\Gamma_z$ (the disjoint union of the graphs $\Gamma_z$).

We now describe the $V$-involution $\mathbf a$ appearing as part of data in ${\mathcal R}_G$.
Recall that the extended Whitehead graphs $\widetilde W_z$ come equipped with
$V$-involutions ${\mathbf{\tilde a}}_z=(\tilde a_z,\{ (\tilde\alpha_z)_v:v\in V_{\Gamma_z} \})$.
Equip also all the $\theta$-graphs 
$\Gamma_z=\Xi_{k(z)}\subset\Gamma$ with their standard $V$-involutions
(with which they form standard connecting $V$-systems of $\theta$-graphs),
as described in the last paragraph of 
\Exref{V.8.2}, and denote these involutions also by 
${\mathbf{\tilde a}}_z=(\tilde a_z,\{ (\tilde\alpha_z)_v:v\in V_{\Gamma_z} \})$.
Define the $V$-involution ${\mathbf a}=(a,\{ \alpha_v\})$ for $\Gamma$ 
as follows. 
The involution $a:V_\Gamma\to V_\Gamma$
is given by $a(v):=\tilde a_z(v)$ for each $v\in V_{\Gamma_z}\subset V_\Gamma$.
Moreover, for each $v\in V_\Gamma$, if $v$ is a vertex of a component $\Gamma_z$,
put $\alpha_v=(\tilde\alpha_z)_v$.

It remains to describe the set $\mathcal A$ of $E$-connections for $\Gamma$ appearing in ${\mathcal R}_G$. 
For this we need some preparations which concern the action of $G$ on the tree $T^r_G$. 
For each vertex $z\in Z$ of type (v2), fix any bijective correspondence between the set of edges
of the $\theta$-graph $\Gamma_z=\Xi_{k(z)}$ and the set $N_z$ of vertices of $T^r_G$
adjacent to $z$ (which are all flexible). 
Denote this bijection by $\lambda_z^+$.
Fix also an identification of the vertex set of $\Gamma_z$
(which consists of 2 vertices) with the doubleton $\partial G(z)\subset\partial G$.
Next, recall that
for each rigid cluster vertex $z\in Z$, we have a natural bijective correspondence between the set 
of orbits in $N_z$ under the action of $G(z)$ (i.e. the quotient $N_z/G(z)$) 
and the set ${\mathcal H}_z$
(of conjugacy classes of peripheral subgroups of $G(z)$, 
under the view of  $G(z)$ as an abstract factor, see \Exref{F.1.1}).
This correspondence is given by assigning to any orbit $[u]\in N_z/G(z)$
(with $u\in N_z$ as any representative of this orbit) the conjugacy class of the subgroup
$G(u)<G(z)$ (note that, due to \Lemref{E.1}, 
the groups $G(u)$ are indeed contained in $G(z)$). For any $\xi\in{\mathcal H}_z$, we fix some representative $u_\xi\in N_z$
of the orbit in $N_z/G(z)$ corresponding to $\xi$, and we also fix a bijective correspondence
between the vertices in the set $N_{u_\xi}\setminus\{ z \}$ (which are all flexible, by the fact that
$z$ is isolated, see \Factref{3.4}) and the set of numbers
$1,2,\dots,n_\xi-1$, where $n_\xi$ is the degree of $u_\xi$ in $T^r_G$.
This induces a bijective correspondence between the set 
$\widetilde{\mathcal H}_z=\{ (\xi,j):\xi\in{\mathcal H}_z, 1\le j\le n_\xi-1 \}$
and the set $N^2_z$ of all those vertices $y$ in $T^r_G$ lying at distance 2 from $z$,
for which the shortest path between $y$ and $z$ in $T^r_G$ 
passes through one of the vertices $u_\xi$.
We denote by $\lambda^+_z:E_{\widetilde W_z}\to N^2_z$ the labelling 
obtained from $\tilde\lambda_z$
(where $\tilde\lambda_z$ is the labelling described in \Secref{E},
right before Fact \factref{E.2})
by replacing the labels from $\widetilde{\mathcal H}_z$ with the corresponding
elements from $N^2_z$.

To describe the set $\mathcal A$,
consider the set $Y$ of orbits of the action of $G$ on the set of all flexible
vertices of $T^r_G$. The set $\mathcal A$ will be given as the (disjoint) 
union of the sets ${\mathcal A}_y:y\in Y$, which we call \emph{blocks}. We now pass to describing
these blocks ${\mathcal A}_y$.
Fix any orbit $y\in Y$.
For any $z\in Z$ of type $(v2)$, let $y^z_i:i\in I_{y,z}$ be the set of all vertices adjacent
to $z$ in the orbit $y$. Similarly, for any rigid cluster vertex $z\in Z$,  let $y^z_i:i\in I_{y,z}$ 
be the set of all vertices from the set $N^2_z$ in the orbit $y$. (Note that for some $z$,
the index sets $I_{y,z}$ may be empty.)
Now, if $z\in Z$ is of type (v2), each $y_i^z$ as above belongs to $N_z$, 
and hence  the preimage 
$(\lambda^+_z)^{-1}(y_i^z)$ is a subset of the edge set of $\Gamma_z=\Xi_{k(z)}$ 
consistsing of a single edge, namely the edge labelled with $y_i^z$. 
If $z\in Z$ is a rigid cluster vertex, each $y_i^z$ as above belongs to the set $N^2_z$, 
and hence the preimage 
$(\lambda^+_z)^{-1}(y_i^z)$ is a subset of the edge set of the graph 
$\Gamma_z=\widetilde W_z$, namely it also consists of these edges of $\widetilde W_z$
which are labelled with $y_i^z$ (this time there may be more than one such edge).
Put 
\[
E_y=\bigcup_{z\in Z}\bigcup_{i\in I_{y,z}} (\lambda^+_z)^{-1}(y_i^z).
\]
It is not hard to observe
that the family $E_y:y\in Y$ is then a partition of the edge set of $\Gamma$, i.e., each edge of
$\Gamma$ belongs to precisely one of the sets $E_y$.

For any orbit $y\in Y$, fix any vertex $v_y\in y$, and note that the group $G(v_y)$ is,
up to an isomorphism (a conjugation in $G$), independent of this choice.
If $G(v_y)$ is non-orientable (when viewed as a flexible factor),
denote by $E^\pm_y$ the set of all oriented edges of $\Gamma$ (with both orientations) whose underlying
non-oriented edges are in $E_y$. Put ${\mathcal A}_y:=E^\pm_y\times E^\pm_y$.
On the other hand,
if $G(v_y)$ is orientable, in order to describe the set ${\mathcal A}_y$, 
we need to take more subtle aspects of orientations into account.
In particular, we will now describe how to associate to any edge
$\varepsilon\in E_y$ some distinguished orientation, with which 
it will be denoted $\varepsilon^\#$.
To this aim, fix in advance, in a $G$-invariant way, 
orientations for all flexible factors $G(v)$ with $v\in y$.
For any edge $\varepsilon\in E_y$, let $\Gamma_z$ be the component of $\Gamma$ to which
$\varepsilon$ belongs. Put $y_\varepsilon=\lambda_z^+(\varepsilon)$,
note that it belongs to the orbit $y$,
and consider the orientation of the flexible factor $G(y_\varepsilon)$ as fixed above.
This orientation induces orientations of all peripheral subgroups of this factor
(in the sense of compatibility of orientations as in \Defnref{X.4}; see also \Lemref{X.3}(2)).
To describe the oriented edge $\varepsilon^\#$, 
we proceed separately in the following two cases.

\emph{Case 1.}
If $z$ is a rigid cluster vertex 
(so that $\Gamma_z$ is the extended Whitehead graph for 
the rigid cluster factor ${\mathcal Q}_z = (G(z),\mathcal{H}_z)$),
put $\xi=\lambda_z(\varepsilon)\in{\mathcal H}_z$. 
We then have $y_\varepsilon\in N_{u_\xi}\setminus\{ z \}$,
so the fixed orientation of $G(y_\varepsilon)$ induces an orientation of 
its peripheral subgroup $G(y_\varepsilon)\cap G(u_\xi)$.
This obviously determines an 
order of the
doubleton $\partial G(u_\xi)$.
Since $G(u_\xi)$ is a peripheral subgroup of the factor $G(z)$ (compare \Lemref{E.1}),
$\partial G(u_\xi)$ is a peripheral of the space with peripherals associated to $\mathcal{Q}_z$.
Identifying the latter naturally with the core $C(\widetilde{\mathcal V}_z)$
of the associated $V$-tree ${\mathcal X}_V(\widetilde{\mathcal V}_v)$ (as in \Lemref{E.4}(1)),
$\partial G(u_\xi)$ is a peripheral of this core.
Since $\varepsilon$ is an edge of the extended Whitehead graph $\widetilde W_z=\Gamma_z$,
and since it belongs to edges associated to the peripheral $\partial G(u_\xi)$,
it carries the order-induced orientation (with respect to the above described order
of $\partial G(u_\xi)$).  (See \Rmkref{E.5}(3).)
We denote by $\varepsilon^\#$ the edge $\varepsilon$ with this order-induced orientation.

\emph{Case 2.}
If $z$ is of type (v2) (and non-adjacent to any rigid cluster vertex), $\Gamma_z=\Xi_{k(z)}$
is a $\theta$-graph. The edge $\varepsilon$ corresponds then to the unique vertex $y_i^z\in N_z$.
Again, since $G(y_i^z)$ is oriented,
it induces the order on the boundary set $\partial G(z)$, and hence
(by our earlier fixed identification) on the vertex set of $\Gamma_z$.
The latter induces an orientation on $\varepsilon$, and we again denote by $\varepsilon^\#$
the edge $\varepsilon$ with this induced orientation.

Put $E_y^\#:=\{ \varepsilon^\#:\varepsilon\in E_y \}$, and put also
\[
{\mathcal A}_y:=
\{ (o_1,\bar o_2): o_1,o_2\in E_y^\# \} \cup
\{ (\bar o_1,o_2): o_1,o_2\in E_y^\# \}
\]
where $\bar o_i$ denotes the oppositely oriented edge $o_i$,
for $i=1,2$.

Having defined all the blocks ${\mathcal A}_y$, we put ${\mathcal A}=\bigcup_{y\in Y}{\mathcal A}_y$.
We skip the straightforward verification of the fact 
that the so described set ${\mathcal A}$ does not
depend on the choices of the set of representatives $Z$, and the representatives $v_y$ 
of the orbits $y\in Y$.
This completes the description of the graphical connecting system ${\mathcal R}_G$.
The fact that ${\mathcal R}_G$ 
satisfies the transitivity property (4) from \Defnref{R.1}
follows, roughly speaking, from the fact that the tree $T^r_G$ is connected,
and we omit the details.

Given the graphical connecting system ${\mathcal R}_G$, as introduced above, we may
state the following result, which describes explicitly the form of a regular tree of graphs
which appears only implicitly in our statement of Therorem A
(in condition (2)). This result, 
\Thmref{S.1} below, may be viewed as the third main result of the paper.
Its proof, after necessary preparations provided in the next two sections, will be presented in \Secref{P}.

\begin{thm}
  \thmlabel{S.1}
Let $G$ be a 1-ended hyperbolic group which is not cocompact Fuchsian,
and suppose that each rigid cluster factor of its reduced Bowditch JSJ splitting
is virtually free.
%
Let 
${\mathcal R}_G$ be the graphical connecting system associated to $G$
(as described above in this section).
Then the Gromov boundary $\partial G$ is homeomorphic to the 
regular tree of graphs ${\mathcal X}({\mathcal R}_G)$.
\end{thm}

We finish this section with some further observations concerning the 
edges in the graph $\Gamma$ of the connecting system ${\mathcal R}_G$.
These observations will be used
in the  proof of \Thmref{S.1} given in \Secref{P}.

Let $z\in Z$ be one of the orbit representatives fixed at the beginning of the above 
description of the graphical connecting system ${\mathcal R}_G$. Suppose first that
$z$ is a rigid cluster vertex of $T^r_G$ (the remaining case when $z$ is of type (v2) will be
discussed later). A \emph{bunch} of edges in the extended Whitehead graph
$\widetilde W_z$ is any collection $\{ e_1,\dots,e_{n_\xi-1} \}$ of edges that replace
an edge $e$ of the Whitehead graph $W_z$, as in the description of
$\widetilde W_z$ given in \Secref{E}, where $\xi=\lambda_z(e)\in{\mathcal H}_z$
is the label of $e$. Clearly, all edges in any bunch share the set of endpoints.
We assign to a bunch as above the same label $\xi$.

Now, as mentioned in the above description of ${\mathcal R}_G$,
any bunch of edges $\mathcal M$ labelled with $\xi$ is in a fixed bijective correspondence
(provided by the labelling $\lambda_z^+$) with the vertex set 
$N_{u_\xi}\setminus\{ z \}$. This correspondence allows us to split $\mathcal M$ into subsets
corresponding to the orbits of the action of $G(u_\xi)$ on $N_{u_\xi}\setminus\{ z \}$.
More precisely, observe that due to \Lemref{E.1}, $z$ is a global fixed point of the action
of $G(u_\xi)$ on $N_{u_\xi}$ (and an alternative justification follows by the fact
that $z$ is the unique rigid cluster vertex in $N_{u_\xi}$, due to its isolatedness). Consequently, the $G(u_\xi)$-orbit
$[x]$ of any vertex $x\in N_{u_\xi}\setminus\{ z \}$ is contained in
$N_{u_\xi}\setminus\{ z \}$. The subset of $\mathcal M$ corresponding to any such 
orbit $[x]$ will be called a \emph{packet} of edges in $\widetilde W_z$.

We will later need the following easy but a bit technical observation,
the proof of which we omit. 

\begin{fact}
  \factlabel{S.2}
Let $z\in Z$ be a rigid cluster vertex, let $\xi\in{\mathcal H}_z$, 
let $\mathcal M$ be a bunch of edges in $\Gamma_z=\widetilde W_z$ labelled with $\xi$,
and let $M\subset{\mathcal M}$ be a packet corresponding to the
$G(u_\xi)$-orbit $[x]$ of some vertex $x\in N_{u_\xi}\setminus\{ z \}$.

\begin{enumerate}
\item[(1)]
The cardinality of the packet $M$ coincides with the index
$[G(u_\xi):G(u_\xi)\cap G(x)]=[G(u_\xi):G([u_\xi,x])]$.
\item[(2)]
Having fixed $\mathcal M$, for all $M$ and $x$ as above for which the corresponding
flexible factor $G(x)$ is orientable, we have

\begin{enumerate} 
\item[(a)]
if the group $G(u_\xi)$ is orientable then
the orientations of all (oriented)
edges
$\varepsilon^\#$ with $\varepsilon$ in $M$ (as defined in the above description of the set
of E-connections for ${\mathcal R}_G$) are compatible, i.e., they induce the same order
on the common set of endpoints of these edges;
\item[(b)]
if the group $G(u_\xi)$ is non-orientable 
then the cardinality of the packet $M$ is even,
and half of the orientations of the edges $\{ \varepsilon^\#:\varepsilon\in M \}$
are compatible with one of the orders in the common set of endpoints,
while the other half of these orientations are compatible with the opposite order.
\end{enumerate}
\end{enumerate}
\end{fact}

Let now $z\in Z$ be a vertex of type (v2). Recall that, by the description of $Z$,
the set $N_z$ of all vertices in $T^r_G$ adjacent to $z$ contains no rigid cluster vertex.
Recall also that the component $\Gamma_z$ of the graph $\Gamma$
is isomorphic to the $\theta$-graph $\Xi_{k(z)}$, where $k(z)$ is the degree
of $z$ in $T^r_G$, and that we have fixed a bijective correspondence between
the edge set $E_{\Gamma_z}$ and the set $N_z$. For any $G(z)$-orbit $[x]$
(where $x$ is any vertex of $N_z$), the corresponding set of edges of the $\theta$-graph $\Gamma_z$
will be again called a \emph{packet} of edges in $\Gamma_z$.

We have the following observation similar to \Factref{S.2}, the proof of which we also skip.

\begin{fact}
  \factlabel{S.3}
Let $z\in Z$ be a vertex of type (v2), and let $M$ be the packet of 
edges in the $\theta$-graph $\Gamma_z$ corresponding to the
$G(z)$-orbit $[x]$ of some vertex $x\in N_z$.

\begin{enumerate}
\item[(1)]
The cardinality of the packet $M$ coincides with the index
\[[G(z):G(z)\cap G(x)]=[G(z):G([z,x])].\]
\item[(2)]
For all $M$ and $x$ as above for which the corresponding
flexible factor $G(x)$ is orientable, we have
\begin{enumerate}
\item[(a)]  
if $G(z)$ is orientable then
the orientations of all (oriented)
edges
$\varepsilon^\#$ with $\varepsilon$ in any such $M$ (as defined in the above description of the set
of E-connections for ${\mathcal R}_G$) are compatible, i.e., they induce the same order
on the common set of endpoints of these edges;
\item[(b)]
if $G(z)$ is non-orientable then
the cardinality of the packet $M$ is even,
and half of the orientations of the edges $\{ \varepsilon^\#:\varepsilon\in M \}$
are compatible with one of the orders in the common set of endpoints,
while the other half of these orientations are compatible with the opposite order.
\end{enumerate}
\end{enumerate}
\end{fact}

\section{Regular tree systems of spaces with doubleton peripherals}
\seclabel{D}

In this section we prepare some tools and terminology needed in our proof of \Thmref{S.1}
(from which we then deduce the implications \pitmref{g_rigid_cluster_factors_vf}$\Rightarrow$\pitmref{bd_g_reg_twocon_tog} in \Mainthmref{1}
and (3)$\Rightarrow$(2) in \Mainthmref{2}).  By a 
\emph{space with doubleton peripherals} we mean
a space with peripherals in which each peripheral subspace consists of precisely
two points (i.e. is a doubleton). We start with a rather long definition.

\begin{defn}
  \defnlabel{D.1}
  \begin{enumerate}
\item A \emph{decoration} $\mathcal D$ of a space with doubleton peripherals $X$ is a collection
of the following data:
\begin{itemize}
\item a labelling $\tau$ which to any peripheral subspace $P$ of $X$ 
associates an element $\tau(P)$ from a finite set $\mathcal T$; $\tau(P)$ is then called
\emph{the type} of $P$; 
\item
for each peripheral $P$, either a single fixed order of the set of two elements of $P$,
called \emph{the orientation} of $P$ (in which case $P$ is called \emph{orientable}),
or an artificial assignment to $P$ of both orders of its two elements
(in which case we say that $P$ is \emph{non-orientable});
we also assume that the peripherals of a fixed type are either all orientable
or all non-orienbtable, so that we can speak of orientable and non-orientable types.
\end{itemize}

A pair $(X,{\mathcal D})$ as above will be called a \emph{decorated space with doubleton peripherals}.

\item Given a space with doubleton peripherals $X$ and its decoration $\mathcal D$,
the set of \emph{directed peripherals} of $(X,{\mathcal D})$ consists of all peripherals
with all assigned orientations (one directed peripheral for each orientable $P$, and two
directed peripherals for each non-orientable $P$). A \emph{type} of a directed peripheral
is the type of its underlying counterpart (i.e. of the same peripheral with
its orientation ignored).

\item A \emph{decoration-preserving} homeomorphism of $(X,{\mathcal D})$
is a homeomorphism $h$ of $X$ which maps peripherals to peripherals,
which preserves the types (i.e. $\tau(h(P))=\tau(P)$ for any peripheral $P$),
and which preserves the orientations of the peripherals. The latter means
that $h$ maps non-orientable peripherals to non-orientable ones, and it maps
orientable peripherals to orientable ones, preserving their assigned orientations
(in other words, the set of all directed peripherals of $(X,{\mathcal D})$ is preserved by $h$).
Note that the inverse $h^{-1}$ of any decoration-preserving homeomorphism $h$
of $(X,{\mathcal D})$ is also a decoration-preserving homeomorphism
of $(X,{\mathcal D})$.

\item A pair $(X,{\mathcal D})$ as above is called a \emph{regular space with doubleton peripherals}
if it satisfies the following \emph{regularity} condition: for any two directed peripherals
of $(X,{\mathcal D})$ having the same type there is a decoration-preserving homeomorphism
of $(X,{\mathcal D})$ which maps the first of these two directed peripherals to the second one.
In particular, if $P$ is a non-orientable peripheral, there is a decoration-preserving
homeomorphism which maps $P$ to $P$ and for which the restriction $h|_P$
is a transposition.
  \end{enumerate}
\end{defn}

We now present a few classes of examples of regular spaces with doubleton peripherals.
These examples will be used in  our later considerations and arguments
concerning
the proof of \Thmref{S.1}.

%

\begin{ex}
  \exlabel{D.3}
Recall that to any abstract factor ${\mathcal F}=(K,\{ H_i \}_{i\in I})$ 
(as in \Ssecref{F})
there is associated the induced space with doubleton peripherals, 
as described in \Defnref{F.1.2}; we denote this space with peripherals by $X_{\mathcal F}$.
We describe a decoration ${\mathcal D}={\mathcal D}_{\mathcal F}$ 
naturally associated to $X_{\mathcal F}$,
which turns out to satisfy the regularity condition of \Defnref{D.1}(4),
and which yields what we call \emph{the regular space with doubleton peripherals
associated to an abstract factor}.

The labelling $\tau$ of $\mathcal D$ 
has values in the set ${\mathcal T}:=I$ and
is given by 
the formula 
\[\tau(\partial(gH_ig^{-1})):=i.\]
The orientability (and non-orientability) data for the peripherals, as well as the assigned
orientations, are described by inducing in the obvious way the similar data for the
corresponding peripheral subgroups (provided by fixed orientation data for the factor,
as given in \Defnref{F.3}). By \Lemref{F.4}, 
the so described decoration ${\mathcal D}_{\mathcal F}$ satisfies the regularity condition.
\end{ex}

\begin{ex}[densely decorated punctured circle]
  \exlabel{D.4}
Let $(K,\Omega)$ be a punctured circle, i.e., a space with peripherals (as described
right before \Defnref{T.2}) in which $K=(S^1)^\circ$, with its standard family
$\Omega$ of peripheral subspaces, which are doubletons.
Fix an arbitrary auxiliary orientation of $(K,\Omega)$ (in the sense of \Rmkref{X.1.2}).
Let ${\mathcal D}$ be any decoration for $(K,\Omega)$ which is \emph{dense},
i.e., satisfies one of the following two conditions:
\begin{enumerate}
\item[(a)]
for each type of peripheral, the union of all
directed peripherals of this type in $((K,\Omega),{\mathcal D})$, 
with orientations compatible (respectively, incompatible)
with the fixed orientation of  $(K,\Omega)$, is dense in $K$;
that is, if a type of peripherals is  non-orientable, we require simply
that the union of peripherals of this type is dense in $K$,
and if a type is orientable, we require both that the union of all those peripherals of this type 
whose orientations are \emph{compatible} with the fixed
orientation of $(K,\Omega)$ is dense in $K$ and that the union of all those peripherals of this type 
whose orientations are \emph{incompatible} with the fixed
orientation of $(K,\Omega)$ is dense in $K$;
\item[(b)]
all peripherals are orientable, and for each type either all peripherals of this type are
oriented consistently with the fixed orientation of $(K,\Omega)$, or they are all oriented inconsistently with this orientation; in any case, the union of the peripherals of any fixed type is dense in $K$.
\end{enumerate}

\noindent 
We then say that the pair $((K,\Omega),{\mathcal D})$ is a \emph{densely decorated punctured circle}.
We will speak of two kinds of densely decorated punctured circles: those satisfying
condition (a) above will be called \emph{non-orientable}, and those satisying condition (b) will be called \emph{orientable}. The next lemma collects some useful properties of densely decorated punctured circles,
and we omit its straightforward proof which can be performed by a standard
back and forth argument.
\end{ex}

\begin{lem}
  \lemlabel{D.4.1}
\begin{enumerate}
\item[(1)] Any densely 
decorated punctured circle is a regular space with doubleton peripherals.
\item[(2)] Any two orientable densely decorated punctured circles with the same number of types of peripherals are isomorphic (via a decoration-preserving homeomorhism $h$), 
after appropriate
identification of the types, and after possibly reversing orientations of all peripherals
of some of the types. Moreover, for any bijection $\beta$ between the sets of types of the
punctured circles above 
(which we view as the identification of the types)
there is a homeomorphism $h$ as above which is compatible with $\beta$
(i.e. such that for any peripheral $P$ we have $\tau(h(P))=\beta(\tau(P))$).
\item[(3)] Any two 
non-orientable densely decorated punctured circles with the same number
of types of orientable peripherals, and the same number 
of types of non-orientable peripherals,
are isomorphic, after appropriate
identification of the types. Moreover, for any bijection $\beta$ between the sets of types of the
punctured circles above, if $\beta$ respects orientability and non-orientability of the types,
then there is an isomorphism $h$ between those decorated punctured circles 
which is compatible with $\beta$.
\end{enumerate}
\end{lem}

%

\begin{rmk}
  \rmklabel{D.4.2}
We should note here that if ${\mathcal F}=(K,\{ H_i \}_{i\in I})$ is an abstract
flexible factor, then the space with peripherals  $(X_{\mathcal F},{\mathcal D}_{{\mathcal F}})$
associated to $\mathcal F$ (as described in \Exref{D.3}), is a densely decorated
punctured circle (see \Factref{X.5} for the justification).
This densely decorated punctured circle 
is easily seen to be orientable if and only if the flexible factor $\mathcal F$ itself
is orientable. The numbers of types of orientable and non-orientable peripherals
in $(X_{\mathcal F},{\mathcal D}_{{\mathcal F}})$ coincide then with the numbers of 
(the conjugacy classes of) the peripheral subgroups $H_i$ which are orientable and non-orientable,
respectively.
\end{rmk}

We now proceed one level further, and discuss some tree systems made out of
regular spaces with doubleton peripherals. As we will see, such tree systems
are induced, in an essentially unique way,  by rather simple data described in the next definitions.

\begin{defn}
  \defnlabel{D.0}
Given an arbitrary regular (decorated) space with doubleton peripherals $((K,\Omega),{\mathcal D})$,
a \emph{degree function} for this space is a positive integer valued function $d:\Omega\to \Z$
which is constant on types, i.e., such that 
$d(P_1)=d(P_2)$ for any two peripherals $P_1,P_2$ 
of the same type. We call a collection of data ${\mathcal D}^+=({\mathcal D},d)$ 
consisting of a decoration $\mathcal D$ and
a degree function $d$ an \emph{extra-decoration} for $(K,\Omega)$.
The pair $((K,\Omega),{\mathcal D}^+)$ is then called an \emph{extra-decorated regular space
with doubleton peripherals}.
\end{defn}

\begin{defn}
  \defnlabel{D.5}
Given a finite family ${\mathcal K}=((K_j,\Omega_j),{\mathcal D}_j^+)_{j\in J}$ of regular 
extra-decorated
spaces with doubleton peripherals, denote by 
${\mathcal T}={\mathcal T}_{\mathcal K}=\bigsqcup_{j\in J}{\mathcal T}_j$ the disjoint union of the sets
${\mathcal T}_j$ of types of the peripherals in the families $\Omega_j$.
\emph{A connecting system for regular spaces with doubleton
peripherals} is a pair ${\mathcal P}=({\mathcal K},{\mathcal B})$, where $\mathcal K$ is a family as above,
and where $\mathcal B$, called a \emph{block structure},  consists of the following data:
\begin{enumerate}
\item a partition  ${\mathcal T}_{\mathcal K}=\bigsqcup_{a\in A}B_a$ of 
the set of types ${\mathcal T}_{\mathcal K}$ into subsets $B_a$ called \emph{blocks};  
\item for each block $B_a$, additional on information whether $B_a$ is \emph{orientable}
or \emph{non-orienta\-ble}; we require that if $B_a$ is orientable then it necessarily consists of
orientable types only, but if $B_a$ is non-orientable, there is no restriction on orientability
and non-orientability of the types that it contains.
\end{enumerate}

\noindent
We also require in the definition 
that the block structure $\mathcal B$ fulfills the following \emph{transitivity} condition:
for any two indices $j,j'\in J$ there is a sequence of indices
$j=j_0,j_1,\dots,j_k=j'$ such that for each $i=1,\dots,k$ there exist types
$\tau\in{\mathcal T}_{j_{i-1}}$ and $\tau'\in{\mathcal T}_{j_i}$ which belong
to the same block.
\end{defn}

\begin{defn}
  \defnlabel{D.6}
Let $\Theta=(T, \{ L_t\}), \{ \Sigma_u \}, \{ \varphi_e \})$ be a tree system,
and let \[{\mathcal P}=(((K_j,\Omega_j),{\mathcal D}_j^+)_{j\in J},{\mathcal B})\] be
a connecting system for regular spaces with doubleton peripherals. 
For each $j\in J$, denote by 
$d_j:\Omega_j\to \Z$ the degree function of the extra-decoration ${\mathcal D}_j^+$.
We say that $\Theta$ is \emph{modelled on} $\mathcal P$ if each constituent space $L_t$ of $\Theta$ 
can be equipped with an identification (as a space with peripherals) 
with one of the spaces $K_j$ of $\mathcal P$ (and hence also equipped with the corresponding 
extra-decoration
${\mathcal D}_j^+$) so that
the following conditions are satisfied:

\begin{enumerate}
\item[(1)] for each white vertex $u$ of $T$ there is a block $B$ of ${\mathcal B}$ such that
the set of types of the peripherals of the form $\varphi_{[u,t]}(\Sigma_u)\subset L_t$
(under the identification of $L_t$ with one of the $K_j$, as above),
where $t$ runs through the set of all vertices of $T$ adjacent to $u$, coincides with $B$;

\item[(2)] for each $u$ and $B$ as in (1), if the block $B$ is orientable,
then for any two edges $e_1=[u,t_1]$ and $e_2=[u,t_2]$ adjacent to $u$ the map
$\varphi_{e_2}\varphi_{e_1}^{-1}$ between the appropriate peripherals
in $L_{t_1}$ and $L_{t_2}$ 
preserves the orientations of these peripherals
(provided by the identifications of both $L_{t_i}$'s with the corresponding $K_j$'s, as above);
moreover, we assume in this case that for any $\lambda\in B$, denoting by $K_j$
this unique space of $\mathcal P$ which contains peripherals of type $\lambda$,
the number of (black) vertices $t$ adjacent to $u$ for which the spaces $L_t$ 
are identified with $K_j$ and the peripherals $\varphi_{[u,t]}(\Sigma_u)$ are of type $\lambda$
is equal to the degree $d_j(\varphi_{[u,t]}(\Sigma_u))$ (for any $t$ as above);
\item[(3)] for any $u$ and $B$ as in (1), if the block $B$ is non-orientable,
and if $\lambda\in B$ is a non-orientable type then, denoting by $K_j$
this unique space of $\mathcal P$ which contains peripherals of type $\lambda$,
the number of (black) vertices $t$ adjacent to $u$ for which the spaces $L_t$ 
are identified with $K_j$ and the peripherals $\varphi_{[u,t]}(\Sigma_u)$ are of type $\lambda$
is equal to the degree $d_j(\varphi_{[u,t]}(\Sigma_u))$ (for any $t$ as above);
\item[(4)]
for any $u$ and $B$ as in (3) (i.e. with $B$ non-orientable), 
if $\lambda\in B$ is an orientable type then,
again denoting by $K_j$
this unique space of $\mathcal P$ which contains peripherals of type $\lambda$,
the number of (black) vertices $t$ adjacent to $u$ for which the spaces $L_t$ 
are identified with $K_j$ and the peripherals $\varphi_{[u,t]}(\Sigma_u)$ are of type $\lambda$
is equal to twice the degree, i.e., to $2\cdot d_j(\varphi_{[u,t]}(\Sigma_u))$ (for any $t$ as above);
moreover, if we choose an auxiliary orientation of the doubleton $\Sigma_u$
(i.e. an order for its two points) then for exactly half of the vertices $t$ as above (i.e. for
$d_j(\varphi_{[u,t]}(\Sigma_u))$ of those vertices) the  map $\varphi_{[u,t]}$ preserves
the orientations, and for another half of such vertices it reverses the orientations.
\end{enumerate}
\end{defn}

Informally speaking, the above conditions mean that the constituent spaces of $\Theta$,
after their identifications with the spaces $K_j$,
are glued to each other in $\Theta$ along the peripherals belonging to the same blocks 
$B$ of $\mathcal B$,
in numbers appropriately determined by the degrees.
Moreover, if a block $B$ is orientable, all the gluings preserve the orientations,
and if $B$ is non-orientable, the condition is a bit more complicated, namely:
gluings along non-orientable peripherals are arbitrary, and those along orientable
peripherals are performed in equal numbers with both possible orientations, for each type.

Any tree system $\Theta$ which is modelled on some connecting system for
regular spaces with doubleton peripherals will be called
a \emph{regular tree system of spaces with doubleton peripherals}.
Any family of identifications of the constituent spaces $(L_t,\Omega_t)$ of $\Theta$
with the spaces $(K_j,\Omega_j)$ of $\mathcal P$ for which conditions (1)-(4) above are satisfied
will be called a \emph{$\mathcal P$-atlas} for $\Theta$.

The significance of the concept of being \emph{modelled on a connecting system
for regular spaces with doubleton peripherals} for our arguments in the proof
of \Thmref{S.1}, as given in \Secref{P} below, comes from the following observation.

\begin{lem}
  \lemlabel{D.7}
Any two tree systems modelled on the same connecting system $\mathcal P$
for regular spaces with doubleton peripherals are isomorphic (as tree systems).
\end{lem}

We omit a rather straightforward (though technically inconvenient to expose)
proof of this lemma, indicating only that
an isomorphism from the assertion can be constructed recursively,
in a way respecting the identifications from the $\mathcal P$-atlases, by refering
to the regularity of the involved decorated spaces with doubleton peripherals
(see part (4) of \Defnref{D.1}).
To start such a recursive construction, one needs to refer
to the transitivity condition in \Defnref{D.5} to find some initial 
constituent spaces in both systems which are identified with the same space $K_j$
of the corresponding connecting system.

The next example will play a central role in the final step of the proof of \Thmref{S.1},
as given in \Secref{P}.

\begin{ex}[$\Theta^r(G)$ as a regular tree system of spaces with doubleton
peripherals]
  \exlabel{D.8}
We will show that any tree system $\Theta^r(G)$, as described in \Ssecref{reduced_JSJ}, has
a natural structure of a regular tree system of spaces with doubleton peripherals.
We will do this by explicitly describing a connecting system ${\mathcal P}={\mathcal P}^r_G$
on which the corresponding tree system $\Theta^r(G)$ is modelled.
We refer to the terminology and notation as in \Ssecref{reduced_JSJ}.

To describe the appropriate family $(K_j,\Omega_j)_{j\in J}$ of regular spaces with 
doubleton peripherals in the connecting system ${\mathcal P}^r_G$, consider a set $v_j:j\in J$ of
representatives of orbits of the action of $G$ on the set of black vertices of $T^r_G$
(i.e. on the set of flexible vertices and rigid cluster vertices).
For each $j\in J$ put 
\[K_j:=\partial G(v_j)\,\,\,\, \hbox{and}\,\,\,\, 
\Omega_j:=\{ \partial G(u): u\in N^{T^r_G}_{v_j} \},\]
where $N^{T^r_G}_{v_j}$ is the set of (white) vertices in $T^r_G$
adjacent to $v_j$.
(In other words, put $(K_j,\Omega_j)$ to be the space with peripherals associated to
the group $G(v_j)$ viewed as an abstract factor in a way described in \Exref{F.1.1}.)

In order to describe appropriate decorations ${\mathcal D}_j$ for the spaces $(K_j,\Omega_j)$,
we need some preparations. Fix in advance, in a conjugacy invariant way,
orientations for all those 2-ended groups $G(u)$ (at the white vertices $u$ of $T^r_G$) 
which are orientable.
Note that this induces the orientations of all those edge groups $G(e)$ for which the edge $e$
has as its endpoint a white vertex $u$ with orientable $G(u)$.
Fix also in a conjugacy invariant way orientations for all orientable groups $G(e)$
at those edges $e$ whose white endpoint groups $G(u)$ are non-orientable.
Altogether, this results with some conjugacy invariant and consistent choices of orientations for
all edge groups $G(e)$ which are orientable. As a consequence, this induces also
the choices of orientations (in a conjugacy invariant way) for all orientable peripheral subgroups
of all black vertex groups $G(v_j)$ viewed as abstract factors.

For any $(K_j,\Omega_j)$ as above, take as
a decoration ${\mathcal D}_j$  all the data
as in \Exref{D.3}, under viewing $G(v_j)$ as an abstract factor,
with the orientations of orientable peripheral subgroups chosen in a way 
described in the previous paragraph. 
To describe the appropriate degree function $d_j$,
consider any peripheral $\partial G(u)$ in $K_j=\partial G(v_j)$, 
where $u$ is some (white) vertex of $T^r_G$
adjacent to $v_j$. Recall that the peripheral subgroup corresponding
to this peripheral of $K_j$ is $G([u,v_j])=G(u)\cap G(v_j)$, 
which we denote briefly by $H_u$.
If $G(u)$ (which is a 2-ended group) is orientable,
we take as $d_j(\partial G(u))$ the index $[G(u):H_u]$.
If both $G(u)$ and $H_u$ are non-orientable, we similarly put 
$d_j(\partial G(u)):=[G(u):H_u]$. However, if $G(u)$ is non-orientable, and $H_u$
is orientable, we notice that the index $[G(u):H_u]$ is an even number,
and we take as $d_j(\partial G(u))$ half of this number. Note that, due to the final comment
in \Exref{D.3}, with the so described extra-decoration ${\mathcal D}_j^+=({\mathcal D}_j,d_j)$,
each of the spaces
$(K_j,\Omega_j)$ becomes an extra-decorated regular space with doubleton
peripherals.

We now pass to the description of appropriate data ${\mathcal B}=\{ B_a:a\in A \}$.
Take as $A$ the set of $G$-orbits of the white vertices of $T^r_G$.
For any orbit $a\in A$, let $u$ be any vertex in this orbit. 
Note that for any (black) vertex $t$ adjacent to $u$, 
the corresponding space with peripherals $\partial G(t)$
is identified, uniquely up to a decoration preserving homeomorphism,
with the appropriate space $K_j=\partial G(v_j)$, where $v_j$ is in the same orbit as $t$,
via any of the maps from the $G$-action of $G$
on $\partial G$ which sends $\partial G(v_j)$ to $\partial G(t)$.
In particular, under the latter identifications, all peripherals in all such  spaces $\partial G(t)$
have well defined types.
Define $B_a$ to be the set
of all types appearing among the peripherals of the form $\varphi_{[u,t]}(\Sigma_u)$
(in the corresponding spaces $\partial G(t)$),
where $t$ runs through all (black) vertices of $T^r_G$ adjacent to $u$. 
Furthermore, we declare that the block $B_a$ is orientable if and only if
the corresponding 2-ended group $G(u)$ is orientable.

We skip the straightforward verification of the fact that all the data above
are well defined.
It is also not hard to verify that the above described data 
satisfy all requirements for being a connecting system for regular spaces with
doubleton peripherals,  and we denote this system by ${\mathcal P}^r_G$. 

It remains to show that the tree system $\Theta^r(G)$ is indeed modelled on 
the connecting system ${\mathcal P}^r_G$.
To do this, we will now describe some ${\mathcal P}^r_G$-atlas ${\mathbf A}^r_G$ for $\Theta^r(G)$.
Given any constituent space $\partial G(t)$ of $\Theta^r(G)$
(where $t$ is a black vertex of $T^r_G$),
consider the unique vertex $v_j$ which lies in the same $G$-orbit as $t$.
Consider then any $g\in G$ which sends $t$ to $v_j$,
and take as an identification in ${\mathbf A}^r_G$ the restriction to $\partial G(t)$
of the homeomorphism of $\partial G$ induced by $g$ (which sends $\partial G(t)$
to $\partial G(v_j)=K_j$).
The verification that ${\mathbf A}^r_G$ is indeed a ${\mathcal P}^r_G$-atlas (i.e. that it satisfies
conditions (1)-(4) of \Defnref{D.6}) becomes (nearly tautologically) clear in view of
the following observation, the easy proof of which we also skip.
\end{ex}

\begin{obs}
  \obslabel{D.8+}
Under the identifications from the atlas ${\mathbf A}^r_G$, the data from the connecting system
${\mathcal P}^r_G$ can be a posteriori interpreted in the following terms corresponding to the
features of the action of $G$ on $T^r_G$:
\begin{itemize}
\item
the set $\mathcal{T}_G$ of all types of the peripherals in all constituent
spaces in the system $\Theta^r(G)$ can be identified with the set of all $G$-orbits $[e]$
of the edges $e$ of $T^r_G$; under this identification, each peripheral $\varphi_e(\Sigma_u)$
is assigned with the type $[e]\in\mathcal{T}_G$;

\item
a type $[e]\in\mathcal{T}_G$ is orientable iff the corresponding 2-ended group $G(e)$
is orientable;

\item
the set $\mathcal B$ of blocks is indexed by the set of $G$-orbits $[u]$ in the set of all white
vertices $u$ in $T^r_G$, and each block $B_{[u]}$ consists of those types $[e]$ which
correspond to the edges $e$ adjacent to $u$;

\item the set of types $\mathcal{T}_G$ is partitioned into subsets 
\[
{\mathcal T}_j=\{ [e]:\text{ $e$ is an edge of $T^r_G$ adjacent to $v_j$} \},
\]
 and each ${\mathcal T}_j$
is the set of types of peripherals in $(K_j,\Omega_j)$;
\item
  if $[e]\in{\mathcal T}_j$, and if $u$ denotes the white endpoint of $e$, then
  \begin{enumerate}
  \item[(a)] if both groups $G(e),G(u)$ are orientable, or both are
    non-orientable, the degree $d_j([e])$ coincides with the number of
    edges adjacent to $u$ in the $G$-orbit of $e$;
  \item[(b)] if $G(e)$ is orientable and $G(u)$ is non-orientable, the
    degree $d_j([e])$ is equal to half of the number of edges adjacent
    to $u$ in the $G$-orbit of $e$ (where the latter number is easily
    seen to be even).
  \end{enumerate}
\end{itemize}
\end{obs}

\begin{rmk}
  \rmklabel{D.9}
Note that, under the notation of the above \Exref{D.8}, if $v_j$ is a flexible vertex, 
the decorated space with doubleton peripherals $((K_j,\Omega_j),{\mathcal D}_j)$
is a densely decorated punctured circle, as described in \Exref{D.4}
(for justification, see \Exref{X.1.1} and the last paragraph in \Exref{D.4}).
This densely decorated punctured circle is easily seen to be orientable if and only if the flexible factor corresponding to the group $G(v_j)$
is orientable. The numbers of types of orientable and non-orientable peripherals
in $((K_j,\Omega_j),{\mathcal D}_j)$ coincide with the numbers of conjugacy classes
in $G(v_j)$ among the orientable (and non-orientable, respectively) subgroups
of the form $G(u)\cap G(v_j):u\in N^{T^r_G}_{v_j}$.
As we have already noticed before, the information from the last two sentences determines
the decorated space $((K_j,\Omega_j),{\mathcal D}_j)$ uniquely up to an isomorphism
(see \Lemref{D.4.1}(2)(3)).
\end{rmk}

\section{Trees of internally punctured segments}
\seclabel{I}

In this section we describe one more ingredient necessary in our proof of \Thmref{S.1},
as given in \Secref{P}. This ingredient consists of expressing various decorated punctured circles
as trees of appropriate internally punctured segments. The section is split into two subsections
corresponding to two kinds of the described objects.

\noindent
\subsection{Orientable trees of segments}
\sseclabel{I.1}

Fix any positive integer $m\ge1$, and let ${\mathcal T}$
be a set of cardinality $m$, which we will view as the set of types.
For each type $\tau\in{\mathcal T}$, denote by $I^\circ_\tau$ the internally punctured 
segment $I^\circ_{\rm int}$, as described in \Exref{V.2}, equipped with the
structure consisting of the following data:

\begin{enumerate}
\item[(1)] except for the standard peripheral subspaces in $I^\circ_{\rm int}$,
which are all doubletons, and which we call the \emph{ordinary} peripherals of $I^\circ_\tau$,
we consider also an additional subspace, denoted $\partial I^\circ_{\tau}$,
consisting of the two endpoints of $I^\circ_{\rm int}$,
and we call it \emph{the special peripheral} of $I^\circ_{\tau}$;

\item[(2)] identifying naturally $I^\circ_{\rm int}$ with the punctured circle $(S^1)^\circ$,
and fixing one of its orientations (in the sense of the last sentence of \Defnref{X.2}),
we equip all its peripherals (both ordinary and special) with the orientations compatible
with the just fixed orientation of the whole space;

\item[(3)] we associate the type $\tau$ to the special peripheral of $I^\circ_{\tau}$,
and we do not associate any types to its ordinary peripherals.
\end{enumerate}

Define the tree system $\Psi^{\text{o}}_{m}(I^\circ_{\rm int})$ by requiring
the following (the superscript ``o'' on $\Psi^{\text{o}}_{m}$ indicates ``orientable''):
\begin{enumerate}
\item[(a)] its constituent spaces are the copies of the internally punctured segments
$I_\tau^\circ:\tau\in{\mathcal T}$, with ordinary peripherals as the peripherals of these spaces
(this means that the special peripherals are not among the peripherals of those constituent
spaces);
\item[(b)] each gluing in the system involves precisely two constituent spaces
(equivalently, each white vertex of the underlying tree of the system has 
degree 2);
\item[(c)]  gluing maps \emph{respect} the orientations of the glued peripherals,
which means that those maps reverse the orientations of the glued peripherals;
\item[(d)] for each $\tau,\tau'\in{\mathcal T}$ (not necessarily distinct),
and for any constituent space $K_t$ being a copy
of $I^\circ_\tau$,
the union of those peripherals of $K_t$
along which $K_t$ is glued to a copy of $I^\circ_{\tau'}$, is dense in $K_t$.
\end{enumerate}

The fact that such a tree system is unique up to an isomorphism 
follows by discussion and arguments similar to those in \Exref{D.4}, 
\Lemref{D.4.1} and \Lemref{D.7}, but this fact is not essential for our
further considerations.

\begin{lem}
  \lemlabel{I.1.1}
Consider the limit $X$ of the tree system $\Psi^{\text{o}}_m(I^\circ_{\rm int})$
of punctured segments,
and view it as a space with peripherals, where peripherals are the images of all
special peripherals in all constituent punctured segments. Equip also all those
peripherals in a natural way with the types from the set $\mathcal T$,
and with the naturally induced orientations. Then $X$  
is homeomorphic, as a space with peripherals, to an orientable densely
decorated punctured circle, as described in \Exref{D.4}, and it has precisely $m$
types of peripherals.
\end{lem}
\begin{proof}
At first, we will ignore the types and the orientations of the peripherals,
and show that, as a pure space with peripherals, 
the limit of the tree system $\Psi^{\text{o}}_m(I^\circ_{\rm int})$ is homeomorphic to the
punctured circle $(S^1)^\circ$. To do this, consider a related tree system
$\overline\Psi^{\text{o}}_m(I^\circ_{\rm int})$ with the same underlying tree $T$,
in which the constituent spaces $\overline K_t$ are obtained from the corresponding
constituent spaces $K_t$ of $\Psi^{\text{o}}_m(I^\circ_{\rm int})$ 
(which are punctured segments) as follows. Glue to each $K_t$ a segment $I_t$ via
a gluing map which sends the special peripheral of $K_t$ bijectively to the 
boundary $\partial I_t$. Identify the (ordinary) peripherals in $K_t$ as subspaces of
$\overline K_t$, and view them as the peripherals of $\overline K_t$.
For the gluing maps (or more formally, connecting maps) in the new system
$\overline\Psi^{\text{o}}_m(I^\circ_{\rm int})$ take then the same maps which appeared
as connecting maps in the system $\Psi^{\text{o}}_m(I^\circ_{\rm int})$.

As in \Ssecref{inverse_associated}, we associate to the tree system
$\overline\Psi:=\overline\Psi^{\text{o}}_m(I^\circ_{\rm int})$ the
inverse system ${\mathcal S}_{\overline \Psi}$ such that, by
\Propref{T.6},
$\lim_\leftarrow{\mathcal S}_{\overline\Psi}\cong\lim\overline\Psi$.
Now, since all the spaces in the inverse system
${\mathcal S}_{\overline \Psi}$ are clearly homeomorphic to the circle
$S^1$, and since all bonding maps in this system consist of shrinking
finitely many pairwise disjoint closed subsegments of $S^1$ to points,
these maps are obviously \emph{near homeomorphisms}, i.e., they can be
approximated by homeomorphisms.  Since, by a result of M. Brown
(\cite[Theorem~4]{MBrown:inverse_limits:1960}), the inverse limit maps
in an inverse system consisting of near homeomorphisms are near
homeomorphisms, it follows that the inverse limit
$\lim_\leftarrow{\mathcal S}_{\overline\Psi}$ is homeomorphic to
$S^1$, and thus we also have $\lim\overline\Psi\cong S^1$.

The segments $I_t$ from the constituent spaces $\overline K_t$ obviously become
subsegments in the limit circle $\lim\overline\Psi$, and moreover they are easily seen
to be pairwise disjoint. A slightly less clear fact, the straightforward proof of which we
omit, is that the union of the segments $I_t$ (viewed as subsegments in the circle
$\lim\overline\Psi$)  is dense in $\lim\overline\Psi$. On the other hand, the limit
$\lim\Psi$ of the original tree system $\Psi=\Psi^{\text{o}}_m(I^\circ_{\rm int})$ can be identified
with a subspace in the limit $\lim\overline\Psi$, namely the subspace obtained by
deleting from $\lim\overline\Psi$ the interiors of all subsegments $I_t$.
Moreover, the peripherals in $\lim\Psi$ correspond to the pairs of endpoints of the
deleted subsegments $I_t$. Consequently, as a space with peripherals, $\lim\Psi$
is homeomorphic to the punctured circle, as required.

It remains to check the conditions related to the types and orientations of the peripherals.
The condition which requires that the orientations of all peripherals in the limit space
$\lim\Psi$ are consistent (as in condition (b) in \Exref{D.4})
follows directly from condition (c) above in this subsection. The condition which requires
that the union of the peripherals of any fixed type is dense follows easily from
condition (d) above.
This completes the proof.
\end{proof}

\subsection{Non-orientable trees of segments}
\sseclabel{I.2}

Fix any non-negative integers $m,n$ such that $m+n\ge1$, and let 
${\mathcal T}={\mathcal T}_1\sqcup{\mathcal T}_2$
be a set of two kinds of types, with cardinalities $|{\mathcal T}_1|=m$ and $|{\mathcal T}_2|=n$.
For each type $\tau\in{\mathcal T}$, denote by $I^\circ_\tau$ the internally punctured 
segment $I^\circ_{\rm int}$, as described in \Exref{V.2}, equipped with the
structure consisting of the following data:

\begin{enumerate}
\item[(1)] as in the orientable case, we consider the ordinary peripherals and the special
peripheral in $I^\circ_{\rm int}$, and we equip the special peripheral with the type $\tau$
(and we do not associate any types to the ordinary peripherals);

\item[(2)] for each type $\tau\in{\mathcal T}_1$, we equip the special peripheral of 
$I^\circ_\tau$ with one of the orientations, and for each type $\tau\in{\mathcal T}_2$
we associate to the special peripheral of $I^\circ_\tau$ both of its orientations.
\end{enumerate}

Define the tree system $\Psi^{\text{n-o}}_{m,n}(I^\circ_{\rm int})$ by requiring
the following (the superscript ``n-o'' on $\Psi^{\text{n-o}}_{m,n}$ indicates ``non-orientable''):
\begin{enumerate}
\item[(a)] its constituent spaces are the copies of the internally punctured segments
$I_\tau^\circ:\tau\in{\mathcal T}$, with ordinary peripherals as the peripherals of these spaces;
\item[(b)] each gluing in the system involves precisely two constituent spaces
(equivalently, each white vertex of the underlying tree of the system has 
degree 2);

\item[(c)] for each type $\tau\in{\mathcal T}$, and each type $\tau'\in{\mathcal T}_2$,
for any constituent space $K_t$ being a copy of $I^\circ_\tau$,
the union of those peripherals of $K_t$
along which $K_t$ is glued to a copy of $I^\circ_{\tau'}$ is dense in $K_t$;

\item[(d)] for each type $\tau\in{\mathcal T}$, and each type $\tau'\in{\mathcal T}_1$,
any constituent space $K_t$  being a copy
of $I^\circ_\tau$, after fixing an auxiliary orientation, satisfies the following condition:
\begin{itemize}
\item[(\#)]
consider any peripheral $P$ of $K_t$ along which $K_t$ is glued to a copy of $I^\circ_{\tau'}$;
then this copy of $I^\circ_{\tau'}$ inherits an orientation from its special peripheral,
and thus induces one of the orientations on the peripheral $P$; we will call the latter
orientation of $P$ \emph{the induced orientation}; we require that the union
of those peripherals of $K_t$
along which $K_t$ is glued to a copy of $I^\circ_{\tau'}$, and whose induced
orientations are compatible (respectively, incompatible) with the earlier fixed
orientation of $K_t$, is dense in $K_t$. 
\end{itemize}
\end{enumerate}

Similarly as we have mentioned in the previous subsection, the fact that 
a tree system as above is unique up to an isomorphism 
follows by discussion and arguments similar to those in \Exref{D.4}, 
\Lemref{D.4.1} and \Lemref{D.7}. However, this fact is not essential for our
further considerations.

\begin{lem}
  \lemlabel{I.2.1}
Consider the limit $X$ of the tree system $\Psi^{\text{n-o}}_{m,n}(I^\circ_{\rm int})$
of punctured segments,
and view it as a space with peripherals, where peripherals are the images of all
special peripherals in all constituent punctured segments. Equip also all those
peripherals in a natural way with the types from the set $\mathcal T$,
and with the naturally induced orientations
(one orientation if $\tau\in{\mathcal T}_1$, and both orientations if $\tau\in{\mathcal T}_2$).
Then $X$
is homeomorphic, as a space with peripherals, to a non-orientable densely
decorated punctured circle, as described in \Exref{D.4}, and it has precisely $m$
types of orientable peripherals, and $n$ types of non-orientable peripherals.
\end{lem}

The proof of this lemma is similar to the proof of \Lemref{I.1.1}, and we omit it.

\section{Proof of \Thmref{S.1} and of the missing parts of Theorems \mainthmref{1} and \mainthmref{2}}
\seclabel{P}

The major part of this section is devoted to the proof of \Thmref{S.1}.
After doing this, we also provide arguments for getting
(as rather immediate corollaries to \Thmref{S.1}) the proofs of the implications
\pitmref{g_rigid_cluster_factors_vf}$\Rightarrow$\pitmref{bd_g_reg_twocon_tog} in \Mainthmref{1}
and (3)$\Rightarrow$(2) in \Mainthmref{2}.

 The outline of the proof of \Thmref{S.1} 
is as follows. Given a group $G$ as in the assumption of \Thmref{S.1}, we have described two
tree systems canonically associated to $G$. The first one is the tree system $\Theta^r(G)$
associated to the reduced Bowditch JSJ splitting of $G$, as described in \Ssecref{reduced_JSJ},
and it has the property that $\lim\Theta^r(G)\cong\partial G$ (see \Thmref{T.9+}).
The second tree system associated to $G$ is the (regular) tree system of punctured
graphs $\Theta[{\mathcal R}_G]$ associated to the graphical connecting system ${\mathcal R}_G$, 
as described in \Secref{tr_gr} (together with the description of ${\mathcal R}_G$ given in
\Secref{S}),
and as appearing in the statement of \Thmref{S.1}.
In view of the equality $\partial G\cong\lim\Theta^r(G)$,
our goal is to show that $\lim\Theta^r(G)\cong\lim\Theta[{\mathcal R}_G]$.

In order to show the latter equality, we shall describe three more tree systems, denoted 
$\Theta^{V}[{\mathcal R}_G]$, $\Theta^{V}_*[{\mathcal R}_G]$ and 
$\Theta^{V}_*[{\mathcal R}_G]^a$, with the following
properties. The tree system $\Theta^{V}[{\mathcal R}_G]$, called \emph{the tree system of V-trees
associated to $G$}, will be described as obtained from $\Theta[{\mathcal R}_G]$
via a certain consolidation (an operation described in \Ssecref{consolidation_of_tree_sys}, which does not
affect the limit space of the system). The constituent spaces of $\Theta^{V}[{\mathcal R}_G]$
will be the V-trees of graphs given as limits of certain V-tree systems of graphs
naturally appearing inside the system $\Theta[{\mathcal R}_G]$.
The tree system $\Theta^{V}_*[{\mathcal R}_G]$,
called \emph{the tree system of cores and arms associated to $G$}, will be described
as obtained from $\Theta^{V}[{\mathcal R}_G]$ by an operation of splitting (i.e. operation 
opposite to some operation of consolidation) consisting of appropriate splittings of all V-tree
constituent spaces of $\Theta^{V}[{\mathcal R}_G]$ into cores and arms
(as described in \Secref{V}).
Finally, the system $\Theta^V_*[{\mathcal R}_G]^a$, called 
\emph{the tree system of cores and consolidated arms associated to $G$},
will be described as obtained from $\Theta^V_*[{\mathcal R}_G]$ via one more
operation of consolidation, which will join together certain configurations of arms in 
the latter system
into bigger constituent spaces which have the form of punctured circles
(here we will refer to the results of \Secref{I}).

By \Lemref{V.9} and \Thmref{T.7}, we will then have the natural identifications
\[\lim\Theta[{\mathcal R}_G]\cong\lim\Theta^V[{\mathcal R}_G]\cong\lim\Theta^V_*[{\mathcal R}_G]
\cong\lim\Theta^V_*[{\mathcal R}_G]^a.
\]
\noindent
As the last step in the proof of \Thmref{S.1}, we will show that the tree system 
$\Theta^V_*[{\mathcal R}_G]^a$ is isomorphic to the tree system $\Theta^r(G)$, 
thus yielding the desired assertion by referring
to the fact that isomorphic tree systems have homeomorphic limits. 
The existence of an isomorphism between the two just mentioned tree systems
will be deduced from \Lemref{D.7}, by showing that 
the system $\Theta^V_*[{\mathcal R}_G]^a$ is modelled on the
same connecting system (for regular spaces with doubleton peripherals)
as the tree system $\Theta^r(G)$, namely on 
the system ${\mathcal P}^r_G$ described in \Exref{D.8}.

We now pass to the details of the above outlined proof.
We split the exposition into subsections which correspond to the steps of this proof.

\subsection{Description of $\Theta^{V}[{\mathcal R}_G]$  - the tree system of V-trees
associated to $G$}

Recall that the tree system $\Theta[{\mathcal R}_G]$ is a tree system of punctured graphs,
with
\[
\Theta[{\mathcal R}_G]=(T, \{ \Gamma^\circ_t \}, \{ \Sigma_u \}, \{ \varphi_e \}  ),
\]
where each graph $\Gamma_t$ is either the extended Whitehead graph 
associated to some rigid cluster factor of $G$ or some $\theta$-graph.
Recall also that the set of white vertices $u$ of the underlying tree $T$ is partitioned
into two subsets according to whether the corresponding subspaces 
$\varphi_e(\Sigma_u)$
are V-peripherals or E-peripherals
(compare the second paragraph after \Defnref{R.1}).
View the tree $T$ as the first barycentric subdivision 
of another tree $T^\bullet$, so that black vertices of $T$ are the vertices of $T^\bullet$
and the white vertices of $T$ are the barycenters of the edges of $T^\bullet$.
Then, by what was said above about white vertices of $T$, we can speak of
V-edges and E-edges of the tree $T^\bullet$. Note also that, under this interpretation of $T$, 
the b-subtrees of $T$ correspond to arbitrary subtrees of $T^\bullet$.

Call any subtree of the tree $T^\bullet$ whose edge set consists only of V-edges 
a \emph{V-subtree} of $T^\bullet$.
Let $\Pi^V$ be the partition of the tree $T$
(in the sense explained in \Ssecref{consolidation_of_tree_sys})  into b-subtrees
which
correspond to all maximal V-subtrees of $T^\bullet$. 
Take as $\Theta^{V}[{\mathcal R}_G]$ the tree system obtained from $\Theta[{\mathcal R}_G]$
by consolidation with respect to the partition $\Pi^V$, i.e.,
\[
\Theta^{V}[{\mathcal R}_G]:=(\Theta[{\mathcal R}_G])^{\Pi^V}.
\]

We will now describe in more detail the constituent spaces, the peripherals, and 
the connecting maps of the tree system $\Theta^{V}[{\mathcal R}_G]$.
Suppose first that a b-subtree $S\in\Pi^V$ contains a (black) vertex $t$ 
for which $\Gamma_t$ is
the extended Whitehead graph, say $\widetilde W_z$, associated to some 
virtually free rigid cluster factor $G(z)$ of $G$. It follows from the form of the V-involution
in the graphical connecting system ${\mathcal R}_G$ that the restriction of the system
$\Theta[{\mathcal R}_G]$ to $S$ is then isomorphic (as a tree system) to the 
V-tree system $\Theta[\widetilde{\mathcal V}_z]$ of the extended Whitehead graph $\widetilde W_z$. In particular,
for any black vertex $s$ of $S$ the corresponding graph $\Gamma_s$
is isomorphic to $\widetilde W_z$, and the constituent space $K_S$ of
the consolidated system $\Theta^{V}[{\mathcal R}_G]$ 
is homeomorphic (as a space with peripherals) to the V-tree 
${\mathcal X}_V((\widetilde W_z,\tilde {\mathbf a}_z))$ 
(where $\tilde{\mathbf a}_z$ is the V-involution canonically associated to $\widetilde W_z$).

Consider now the remaining case when a b-subtree $S\in\Pi^V$ contains a (black) vertex $t$ 
for which $\Gamma_t$ is some $\theta$-graph $\Xi_k$.
 It follows from the form of the V-involution
in the graphical connecting system ${\mathcal R}_G$ that the restriction of the system
$\Theta[{\mathcal R}_G]$ to $S$ is then isomorphic (as a tree system) to the 
V-tree system $\Theta[{\mathcal V}_{\Xi_k}]$
induced by the connecting V-system ${\mathcal V}_{\Xi_k}$
(i.e. the standard connecting V-system for $\Xi_k$, as described in \Exref{V.8.2}).
In particular, for any black vertex $s$ of $S$ the corresponding graph $\Gamma_s$
is isomorhic to $\Xi_k$, and the constituent space $K_S$ of
the consolidated system $\Theta^{V}[{\mathcal R}_G]$ 
is homeomorphic (as a space with peripherals) to the internally punctured graph
$\Xi_k$, i.e., to the space with peripherals $(\Xi_k)^\circ_{int}$
(as in \Exref{V.8.2}).

Recall again that we have two kinds of peripherals in the initial tree system 
$\Theta[{\mathcal R}_G]$,
namely V-peripherals and E-peripherals.
All V-peripherals disappear in the process of passing, via consolidation, 
to the consolidated tree system $\Theta^{V}[{\mathcal R}_G]$, while all E-peripherals
survive this process. Thus, we can naturally identify the peripherals of the system
$\Theta^{V}[{\mathcal R}_G]$ with the E-peripherals of the system $\Theta[{\mathcal R}_G]$.
Under this identification,
the connecting maps
in the system $\Theta^{V}[{\mathcal R}_G]$ are induced in the obvious way from
those maps in the system $\Theta[{\mathcal R}_G]$ which concern E-peripherals.

\subsection{Description of $\Theta^{V}_*[{\mathcal R}_G]$  - the tree system of cores and arms
associated to $G$}

Recall that if $S\in\Pi^V$ is a subtree such that for its black vertices $t$ the graphs
$\Gamma_t$ are the extended Whitehead graphs, then the constituent space  $K_S$ 
of the system $\Theta^{V}[{\mathcal R}_G]$
has the form of a V-tree of extended Whitehead graphs
${\mathcal X}_V(\widetilde{\mathcal V}_z)={\mathcal X}_V((\widetilde W_z,\tilde{\mathbf a}_z))$, 
and that this V-tree can be expressed as the limit
of the tree system $\Psi_S:=\Psi[\widetilde{\mathcal V}_z]$
in which the constituent spaces are the core and the arms
of ${\mathcal X}_V(\widetilde{\mathcal V}_z)$, see 
\Lemref{E.4}(2) and \Lemref{V.9}.
\Factref{V.6} and \Lemref{E.4}(1) explain then the nature of  
the core and the arms appearing in this expression.
In the remaining case when $S\in\Pi^V$ is a subtree with black vertex graphs 
$\Gamma_t$ isomorphic to a $\theta$-graph $\Xi_k$ (for some $k$
depending only on $S$), the constituent space $K_S$ coincides with
the V-tree of $\theta$-graphs ${\mathcal X}_V({\mathcal V}_{\Xi_k})$, as described in 
\Exref{V.8.2}. In particular, in this case 
the constituent space $K_S$ is isomorphic, as a space with peripherals,
to the internally punctured $\theta$-graph $(\Xi_k)^\circ_{int}$.
Consequently, this space $K_S$ can be expressed as the wedge of ($k$ copies of)
punctured segments, 
$K_S=\lim\Psi_w[{\mathcal V}_{\Xi_k}]$,
as described in \Exref{V.9.1}.
For $S$ as in this case,
we put $\Psi_S=\Psi_w[{\mathcal V}_{\Xi_k}]$, 
and we recall
that this system $\Psi_S$ consists of precisely  one white vertex, $u_S$,
having degree $k$ in the corresponding underlying tree, and of precisely $k$ black
vertices, all adjacent to $u_S$, the constituent spaces at which are punctured segments.
We call these constituent spaces the \emph{wedge arms}, and we call their common peripheral,
which is a doubleton, the \emph{wedge core}.


Since for any $S\in\Pi^V$, each peripheral of $K_S$ is contained in precisely one arm
(or wedge arm),
and disjoint with the core (respectively, wedge core), 
we can naturally view the tree system $\Theta^{V}[{\mathcal R}_G]$
as obtained by consolidation from another tree system, denoted 
$\Theta^{V}_*[{\mathcal R}_G]$, for which the constituent spaces are the cores, the arms
and the wedge arms
in all spaces $K_S:S\in\Pi^V$. 
The above described systems $\Psi_S$ correspond then to the restricted subsystems
of such a new system  $\Theta^{V}_*[{\mathcal R}_G]$ along which one consolidates
the latter
to obtain back the system $\Theta^{V}[{\mathcal R}_G]$.   We skip the straightforward
further details of this construction.

Note that all peripherals of the system $\Theta^V[{\mathcal R}_G]$
(which we naturally identify with the E-peripherals of the initial system $\Theta[{\mathcal R}_G]$)
can be naturally
identified as some of the peripherals in the system $\Theta^{V}_*[{\mathcal R}_G]$.
Thus, we will likewise refer to the peripherals in $\Theta^{V}_*[{\mathcal R}_G]$ 
of the latter form as \emph{E-peripherals}. Obviously, there are also some other peripherals
in $\Theta^{V}_*[{\mathcal R}_G]$, and they correspond precisely to the peripherals
in the cores and to the special peripherals in the arms and wedge arms. We will call the peripherals in $\Theta^{V}_*[{\mathcal R}_G]$  of the latter 
forms \emph{C-peripherals}.

\subsection{U-types, Y-labels and $\#$-orientations of the arms and wedge arms of
$\Theta^V_*[{\mathcal R}_G]$}

Before getting to the further modifications of the tree system $\Theta^V_*[{\mathcal R}_G]$,
we need to describe some natural decorations for arms and wedge arms of this system,
and to record some of their properties.
We start with describing the so called U-types.

Let $U$ be the set 
of $G$-orbits in the set of all those edges of the reduced Bowditch JSJ-tree $T^r_G$
which have a flexible vertex as one of their endpoints.
We will use $U$ as a set of types for arms and wedge arms
in $\Theta^V_*[{\mathcal R}_G]$, and we will say that a type $\varepsilon\in U$ is orientable
(respectively, non-orientable)
if the edge stabilizing group $G(e)$ for any $e\in\varepsilon$ 
is orientable (respectively, non-orientable). 
Now, let $A$ be any arm in the tree system $\Theta^{V}_*[{\mathcal R}_G]$.
Recall that $A$ corresponds to a class (or a sequence of copies) of edges
in one of the extended Whitehead graphs $\Gamma_z=\widetilde W_z$, and 
the set of edges in $\widetilde W_z$ whose copies are involved in $A$
(which we denote by $E_A$)
coincides with the preimage $(\lambda_z^+)^{-1}(x)$ for some (precisely one)
vertex $x\in N_z^2$ (see the notation introduced in the fourth paragraph of \Secref{S}).
The vertex $x$ as above is adjacent to precisely one vertex $u_\xi$  (as described
in the same \Secref{S}), and we put $e=[x,u_\xi]$.
We associate to $A$ (and call the \emph{U-type} of $A$) 
the type $\varepsilon=[e]\in U$, i.e., the $G$-orbit of $e$.
This orbit clearly belongs to $U$, since $x$ is a flexible vertex of $T^r_G$
(which follows from the fact that rigid cluster vertices in the tree $T^r_G$
are isolated, see \Factref{3.4}).
Consider now the remaining case when $A$ is a wedge arm
in $\Theta^{V}_*[{\mathcal R}_G]$, and recall that $A$ is composed of
punctured edges obtained by puncturing precisely one edge $\sigma$
in precisely one $\theta$-graph $\Gamma_z=\Xi_{k(z)}$.
(We then put $E_A:=\{ \sigma \}$.)
Since the edges of $\Gamma_z$ are in a fixed bijective correspondence with the
vertices of $T^r_G$ adjacent to $z$ (via the bijection $\lambda_z^+$
described in the fourth paragraph of \Secref{S}), 
we denote by $e_\sigma=[z,\lambda_z^+(\sigma)]$ this edge joining $z$
to the vertex which corresponds to $\sigma$. We associate to $A$ (and call the \emph{U-type} of A) 
the type $\varepsilon=[e_\sigma]\in U$, i.e., the $G$-orbit of $e_\sigma$.

Accordingly with the above described assignment of U-types,
we call an arm or wedge arm \emph{U-orientable} if its U-type is orientable,
and in the remaining case we call it \emph{U-non-orientable}.

If some U-type $\varepsilon\in U$ is orientable, its \emph{orientation}
is a choice of an orientation for any 2-ended group $G(e)$ with $e\in\varepsilon$
(and such a choice extends, via conjugations in $G$, to the choices of orientations 
for all groups $G(e)$ with $e\in\varepsilon$). We will now explain how the choice
of an orientation of a U-type induces some orientation on any arm or wedge arm
equipped with this U-type.
If $\varepsilon$ is the U-type of some arm $A$, there is precisely one rigid
cluster vertex $z\in Z$
such that $A$ is an arm in a V-tree of extended Whitehead graphs $\Gamma_z$.
Moreover, there is precisely one $x\in N^2_z$ 
such that $A$ is composed of copies of edges from $(\lambda_z^+)^{-1}(x)$,
and if $u_\xi$ denotes the white vertex on the path from $x$ to $z$ then
$[u_\xi,x]\in\varepsilon$.
If we fix an orientation of $\varepsilon$, it induces an orientation of the group
$G([u_\xi,x])$, and this induces an order of the doubleton
$P=\partial G([u_\xi,x])=\partial G(u_\xi)=\partial G([u_\xi,z])\subset \partial G(z)=C_z$,
where $C_z$ is the core of the V-tree ${\mathcal X}_V(\widetilde{\mathcal V}_z)$
(for the last equality, compare \Lemref{E.4}).
Note that $P$ coincides with the set of endpoints of precisely one arm $A'$ in 
the V-tree ${\mathcal X}_V(\widetilde{\mathcal V}_z)$ that is made of punctured copies
of the same edges of $\widetilde W_z$ as the arm $A$.
Moreover, the above order of $P$ uniquely determines the order-induced orientations
of all edges of $\widetilde W_z$ whose copies are involved in $A'$
(compare \Rmkdefnref{V.11}, \Lemref{Y.2a} and \Rmkref{E.5}(3)).
Consequently,
the above order-induced orientations of the edges
uniquely induce an orientation of the arm $A$, and we will call it
\emph{the orientation of $A$ induced from the orientation of $\varepsilon$}.
Obviously, any orientation of $\varepsilon$ uniquely induces orientations of all
peripherals in any arm $A$ of the U-type $\varepsilon$.

Now, suppose that $\varepsilon$ is the U-type of some wedge arm $A$ of the system 
$\Theta^V_*[{\mathcal R}_G]$, and that this type is orientable, with one of the orientations fixed.
Let $z\in Z$ be the (unique) vertex of type (v2) for which $A$ is an arm in a V-tree
of $\theta$-graphs $\Gamma_z=\Xi_{k(z)}$. Then there is precisely one $x\in N_z$
such that $A$ is composed of copies of the edge $(\lambda_z^+)^{-1}(x)$.
The fixed orientation of the U-type $\varepsilon$ induces then an orientation
of the group $G([z,x])$, and this induces an order on the doubleton
$\partial G([z,x])=\partial G(z)$. By our convention from \Secref{S}, we have fixed
an identification of the set $\partial G(z)$ with the vertex set of the
$\theta$-graph $\Gamma_z=\Xi_{k(z)}$, so the above orientation of $\varepsilon$
induces an order on the vertex set of $\Gamma_z$. This in turn induces,
in the obvious way, an orientation of the edge $(\lambda_z^+)^{-1}(x)$
in $\Gamma_z$, and since all punctured copies of this edge appearing inside
the arm $A$ are oriented consistently, this induces also an orientation of $A$.

We now pass to describing the \emph{Y-labels} associated to arms and wedge arms.
To describe them, we need to look closer at the set $\mathcal A$
of E-connections in the graphical connecting system ${\mathcal R}_G$. 
Recall that 
$\mathcal A$ 
is given as the disjoint union of the blocks ${\mathcal A}_y:y\in Y$,
where $Y$ is the set of $G$-orbits in the set of flexible vertices of $T^r_G$.
Each block ${\mathcal A}_y$ is described in terms of some set $E_y$,
where the family $E_y:y\in Y$ is a certain partition of the set of edges of $\Gamma$
(see \Secref{S}). It follows from the descriptions of the blocks ${\mathcal A}_y$
and the sets $E_y$ (in \Secref{S}) that for any arm or wedge arm $A$ the corresponding set 
$E_A$ is contained in precisely one of the sets $E_y$; we then say that $y$ is the \emph{Y-label}
associated to the arm $A$. Observe that any two arms (or wedge arms) with the same U-type also
have the same Y-label, so that the partition of the set of arms according to the U-types is a refinement
of the partition according to the Y-labels.
Moreover, it is not hard to see that for any $y\in Y$ the set
\[U_y:=
\{ \varepsilon\in U: \hbox{ $y$ is the Y-label of any arm or wedge arm having U-type 
$\varepsilon$} \}
\]
coincides with the set of $G$-orbits of all those edges of $T^r_G$ which are adjacent
to any fixed vertex $v\in y$.

Now, consider any label $y\in Y$ such that the flexible factors $G(v):v\in y$
(in the reduced Bowditch JSJ splitting of $G$) are orientable. 
Then to any arm or wedge arm labelled with $y$
we associate the so called \emph{$\#$-orientation} in the following way. 
Recall that in our description of the connecting system ${\mathcal R}_G$
as given in \Secref{S}, we fixed in advance, in a $G$-invariant way,
orientations for all flexible factors $G(v)$ with $v\in y$, and that this choice of orientations
led to the assignment of an orientation 
to any edge $\epsilon\in E_y$, yielding a collection of oriented edges
$E_y^\#$. Note that for any arm (or wedge arm) $A$ labelled with $y$, 
the copies of edges from $E_y$ appearing in this arm,
viewed with the orientations as in $E_y^\#$, are oriented consistently.
Consequently, these orientations induce an orientation of the entire arm $A$.
Accordingly, we will denote by $A^\#$ the arm $A$ equipped with the just described 
orientation, and we will call this orientation the \emph{$\#$-orientation} of $A$.
We will then call any arm or wedge arm $A$ as above \emph{$\#$-oriented}.
The arms or wedge arms not equipped with $\#$-orientations 
(i.e. those having the Y-label $y$ such that the groups $G(v):v\in y$ are
non-orientable, as flexible factors)
will be called
\emph{$\#$-non-orientable}.

The next observation follows directly from the description of the connecting system 
${\mathcal R}_G$, and from the above descriptions of U-types, Y-labels and $\#$-orientations.
To justify part (3), one also needs to refer to the density condition (4)
in \Defnref{6.2}, as applied to the tree system $\Theta[{\mathcal R}_G]$
(see also condition (L3) in the construction from the proof of \Lemref{6.3}).
In part (3) we use the following terminological convention: 
under notation as in the whole of \Factref{P.1},
given a constituent
space $K^*_t$ of $\Theta^V_*[{\mathcal R}_G]$ which is an arm or a wedge arm,
and one of its E-peripherals $P=\varphi^*_{[u,t]}(\Sigma^*_u)$,  
\emph{the peripheral adjacent to $P$} is the peripheral $P'=\varphi^*_{[u,t']}(\Sigma^*_u)$
in the constituent space $K^*_{t'}$, where $t'$ is the unique black vertex of $T_*$
adjacent to $u$ and distinct from $t$.

\begin{fact}
  \factlabel{P.1}
Denote the data of the tree system $\Theta^V_*[{\mathcal R}_G]$ by
$(T_*,\{ K_t^* \}, \{ \Sigma_u^* \}, \{ \varphi_e^* \})$. Let $A=K_t^*$ and 
$\hat A=K_{\hat t}^*$ be some two constituent spaces  
in $\Theta^V_*[{\mathcal R}_G]$, each of which is an arm or a wedge arm, and suppose that there is a 
white vertex $u$ in $T_*$ which is adjacent to both $t$ and $\hat t$.

\begin{enumerate}
\item[(1)] If the Y-label of $A$ is $y$, then the Y-label of $\hat A$ is also $y$. 
 
\item[(2)] If $A$ is $\#$-oriented, then $\hat A$ is also $\#$-oriented.
Moreover, if we induce the $\#$-orien\-tations of $A$ and $\hat A$ to the peripherals
$\varphi^*_{[u,t]}(\Sigma^*_u)\subset A$ and 
$\varphi^*_{[u,\hat t]}(\Sigma^*_u)\subset \hat A$, then the gluing map
$\varphi^*_{[u,\hat t]}(\varphi^*_{[u,t]})^{-1}$ between those peripherals
respects the $\#$-orientations, i.e., it reverses the induced orientations of the peripherals.

\item[(3)] If the Y-label of $A$ is $y$, and if $\varepsilon$ is any U-type from $U_y$,
then:

\begin{enumerate}
\item[(A)] the union of all those E-peripherals $P$ of $A$ for which
  the peripheral $P'$ adjacent to $P$ is contained in an arm or a
  wedge arm of U-type $\varepsilon$, is dense in $A$;
\item[(B)] moreover, if $A$ is $\#$-non-orientable and $\varepsilon$
  is orientable then, after fixing an auxiliary orientation of $A$ and
  some orientation for $\varepsilon$, the union of all those
  E-peripherals $P$ of $A$ for which the peripheral $P'$ adjacent to
  $P$ is contained in an arm or wedge arm $A'$ of U-type
  $\varepsilon$, and for which the orientation induced from the fixed
  orientation of $\varepsilon$ (via the induced orientations of $A'$
  and $P'$) is compatible (respectively, incompatible) with the fixed
  orientation of $A$, is dense in $A$.
\end{enumerate}
\end{enumerate}
\end{fact}

We finish this subsection with another two observations concerning the concepts 
just introduced,
which describe some features of the C-peripherals of the system $\Theta^V_*[{\mathcal R}_G]$.
These observations follow rather easily from Facts \factref{S.2} and \factref{S.3}, respectively,
and we omit their proofs.

\begin{fact}
  \factlabel{P.2}
Under the notation $\Theta^V_*[{\mathcal R}_G]=
(T_*,\{ K_t^* \}, \{ \Sigma_u^* \}, \{ \varphi_e^* \})$,
let $t$ be a vertex of the tree $T_*$ such that $K^*_t$ is a constituent space which is
a copy of the core $C_z$ of the V-tree of extended Whitehead graphs $\widetilde W_z$,
for some rigid cluster vertex $z\in Z$. Let $u$ be any white vertex adjacent to $t$,
and denote by $\xi\in{\mathcal H}_z$ the type of the peripheral $\varphi^*_{[u,t]}(\Sigma^*_u)$ 
(as introduced in \Defnref{E.5}).

\begin{enumerate}
\item[(0)] The degree of the vertex $u$ in $T_*$ is equal to 
$n_\xi$, and there is a natural
bijective correspondence $b:N_u^{T_*}\to N_{u_\xi}^{T^r_G}$ 
between the sets of vertices  adjacent to $u$ and $u_\xi$ in the corresponding trees,
such that $b(t)=z$ and for each arm $K^*_s$ with $s\in N_u^{T_*}\setminus\{ t \}$
the U-type of $K_s^*$ is equal to the $G$-orbit of the edge $[u_\xi,b(s)]$.


\item[(1)] For any $s\in N_u^{T_*}\setminus\{ t \} $,
denoting by $\varepsilon$ the U-type of the arm $K_s^*$, the cardinality of the set
\[
\{ s'\in N_u^{T_*}\setminus\{ t \}: \varepsilon\hbox{ is the U-type of the arm }K^*_{s'} \}
\]
coincides with the index $[G(u_\xi):G([u_\xi,b(s)])]=[G(u_\xi):G(u_\xi)\cap G(b(s))]$.

\item[(2)] Given $s\in N_u^{T_*}\setminus\{ t \} $, consider the subset
$[s]\subset N_u^{T_*}\setminus\{ t \}$ corresponding via the bijection  
$b$ to the $G(u_\xi)$-orbit $[b(s)]$ of the vertex 
$b(s)\in N^{T^r_G}_{u_\xi}\setminus\{z\}$.
(Equivalently, in accordance with the above part (0), $[s]$ is precisely the set of those $s'\in N_u^{T_*}\setminus\{ t \}$ for which
the arm $K_{s'}^*$ has the same U-label as $K^*_s$.)
Suppose also that the group $G(b(s))$ (as well as $G(b(s'))$
for any $s'\in[s]$) is orientable, so that all arms $K^*_{s'}:s'\in[s]$
are $\#$-oriented. Then:

\begin{enumerate}
\item[(A)] if $G(u_\xi)$ is orientable, then the $\#$-orientations
of all arms $K^*_{s'}:s'\in[s]$ are compatible, 
i.e., they all induce 
(via the maps $(\varphi_{[s,u]}^*)^{-1}$) the same
order on the doubleton $\Sigma^*_{u}$;

\item[(B)] if $G(u_\xi)$ is non-orientable, then the cardinality of $[s]$
is even, and the $\#$-orientations of half of the arms $K^*_{s'}:s'\in[s]$
are compatible with one of the orders of the doubleton $\Sigma_u^*$,
while the $\#$-orientations of the other half of these arms
are compatible with the opposite order of $\Sigma^*_{u}$.

\end{enumerate}

\end{enumerate}
\end{fact}

\begin{fact}
  \factlabel{P.3}
Under the notation $\Theta^V_*[{\mathcal R}_G]=
(T_*,\{ K_t^* \}, \{ \Sigma_u^* \}, \{ \varphi_e^* \})$,
let $u$ be a vertex of the tree $T_*$ such that $\Sigma^*_u$ is a peripheral space
of the tree system $\Theta^V_*[{\mathcal R}_G]$ which is a copy of the wedge core
of the V-tree of the $\theta$-graphs $\Gamma_z=\Xi_{k(z)}$, for some $z\in Z$
of type (v2).

\begin{enumerate}
\item[(0)]
The degree of $u$ in $T_*$ coincides with the degree $k(z)$ of $z$ in $T^r_G$,
and there is a natural bijective correspondence $b:N_u^{T_*}\to N_z^{T^r_G}$
such that for each wedge arm $K_s^*$ with $s\in N_u^{T_*}$ 
the U-type of $K_s^*$ is equal to the $G$-orbit of the edge $[z,b(s)]$.


\item[(1)] For any $s\in N_u^{T_*} $,
denoting by $\varepsilon$ the U-type of the arm $K_s^*$, the cardinality of the set
\[
\{ s'\in N_u^{T_*}: \varepsilon\hbox{ is the U-type of the wedge arm }K^*_{s'} \}
\]
coincides with the index $[G(z):G([z,b(s)])]=[G(z):G(z)\cap G(b(s))]$.

\item[(2)] Given $s\in N_u^{T_*}\setminus\{ t \} $, consider the subset
$[s]\subset N_u^{T_*}$ corresponding via the bijection  
$b$ to the $G(z)$-orbit $[b(s)]$ of the vertex 
$b(s)\in N^{T^r_G}_{z}$.
(Equivalently, in accordance with the above part (0), $[s]$ is precisely the set of those $s'\in N_u^{T_*}$ for which
the wedge arm $K_{s'}^*$ has the same U-label as $K^*_s$.)
Suppose also that the group $G(b(s))$ (as well as $G(b(s'))$
for any $s'\in[s]$) is orientable, so that all wedge arms $K^*_{s'}:s'\in[s]$
are $\#$-oriented. Then:

\begin{enumerate}
\item[(A)] if $G(z)$ is orientable, then the $\#$-orientations
of all arms $K^*_{s'}:s'\in[s]$ are compatible, 
i.e., they all induce 
(via the maps $(\varphi_{[s,u]}^*)^{-1}$) the same
order on the doubleton $\Sigma^*_{u}$;

\item[(B)] if $G(z)$ is non-orientable, then the cardinality of $[s]$
is even, and the $\#$-orientations of half of the arms $K^*_{s'}:s'\in[s]$
are compatible with one of the orders of the doubleton $\Sigma_u^*$,
while the $\#$-orientations of the other half of these arms
are compatible with the opposite order of $\Sigma^*_{u}$.

\end{enumerate}

\end{enumerate}
\end{fact}

\subsection{Consolidation $\Theta^{V}_*[{\mathcal R}_G]^a$ of the  system 
$\Theta^{V}_*[{\mathcal R}_G]$ along 
arm-subtrees}

In this step of the proof, we describe some consolidation of the tree system
$\Theta^{V}_*[{\mathcal R}_G]$. The proof of \Thmref{S.1} will be then concluded by showing,
in the next step, that the tree system $\Theta^{V}_*[{\mathcal R}_G]^a$
obtained by this consolidation is isomorphic
to the tree system $\Theta^r(G)$.

Recall that we denote the underlying tree of the system $\Theta^{V}_*[{\mathcal R}_G]$
by $T_*$.
Let $\Pi^a$ be the partition of $T_*$ into b-subtrees described as follows.
The first kind of b-subtrees in $\Pi^a$ are the trivial subtrees $R$
which consist of a single ``core vertex'',
i.e., a vertex $r$ with one of the cores (of a V-tree of extended Whitehead graphs)
as the constituent space $K_r^*$. The b-subtrees in $\Pi^a$ of the second kind are described
as follows.
Cut $T_*$ along all core vertices  (as mentioned above), and along all wedge core vertices
(which are white),
and consider the b-subtrees $R$ spanned on all black vertices in any of
the connected components obtained by cutting $T_*$ in this way.
Note that for such subtrees $R$ all constituent spaces $K_r^*$ 
at the black vertices $r$ in $R$
are arms (in some V-trees of exteneded Whitehead graphs) or wedge arms. 
For this reason the b-subtrees $R\in \Pi^a$ of the latter form will be called  
\emph{arm-subtrees}
of $T_*$.
Define the tree system $\Theta^V_*[{\mathcal R}_G]^a$ as the consolidation of
$\Theta^V_*[{\mathcal R}_G]$ with respect to the partition $\Pi^a$, i.e., as
$\Theta^V_*[{\mathcal R}_G]^a:=(\Theta^V_*[{\mathcal R}_G])^{\Pi^a}$.

An alternative way of describing the arm-subtrees $R$ of $T_*$ is by saying that
they are the maximal b-subtrees of $T_*$ in which all white vertices correspond to
E-peripherals of the system $\Theta^{V}_*[{\mathcal R}_G]$. It becomes clear from this
perspective, that the peripherals of the final tree system $\Theta^{V}_*[{\mathcal R}_G]^a$
correspond precisely to (and are naturally identified with) the C-peripherals of
the system $\Theta^{V}_*[{\mathcal R}_G]$. These C-peripherals can be further split into
peripherals of the following two kinds. 
Peripherals of the first kind are those which appear in the cores (and which are
identified via the connecting maps with the
special peripherals in the arms, which we view as being of the same kind); we will
call them \emph{core-peripherals}. Peripherals of the second kind 
correspond to the wedge cores, and appear as the special peripherals 
in the wedge arms; we will call them \emph{wedge-peripherals}.

We now pass to describing closer those constituent spaces $K_R$ of the system 
$\Theta^{V}_*[{\mathcal R}_G]^a$ which correspond to the arm-subtrees $R$ of $T_*$.
We will show that, after a suitable interpretation of types, the restriction of the system
$\Theta^{V}_*[{\mathcal R}_G]$ to any arm-subtree $R$ is isomorphic
to a tree system of internally punctured segments of one of the kinds described
in \Secref{I}. As a concequence, by referring to \Lemref{I.1.1} or \Lemref{I.2.1},
we will identify each of the constituent spaces $K_R$
corresponding to arm-subtrees with an appropriate densely decorated punctured circle.
More precisely, note first that each arm or wedge arm (viewed as a constituent space in
the system $\Theta^{V}_*[{\mathcal R}_G]$)
contains precisely one C-peripheral. Obviously, we will view this arm (or wedge arm)
as an internally punctured segment, with its C-peripheral viewed as a special peripheral,
and with its E-peripherals viewed as ordinary peripherals. 
We associate to the special peripheral of any arm (or wedge arm) $A$ the type $\tau$
coinciding with the U-type of $A$ (as described in the previous subsection).
Moreover, if the Y-label $y$ associated to $A$ (as in the previous subsection)
is orientable (i.e. if the flexible factor $G(v)$, for any $v\in y$,
is orientable),
we associate to $A$
the orientation which coincides with the $\#$-orientation (as described also in the
previous subsection). Observe that, due to \Factref{P.1}(1), 
for any arm-subtree $R$, all arms and wedge arms
in the restricted tree
system $\Theta^V_*[{\mathcal R}_G]|_R$ have the same Y-label $y$,
which we call therefore the \emph{Y-label associated to $R$}.
It is not hard to observe that the set of U-types of arms and wedge arms
appearing in the tree system $\Theta^V_*[{\mathcal R}_G]|_R$ coincides then
with the set $U_y$ described in the previous subsection.
For any such $R$, we take $U_y$ as the corresponding set $\mathcal T$ of types.
It follows then from \Factref{P.1}(2)(3) 
that for any arm-subtree $R$ the restricted tree system 
$\Theta^V_*[{\mathcal R}_G]|_R$ is isomorphic, as equipped with the above
mentioned types and orientations, to one of the tree systems of internally
punctured segments described in \Secref{I}.
More precisely, if $y$ is the Y-label
associated to $R$, and if the corresponding flexible factor $G(v)$ for
any $v\in y$ is orientable, the restricted tree system $\Theta^V_*[{\mathcal R}_G]|_R$
with the above described types and orientations
is isomorphic to the tree system $\Psi^{\text{o}}_m(I^\circ_{\rm int})$, as described
in \Ssecref{I.1}, where $m$ is the cardinality of the subset $U_y\subset U$.
Similarly,  if the corresponding flexible factor $G(v)$ for any $v\in y$ is non-orientable,
then the restricted tree system $\Theta^V_*[{\mathcal R}_G]|_R$
is isomorphic (via type and orientation respecting isomorphism) with the tree system 
$\Psi^{\text{n-o}}_{m,n}(I^\circ_{\rm int})$, 
as described in \Ssecref{I.2},
where $m$ and $n$ are the cardinalities
of the subsets $U_{y,1},U_{y,2}$ consisting of those types from $U_y$ which are 
orientable and nonorientable, respectively. This completes our analysis
of the constituent spaces $K_R$ corresponding to the arm-subtrees $R$ in the 
consolidated tree system $\Theta^V_*[{\mathcal R}_G]^a$.

\subsection{The final step of the proof of \Thmref{S.1}}

We need to show that the tree system  $\Theta^{V}_*[{\mathcal R}_G]^a$
is isomorphic to the tree system $\Theta^r(G)$.
We will argue by showing that $\Theta^{V}_*[{\mathcal R}_G]^a$ can be naturally viewed as
a regular tree system with doubleton peripherals, as defined in \Secref{D},
and that it is modelled on the same connecting system (for regular spaces
with doubleton peripherals) as the system $\Theta^r(G)$, namely on the connecting
system ${\mathcal P}^r_G$ described in \Exref{D.8}.
The proof will be then concluded by referring to Lemmas \lemref{D.7} and \lemref{T.5.5}.

Let ${\mathcal P}^r_G$ be the connecting system described in \Exref{D.8},
and denote the data which form this system as 
${\mathcal P}^r_G=(((K_j,\Omega_j),({\mathcal D}_j,d_j))_{j\in J},{\mathcal B})$.
In view of what was said in the previous paragraph, to conclude the proof of
\Thmref{S.1} we need to show that the tree system $\Theta^{V}_*[{\mathcal R}_G]^a$
is modelled on ${\mathcal P}^r_G$. To do this, we will explicitly describe some
${\mathcal P}^r_G$-atlas ${\mathbf A}^a$ for $\Theta^{V}_*[{\mathcal R}_G]^a$.
Recall first that we have two kinds of constituent spaces in the tree system
$\Theta^{V}_*[{\mathcal R}_G]^a$. The constituent spaces of the first kind are
(the copies of)
the cores $C_z$ of the V-trees ${\mathcal X}_V(\widetilde{\mathcal V}_z)$
of extended Whitehead graphs $\widetilde W_z$
(these graphs are described in \Secref{S}, as part of the structure of ${\mathcal R}_G$,
and in \Secref{E}). 
By \Lemref{E.4}(1), each such core $C_z$ (and hence each constituent space
of $\Theta^{V}_*[{\mathcal R}_G]^a$ which is its copy)
is canonically homeomorphic to the 
space with peripherals $(X_z,{\mathcal D}_z)$
associated to the corresponding rigid cluster factor $(G(z),{\mathcal H}_z)$ of $G$.
The corresponding rigid cluster vertex $z$ of $T^r_G$ is then in the same $G$-orbit
as precisely one of the vertices $v_j$ mentioned in \Exref{D.8},
and therefore there is $g\in G$ which maps $z$ to $v_j$. This element $g$,
in its action on the boundary $\partial G$, maps the space with peripherals 
$\partial G(z)=X_z$
onto the analogous space with peripherals $\partial G(v_j)=K_j$, and we identify $C_z$
(via $\partial G(z)$ and $\partial G(v_j)$) with the corresponding $K_j$,
using this map (and this identification essentially does not depend on a choice of $g$
as above, since any two such identifications differ by a decoration-preserving
homeomorphism).
Thus, for each constituent space of the first kind we get in this way its identification
with an appropriate space $K_j$, and we take it as the corresponding element
of the atlas ${\mathbf A}^a$.
The constituent spaces of the second kind in the tree system $\Theta^{V}_*[{\mathcal R}_G]^a$
are those spaces $K_R$ which correspond to the arm-subtrees $R\subset T_*$.
By the discussion in the previous subsection, each such $K_R$
is naturally a densely decorated punctured circle, with types in the set $U_y$,
where $y\in Y$ is the Y-label associated to $R$.
Recalling that $y$ is a $G$-orbit in the set of flexible vertices in $T^r_G$,
we note that the decorated space with peripherals $K_j$
with this index $j\in J$ for which $v_j\in y$, is also a densely decorated 
punctured circle.
Recalling further  that
$U_y$ is the set of $G$-orbits of the edges
of $T^r_G$ which are adjacent to any fixed vertex $v\in y$,
we deduce (by referring to \Lemref{D.4.1}(2)(3))
that $K_R$ is then isomorphic to  $K_j$. Moreover,
by referring to \Lemref{D.4.1}(2)(3),
we can (and will) identify $K_R$ with the above mentioned $K_j$,
in an essentially unique way,
via a decoration-preserving homeomorphism which sends each peripheral 
of any type $\varepsilon\in U_y$ in $K_R$ to a peripheral of type $\tau(\partial G(u'))$ in $K_j$,
where $u'$ is any vertex adjacent to $v_j$ such that the edge $[u',v_j]$ 
belongs to the $G$-orbit $\varepsilon$. 
Thus, for each constituent space of the second kind in the tree system 
$\Theta^{V}_*[{\mathcal R}_G]^a$,
this gives us an identification with an appropriate space $K_j$ of ${\mathcal P}^r_G$,
and we take this identification as the corresponding element of the atlas ${\mathbf A}^a$. 

It remains to show that the above described family ${\mathbf A}^a$ of identifications 
satisfies conditions (1)-(4)
of \Defnref{D.6} (as applied to the connecting system ${\mathcal P}^r_G$).
Fix the notation 
$(T^a,\{ K^a_{\tilde t} \},\{ \Sigma^a_{\tilde u} \}, 
\{ \varphi^a_{\tilde e}=\varphi^a_{[\tilde u, \tilde t]} \})$
for the data in $\Theta^V_*[{\mathcal R}_G]^a$.
We pass to checking that condition (1) of the \Defnref{D.6} holds.
Let $\Sigma^a_{\tilde u}$ be some peripheral space in the tree system
$\Theta^{V}_*[{\mathcal R}_G]^a$. We need to show that the set of types 
(induced by the above identifications) of the corresponding peripheral subspaces 
$\varphi^a_{[\tilde u,\tilde t]}(\Sigma^a_{\tilde u})$, where $\tilde t$ runs through
all vertices adjacent to $\tilde u$ in the tree $T^a$, coincides with some single
block of $\mathcal B$. To do this, we consider first the case when $\Sigma^a_{\tilde u}$ is
a core-peripheral (as explained in the previous subsection). 
Denote by $\tilde v$ the (unique) core vertex of $T^a$ adjacent to $\tilde u$,
where by a core vertex we mean here a vertex for which the constituent space
$K^a_{\tilde v}$ is a copy of some core $C_z=\partial G(z)$.
View then the peripheral 
$P=\varphi^a_{[\tilde u,\tilde v]}(\Sigma^a_{\tilde u})$ of $K^a_{\tilde v}$
as a peripheral in the space $\partial G(z)$ (viewed as a space with peripherals
associated to the rigid cluster factor $G(z)$), and denote by $u$ the white vertex of 
$T^r_G$ for which $P=\partial G(u)$. Denoting by $[u]\in A$ the $G$-orbit
of the vertex $u$
(where $A$ is as in \Exref{D.8}), 
we claim that the set of types of the peripherals of the form
$\varphi^a_{[\tilde t,\tilde u]}(\Sigma^a_{\tilde u})$ coincides with the block
$B_{[u]}$ of $\mathcal B$. 
To see this, note that if we interpret the types of the peripherals in ${\mathcal P}^r_G$
as $G$-orbits of edges, as in \Obsref{D.8+}, then the above described
identifications from the atlas ${\mathbf A}^a$ simply preserve those types.
Moreover, since the peripherals in the tree system $\Theta^V_*[{\mathcal R}_G]^a$
correspond to the C-peripherals of the system $\Theta^V_*[{\mathcal R}_G]$,
we may view the observations from Facts \factref{P.2} and \factref{P.3}, after appropriate
reinterpretation, as concerning the
tree system $\Theta^V_*[{\mathcal R}_G]^a$.
Our claim 
follows then fairly directly from \Factref{P.2}(0).
 We pass now to considering the
remaining case when $\Sigma^a_{\tilde u}$ is a wedge-peripheral,
corresponding to some wedge core.
This wedge core is contained in a V-tree of copies of precisely one $\theta$-graph
$\Gamma_z$ from the connecting system ${\mathcal R}_G$, and the vertex $z$ 
is white in $T^r_G$ and not adjacent to any rigid cluster vertex.
We  denote by $[z]\in A$ the $G$-orbit of the vertex $z$,
and we observe (in a way similar as in the previous case) that,
in view of \Factref{P.3}(0), the set of types 
of the peripherals of the form 
$\varphi^a_{[\tilde t,\tilde u]}(\Sigma^a_{\tilde u})$ coincides with the block
$B_{[z]}$ of $\mathcal B$.
This completes the verification of condition (1).

In order to verify the remaining conditions (2)-(4) of \Defnref{D.6},
we again interpret Facts \factref{P.2} and \factref{P.3} as concerning the tree system 
$\Theta^V_*[{\mathcal R}_G]^a$, as in the previous paragraph.
In view of this, condition (2) of \Defnref{D.6} follows directly from Facts \factref{P.2}(1)(2A)
and \factref{P.3}(1)(2A), condition (3) follows from Facts \factref{P.2}(1) and \factref{P.3}(1),
while condition (4) is a consequence of Facts \factref{P.2}(1)(2B) and \factref{P.3}(1)(2B).
We omit further details, thus completing the proof of the fact that the
tree system $\Theta^V_*[{\mathcal R}_G]^a$ is modelled on the connecting system 
${\mathcal P}^r_G$. In view of the comments at the beginning of this section
and at the beginning of this subsection,
this completes also the whole proof of \Thmref{S.1}.


\subsection{Proofs of the missing parts of Theorems \mainthmref{1} and \mainthmref{2}}

We complete the proofs of Theorems \mainthmref{1} and \mainthmref{2} of the introduction. To do this,
we only need to justify the implications \pitmref{g_rigid_cluster_factors_vf}$\Rightarrow$\pitmref{bd_g_reg_twocon_tog} in \Mainthmref{1}
and (3)$\Rightarrow$(2) in \Mainthmref{2}.

\noindent
\begin{proof}[Proof of the implication \pitmref{g_rigid_cluster_factors_vf}$\Rightarrow$\pitmref{bd_g_reg_twocon_tog} in \Mainthmref{1}]

It obviously follows from \Thmref{S.1} that, assuming condition \pitmref{g_rigid_cluster_factors_vf} in \Mainthmref{1},
the Gromov boundary $\partial G$ is homeomorphic to a regular tree of graphs.
It is thus sufficient to show that the graphs $\Gamma_z$ appearing in the
graphical connecting system ${\mathcal R}_G$ are all 2-connected. To do this,
note that if $\Gamma_z$ is the extended Whitehead graph of a rigid cluster factor
$G(z)$ of $G$ (i.e. if $z$ is a rigid cluster vertex of $T^r_G$), then 
$\Gamma_z$ is 2-connected by
\Factref{E.2}. On the other hand, if $z$ is a white vertex of $T^r_G$ not adjacent
to a rigid cluster vertex, $\Gamma_z$ is a $\theta$-graph with $k(z)$ edges,
where $k(z)$ is the degree of the vertex $z$ in the tree $T^r_G$.
Since $k(z)\ge3$, this $\theta$-graph $\Gamma_z$ is thick and 2-connected. It follows that
$\partial G$ is indeed homeomorphic to a regular tree of 2-connected graphs,
as asserted in condition \pitmref{bd_g_reg_twocon_tog} of \Mainthmref{1}, which completes the proof. 
\end{proof}

\begin{proof}[Proof of the part (3)$\Rightarrow$(2) in \Mainthmref{2}]
Suppose that condition (3) of \Mainthmref{2} holds, i.e., $G$ has no rigid factor
in its Bowditch JSJ splitting. Obviously, $G$ has then no rigid cluster factor
in its reduced Bowditch JSJ splitting. As a consequence, all graphs $\Gamma_z$
in the graphical connecting system ${\mathcal R}_G$ are some thick $\theta$-graphs.
It follows then from \Thmref{S.1} that $\partial G$ is a regular tree of 
$\theta$-graphs, as asserted in condition (2) of \Mainthmref{2},
which completes the proof.
\end{proof}

\section{Examples}
\seclabel{examples}

We start this section with the explanations concerning the examples provided in the
introduction (in Figures~1--4). In the later part of the section,
in \Ssecref{ex_cox}, we discuss some Coxeter group examples.

\subsection{Examples from the introduction.}
\sseclabel{ex_intro}

In order to show that the examples from the introduction are correct,
we need to justify that the described splittings of the corresponding groups $G_i$
are indeed their Bowditch (canonical) JSJ splittings. To do this, we will need
several auxilliary results.

The following result characterizes the Bowditch JSJ splittings of 
1-ended hyperbolic groups in terms of the features of their associated graph of groups
decompositions (see Theorem~3.22 in \cite{Dahmani_Touikan:isomorphy_dehn_fillings:2019}, where a more general case
of relatively hyperbolic groups is addressed). Recall that a splitting 
$\mathcal{G}$ of some group $G$,
as a graph of groups, is {\it nontrivial} if no vertex group
of $\mathcal{G}$ coincides with $G$.
A splitting of $G$ is {\it relative to a family $P$ of subgroups of $G$} 
if each subgroup from $P$ conjugates in $G$ into some vertex group of 
$\mathcal{G}$.

\begin{thm}[F. Dahmani, N. Touikan \cite{Dahmani_Touikan:isomorphy_dehn_fillings:2019}]
\thmlabel{EX.1.1}
Let $G$ be a 1-ended hyperbolic group which is not cocompact Fuchsian.
The canonical JSJ splitting (i.e. the Bowditch JSJ splitting) of $G$
is the unique splitting over 2-ended groups that satisfies the following condition:

\noindent {The underlying graph $X$ of the associated graph of groups 
$\mathcal{G}=(\{ G_v \}, \{ G_e \}, \{ \varphi_e \})$
is bipartite, with black and white vertices, so that}
\begin{enumerate}
\item {the white vertex groups $G_u$ are maximal 2-ended subgroups of $G$;}
\item {the black vertex groups $G_t$ are non-elementary 
(i.e. infinite and not 2-ended); 
moreover, the images of the incident edge groups in each such $G_t$
are pairwise distinct, and denoting by 
$P_t$ the family of all such images, the pair $(G_t,P_t)$
 falls into one of the following two categories}
\begin{enumerate} 
\item {$(G_t,P_t)$ is an abstract flexible factor 
(as in \Defnref{X.1}),}
\item {$(G_t,P_t)$ is \hbox{\rm algebraically rigid} 
(i.e. there is no 
nontrivial
splitting of $G_t$ relative to $P_t$ along finite or 2-ended subgroups) 
and $(G_t,P_t)$ is not an abstract flexible factor;
moreover, we require in this case that the images in $G_t$ of any two incident edge groups
are not conjugated in $G_t$ into the same 2-ended subgroup of $G_t$;}
\end{enumerate}

\item {for any white vertex $u$ of degree 1 in $X$, the incident edge group $G_e$
does not coincide with $G_u$ (i.e. it is a proper subgroup of $G_u$);
moreover, if the neighbouring vertex of $u$ in $X$ carries a flexible factor,
then the index of $G_e$ in $G_u$ is at least 3;}
\item {for any white vertex $u$ of degree 2 in $X$, 
if both neighbours of $u$
carry flexible factors then 
at least one of the incident edge groups $G_e$ is a proper
subgroup of $G_u$.} 
\end{enumerate}
\end{thm}

The next result, which seems to be well known, 
allows us to recognize the rigid factors of the groups $G_i$ from the introduction. For completeness, we include a sketch of a proof of this result,
since we haven't found an appropriate reference in the literature.

\begin{prop}
\proplabel{EX.1.2}
Let $S$ be a compact surface, possibly with boundary, equipped with a hyperbolic metric (so that the boundary is totally geodesic), and let $L$ be
a finite collection of closed geodesics contained in the interior of $S$.
Suppose that the family $L$ is \hbox{\rm filling}, i.e., any connected component
obtained by cutting $S$ along the union of $L$ is either a disk or an annulus,
one of whose boundary components coincides with some boundary component of $S$.
Let $G=\pi_1S$, and let $P$ be the family of (the conjugacy classes of) the cyclic
subgroups of $G$ corresponding to the curves in $L$ 
and to the boundary components of $S$.
Then $(G,P)$ is algebraically
rigid, i.e., $G$ has no nontrivial splitting relative to $P$ over a finite or 2-ended subgroup.
\end{prop}

In order to justify the above proposition, we will make use of the
following lemma, which is essentially a restatement of Theorem~III.2.6
in \cite{Morgan_Shalen:valuations_trees:1984} (see also Lemmas 2.2 and
2.3 in \cite{Dunwoody_Sageev:jsj_slender:1999}). Recall that a
splitting of a group $G$ as a graph of groups is {\it minimal} if the
corresponding action of $G$ on the associated Bass-Serre tree $T$ is
minimal, i.e., there is no $G$-invariant proper subtree
$T_0\subset T$, $T_0\ne T$. Recall also that a simple closed curve on
a compact surface is {\it essential} if it is homotopically
non-trivial and not parallel to a boundary component.

\begin{lem}
\lemlabel{EX.1.3}
Let $G=\pi_1S$ be the fundamental group of a compact connected 
hyperbolic surface $S$
(we allow that $S$ has nonempty totally geodesic boundary), 
and let $Q$ be the family of
the conjugacy classes of cyclic subgroups of $G$ corresponding to the
boundary components of $S$. If $G$ splits 
non-trivially and minimally relative to $Q$ as a graph of groups
$\mathcal{G}$ with cyclic edge groups, then there is a collection $\Lambda$
of pairwise disjoint essential simple closed curves contained in the interior of $S$ such that
every edge group of $\mathcal{G}$ is conjugate to the $\pi_1$-image of some
component of $\Lambda$ (in particular, every edge group is infinite cyclic), 
and every vertex group is conjugate to the $\pi_1$-image
of some connected component of the complement $S\setminus\bigcup\Lambda$. 
Moreover,
the components of $\Lambda$  can be realized as geodesics,
except for the families of mutually parallel components, which can be realized as
arbitrarily close to a geodesic.
\end{lem}

\begin{rmk}
\rmklabel{EX.1.3+}
Note that the above lemma implies in particular a rather well known but useful
fact that the fundamental group of a compact hyperbolic surface is freely
indecomposable relative to its boundary subgroups, i.e., that the pair $(G,Q)$
as in the lemma is freely indecomposable.
\end{rmk}

\begin{proof}[Proof of \Propref{EX.1.2}:]
Suppose, a contrario, that $G$ admits a notrivial splitting $\mathcal{G}$
(as a graph of groups) relative to $P$ over finite or 2-ended subgroups.
We can assume, without loss of generality, that this splitting is minimal.
Since $G$ is torsion-free, this splitting is over infinite cyclic or trivial subgroups.
Let $Q$ denote the set of the conjugacy classes of those cyclic subgroups of $G$
which correspond to the boundary components of $S$. 
Since we have $Q\subset P$, $\mathcal{G}$ is obviously 
a splitting relative to $Q$.
By \Lemref{EX.1.3}, 
we can realize this splitting geometrically, as induced by 
some collection $\Lambda$ of pairwise disjoint essential simple closed curves on $S$.
This shows that the splitting is actually over infinite cyclic subgroups only
(compare \Rmkref{EX.1.3+}).
Furthermore, we may assume that the curves of $\Lambda$ are either geodesic,
or very close to a geodesic. 
Since $\mathcal{G}$ is a relative
splitting (with respect to $P$), an easy geometric observation shows that 
one can slightly perturb $\Lambda$ so that
each curve $C$ of $\Lambda$
is contained in a single connected component of the complement $S\setminus\bigcup L$.
By the fact that $L$ is a filling collection, each connected component of
$S\setminus\bigcup L$ is
either a disk or an annulus adjacent to a boundary component of $S$.
Obviously, being an essential simple closed curve,
$C$ cannot be contained in a disk component since it is homotopically
nontrivial in $S$, but it also cannot be contained in an annulus component
since it is not parallel to a boundary component.
It follows that the family $\Lambda$ is empty. However, if this is the case,
the induced splitting is obviously trivial, which contradicts our assumption
that the splitting was nontrivial. This completes the proof.
\end{proof}

We are now ready to discuss and explain the examples from the introduction.
Note first that for each graph of groups $\mathcal{G}_i$, $i=1,2,3,4$,
the underlying graph $X_i$ carries a natural bipartite structure,
with white vertices corresponding to the curves of the gluings,
and with black vertices corresponding to the surface pieces 
($F_j$ and $H_j$) of $M_i$.
It is also easy to realize that each of the white vertex groups is a maximal cyclic
(and hence also maximal 2-ended) subgroup of the corresponding torsion free group $G_i$. 
On the other hand, each of the black vertex groups is the fundamental group of
a hyperbolic surface with nonempty boundary, so it is a non abelian free group,
and hence it is non-elementary.
Moreover, each of the black vertex groups of the form $\pi_1(F_j)$
is easily seen to be an abstract flexible factor (when equipped with the
family of images of the incident edge groups). We will show that any black vertex
group of the form $G_t=\pi_1(H_j)$, equipped with the corresponding family $P_t$
of images of incident edge groups, satisfies condition (2)(b) of 
\Thmref{EX.1.1}.
The fact that the pair $(G_t,P_t)$ as above is algebraically rigid follows directly
from \Propref{EX.1.2}. 
To see that $(G_t,P_t)$ is not an abstract flexible factor,
note that any flexible abstract factor $(G,P)$ fails to be still flexible after deleting
from $P$ some of the subgroups. On the other hand, if we delete from $P_t$
those subgroups of $G_t=\pi_1H_j$ which correspond to the gluing curves contained
in the interior of $H_j$, the resulting new abstract factor $(G_t,Q_t)$ will be obviously
flexible.  
It remains to check the last assertion of condition (2)(b),
namely that no two incident edge groups from $P_t$ conjugate into the same
2-ended subgroup of $G_t$. This follows however from the fact that these
subgroups (as well as any pair of their conjugates), 
when viewed as groups of isometries of a hyperbolic plane,
are generated by loxodromic isometries having distinct axes,
so their intersection is trivial.

Since the verification of the remaining conditions (3) and (4) 
of \Thmref{EX.1.1}
is straightforward, we omit its details, and we conclude that the splittings
of the groups $G_i$ related to the graph of groups expressions 
$\mathcal{G}_i$
are indeed the Bowditch JSJ splittings.

\medskip
We now pass to a more detailed description of the Gromov boundary in one
of the above examples, namely the Gromov boundary $\partial G_2$. 
We will do this by explicitly describing the graphical connecting system 
$\mathcal{R}_{G_2}$, as introduced in \Secref{S}, and as appearing in the statement of
\Thmref{S.1}. This will do the job, since this graphical connecting system uniquely determines
the corresponding regular tree of graphs $\mathcal{X}(\mathcal{R}_{G_2})$,
and by \Thmref{S.1} we have 
$\partial G_2\cong \mathcal{X}(\mathcal{R}_{G_2})$.

Observe that the set $Z$ of orbit representatives (as in \Secref{S}) 
consists of a single vertex, $z_0$,
which corresponds to the rigid cluster factor $\pi_1H_0$ 
in the splitting $\mathcal{G}_2$. 
This is so since, as it is easy to realize, the Bass-Serre tree 
$T^r_{G_2}$ 
corresponding to $\mathcal{G}_2$
contains no white vertices which are not adjacent to a rigid cluster vertex.
The graph $\Gamma$ in $\mathcal{R}_{G_2}$ coincides then with the graph 
$\Gamma_{z_0}$, and the latter is isomorphic to the extended Whitehead graph
$\widetilde W$ 
of the corresponding rigid cluster factor $\pi_1H_0$. It is also not hard to realize
that the degree in $T^r_{G_2}$ of any (white) vertex adjacent to $z_0$ equals 2,
and hence the extended Whitehead graph $\widetilde W$
coincides with the ordinary Whitehead graph $W$ of this factor.

To get the form of the Whitehead graph $W$, observe that for an appropriately chosen
freely generating set $\{  a,b \}$ in $\pi_1H_0$, the peripheral 
subgroups of $\pi_1H_0$
are generated by the elements $a$, $b$ and $aba^{-1}b^{-1}$, respectively.
As a consequence, the Whitehead graph $W$ has the form of a square with both diagonals
(i.e. $W=K(4)$), where the four boundary edges correspond to the generator
$aba^{-1}b^{-1}$. 
The peripheral subgroups are obviously orientable,
and if we choose their orientations corresponding to the above choices of the generators, the order induced orientations of the edges of $W$ are arranged so that the edges
on the boundary of the square are oriented consistently with one of their cyclic order.
The associated V-involution $\bf a$ is described by transpositions
of the pairs of opposite vertices in the square, and the corresponding maps
$\alpha_v:{\rm Lk}_v\to{\rm Lk}_{a(v)}$ are given by the assignments
of points corresponding to the diagonals to the relevant points corresponding 
to the diagonals, and by the assignments of points corresponding to the remaining edges
in the unique possible way respecting the orientations. 

Let $\varepsilon_1,\varepsilon_2,\varepsilon_3, \varepsilon_4$ be the edges of $W$
on the boundary of the square, and let $\delta_1, \delta_2$ be the diagonals,
with the above mentioned orientations. 
Then the set $\mathcal{A}$ of E-connections
is easily seen to have the following form:
$$
\mathcal{A}=\{ (\delta_1,\delta_1) \}\cup\{ (\delta_2,\delta_2) \}\cup
(\{ \varepsilon_1,\varepsilon_2,\varepsilon_3, \varepsilon_4 \}\times
\{ \varepsilon_1,\varepsilon_2,\varepsilon_3, \varepsilon_4 \}).
$$
 
By the assertion of \Thmref{S.1}, the Gromov boundary $\partial G_2$
is homemorphic to the regular tree of graphs $K(4)$ associated to the above
described graphical connecting system $\mathcal{R}_{G_2}$. However,
due to the large symmetry of the graph $K(4)$, any two tree systems
of punctured graphs $K(4)$ are in fact isomorphic, and consequently their limits 
are homeomorphic,
so we can just say that
$\partial G_2$ is homeomorphic to the (unique) tree of graphs $K(4)$. 

This finishes our discussion of the examples from the introduction.

\subsection{More examples - right angled Coxeter groups
with 1-dimensional nerves.}
\sseclabel{ex_cox}

In this subsection we illustrate 
the results of the present paper by means of various
examples of (hyperbolic) Coxeter groups, or more precisely 
{\it right angled} Coxeter groups
(which we shortly call RACG's).
We assume that the reader is familiar with basic terminology and facts related
to Coxeter groups, and we refer to the book \cite{Davis:coxeter_groups:2008} for
an introduction.

\medskip
In an earlier paper \cite{Swiatkowski:refl_trees_of_graphs:2021} the second author of the present paper has 
shown that for a large class of 1-ended
hyperbolic RACG's their Gromov boundaries are trees of graphs. This class
of RACG's can be described in terms of their defining graphs (or nerves) $N$ as follows:

\medskip
$N$ is a triangulation of a topological graph $X$ such that
\begin{itemize}
\item $X$ is connected and has no cut vertex, and it is distinct from the circle $S^1$;
\item each (essential) edge of $X$, when viewed as a subcomplex of 
$N$, consists of at least 3 edges of $N$.
\end{itemize}

\noindent
We denote the class of graphs $N$ satisfying the above two conditions 
by $\mathcal{C}_{JS}$,
and the corresponding RACG's associated to such graphs by $W_N$. 
The argument in  \cite{Swiatkowski:refl_trees_of_graphs:2021} for showing that Gromov boundaries
$\partial W_N$ of such groups are trees of graphs is different than
the one given in the present paper. The trees of graphs appearing as boundaries 
$\partial W_N$ for $N\in\mathcal{C}_{JS}$ belong to a subclass consisting
of so called {\it reflection trees of graphs},  as mentioned in \Exref{R.4}
in this paper. Strictly speaking, it is shown in \cite{Swiatkowski:refl_trees_of_graphs:2021} that if 
$N\in\mathcal{C}_{JS}$
is a triangulation of a topological graph $X$ then 
$\partial W_N\cong \mathcal{X}^r(X)$
(the reflection tree of graphs $X$).

\medskip
A wider class of examples has been studied in the paper \cite{Dani_Thomas:jsj_coxeter:2017}
by Pallavi Dani and Anne Thomas. That paper contains a complete description
of the Bowditch JSJ splittings for a family of 1-ended hyperbolic RACG's
whose defining graphs $N$ satisfy the following set
of conditions (which appear in \cite{Dani_Thomas:jsj_coxeter:2017} under the name Standing Assumptions~1.1):
\begin{itemize}
\item $N$ has no embedded cycles consisting of less than 5 edges,
\item $N$ is connected and has no cut vertex and no cut edge,
\item $N$ is not a cycle.
\end{itemize}

\noindent
We denote the class of graphs $N$ satisfying the above conditions by
$\mathcal{C}_{DT}$. The class is strictly larger than $\mathcal{C}_{JS}$.
In the light of the results in \cite{Dani_Thomas:jsj_coxeter:2017}, and using Theorems 
\mainthmref{1} and \mainthmref{2} 
of the present paper (or rather
their more precise version given as \Thmref{S.1}), 
one can indicate many more
hyperbolic RACG's $W_N$ with the boundary $\partial W_N$ recognized as
an explicitly given tree of graphs 
(see Examples \exref{EX.2.1} and \exref{EX.2.2} below,
Figures 5 and 6). 
One can also indicate examples with boundaries $\partial W_N$ of topological
dimension 1 and not homeomorphic to any tree of graphs, even the ones 
for which all rigid factors are virtually free
(see e.g. an example presented in \Figref{fig7} below).

\medskip
We now describe and discuss the promised examples
from the class $\mathcal{C}_{DT}$. 
In all of them, as presented in Figures~5--7,
we indicate their graph of groups decompositions related to the
Bowditch JSJ splittings. 
These
decompositions are obtained by applying
the results of \cite{Dani_Thomas:jsj_coxeter:2017}, and we do not provide the details, referring the interested
readers directly to that paper. 
The underlying graphs of the indicated
graph of groups decompositions in all of these examples are
the quotients of the corresponding Bass-Serre trees 
by the actions of the groups, so one can easily deduce various properties
of the corresponding Bowditch JSJ splittings (as well as of their reduced
variants).

\begin{ex}
\exlabel{EX.2.1} 
Let $G=W_N$  be the RACG for 
the defining graph $N$ as presented
in the left part of \Figref{fig5}. The right part of the same figure presents
the graph of groups decomposition $\mathcal{G}$ of $G$ related to the Bowditch JSJ
splitting of this group. Vertices represented by circles are flexible, i.e., of type (v1),
while those represented by little squares are of type (v2). 
The symbols inside circles and squares, as well as those which accompany the edges, 
are the generators
of the corresponding special subgroups of $W_N$
appearing as vertex groups and edge groups of $\mathcal{G}$, respectively.

\begin{figure}[ht]
    \centering
   
        \includegraphics[width=\textwidth]{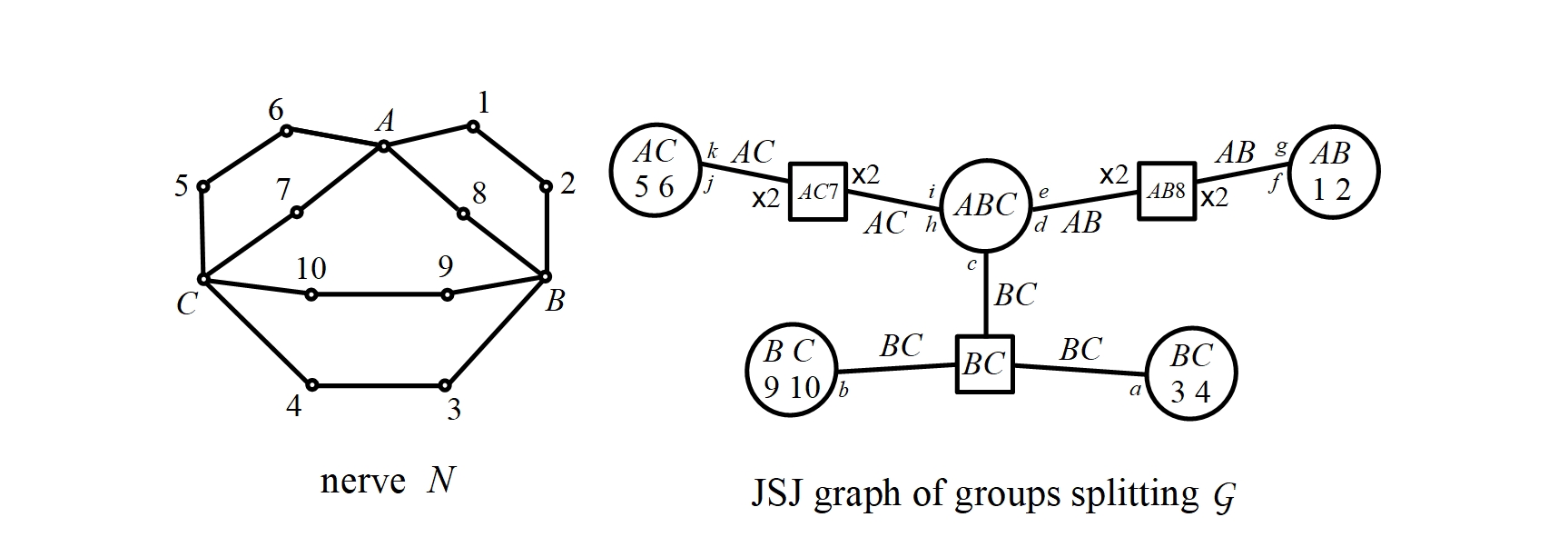}
        \caption{}
        \figlabel{fig5}
\end{figure}

There are no rigid vertices in $\mathcal{G}$, 
so there are also no
rigid vertices in the corresponding Bass-Serre tree $T_G$.
It follows then from \Mainthmref{2} that the boundary $\partial G$ is 
a regular tree of $\theta$-graphs. 
Note that the defining graph $N$ clearly does not belong to the class 
$\mathcal{C}_{JS}$, so this examples goes beyond the class 
studied in an earlier mentioned 
paper \cite{Swiatkowski:refl_trees_of_graphs:2021}. 
(Obviously, it is easy to show many more examples of this sort,
and another one,
whose JSJ splitting contains a rigid factor,
is described below in Example \exref{EX.2.2}, see the group $G_1$ presented 
in \Figref{fig6}.)
To get the explicit form of the regular tree
of $\theta$-graphs homeomorphic to $\partial G$, i.e., to describe the corresponding graphical connecting system
$\mathcal{R}_G$ as in \Secref{S}, we need a closer analysis.

One deduces from the form of the graph of groups $\mathcal{G}$ that
preimages in $T_G$ of the vertices $\langle A,B,8\rangle$ 
and $\langle A,C,7\rangle$ have degree 4,
so that for any such preimage vertex $v$ there are precisely two edges issuing from $v$ that
are mapped to each edge issuing from  $\langle A,B,8\rangle$ or $\langle A,C,7\rangle$
(this is indicated in \Figref{fig5} by little symbols ``$\times2$''). Preimages in
$T_G$ of the vertex $\langle B,C\rangle$ have degree 3. Each of the 2-ended groups
at vertices of type (v2) is non-orientable, because it
is either an infinite dihedral group, or the product of an infinite dihedral group
with a group of order 2. It follows that all flexible factors of the Bowditch
JSJ splitting of $G$ are also non-orientable.

According to our exposition in \Secref{S}, the graphical connecting system
$\mathcal{R}_G=(\Gamma,{\bf a},\mathcal{A})$ is described as follows.
First, $\Gamma$ is the disjoint union of two copies of the $\theta$-graph $\Xi_4$
(corresponding to vertices $\langle A,B,8\rangle$ and $\langle A,C,7\rangle$), 
which we denote $\Gamma_2$
and $\Gamma_3$, respectively, and one copy of the $\theta$-graph $\Xi_3$
(corresponding to the vertex $\langle B,C\rangle$), which we denote $\Gamma_1$.
The $V$-involution ${\bf a}$ of $\mathcal{R}_G$ is just the union of the standard
$V$-involutions (as described in \Exref{V.8.2})
for the $\theta$-graphs $\Gamma_1$, $\Gamma_2$ and $\Gamma_3$.

To describe the set $\mathcal{A}$ of $E$-connections for $\Gamma$,
denote by $a,b,c$ the edges of $\Gamma_1$, by $d,e,f,g$ the edges of
$\Gamma_2$, and by $h,i,j,k$ the edges of $\Gamma_3$,
and for each of those edges fix one of its orientations (so that these letters
actually denote oriented edges). 
According to the description in \Secref{S}, for each $\Gamma_i$
we fix a bijection between the set of edges of $\Gamma_i$ and the set
of vertices in $T_G$ adjacent to the vertex $z\in Z$ representing this component.
We have indicated our choices of these bijections by assigning letters $a,\dots,k$
to flexible vertices of the graph in the right part of \Figref{fig5}. 
Any set $Y$ of representatives of flexible
vertices in $T_G$ can be canonically identified with the set
of vertices represented by circles
in \Figref{fig5}. It follows that the sets $E_y:y\in Y$ coincide with the sets
$\{ a \}$, $\{ b \}$, $\{ c,d,e,h,i \}$, $\{ f,g \}$ and $\{ j,k \}$.
Due to non-orientability of all flexible factors, the set of $E$-connections
has then the form 
$$
\mathcal{A}=\bigcup_{y\in Y}E_y^\pm\times E_y^\pm,
$$
with $E_y:y\in Y$ as mentioned above. 
\end{ex}

\begin{ex}
\exlabel{EX.2.2}
Let $G_1=W_{N_1}$ and $G_2=W_{N_2}$ be the 1-ended hyperbolic 
RACG's for the defining graphs $N_i$ as presented
in the left parts of Figures 6 and 7, respectively. The central parts of these figures present
the graph of groups decompositions $\mathcal{G}_i$ of these groups related
to their Bowditch JSJ splittings. These decompositions have been obtained by referring to
the results in \cite{Dani_Thomas:jsj_coxeter:2017}. Both underlying graphs $X_i$ of $\mathcal{G}_i$ contain
a single vertex related to a rigid factor, namely the vertex
represented by a hexagonal symbol, which corresponds to the special subgroup
$H=\langle A,B,C,D\rangle<G_i$ in each of the two cases.
It is also not hard to deduce from the form of $\mathcal{G}_i$'s that the corresponding
rigid vertices in the Bass-Serre trees $T_{G_i}$ are not isolated, i.e they have
some rigid neighbours at combinatorial distance 2. Thus, in both cases we have 
nontrivial rigid clusters, and we have rigid cluster factors which contain
rigid factors as proper subgroups. 

\begin{figure}[ht]
    \centering
   
        \includegraphics[width=\textwidth]{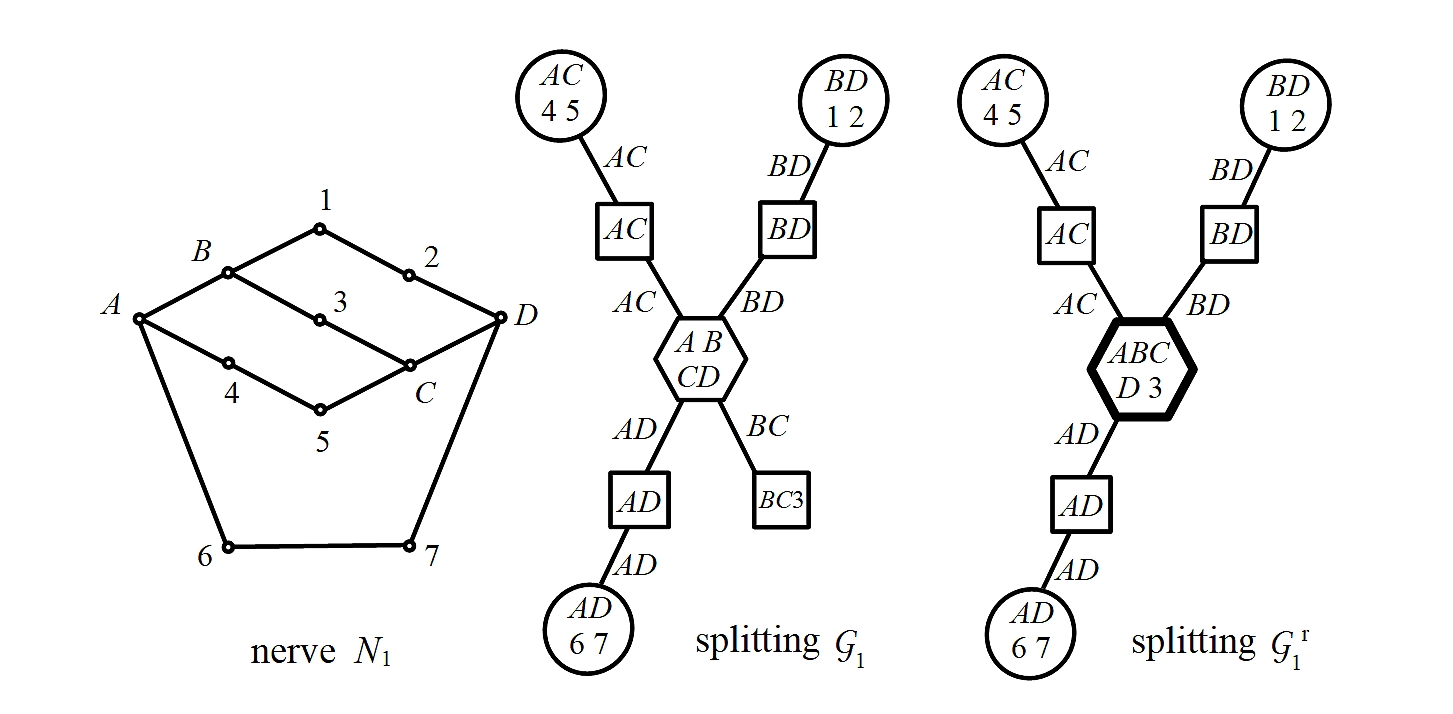}
        \caption{}
        \figlabel{fig6}
\end{figure}

\begin{figure}[ht]
    \centering
   
        \includegraphics[width=\textwidth]{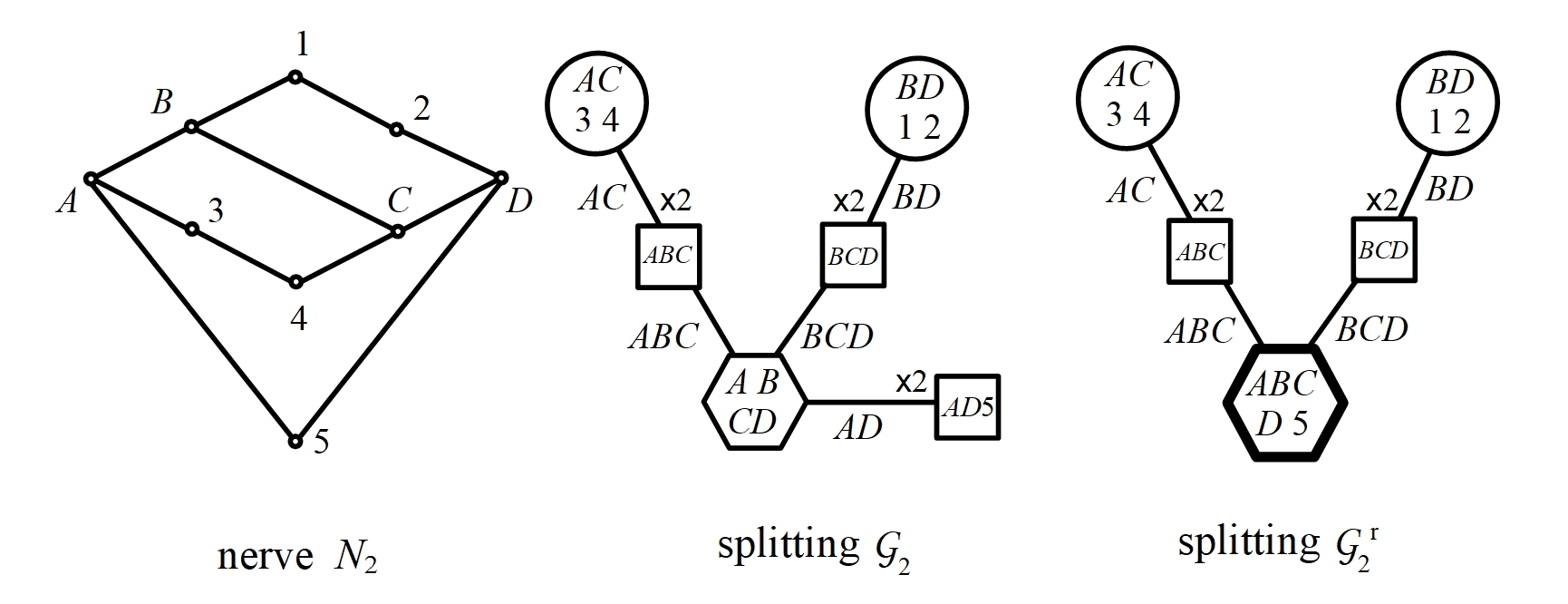}
        \caption{}
        \figlabel{fig7}
\end{figure}

Consider the graph of groups decompositions $\mathcal{G}^r_i$ of the groups $G_i$
related to their reduced Bowditch JSJ splittings. It is not hard to realize
that they have forms as presented in the right parts of Figures 6 and 7.
The rigid cluster factors correspond to the vertices indicated by the symbols
of  the bold hexagons. The rigid cluster factors in the group $G_1$ are the
conjugates of the special subgroup $H_1=\langle A,B,C,D,3\rangle$,
and it is not hard to see that they are virtually free. It follows then from
\Mainthmref{1} that the Gromov boundary $\partial G_1$ is a regular tree of
2-connected graphs. The explicit form of this tree of graphs,
i.e., the corresponding graphical connecting system $\mathcal{R}_{G_1}$,
can be extracted from the data in $\mathcal{G}^r_1$, and the least obvious,
though still starightforward
part of this extraction is the determination of the Whitehead graph associated
to the rigid cluster factor $H_1$. We skip the details.

On the other hand, the rigid cluster factors in the group $G_2$ are the conjugates
of the special subgroup $H_2=\langle A,B,C,D,5 \rangle$ (see \Figref{fig7}).
Since the defining graph of this subgroup is a 5-cycle,
it is a cocompact Fuchsian group (it can be realized e.g. as the group
generated by reflections in the sides of a right angled hyperbolic pentagon).
The Gromov boundary of this subgroup is a circle, so the subgroup is not virtually free,
even though the rigid factors from which it is formed are virtually free.
It follows then from \Mainthmref{1}
that the Gromov boundary $\partial G_2$
is not homeomorphic to any tree of graphs.

\end{ex}

\section{Final remarks and comments}
\seclabel{final}

\begin{rmk}\rmklabel{splitting_generalization}
It is worth noting that the observations and results from 
sections 9-13 above (concerning an explicit expression of the Gromov boundary $\partial G$ as a regular tree of graphs) are still valid in more general situations than those
corresponding  to Bowditch JSJ
or reduced Bowditch JSJ splittings.
We now state a few such generalizations.

\begin{enumerate}
\item
Let ${\mathcal{G}}=(X, \{ G_t \}\sqcup \{ G_u \}, \{ G_e \})$
be a graph of groups such that:
\begin{enumerate}
\item the underlying graph $X$ of $\mathcal{G}$ is bipartite,
and its vertex set is accordingly bipartitioned
into black vertices $t$ and white vertices $u$;
\item the white vertex groups $G_u$ of $\mathcal{G}$ are 2-ended,
and the edge groups $G_e$ of $\mathcal{G}$ are 2-ended as well;
\item each black vertex group $G_t$ of $\mathcal{G}$ is 
non-elementary
word hyperbolic, and its adjacent edge subgroups $G_e$ are
maximal 2-ended and pairwise non-conjugate in $G_t$
(so that the pair $(G_t, \{ \hbox{adjacent $G_e$'s} \})$ is an abstract
factor, in the sense of \Defnref{F.1}).
\end{enumerate}

\noindent
Under these assumptions,
 the fundamental group $G:=\pi_1({\mathcal{G}})$ is hyperbolic
(by Bestvi\-na-Feighn combination theorem), and
the vertex groups of $\mathcal{G}$ are quasi-convex in $G$
(e.g. by Theorem~1.2 in [Kapo]).
Consider the action of $G$ on the Bass-Serre tree $T_{\mathcal{G}}$
of $\mathcal{G}$. The vertex set of $T_{\mathcal{G}}$ is then equipped with the 
induced bipartition into black and white vertices. 
By our assumptions on $\mathcal{G}$, no two distinct white vertices
of $T_{\mathcal{G}}$ are fixed (for the above action) by an infinite
subgroup of $G$, and hence the action (as well as $\mathcal{G}$ itself)
is acylindrical.
If we associate to the splitting $\mathcal{G}$ of $G$ the tree system
$\Theta^{\mathcal{G}}(G)$, in the analogous way as in \Secref{reduced_JSJ} for
Bowditch and reduced Bowditch splittings, 
the arguments as in the proof of \Thmref{T.9}
(based on the discussion in Section~2.6.2 of [CaMa])
still yield that 
$\partial G\cong\lim\Theta^{\mathcal{G}}(G)$. 
In particular, this holds even if $G$ is not 1-ended.

\item 
Now, in addition to the assumptions of (a)-(c) above, suppose that
each black vertex factor $\mathcal{F}_t=(G_t, \hbox{adjacent $G_e$'s})$
which is not an abstract flexible factor (as in \Defnref{X.1})
is virtually free, and isolated in $\mathcal{G}$ (which means that
all black vertex factors $G_{t'}$ corresponding to the vertices of $X$
lying at distance 2 from $t$ are flexible).
Suppose also that to any such factor $\mathcal{F}_t$
we have associated some Whitehead graph $W_t$
(for some choice of a finite index free subgroup $F<G_t$,
and some choice of a free basis $B$ of $F$),
which is not necessarily minimal, but which is essential
(i.e. has no isolated vertex).
Then we can associate to $\mathcal{G}$ an appropriate graphical connecting
system ${\mathcal R}_{\mathcal{G}}$ (in the same way as
the system ${\mathcal R}_G$ associated in \Secref{S} to
the reduced Bowditch splitting of $G$).
By the same arguments as in \Secref{P}, we get that
$\partial G\cong\lim\Theta[\mathcal{R}_{\mathcal{G}}]$.
This means in particular that Gromov boundary $\partial G$
is an appropriate regular tree of extended Whitehead graphs $\widetilde W_t$
of non-flexible factors of $\mathcal{G}$, and of $\theta$-graphs
corresponding to the white vertices in $X$ having no non-flexible
black neighbours. Note however that the above mentioned $\theta$-graphs
in $\mathcal{R}_{\mathcal{G}}$
may have a form $\Xi_k$ with arbitrary $k\ge1$
(and not only with $k\ge3$, as this is the case for the systems $\mathcal{R}_G$
described in \Secref{S}). (We remind that 
this number $k$ is equal to the degree
of any vertex $\tilde u$ of the Bass-Serre tree of $\mathcal{G}$ in the
preimage of the corresponding vertex $u$ of $X$.) 
Note also
that in this generality 
the above mentioned extended Whitehead graphs in $\mathcal{R}_{\mathcal{G}}$
needn't be 2-connected,
nor even connected. As a consequence,
the boundary $\partial G$ needn't be connected as well
(because $G$ needn't be 1-ended).
(This can be compared with our discussions in \Propref{M.2}, \Claimref{9.4},
and in their proofs; see also the comment in part (3) below,
and in \Exref{15.2} below.)

\item 
Finally, in addition to the assumptions of (1) and (2),
suppose that each black vertex factor $G_t$ is {\it weakly rigid},
which means that $G_t$ has no nontrivial splitting along finite subgroups
relative to its incident edge groups $G_e$.
(This assumption is automatically fulfilled by the factors $G_t$
which are flexible, so it is sufficient to demand that all
non-flexible factors of $\mathcal{G}$ are weakly rigid.) 
Suppose also that the corresponding action of $G$ on the 
associated Bass-Serre tree is minimal, 
which is equivalent to saying that for each white vertex $u$ of $X$,
either the degree of $u$ in $X$ is at least 2, or this degree is 1
and the image in $G_u$ of the unique incident edge group $G_e$
is a proper subgroup. 
Then $G$ is 1-ended (e.g. by Corollary~1.5 in [Toui]). 
Note that under the above minimality assumption for the $G$-action
on the Bass-Serre tree, the $\theta$-graphs appearing in $\mathcal{R}_{\mathcal{G}}$
have the forms $\Xi_k$ with $k\ge2$, so they are all 2-connected.
Moreover,
if we choose the Whitehead graphs $W_t$ mentioned above
in part (2) to be minimal, then they are 2-connected as well
(by the argument as in the proof of \Propref{M.2}).
(Such a choice is then analogous to the one declared 
in our previous considerations in \Secref{abs_fac}, as explained in \Rmkref{8.24}.)
Under such a choice of the Whitehead graphs 
(and of their extended variants), the above mentioned equality
$\partial G\cong\lim\Theta[\mathcal{R}_{\mathcal{G}}]$ expresses
the boundary $\partial G$ as an explicit regular tree of 2-connected graphs.
\end{enumerate}
\end{rmk}

\begin{ex}\exlabel{15.2}
We present an example which demonstrates that the discussion
of \Rmkref{splitting_generalization} allows us to easily deduce that boundaries of various
hyperbolic groups have an explicit form of a regular tree of graphs,
without any need to first understand their canonical JSJ decompositions.

Consider the graph of groups $\mathcal{G}$ whose underlying graph $X$
is the barycentric subdivision of a single edge, so that the original vertices
$t,s$ of this edge are black, and the barycenter vertex $u$ is white.
Let $G_t$  be the free group $F_2$ of rank 2, with its standard generators
denoted $x,y$, and let $G_s$  be another free group viewed as
the fundamental group of a compact hyperbolic surface $S$ 
with boundary $\partial S$ consisting
of a single component. Let $G_u$  be the infinite cyclic group,
with fixed generator $a$, and let the edge groups
$G_{[u,t]}$, $G_{[u,s]}$ coincide with $G_u$.
Let the incidence homomorphism $G_{[u,s]}\to G_s$
map the generator $a$ of $G_{[u,s]}$ to some generator
of the subgroup $\pi_1\partial S<\pi_1S=G_s$.
Finally, let the incidence homomorphism $G_{[u,t]}\to G_t$
map the generator $a$ of $G_{[u,t]}$ to the element of $G_t=F_2$
represented by some nontrivial cyclically reduced word 
$w=w(x,y)$ which is not a proper power
of any other word. The fundamental group $G=\pi_1{\mathcal{G}}$ of
this graph of groups, which we will focus on, can be then also viewed
as the fundamental group of the space obtained by gluing $S$ 
along its boundary $\partial S$, via an appropriate map corresponding
to the word $w$, to the wedge of two circles.

Note that any graph of groups $\mathcal{G}$ as above satisfies conditions
(a)-(c) of part (1) of \Rmkref{splitting_generalization}, so that $G=\pi_1{\mathcal{G}}$
is hyperbolic.  
The abstract factor $G_s$ is easily seen to be flexible,
so if we view the other factor $G_t$ as non-flexible
(which we can do even if it is actually flexible), it will be isolated in $\mathcal{G}$.
Moreover, if the word $w$ involves
both generators $x,y$, the Whitehead graph $W_t$ for the basis
$B=\{ x,y \}$ is essential, and the assumtions of part (2) of \Rmkref{splitting_generalization}
are fulfilled. 
The extended Whitehead graph $\widetilde W_t$ coincides then with $W_t$,
and the corresponding graphical connecting system ${\mathcal R}_{\mathcal{G}}$
involves no $\theta$-graphs (because the unique white vertex $u$ of $X$
is adjacent to a non-flexible vertex $t$). By the arguments as in
the proof of \Thmref{S.1} (given in \Secref{P}) we get that the Gromov
boundary $\partial G$ is then a regular tree of the (copies of the) Whitehead
graphs $W_t$.

Finally, if the word $w$ as above is {\it Whitehead-minimal}
(in the sense of the Whitehead algorithm, as described e.g. in Section~3
of [CaMac]),
then the factor $G_t$ is weakly rigid, 
the group $G$ is 1-ended,
and the Whitehead graph $W_t$
is 2-connected, so that $\partial G$ is a regular tree of 2-connected
graphs $W_t$.
\end{ex}

\begin{rmk}\rmklabel{15.3}
  After completing the main part of this paper, we have realized that
  groups $G$ satisfying condition \pitmref{g_rigid_cluster_factors_vf}
  of \Mainthmref{1} (and condition \pitmref{g_no_rigid_factors} of
  \Mainthmref{2}) are actually residually finite, and hence virtually
  torsion free. This has been pointed to us by Marco Linton, and the
  justification goes via virtual specialness of all such $G$,
  guaranteed by the fact that they obviously have quasiconvex virtual
  hierarchies $\mathcal{QVH}$ (compare Theorem~13.1 in [Wise] and
  Theorem~1.3 in [HagW]).  As a consequence, each group $G$ as above
  contains a torsion free subgroup $G'$ of finite index. The
  restricted action of $G'$ on the Bass-Serre tree $T^r_G$ represents
  then the reduced Bowditch JSJ splitting of $G'$, which we denote by
  ${\mathcal{G}}'$.  (The latter follows either by referring to
  Dahmani-Touikan characterization of the canonical JSJ splittings, as
  recalled in \Thmref{EX.1.1} of the present paper, or by using the
  Bowditch approach to canonical JSJ splittings, as presented in
  \cite{Bowditch:cut_points:1998}, where the splitting is expressed in
  terms of configurations of local cut points in the Gromov boundary,
  which is the same for $G$ and $G'$.)  The factors of the reduced
  Bowditch JSJ splitting ${\mathcal{G}}'$ of $G'$ are obviously either
  free or infinite cyclic.  Applying our approach from Sections~6--13
  above to this splitting (for proving the implication
  (3)$\Rightarrow$(1) in \Mainthmref{1}) we would avoid some of the
  technicalities (e.g. the idea of the Whitehead graph for a virtually
  free factor, or many details related to potential non-orientability
  of some peripheral subgroups in the factors).  Since obviously
  $\partial G=\partial G'$, this would lead to a slightly simpler
  proof of the implications (3)$\Rightarrow$(1) in Theorems A and B.
  Some further simplifications could be obtained by passing (via
  residual finiteness of $G$) to an even smaller finite index subgroup
  $G''<G'<G$, for which additionally all flexible factors of the
  corresponding induced splitting ${\mathcal{G}}''$ would be
  orientable (in the sense of \Defnref{X.2}).  However, since the
  passages from $G$ to $G'$ (or to $G''$) and from the reduced
  Bowditch JSJ splitting $\mathcal{G}$ of $G$ to the corresponding
  splitting ${\mathcal{G}}'$ of $G'$, or to ${\mathcal{G}}''$ of
  $G''$, (and, as a consequence, the passages from the connecting
  systems $\mathcal{R}_G$ to $\mathcal{R}_{G'}$ or to
  $\mathcal{R}_{G''}$) do not seem amenable to easy explicit
  calculation, we have decided to leave this part of the paper (and
  the current form of \Thmref{S.1} and its proof) in the current form.
  We believe that the possibility of directly expressing $\partial G$
  as an explicit regular tree of graphs (in the form
  $\partial G\cong\mathcal{X}(\mathcal{R}_G)$, as in \Thmref{S.1} and
  \Rmkref{splitting_generalization}) justifies our decision.
\end{rmk}

\begin{rmk}
As we have noted in \Rmkref{15.3}, we can always pass from a group $G$ satisfying
condition (3) of \Mainthmref{1} to its finite index subgroup $G''$ which is 
torsion free and for which all flexible factors of the corresponding
reduced Bowditch JSJ splitting ${\mathcal{G}}''$ are orientable.
This allows us, by referring to \Thmref{S.1}, to express the Gromov boundary
$\partial G=\partial G''$ of any such $G$ 
as a regular tree of graphs $\mathcal{X}(\mathcal{R}_{G''})$ of a more restricted form than the corresponding tree of graphs $\mathcal{X}(\mathcal{R}_{G})$. 
More precisely, 
in the corresponding graphical connecting systems $\mathcal{R}_{G''}$
the set $\mathcal{A}$ of $E$-connections has a simpler form,
and corresponds to connected sum operations of the involved graphs
which always respect the distinguished 
orientations $e^\#$ of the edges in the graphs
(such orientations are in this case assigned to all edges of the involved graphs, 
see \Secref{S}). This observation suggests the following question.
\end{rmk}

\begin{que}
Does there exist a regular tree of graphs which cannot be realized 
as the Gromov boundary of a hyperbolic group?
\end{que}

\begin{rmk}\rmklabel{15.6}

  We briefly comment on how the ideas and results of this paper
  contribute to the problem of classifying (up to homeomorphism)
  Gromov boundaries of topological dimension 1.

\begin{enumerate}

\item
Recall that on the opposite extreme to trees of graphs among
1-dimensional Gromov boundaries $\partial G$ we have those boundaries
corresponding to 1-ended groups $G$ that do not split over
2-ended subgroups. In these cases the boundary has no local cut points
(cf. \cite{Bowditch:cut_points:1998}), and it is known to be homeomorphic either to the Sierpi\'nski
curve or to the Menger curve (cf. \cite{Kapovich_Kleiner:low_dim_boundary:2000}).
We call the corresponding groups $G$ {\it carpet groups}
and {\it Menger groups}, respectively.

\item
Observe also that a 1-ended hyperbolic group $G$ which is not 
cocompact Fuchsian has 1-dimensional Gromov boundary $\partial G$
if and only if the boundaries $\partial G_v$ of all rigid factors $G_v$ 
(in the Bowditch JSJ splitting of $G$) satisfy $\dim(\partial G_v)\le1$.
To justify this, note that in view of \Thmref{T.9} (which allows us to
express $\partial G$ as the limit of the tree system $\Theta(G)$ composed of
boundaries of the factors in the Bowditch JSJ splitting of $G$), and in view
of the fact that peripheral spaces in the tree system $\Theta(G)$ are 
all doubletons, this observation follows fairly directly from Proposition 2.D.5 in \cite{Swiatkowski:trees_metric_compacta:2020}.

\item
In view of the above observation (2), Theorem A of the present paper
shows that trees of graphs are
the simplest 1-dimensional Gromov boundaries,
and that there is a large variety of 1-dimensional Gromov boundaries that 
are not trees of graphs. 
It is in particular worth observing that
the latter Gromov boundaries include,
among others, a vast class of examples in which the corresponding
groups $G$ have no carpet factors or Menger factors 
(or factors containing such groups as subgroups) in their JSJ splittings.
For instance, it is possible that (as in the example $G_4$ in the introduction)
a group has all rigid factors virtually free, but some of its rigid cluster factors
are not virtually free. Furthermore, there are obviously groups $G$ with 
some rigid factors having boundaries of topological dimension 1 and
not coinciding with a carpet group ar a Menger group (or a group containing such a group as subgroup).
It is possible e.g. that a rigid factor $G_v$ of $G$
is itself a group with $\partial G_v$ homeomorphic to some tree
of graphs.
This shows that the zoo of 1-dimensional connected Gromov boundaries
$\partial G$ is
quite rich and consists of quite complex species and we believe it deserves
further study from the topological perspective.

\end{enumerate}

\end{rmk}

\bibliographystyle{alpha}
\bibliography{nima,hypbdry}
\end{document}